\documentclass{article}

\usepackage{amsmath, amssymb, amsfonts}
\usepackage{graphicx, overpic}
\usepackage{multirow, tabulary}
\usepackage{color}
\usepackage{enumitem}  
\usepackage[normalem]{ulem}
\usepackage[colorlinks=true, pdfstartview=FitV, linkcolor=blue, citecolor=blue, urlcolor=blue]{hyperref}
\hypersetup{
  pdftitle={Fractional distortion in hyperbolic groups},
  pdfauthor={Pallavi Dani and Timothy Riley}
}
\usepackage{array}

\usepackage{amsthm}
\usepackage{xpatch}
\xpatchcmd{\proof}{\itshape}{\prooflabelfont}{}{}
\newcommand{\prooflabelfont}{\bfseries}

\newcommand{\tc}[2]{\textcolor{#1}{#2}}
\definecolor{cerulean}{rgb}{0,.48,.65} 
\definecolor{magenta}{rgb}{.5,0,.5} 
\definecolor{dred}{rgb}{.5,0,0} \newcommand{\dred}[1]{\tc{dred}{#1}}
\definecolor{green}{rgb}{0,.5,0} \newcommand{\green}[1]{\tc{green}{#1}}
\definecolor{blue}{rgb}{0,0,0.5} \newcommand{\blue}[1]{\tc{blue}{#1}}
\definecolor{black}{rgb}{0,0,0} 
\definecolor{dgreen}{rgb}{0,.3,0} 
\definecolor{vdred}{rgb}{.3,0,0} 
\definecolor{red}{rgb}{1,0,0} \newcommand{\red}[1]{\tc{red}{#1}}
\definecolor{salmon}{rgb}{0.98,0.50,0.45} 
\definecolor{gray}{rgb}{.5,.5,.5} \newcommand{\gray}[1]{\tc{gray}{#1}}
\definecolor{seagreen}{rgb}{0.13,0.70,0.67} 
\definecolor{chartreuse}{rgb}{0.40,0.80,0.00}
\definecolor{cornflower}{rgb}{0.39,0.58,0.93} 
\definecolor{gold}{rgb}{0.60,0.48,0.00}
\definecolor{orange}{rgb}{1,0.25,0.05} \newcommand{\orange}[1]{\tc{orange}{#1}}
\definecolor{bcyan}{rgb}{0.15,0.59,0.75} \newcommand{\bcyan}[1]{\tc{bcyan}{#1}}

\newtheorem{mainthm}{Theorem}

\newtheorem{maincor}[mainthm]{Corollary}

\newtheorem{thm}{Theorem}[section]
\newtheorem{lemma}[thm]{Lemma}
\newtheorem{prop}[thm]{Proposition}
\newtheorem{cor}[thm]{Corollary}
\newtheorem{corollary}[thm]{Corollary}

\newtheorem{definition}[thm]{Definition}

\newtheorem{remark}[thm]{Remark}

\newcommand{\R}{\mathbb{R}}
\newcommand{\N}{\mathbb{N}}

\newcommand{\Xb}{X_{\ast}}
\newcommand{\Yb}{Y_{\ast}}

\def\ms{\medskip}

\newcommand{\bs}{\bigskip}
\newcommand{\set}[1]{\left\{#1\right\}}
\newcommand{\ssm}{\smallsetminus}
\newcommand{\abs}[1]{\left|#1\right|}
\def\Z{\mathbb{Z}}
\def\N{\mathbb{N}}

\def\CAT{\hbox{\rm CAT}}

\def\Dist{\hbox{\rm Dist}}

\newcommand{\TB}[2]{\mbox{\tiny{$\left(\!\!\!\begin{array}{c}{#1} \\
{#2}\end{array}\!\!\!\right)$}}}
\newcommand{\SB}[2]{\mbox{\footnotesize{$\left(\!\!\!\begin{array}{c}{#1} \\
{#2}\end{array}\!\!\!\right)$}}}

\newcommand{\HNN}{\mathop{\scalebox{1.5}{\raisebox{-0.2ex}{$\ast$}}}}

\def\onto{{\kern3pt\to\kern-8pt\to\kern3pt}}

\newcommand{\hide}[1]{}

\title{Fractional distortion in hyperbolic groups}
\author{Pallavi Dani 
\thanks{This work of the first author was supported by a grant from the Simons Foundation (\#426932, P.~D.) and by NSF Grant Numbers 1812061 and 2407104. A large part of this work was conducted in 2016--17, and the first author thanks the Simons Laufer Mathematical Sciences Institute (formerly MSRI) for its hospitality during the Semester Program on Geometric Group Theory (2016), as well as Cornell University Department of Mathematics and the Association for Women in Mathematics for the opportunity to visit Cornell University as a Michler Fellow in 2017.}
\and Timothy Riley 
\thanks{This work of the second author was supported by a grant from the Simons Foundation (\#318301, T.~R.).
The second author is grateful for the hospitality of Cambridge University's DPMMS 2019--20. 
}
}
\date{}
\begin{document}
\maketitle


\begin{abstract}
For all integers $p>q>0$ and $k >0$, and all non-elementary torsion-free hyperbolic groups $H$, we construct a hyperbolic group $G$ in which $H$ is a subgroup, such that
the distortion function of $H$ in $G$  
 grows like  $\exp^k({n^{p/q}})$.  Here, $\exp^k$ denotes the $k$-fold-iterated exponential function. 
\end{abstract}

\setcounter{tocdepth}{2}
\tableofcontents

\section{Introduction} 
\label{ch:intro}

\subsection{Our results}\label{sec:intro}

The landscape of subgroups of hyperbolic groups is not well understood.  Whether all one-ended hyperbolic groups have surface subgroups is a celebrated open question.  What functions are Dehn functions of subgroups of hyperbolic groups is widely open.  This article addresses another fundamental issue: What distortion can subgroups of hyperbolic groups exhibit?      Indeed, in his 1998 survey  \cite{Mitra2} Mitra (now known as Mj) asked: ``Given any increasing function $f: \N \to \N$, does there exist a hyperbolic subgroup $H$ of a hyperbolic group $G$ such that the distortion of $H$ is of the order of $\exp (f(n))$.'' 

Let $\exp^k$ denote the $k$-fold iterated exponential function $\N  \to \R$ defined by  $\exp^1(n) = \exp(n)$ and, for $k =2, 3, ...$, by  $\exp^k(n) = \exp(\exp^{k-1}(n))$.  The notation $\simeq$ will be explained in Section~\ref{sec:preliminaries}.  Our main result is:

 \begin{mainthm} \label{main}
 Given integers $p>q>0$ and $k>0$, there exists a hyperbolic group $G$ and 
 free subgroup $H \leq G$ of distortion
$\mathrm{Dist}_H^G(n) \simeq \exp^k({n^{p/q}})$.  
 \end{mainthm}

 Our $G$   are of infinite height (so do not speak to an old open question of Swarup)---see Section~\ref{sec:Height}.   In the case $k=1$ they can be made residually finite, $C'(1/6)$, $\CAT(-1)$, and virtually special---see Section~\ref{sec:the defn}.

In Section~\ref{sec:Realizing others} we leverage the examples of  Theorem~\ref{main} and of \cite{BBD, BDR, Mitra, Mitra2}  so as to make the distorted subgroup  be  any given non-elementary torsion-free hyperbolic group:

\begin{mainthm}\label{thm:distorted hyp grp}
Let $H$ be any non-elementary torsion-free hyperbolic group  and let $f$ be any of the following functions:
\begin{enumerate}
\item $f(n) = \exp^m(n^{p/q})$, for any integers $m\ge 1$ and $p \geq q \ge 1$. \label{thm part:p/q}
\item $f$ is any one of the Ackermann-function representatives of the successive levels of the Grzegorczyk hierarchy of primitive recursive functions. \label{thm part:Ackermann} 
\end{enumerate}
Then there exists a hyperbolic group $G$ with $H<G$ such that 
$\Dist_H^G \simeq f$.
\end{mainthm}

This paper also contains results we needed to prove Theorem~\ref{thm:distorted hyp grp}
which may be of independent interest. Theorem~\ref{thm:combination} assembles results of Bowditch, Dahmani, and Osin into a combination theorem for the hyperbolicity of amalgams $\Gamma = A \ast_C B$. 
Theorem~\ref{thm:amalgam distortion} relates the distortion of $C$ in $A$ and of $C$ in $B$ to that of $A$ in $\Gamma = A \ast_C B$.
Lemma~\ref{lem:there are free subgroups} states that in every non-elementary torsion-free  hyperbolic group $H$ there is, for any $k \geq 2$, a malnormal quasiconvex free subgroup $F$ of rank $k$.  It builds on the $k = 2$ case, proved by I.~Kapovich  in \cite{Kap99}.
Lemma~\ref{lem:malnornal qc} states that if a semi-direct product $G = F_l \rtimes  F_m$ of finite rank free groups  is  hyperbolic, then the  $F_m$-factor is quasiconvex and malnormal in $G$.

\subsubsection*{Background}  At first sight,  it is surprising  that    subgroups of hyperbolic groups can display any distortion given the tree-like geometry of the thin-triangle condition that defines hyperbolicity.  Every $\Z$ subgroup of a hyperbolic group is undistorted---e.g., \cite[III.$\Gamma$ Corollary 3.10]{BrH}.  Finitely generated subgroups $H$ of hyperbolic groups $G$ are undistorted (meaning linear distortion, $\mathrm{Dist}_H^G(n) \simeq n$) if and only if they are quasi-convex, and in that event they are themselves hyperbolic.   Above linear there is a gap in the spectrum of possible distortion functions: a consequence of the exponential divergence property of hyperbolic spaces is that if a   finitely generated subgroup of a hyperbolic group is  subexponentially distorted, then it is quasi-convex \cite[Proposition~2.6]{Ksubexp}.  Theorem~\ref{main} sweeps out much of the landscape of possibilities above exponential.

Prior to Theorem~\ref{main}, only sporadic examples of distortion functions for subgroups of hyperbolic groups were known.   Subgroups of finite-rank free groups and of hyperbolic surface groups are undistorted \cite{Pittet2, Short2}.  Wise~\cite{WiseSectional} generalized this result to fundamental groups of  non-positively curved, piecewise Euclidean $2$-complexes which enjoy a suitable negative sectional curvature condition. The free factor in any hyperbolic free-by-cyclic group is exponentially distorted \cite{BF, BFa, Brinkmann}.   Mitra~\cite{Mitra, Mitra2}  constructed,  for each integer $k \ge 1$, 
a  hyperbolic group with  a free subgroup  distorted like $n \mapsto \exp^k(n)$, and an example 
with distortion growing faster than any iterated exponential. 
Barnard, Brady  and Dani~\cite{BBD} developed Mitra's constructions into more explicit examples that are also $\textup{CAT}(-1)$.  Baker and Riley \cite{BaR2} exhibited a finite-rank free subgroup of a hyperbolic group that is distorted like $n \mapsto \exp^2(n)$ and  is also pathological in that there is no   Cannon--Thurston map.
   Brady,  Dison, and Riley~\cite{BDR} constructed, for every primitive recursive function, a hyperbolic `hydra' group with a finite-rank free subgroup whose distortion outgrows that function.   The Rips construction produces examples displaying yet more extreme distortion.  
Applied to a finitely presentable group with unsolvable word problem the construction yields a hyperbolic ($C'(1/6)$ small-cancellation) group $G$ with a finitely generated subgroup $N$ such that $\Dist^G_N$ is not bounded from above by a recursive function---see  \cite[\S3.4]{AO}, \cite[Corollary~8.2]{Farb}, \cite[\S3, $3.K_3''$]{Gromov} and \cite{PittetThesis}.  

The subgroup $N$ in the Rips construction is not finitely presentable.  In fact, it follows from a theorem of  Bieri in \cite{Bieri} that $N$ is finitely presented if and only if the quotient $Q$ is finite.  So the Rips construction cannot be used to construct examples such as those in Theorem~\ref{main}.  Instead, we use a modification of the Rips construction: starting with a particular finitely presented group $Q$, we realize it as the quotient of a group   presentation that satisfies  $C'(1/6)$ and other small-cancellation conditions, and find a free subgroup which is distorted, but not normal. Several additional nuances in our construction guarantee that we get the desired distortion estimates.  We  outline  this in Section~\ref{sec:motivation}.

In contrast to the situation with hyperbolic groups, a broad family of functions are known to  be  distortion functions of subgroups of $\text{CAT}(0)$ groups.  Indeed,  Olshanskii and Sapir  \cite[Theorem 2]{OS}  used a Mihailova-style construction to show that  the set of distortion functions of finitely generated subgroups of $F_2 \times F_2$ coincides with the set of Dehn functions of finitely presented groups.  Such functions are known to have wide scope thanks to the $S$-machines of \cite{SBR, Sapir2}. 
 
 In finitely presented groups, even   $\Z$-subgroups can exhibit essentially any distortion:   
 Olshanskii \cite{Ol3}   showed that every computable function $\N \to \N$, satisfying some straight-forwardly necessary conditions, is $\simeq$-equivalent to the distortion function of such as subgroup.

\subsubsection*{Application to Dehn functions.}  What functions can be $\simeq$-equivalent to Dehn functions is understood in detail thanks to \cite{BB,BBFS,Ol3,SBR}. However, because the most comprehensive results depend on deeply involved constructions, we note that our examples give some explicit examples as follows.

\begin{maincor} \label{Dehn function examples}
	Our groups $G$ yield explicit examples, for integers $p>q>0$ and $k>0$, of groups with Dehn functions growing $\simeq \exp^k({n^{p/q}})$, namely the free product with amalgamation  $G \ast_H G$ of two copies of $G$ along $H$, and the HNN-extension $G \ast_{\tau}$ of $G$ with stable letter $\tau$ that commutes with all elements of $H$.   
\end{maincor}

\begin{proof}
Theorem~6.20 in Chapter III.$\Gamma$ of \cite{BrH} gives upper and lower bounds on the Dehn functions of $G \ast_H G$ and $G \ast_{\tau}$  in terms of the Dehn function of $G$ (which is $\simeq n$ because $G$ is hyperbolic) and $\mathrm{Dist}_H^G$.  Up to $\simeq$, these  bounds agree with each other and with $\mathrm{Dist}_H^G$ since $\mathrm{Dist}_H^G$ is super-exponential.   
\end{proof}

\subsubsection*{Innovations}  
Constructing $H < G$ that realize the  subgroup distortion functions of Theorem~\ref{main} while staying within the universe of hyperbolic groups requires 
 some delicacy.  For example,  a standard strategy of achieving $f \circ g$ distortion by amalgamating a pair realizing distortion $f$ with one realizing distortion $g$ is not available due to 
 the gap between linear and exponential distortion in the hyperbolic group setting. Instead, we develop new tools and techniques. 
We seed the ``$p/q$-distortion'' with a single free-group automorphism from which we extract two growth rates that we play off against each other.  We look to Wise's version of the Rips construction \cite{Wise1} for small-cancellation (hence hyperbolicity) and for an HNN-structure (which facilitates analysis), but we limit the defining relations employed in a way that sacrifices the normality of the subgroup, but gains crucial control on the ``flow of noise'' through van~Kampen diagrams. We further this control by using two families of  ``Rips noise words'' instead of one. And to analyze this flow, we introduce \emph{tracks} which are branching structures that generalize corridors.  Under appropriate hypotheses tracks display rigidity which constrains diagrams sufficiently to allow distortion estimates.     

We explain these novelties more fully in Section~\ref{sec:motivation}.

 \subsubsection*{Next steps}  
Sapir's $S$-machines emulate general computing machines in appropriately constructed (and always non-hyperbolic) finitely presented groups.  One might view the techniques we introduce here as groundwork for doing the same within appropriately constructed hyperbolic groups.
 
Another potential application of our examples is to constructing subgroups of $\CAT(0)$ groups or hyperbolic groups exhibiting a range of Dehn functions.  One might, for example, look to embed the doubles of Corollary~\ref{Dehn function examples}  in $\CAT(0)$ groups in the manner of \cite{BrT}.  However, our distorted subgroups not being  normal is an obstacle to making this work.

\subsubsection*{The organization of this article}  

The remainder of this section contains preliminaries on words, hyperbolicity, distortion, and the equivalence relation $\simeq$ on functions  $\mathbb{R}_{\geq 0} \to \mathbb{R}_{\geq 0}$ (Section~\ref{sec:preliminaries}), and then an overview of our construction (Section~\ref{sec:motivation}). 
 Section~\ref{ch:our groups} contains the definition of our groups $G$ used to prove Theorem~\ref{main} in the case $m=1$ and catalogs their small-cancellation conditions (Section~\ref{sec:the defn}),  some immediate consequences of those  conditions (Section~\ref{sec:consequences of small-cancellation}), a review of the definition of a corridor in a van~Kampen diagram and an introdrucution to a more general dual notion we call \emph{tracks}, which may branch,  unlike corridors (Section~\ref{sec:tracks}), and then two HNN-structures for $G$ and a proof that $H$ is free  (Section~\ref{sec:hnn}).  
Section~\ref{ch:the lower bound} gives our proof of the lower bound on the distortion of $H$ in $G$.  
Section~\ref{ch:tracks and diagram rigidity} establishes results on the rigidity of van~Kampen diagrams that will facilitate our proof of the upper bound.  We examine   how a van~Kampen diagram $\Delta$ over $G$ being reduced limits the patterns of tracks within it (Section~\ref{sec:tracks in reduced diagrams}).  We give  general results about paths across discs, which we will apply to tracks in $\Delta$ (Section~\ref{sec: intersection patterns}).  We argue that tracks are further constrained in what we call a \emph{distortion diagram}, meaning a $\Delta$ exhibiting how a word in the generators of $H$ equals a shorter word in the generators of $G$ (Section~\ref{sec: tracks in distortion diagrams}).   We introduce and analyse $(a_2, b_q)$-tracks, which are a device we use to connect  growth within $\Delta$  to the presence of certain edges in its boundary (Section~\ref{sec: a2bq tracks}).
Section~\ref{ch:upper} concerns estimates which are made possible by this rigidity and which culminate in an upper bound on the distortion of $H$ in $G$ (Section~\ref{sec:upper}) when combined with    calculations in a free-by-cyclic quotient $Q$ of $G$ where the fraction $p/q$ ultimately enters (Section~\ref{sec:p/q}).   
Section~\ref{ch:leveraging} promotes our examples to iterated exponential functions, and so completes our proof of Theorem~\ref{main} (Section~\ref{sec:general k}), and then  explains how we leverage our examples to prove Theorem~\ref{thm:distorted hyp grp} (Section~\ref{sec:Realizing others}).     
Section~\ref{ch:Height} contains a proof that our examples have infinite height.

\subsection{Preliminaries}  \label{sec:preliminaries}

A \emph{word $w$ on} a set of letters $\mathcal{A}$ is an expression $a_1^{\varepsilon_1} \cdots a_m^{\varepsilon_m}$     where $m \geq 0$, $a_i \in \mathcal{A}$, and $\varepsilon_i = \pm 1$ for all $i$.  It is \emph{positive} when $\varepsilon_i =1$ for all $i$.    Its length $|w|$  is $m$.  
The word metric $d_S(g, h)$ on $G$ gives the length of a shortest word on $S$ that represents $g^{-1}h$.   
We use  
$d_G$ or $d$ in place of $d_S$ when the generating set is understood from the context.   

A finitely generated group is \emph{hyperbolic} when its Cayley graph has the property  that there exists $\delta >0$ such that all geodesic triangles are \emph{$\delta$-thin}: that is, each of its three sides is in the $\delta$-neighbourhood of the other two.  The existence of such a $\delta$ does not depend on the finite generating set (but the values of $\delta$ for which the condition holds generally will).  See, for example, \cite{BrH, Gromov4} for further background.

Suppose $S$ and $T$ are  finite generating sets for a group $G$ and subgroup $H$,  respectively.     The \emph{distortion  function} $\Dist^{G}_{H} : \N \to \N$  measures how $H$ sits as a metric space in $G$ by comparing the  restriction  of the word   metric $d_S$ on  $G$ associated to $S$ to the word metric $d_T$ 
on $H$ associated to~$T$:
$$\Dist^{G}_{H}(n) \ := \  \max \set{  \  d_T(e,g)  \  \mid  \   g \in {H} \textup{ with }  d_S(e,g)  \leq n   \ }.$$
Replacing $S$ and $T$ by other finite generating sets will produce a distortion function that is  $\simeq$-equivalent in the following sense.  For    $f,g: \mathbb{R}_{\geq 0} \to \mathbb{R}_{\geq 0}$ write $f \preceq g$ when there exists $C>0$ such that $f(n) \leq Cg(  Cn+C ) +Cn+C$ for all $n \geq 0$, and $f  \simeq g$ when $f \preceq g$ and $g \preceq f$.  Apply these relations to  functions   $\mathbb{N}  \to \mathbb{R}_{\geq 0}$ by extending the domains to $\mathbb{R}_{\geq 0}$ and having  the functions be  constant on the intervals $[n, n+1)$.  

The following two lemmas concern features of the $\simeq$-relation that will be important for us.  The first is routine and we present it without proof. 

 \begin{lemma} \label{lem:exponents equiv} ${}$
 \begin{enumerate}
 	\item For $\alpha, \beta \geq 1$, $2^{n^{\alpha}} \simeq 2^{n^{\beta}}$ if and only if $\alpha = \beta$.
 	\item For $\alpha \geq 1$ and $C>1$, $C^{n+ n^{\alpha}} \simeq C^{n^{\alpha}} \simeq 2^{n^{\alpha}}$.  
 \end{enumerate}	
 \end{lemma}
 
 For our proof of the lower bound in Theorem~\ref{main}, we will exhibit a sequence of words that represent elements of $H$, but can  only be expressed by long words on the generators of $H$.  The force of the following lemma is that, despite the lengths of our words forming  a sparse sequence, we can draw the desired conclusion.   
 
 \begin{lemma} \label{lem:sparsity}
 Suppose $H$ is a subgroup of $G$ and both are finitely generated.  
 Suppose $p >  q > 0$ are integers, $C_1, C_2, C_3 > 0$ are constants, and  $w_1, w_2, \ldots$ is a sequence of words on the generators of $G$.  Suppose that $w_n$ represents an element of $H$ for all $n$, and $$C_1 n^q \leq |w_n| \leq C_2 n^q \ \text{ but } \ d_H(e,w_n) \geq C_3 2^{n^p}.$$  Then $\Dist^{G}_{H}(n) \succeq 2^{n^{p/q}}$.  
 \end{lemma}

\begin{proof}
	Remark~2.1 in \cite{BBFS} is  that to verify $g \succeq f$ for $f,g: \N \to \N$, it suffices to have $g(m_n) \geq f(m_n)$ on a sequence $(m_n)$ of integers such that $m_n \to \infty$ as $n \to \infty$ and such that there exists $C>0$ with  $m_{n+1} \leq C m_n$ for all  $n$.  If $C_4 = (q+1) \displaystyle{\max_{i=0, \ldots,q}} \TB{q}{i}$, then $(n+1)^q \leq C_4 n^q$  for all $n$.   So there is a $C$ such that the sequence $m_n = |w_n|$ satisfies this condition. Now 
	$$\Dist^{G}_{H}(|w_n|) \ \geq \  d_H(e,w_n) \ \geq \  C_3 2^{n^p} \  \geq  \ C_3 2^{\left(\frac{1}{C_2} |w_n|\right)^{p/q}}.$$ So $\Dist^{G}_{H}(n) \succeq 2^{\left(\frac{n}{C_2}\right)^{p/q}}$, and the result then follows from Lemma~\ref{lem:exponents equiv}(2) (by taking $C=2^{(C_2^{-p/q})}$ and $\alpha = p/q$). 	
	\end{proof}

We will work extensively with van~Kampen diagrams.  There are many introductory accounts in the literature.

\subsection{Motivation for our construction}\label{sec:motivation}

In this section, we offer some insights into the origins of our construction.
The formal definition of our group-pair $H<G$, used to prove Theorem~\ref{main} in the case $k=1$,  follows in Section~\ref{sec:the defn}.

Our construction begins with the free-by-cyclic group
\begin{equation} 
Q  \ = \ \langle \  a_1,  b_0, \ldots, b_p  \mid a_1^{-1}b_i a_1 = \varphi(b_i)  \ \, \forall i  \  \rangle \label{eqn QQQ}
\end{equation}
where $F = F(b_0, \ldots, b_p)$ is a free group of rank $p+1$ and $\varphi$ is the polynomially-growing automorphism of 
$F $  mapping  $b_i \mapsto b_{i+1}b_i$ for $i\neq p$ and $b_p \mapsto b_p$.

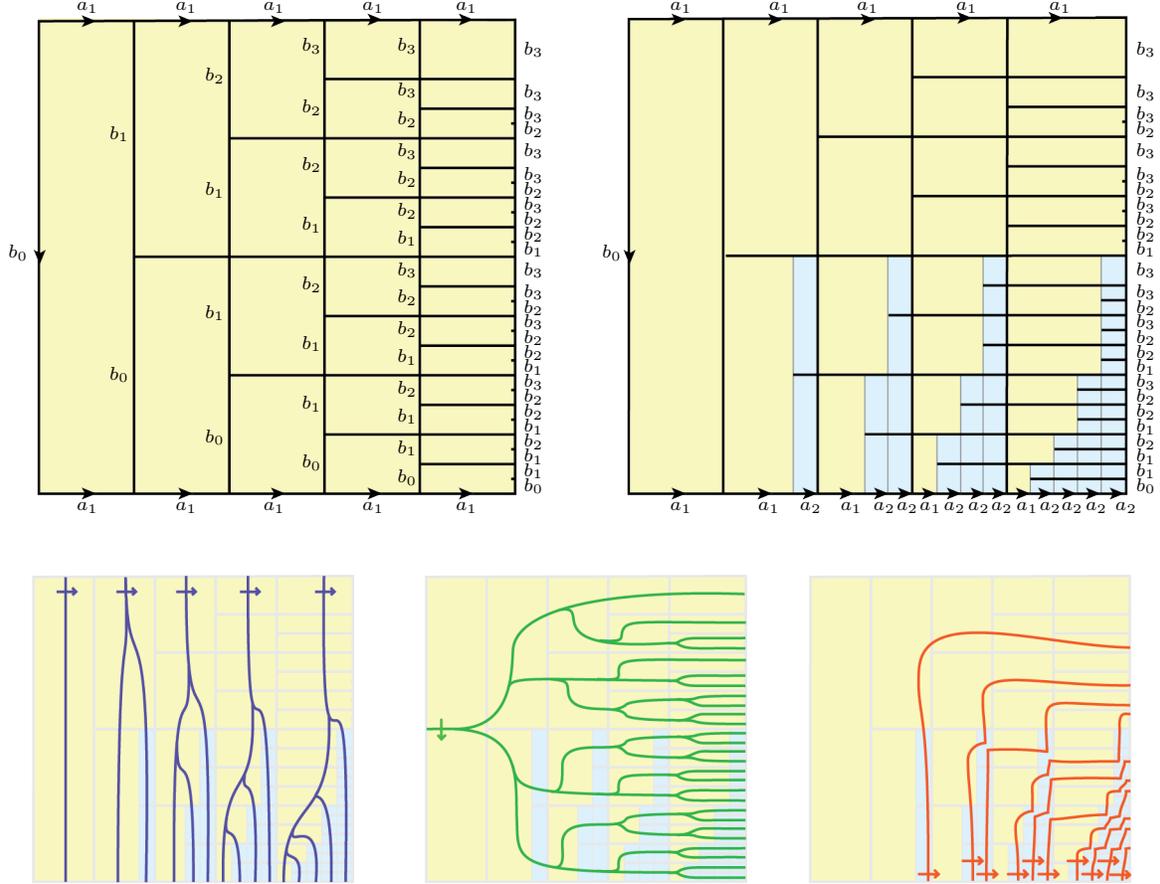
\begin{figure}[ht]
\centering
\begin{overpic}  [scale=0.85]
{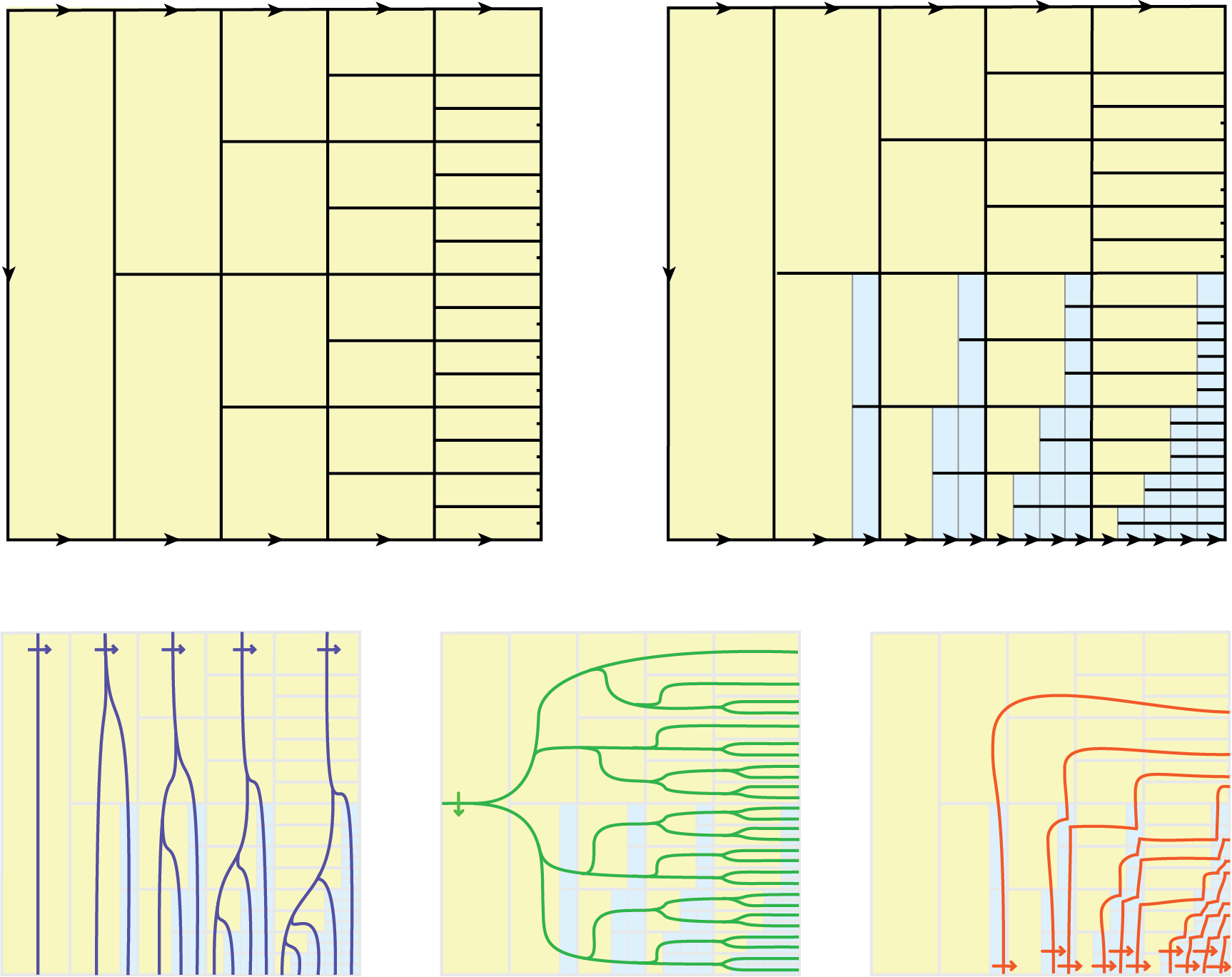}
\put(-2.2, 57){\tiny{$b_0$}}
\put(4, 34){\tiny{$a_1$}}
\put(13, 34){\tiny{$a_1$}}
\put(21, 34){\tiny{$a_1$}}
\put(30, 34){\tiny{$a_1$}}
\put(38.5, 34){\tiny{$a_1$}}
\put(4, 79.5){\tiny{$a_1$}}
\put(13, 79.5){\tiny{$a_1$}}
\put(21, 79.5){\tiny{$a_1$}}
\put(30, 79.5){\tiny{$a_1$}}
\put(38.5, 79.5){\tiny{$a_1$}}
\put(33.2, 75.8){\tiny{$b_3$}}
\put(33.2, 71.8){\tiny{$b_3$}}
\put(33.2, 69.1){\tiny{$b_2$}}
\put(33.2, 66.2){\tiny{$b_3$}}
\put(33.2, 63.7){\tiny{$b_2$}}
\put(33.2, 60.8){\tiny{$b_2$}}
\put(33.2, 58.3){\tiny{$b_1$}}
\put(33.2, 55.4){\tiny{$b_3$}}
\put(33.2, 52.9){\tiny{$b_2$}}
\put(33.2, 50){\tiny{$b_2$}}
\put(33.2, 47.5){\tiny{$b_1$}}
\put(33.2, 44.6){\tiny{$b_2$}}
\put(33.2, 42.1){\tiny{$b_1$}}
\put(33.2, 39.2){\tiny{$b_1$}}
\put(33.2, 36.5){\tiny{$b_0$}}
\put(24.5, 75.8){\tiny{$b_3$}}
\put(24.5, 70.2){\tiny{$b_2$}}
\put(24.5, 65){\tiny{$b_2$}}
\put(24.5, 59.5){\tiny{$b_1$}}
\put(24.5, 54.1){\tiny{$b_2$}}
\put(24.5, 48.7){\tiny{$b_1$}}
\put(24.5, 43.2){\tiny{$b_1$}}
\put(24.5, 37.9){\tiny{$b_0$}}
\put(15.7, 73.1){\tiny{$b_2$}}
\put(15.7, 62.7){\tiny{$b_1$}}
\put(15.7, 51.5){\tiny{$b_1$}}
\put(15.7, 40.2){\tiny{$b_0$}}
\put(7.0, 67.8){\tiny{$b_1$}}
\put(7.0, 45.9){\tiny{$b_0$}}
\put(44.7, 75.5){\tiny{$b_3$}}
\put(44.7, 71.55){\tiny{$b_3$}}
\put(44.7, 69.55){\tiny{$b_3$}}
\put(44.7, 68.2){\tiny{$b_2$}}
\put(44.7, 66.25){\tiny{$b_3$}}
\put(44.7, 64.05){\tiny{$b_3$}}
\put(44.7, 62.7){\tiny{$b_2$}}
\put(44.7, 61.35){\tiny{$b_3$}}
\put(44.7, 60){\tiny{$b_2$}}
\put(44.7, 58.65){\tiny{$b_2$}}
\put(44.7, 57.3){\tiny{$b_1$}}
\put(44.7, 55.4){\tiny{$b_3$}}
\put(44.7, 53.4){\tiny{$b_3$}}
\put(44.7, 51.95){\tiny{$b_2$}}
\put(44.7, 50.7){\tiny{$b_3$}}
\put(44.7, 49.25){\tiny{$b_2$}}
\put(44.7, 47.9){\tiny{$b_2$}}
\put(44.7, 46.55){\tiny{$b_1$}}
\put(44.7, 45.2){\tiny{$b_3$}}
\put(44.7, 43.85){\tiny{$b_2$}}
\put(44.7, 42.5){\tiny{$b_2$}}
\put(44.7, 41.15){\tiny{$b_1$}}
\put(44.7, 39.8){\tiny{$b_2$}}
\put(44.7, 38.45){\tiny{$b_1$}}
\put(44.7, 37.1){\tiny{$b_1$}}
\put(44.7, 35.75){\tiny{$b_0$}}
\put(51.8, 57){\tiny{$b_0$}}
\put(58, 34){\tiny{$a_1$}}
\put(66, 34){\tiny{$a_1$}}
\put(73.5, 34){\tiny{$a_1$}}
\put(80.7, 34){\tiny{$a_1$}}
\put(89.5, 34){\tiny{$a_1$}}
\put(69.8, 34){\tiny{$a_2$}}
\put(76.5, 34){\tiny{$a_2$}}
\put(78.6, 34){\tiny{$a_2$}}
\put(82.9, 34){\tiny{$a_2$}}
\put(85.1, 34){\tiny{$a_2$}}
\put(87.2, 34){\tiny{$a_2$}}
\put(91.6, 34){\tiny{$a_2$}}
\put(93.6, 34){\tiny{$a_2$}}
\put(95.8, 34){\tiny{$a_2$}}
\put(98.5, 34){\tiny{$a_2$}}
\put(58, 79.5){\tiny{$a_1$}}
\put(67, 79.5){\tiny{$a_1$}}
\put(75, 79.5){\tiny{$a_1$}}
\put(84, 79.5){\tiny{$a_1$}}
\put(92.5, 79.5){\tiny{$a_1$}}
\put(100.4, 75.5){\tiny{$b_3$}}
\put(100.4, 71.55){\tiny{$b_3$}}
\put(100.4, 69.55){\tiny{$b_3$}}
\put(100.4, 68.2){\tiny{$b_2$}}
\put(100.4, 66.25){\tiny{$b_3$}}
\put(100.4, 64.05){\tiny{$b_3$}}
\put(100.4, 62.7){\tiny{$b_2$}}
\put(100.4, 61.35){\tiny{$b_3$}}
\put(100.4, 60){\tiny{$b_2$}}
\put(100.4, 58.65){\tiny{$b_2$}}
\put(100.4, 57.3){\tiny{$b_1$}}
\put(100.4, 55.4){\tiny{$b_3$}}
\put(100.4, 53.4){\tiny{$b_3$}}
\put(100.4, 51.95){\tiny{$b_2$}}
\put(100.4, 50.7){\tiny{$b_3$}}
\put(100.4, 49.25){\tiny{$b_2$}}
\put(100.4, 47.9){\tiny{$b_2$}}
\put(100.4, 46.55){\tiny{$b_1$}}
\put(100.4, 45.2){\tiny{$b_3$}}
\put(100.4, 43.85){\tiny{$b_2$}}
\put(100.4, 42.5){\tiny{$b_2$}}
\put(100.4, 41.15){\tiny{$b_1$}}
\put(100.4, 39.8){\tiny{$b_2$}}
\put(100.4, 38.45){\tiny{$b_1$}}
\put(100.4, 37.1){\tiny{$b_1$}}
\put(100.4, 35.75){\tiny{$b_0$}}
\end{overpic}
\vspace{4mm}
\caption{Top left: the van~Kampen diagram $D_1$ over $Q$ for $a_1^{-n}b_0 a_1^n = \varphi^{n} (b_0)$ when $n=5$ and $p=3$.  Top right: the corresponding  diagram $D_2$ over $G_2$ when $q=2$. Lower left, middle and right: $a$-tracks, $b$-tracks, and $(a_2, b_q)$-tracks through $D_2$.}
\label{fig:QDiagram}
\end{figure}

The van~Kampen diagram $D_1$ over $Q$ pictured top-left in Figure~\ref{fig:QDiagram} (for the case $n=5$ and $p=3$) shows how  $a_1^{-n}b_0 a_1^n = \varphi^{n} (b_0)$ equals a positive word $\lambda$ on $b_0, b_1, \dots, b_p$ which contains $\simeq n^i$ letters $b_i$ for $i=0, \ldots, p$ (Lemma~\ref{lem:Bgrowth}). The contribution of $b_p$ dominates, so the length of $\lambda$ is $N = |\lambda| \simeq n^p$.

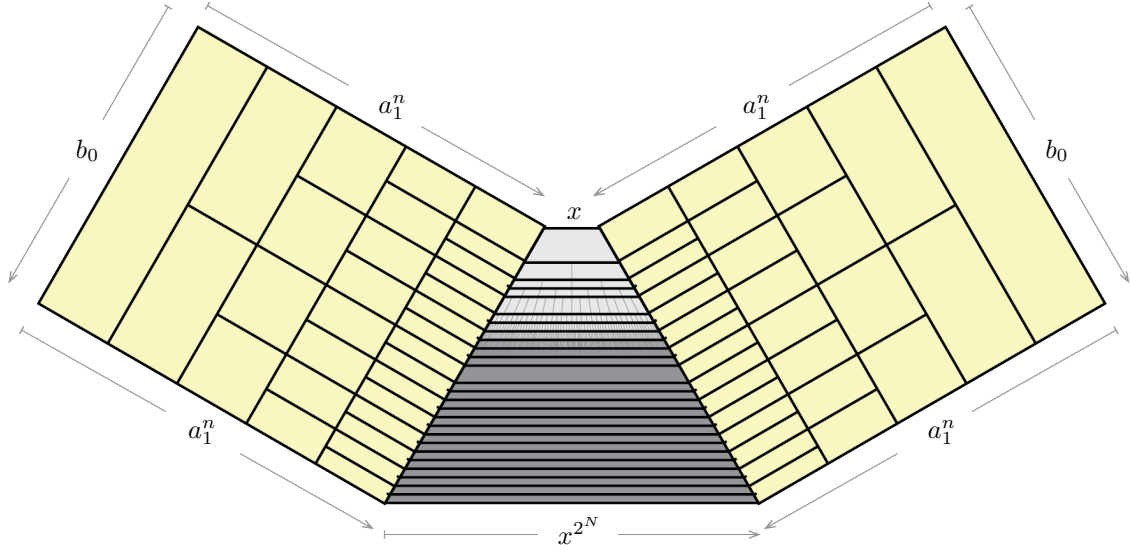
\begin{figure}[htbp]
\centering
\begin{overpic} [scale=0.84]  
{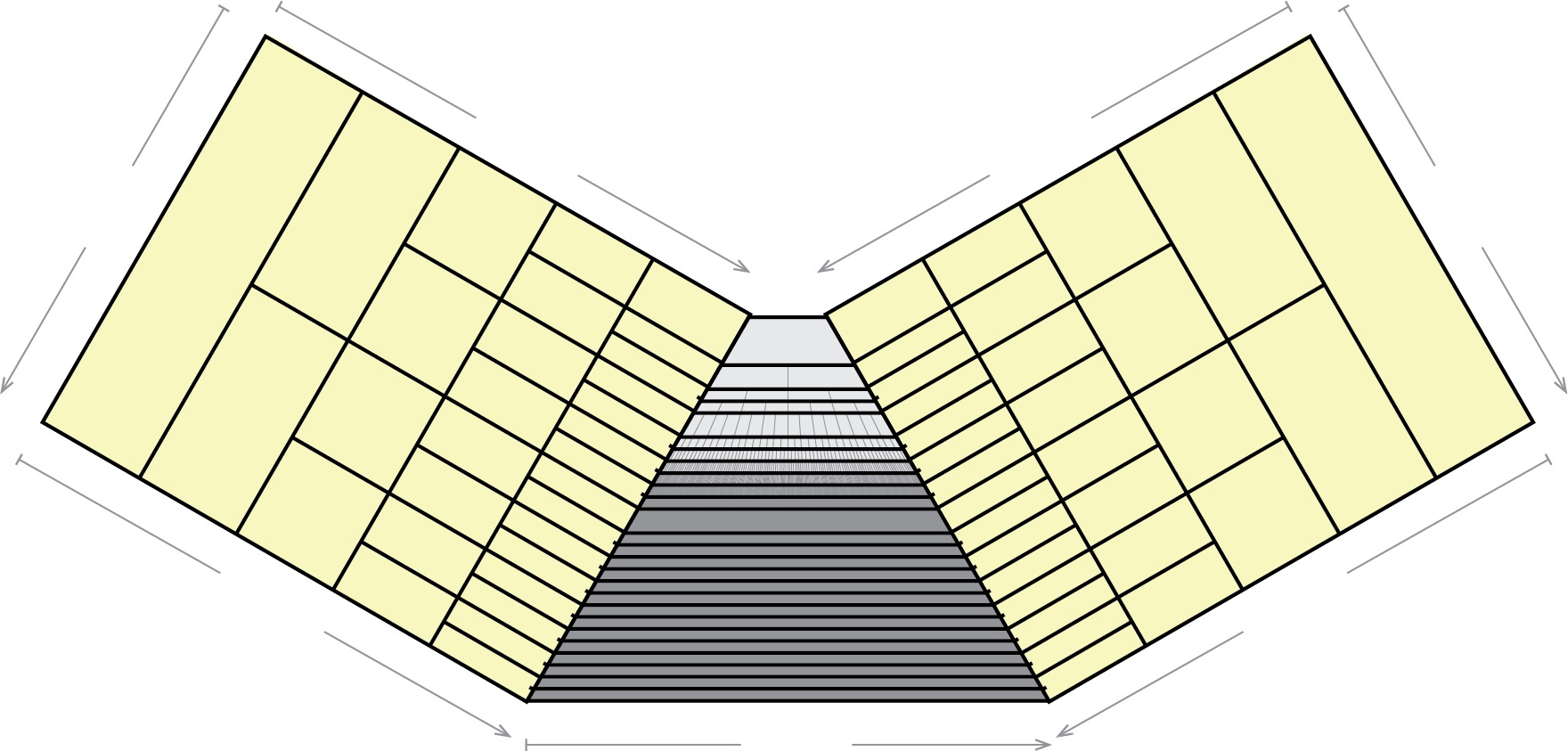}
\put(49, -0.5){\small{$x^{2^N}$}}
\put(49.8, 28.5){\small{$x$}}
\put(16, 9){\small{$a_1^{n}$}}
\put(33, 38){\small{$a_1^{n}$}}
\put(82, 9){\small{$a_1^{n}$}}
\put(65.5, 38){\small{$a_1^{n}$}}
\put(6, 34){\small{$b_0$}}
\put(92.5, 34){\small{$b_0$}}
\end{overpic}
\vspace{4mm}
\caption{A van~Kampen diagram $\Delta_1$ over $G_1$ for $a_1^{-n} b_0^{-1} a_1^{n} x  a_1^{-n} b_0a_1^{n}  =  x^{2{^N}}$ when $n=5$, $p=3$, and $N = | \varphi^n(b_0) | =  26$.   }
\label{fig:G1Diagram}
\end{figure}

Next, we define $$G_1  \ = \ \langle \  Q, x   \mid   b_j^{-1} x b_j  = x^2 \ \,  \forall j \ \rangle.$$  
As shown in Figure~\ref{fig:G1Diagram},  attaching a  copy of $D_1$ and a copy of its mirror image to a diagram for $\lambda^{-1} x \lambda = x^{2^N}$ along its two paths labelled $\lambda$ gives a van Kampen diagram $\Delta_1$ over $G_1$ for the relation 
\begin{equation} \label{eqn:G_1relation}
a_1^{-n} b_0^{-1} a_1^{n} x  a_1^{-n} b_0a_1^{n} \ = \ x^{2{^N}}.
\end{equation}
  This diagram illustrates that there is a word of  length $\simeq 2^{n^p}$ in $H_1 = \langle x \rangle$, whose length in $G_1$ is $\simeq n$.   As there is a family of such diagrams indexed by $n$, this shows that $\Dist^{G_1}_{H_1}(n) \succeq 2^{n^{p}}$. 

Next we elaborate on this construction in a way that plays off  the   $\simeq n^p$ letters $b_p$  against the $\simeq n^q$ letters $b_q$ in $\lambda$.  We introduce a new generator $a_2$ and we modify the relation $a_1^{-1}b_{q-1} a_1 = b_q b_{q-1}$ of $G_1$ to 
$a_1^{-1}b_{q-1} a_1a_2 = b_q b_{q-1}$, so that for every  new $b_q$ created by $\varphi$ within $D_1$, an $a_2$ is created as well.  Furthermore, we add the relations that $a_2$ commutes with $b_j$ for all $j$, allowing these newly created edges to \emph{flow to the boundary} as shown in the diagram on the right in Figure~\ref{fig:QDiagram}. The resulting diagrams $D_2$ can be mapped onto $D_1$ by suitably collapsing all the $a_2$-edges and the commutator 2-cells in which  they occur.  As for the construction of $\Delta_1$, assemble $D_2$, its mirror-image, and our diagram for $\lambda^{-1} x \lambda = x^{2^N}$ to get a diagram $\Delta_2$ that  demonstrates that  $x^{2{^N}}$ equates in a group $G_2$ to a word $a_1^{-n} b_0^{-1} a_1^{n} x  a_1^{-n} b_0a_1^{n}$ with $\simeq n^q$ letters $a_2^{\pm 1}$ inserted.  This construction suggests that the distortion function of $\langle x \rangle$ in $G_2$ grows like $n^q \mapsto 2^{n^{p}}$, and therefore like $n \mapsto 2^{n^{p/q}}$.

Now, $G_2$ is not hyperbolic. So next we  \emph{hyperbolize} its presentation using an approach similar to Wise's version of the Rips construction \cite{Wise1}.  We add \emph{noise} to each relation so that the resulting presentation satisfies small-cancellation conditions including  $C'(1/6)$.  This is achieved by replacing $x$ by three letters $t, x_1, x_2$, and introducing a \emph{noise word} on $t, x_1, x_2$ to each relation. We then add   relations to allow the noise to \emph{flow} to the boundary of the diagram and then (in the two triangles at the bottom of Figure~\ref{fig:G2Diagram}) be moved past the $a_1^{\pm 1}, a_2^{\pm 1}$ and collected together.  These additional relations play a similar role to the commuting relations involving $a_2$ introduced above; they allow noise to move past $a$- and $b$-letters (but only in one direction) at the expense of introducing additional noise.  The  resulting group $G_3$ admits diagrams $\Delta_3$  which map onto $\Delta_2$ on suitably collapsing the edges labelled by noise letters and suitably collapsing the 2-cells that allow the noise to \emph{flow}. We take $H_3 =   \langle t, x_1, x_2 \rangle$.  

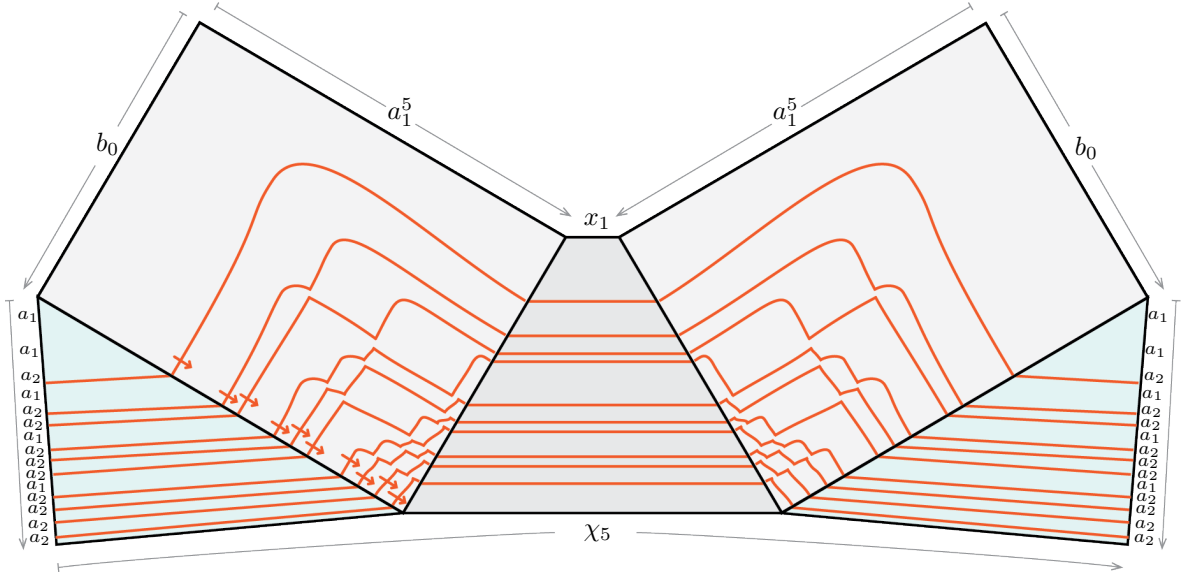
\begin{figure}[htbp]
\centering
\begin{overpic}  [scale=0.8]  
{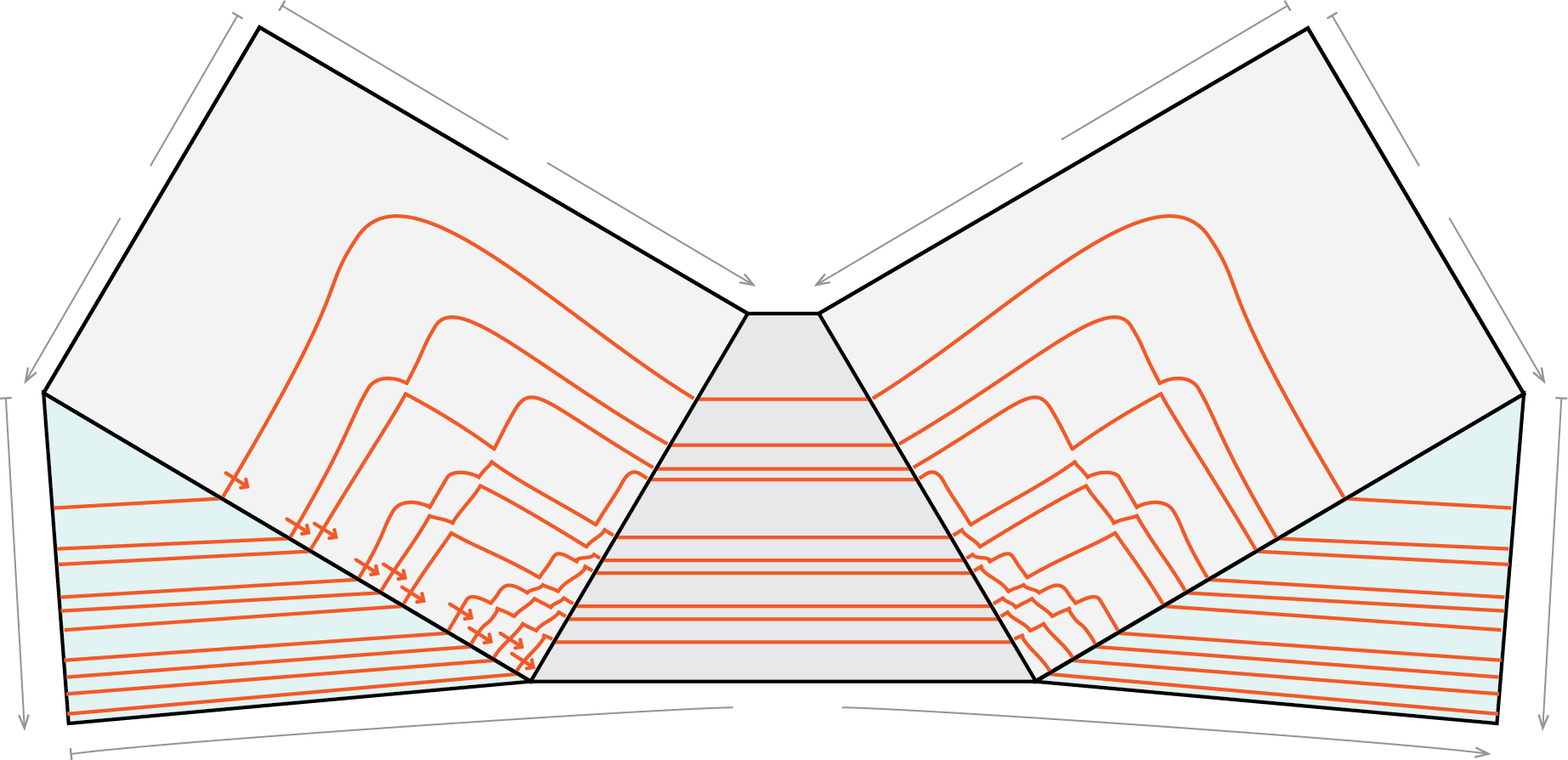}
\put(7.7, 35.8){\small{$b_0$}}
\put(1, 21.5){\tiny{$a_1$}}
\put(1.15, 18.5){\tiny{$a_1$}}
\put(1.25, 16.3){\tiny{$a_2$}}
\put(1.3, 14.8){\tiny{$a_1$}}
\put(1.4, 13.4){\tiny{$a_2$}}
\put(1.45, 12.3){\tiny{$a_2$}}
\put(1.5, 11.2){\tiny{$a_1$}}
\put(1.60, 10){\tiny{$a_2$}}
\put(1.65, 9){\tiny{$a_2$}}
\put(1.70, 8){\tiny{$a_2$}}
\put(1.75, 7){\tiny{$a_1$}}
\put(1.8, 6){\tiny{$a_2$}}
\put(1.85, 5){\tiny{$a_2$}}
\put(1.9, 3.7){\tiny{$a_2$}}
\put(2, 2.5){\tiny{$a_2$}}
\put(97.2, 21.5){\tiny{$a_1$}}
\put(96.95, 18.5){\tiny{$a_1$}}
\put(96.85, 16.3){\tiny{$a_2$}}
\put(96.7, 14.8){\tiny{$a_1$}}
\put(96.6, 13.4){\tiny{$a_2$}}
\put(96.55, 12.3){\tiny{$a_2$}}
\put(96.5, 11.2){\tiny{$a_1$}}
\put(96.4, 10){\tiny{$a_2$}}
\put(96.35, 9){\tiny{$a_2$}}
\put(96.3, 8){\tiny{$a_2$}}
\put(96.25, 7){\tiny{$a_1$}}
\put(96.2, 6){\tiny{$a_2$}}
\put(96.15, 5){\tiny{$a_2$}}
\put(96.1, 3.7){\tiny{$a_2$}}
\put(96, 2.5){\tiny{$a_2$}}
\put(49.2, 29.5){\small{$x_1$}}
\put(49.2, 3){\small{$\chi_5$}}

\put(91, 35.5){\small{$b_0$}}

\put(32.5, 38.5){\small{$a_1^{5}$}}
\put(65.3, 38.5){\small{$a_1^{5}$}}

\end{overpic}
\vspace{4mm}
\caption{A schematic of a van~Kampen diagram $\Delta$ over $G$ showing that, if we define $\nu_5 = a_1 \, a_1 a_2 \, a_1 a_2^2 \, a_1 a_2^3   \, a_1 a_2^4$, then the word   
$w_5 = \nu_5^{-1} \ b_0^{-1} a_1^5 \ x_1 \  a_1^{-5} b_0 \ \nu_5$ on the generators of $G$   equals a word $\chi_5$ on the \emph{noise} letters.  The diagram's $(a_2, b_q)$-tracks are shown. Each meets the boundary at a pair $a_2$-edges.}
\label{fig:G2Diagram}
\end{figure}

The diagram of Figure~\ref{fig:G2Diagram} shows the $n=5$ instance of a family of diagrams demonstrating how words $w_n$ on $a_1, a_2, b_0, a_1, x_1$  represent the same elements of $G_3$ as words $\chi_n$ on $t, x_1, x_2$.  Because the effect is so pronounced,   the figure cannot do justice to the exponential expansion in the direction of $\chi_n$.

While this family of diagrams provides the desired $2^{n^{p/q}}$ lower bound on the distortion of $H_3$ in $G_3$, some issues remain. 
Firstly, with the presentation described,  we cannot get a matching $2^{n^{p/q}}$ upper bound on distortion.  If we replace the two $b_0$ letters in \eqref{eqn:G_1relation} with $b_i$, where $i < q$, and then construct diagrams $\Delta_3$ as described above, then they will exhibit $n \mapsto 2^{n^{(p-i)/(q-i)}}$ distortion of $H_3$, which is greater than $2^{n^{p/q}}$. Secondly, allowing the noise letters to interact with both $a$- and $b$-letters prevents us from establishing an HNN-structure on the group (the iterated HNN-structure of Proposition~\ref{prop:hnn}) which 
will allow us to prove that  our distorted subgroup $H$ is free.   
 
Both issues are solved by making the role of the noise more nuanced.  We introduce two pairs of noise letters, $x_1, x_2$ and $y_1, y_2$ (in addition to the noise letter $t$).  For $i>0$,   $b_i$ interacts with $x_1$ and $x_2$ but not $y_1$ and $y_2$, while  $a_1$ and $a_2$  interact with $y_1$ and $y_2$, and not $x_1$ and $x_2$.  Conjugation by $b_0$ converts $x_1$ and $x_2$ to words on $y_1$ and $y_2$.  This way we arrive at our group $G$ whose defining relations are set out in  Figure~\ref{fig:relations}. We take $H$ to be the subgroup generated by $t,y_1, y_2$.  

Over $G$ there are diagrams $\Delta$ of the form shown in Figure~\ref{fig:G2Diagram} exhibiting $2^{n^{p/q}}$-distortion.  This construction is the heart of our proof in Section~\ref{sec:lower} that  $\Dist_H^G(n) \succeq 2^{n^{p/q}}$.

As for the reverse bound $\Dist_H^G(n) \preceq 2^{n^{p/q}}$, the aforementioned diagrams yielding  larger distortion no longer exist because if we replace  $b_0$ with $b_i$ where $i>0$ in the construction of $\Delta$, then $\partial \Delta$ has a long word in $a_1, a_2, t$ along with $x_1, x_2$  rather than along with $y_1,y_2$. We have  long words on letters that are not all generators for $H$ and we can no longer attach the triangular subdiagrams that  separate the $a_1, a_2$ from the the noise letters. 

However, to establish the upper bound we must prove that no other ``bad'' diagrams exist. To achieve this we study what we call \emph{distortion diagrams}---reduced diagrams $\Delta$, subject to natural simplifying assumptions, which exhibit how a word $\chi$ on $t, y_1, y_2$ can be represented by a shorter word $w$ on the generators of $G$.  We show in Sections~\ref{sec:tracks in reduced diagrams}--\ref{sec: a2bq tracks}  that such a $\Delta$ is subject to considerable rigidity.  Our argument  shows that $\Delta$ is so constrained that it strongly resembles the diagrams  described above and is thereby subject to estimates that yield the $2^{n^{p/q}}$ upper bound.

Three features of $G$  impose this rigidity.

\begin{figure}[htbp]
\centering
\begin{overpic}  
{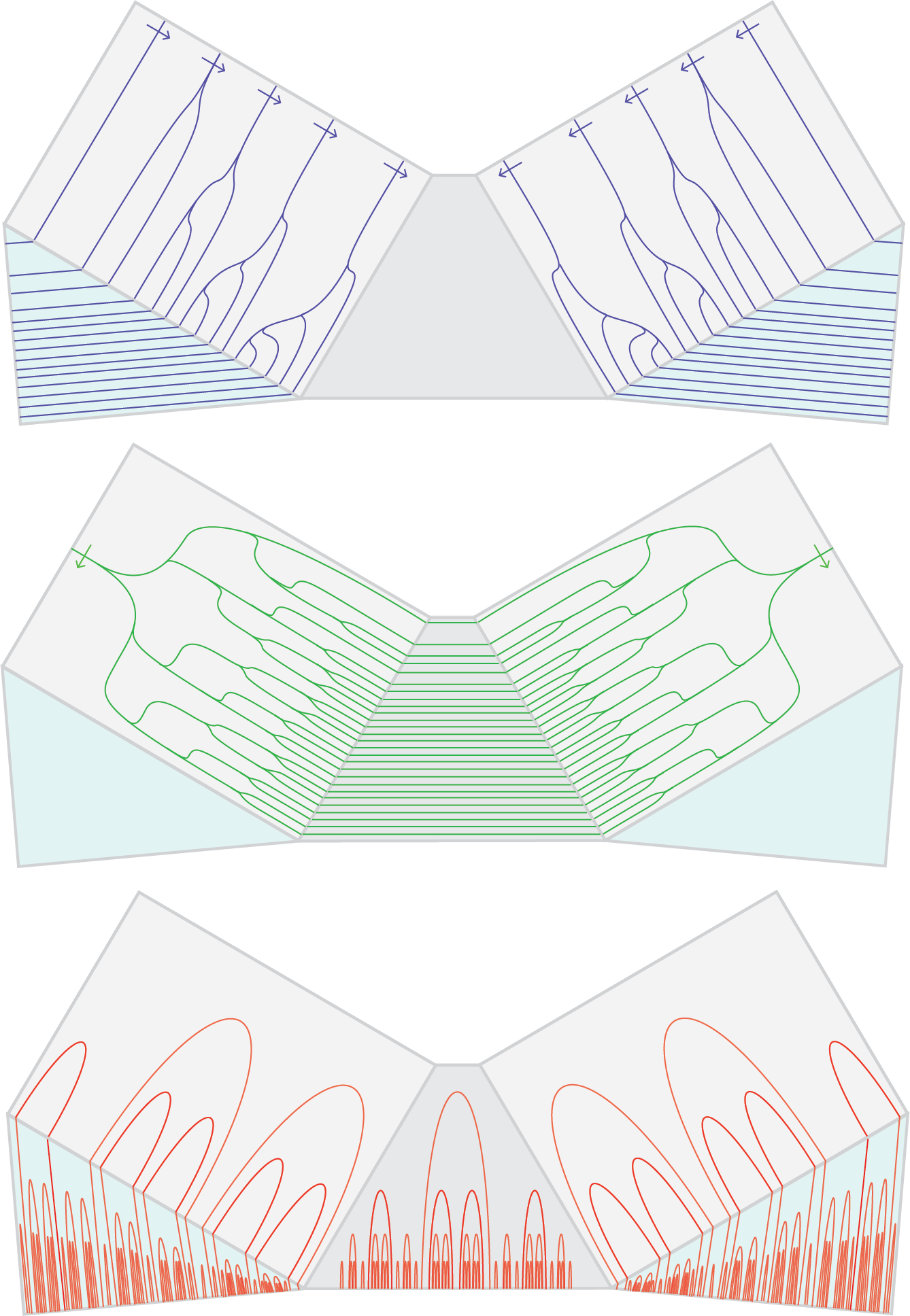}
\end{overpic}
\vspace{4mm}
\caption{Top, middle, lower: $a$-tracks, $b$-tracks, and $t$-tracks through the diagram $\Delta$ of  Figure~\ref{fig:G2Diagram}. The lower diagram is intended only to convey the nesting pattern of the  $t$-tracks. The pattern expands too rapidly towards $\chi$ to be displayed accurately. }
\label{fig:tracks_examples}
\end{figure}

\begin{enumerate}
\item \emph{Noise in $\Delta$ must flow towards $\chi$ and orthogonally to tracks.} This refers to the propagation of (``noise'') letters $t$,  $x_1$, $x_2$, and $y_1$, and $y_2$  through $\Delta$.    Figures~\ref{fig:QDiagram}, \ref{fig:G2Diagram} and~\ref{fig:tracks_examples} show \emph{tracks} through the various diagrams we constructed above.    Introduced in Section~\ref{sec:tracks}, tracks are generalizations of \emph{corridors}.  We will be concerned with four types: $a$-tracks, $b$-tracks, $t$-tracks, and $(a_2,b_q)$-tracks. 

An $a$-track is a path in the dual of $\Delta$ that crosses successive edges labelled by $a$-letters (meaning $a_1$ and $a_2$).  A $b$-track is the same, but for edges labelled by $b_0, \ldots, b_p$.  A $t$-track crosses $t$-edges---the use of $t$ is a distinctive feature of Wise's version of the Rips construction;  it renders the group an HNN-extension of a free group, with $t$ the stable letter (see Proposition~\ref{prop:HNNWise}).  This extra structure, manifested in the geometry of $t$-tracks, facilitates analysis of $G$.  We will describe $(a_2,b_q)$-tracks in (2) below.   As there are three $a$-letters or three $b$-letters in some of the defining relators, $a$-tracks and $b$-tracks can branch.    

	As noise advances  across successive tracks it increases exponentially in length.  
	A consequence of the small-cancellation condition enjoyed by the Rips words used in the defining relators  is that noise cannot substantially cancel within a diagram---it must instead emerge on the boundary.
	  	Therefore, if we assume that $w$ is of minimal length among all words on the generators of $G$ that equal $\chi$ in $G$, then almost all this noise must emerge in $\chi$.  If many noise letters emerge in $w$, then their blow up en route there would result in it being possible to cut a subdiagram out of $\Delta$ to get a new diagram that demonstrated a shorter word than $w$ equals $\chi$ in $G$.    

	This also has helpful consequences for the orientation of tracks---in ways made precise in Lemma~\ref{lem: Layout lemma}. In short, they must be oriented towards $\chi$ because otherwise they would act as blockades for the flow of noise.

\item  \emph{$(a_2, b_q)$-tracks.} These are paths through van~Kampen diagrams that cross successive $a_2$- and $b_q$-edges.  They are the subject of Section~\ref{sec: a2bq tracks}.  Examples are found in Figures~\ref{fig:QDiagram} and \ref{fig:G2Diagram}.  In most defining relators of $G$ there are either zero or two   $a_2$-letters, and ditto for $b_q$-letters.  If an $(a_2, b_q)$-track enters a 2-cell labelled by such a relator across  an $a_2$-edge, then it exists across the other $a_2$-edge, and ditto for $b_q$-edges.        
However our presentation for $G$ has a defining relator ($r_{1, q-1}$ of Figure~\ref{fig:relations}) with   one $a_2$-letter and one $b_q$-letter, and a defining relator  ($r_{2, q}$ of Figure~\ref{fig:relations}) that has two $a_2$-letters and two $b_q$-letters.  On entering the 2-cell of the former type across its $a_2$-edge it exits across its $b_q$-edge (or vice versa). On entering a 2-cell of the latter type  across an $a_2$-edge (resp.\ $b_q$-edge), it exits across the $b_q$-edge (resp.\ $a_2$-edge) that is oriented the same way. These conventions ensure that every $a_2$- and $b_q$-edge in a van~Kampen diagram over $G$ is crossed by exactly one $(a_2, b_q)$-track,  no $(a_2, b_q)$-track can cross itself, and  no two $(a_2, b_q)$-tracks can cross each other.  So $(a_2, b_q)$-tracks  associate to every $b_q$-edge in a diagram $\Delta$ a pair of edges labelled by $a_2$ or $b_q$ on the boundary.  

If the automorphism $\varphi$  gives $\sim \! n^{p}$ growth within $\Delta$, then it creates $\succeq \! n^{q}$ $b_q$-edges within $\Delta$.  It turns out it does so in such a way that  $\succeq \! n^{q}$ of these $b_q$-edges have distinct  $(a_2, b_q)$-tracks through them.  And because those   $(a_2, b_q)$-tracks all run to the boundary,    the length of $w$ must be $\succeq \! n^{q}$.           

\item   \emph{$x$- versus $y$-noise, and $b_0$-tracks.} It is significant that our generating set for $H$ consists of the noise letters $t$, $y_1$, $y_2$ but omits $x_1$ and $x_2$.  It is possible for $x$-noise  to flow across $b$-tracks but impossible for $y$-noise.  And $x$-noise becomes $y$-noise when (and only when) it crosses $b_0$-tracks (particular examples of $b$-tracks).  This means that stacks of nested $b$-tracks must include at most one $b_0$-track and that $b_0$-track  must be the closest to $\chi$.

\end{enumerate}

In Section~\ref{sec:upper} we use these ideas to reduce the problem of bounding $|\chi|$ from above to establishing an inequality concerning the quotient $Q$ of \eqref{eqn QQQ} (specifically, we reduce it to Lemma~\ref{lem:Qp/q}), and this is where the ``$n^{p/q}$'' in our distortion functions is ultimately established, as we explain in Section~\ref{sec:p/q}. Combined with the blow-up that comes from the flow of noise through $\Delta$, it gives our $2^{n^{p/q}}$ upper bound on the distortion of $H$ in $G$.

We leverage our examples to get iterated exponential distortion functions and complete our proof of  Theorem~\ref{main} in Section~\ref{sec:general k}.  The strategy is to  amalgamate $G$ with a chain of hyperbolic free-by-free groups following Brady and Tran \cite{BrT}, and then prove and apply a combination theorem for the hyperbolicity of amalgams.

In Section~\ref{sec:Realizing others} we show  that  the distorted subgroup $H$ need not be free of rank 3, but rather can be taken to be any torsion-free non-elementary hyperbolic group, proving Theorem~\ref{thm:distorted hyp grp}. For this we establish the existence (in Lemma~\ref{lem:there are free subgroups}, after \cite{Kap99}) of undistorted free subgroups of any rank in torsion-free non-elementary hyperbolic groups, apply the same combination theorem to amalgamate these with our examples in a new hyperbolic group, and then we prove the estimates on the distortion function by means of an appropriate general theorem (Theorem~\ref{thm:amalgam distortion}) concerning distortion in amalgams.

\subsection{Acknowledgements} We are grateful to Ilya~Kapovich and Mahan~Mj for suggesting that we promote Theorem~\ref{main} to Theorem~\ref{thm:distorted hyp grp}, and to Jason~Manning for guidance on the associated literature.   We also thank an anonymous referee for a generously thoughtful and detailed reading.

\section{Our groups} \label{ch:our groups}
 
\subsection{The definition} \label{sec:the defn}

Here we will define the group $G$ which will prove Theorem~\ref{main} in the case  $k=1$.  In Section~\ref{sec:general k} we will explain how the case $k=1$ leads to the result for other $k$.  

We fix  integers $p>q>0$. Then $G$ has presentation 
$$
\mathcal{P} \ = \ \langle \; a_1, a_2,   b_0, \ldots, b_p, t, x_1, x_2, y_1, y_2 \;\mid\; \mathcal R 
 \; \rangle
$$
where $\mathcal R$ is the set of $5p+11$ defining relators displayed in Figure~\ref{fig:relations}.  Our notation $\Xb$ and $\Yb$ is intended to indicate indexing that we have chosen to suppress. 
   Every element of $\mathcal R$ is a word of the form $t^{-1} u t  v^{-1}$ where  $u$ and $v$ are words on generators other than $t$.  Each has two or three \emph{Rips subwords},  
    denoted $\Xb$ or $\Yb$,
   from sets $\mathcal{X}    =   \set{X_1, X_2, \ldots, X_{14p}}$  and  $\mathcal{Y}   =    \set{Y_1, Y_2, \ldots, Y_{30}}$ of pairwise disjoint subwords  of  the infinite Rips words
$x_1 x_2^1 \, x_1 x_2^2 \,  x_1 x_2^3  \cdots$ 
and
$y_1 y_2^1 \, y_1 y_2^2 \,  y_1 y_2^3  \cdots$, respectively, chosen in a manner we will explain momentarily.
   We stress that each $\Xb$ and $\Yb$ occurs  once in $\mathcal{P}$  and does so as a subword of one defining relator.  So,  if an $\Xb$ or $\Yb$ can be read around a portion of the boundary circuit of  a 2-cell in a van~Kampen diagram (see Section~\ref{sec:tracks}) over $\mathcal{P}$, then that Rips word uniquely determines the defining relator that 2-cell corresponds to.   
This use of $t$ and  Rips words is  a variation  on Wise's  \cite{Wise1} HNN-version  of Rips' Construction \cite{Rips}.  (Our example $G$ departs in some respects from Wise's framework. Wise has two  $\Xb$ subwords in each defining relator, has only two `noise' generators $x_1$ and $x_2$, and has   additional defining relators that ensure that $\langle t, x_1, x_2  \rangle$ is a normal subgroup.)

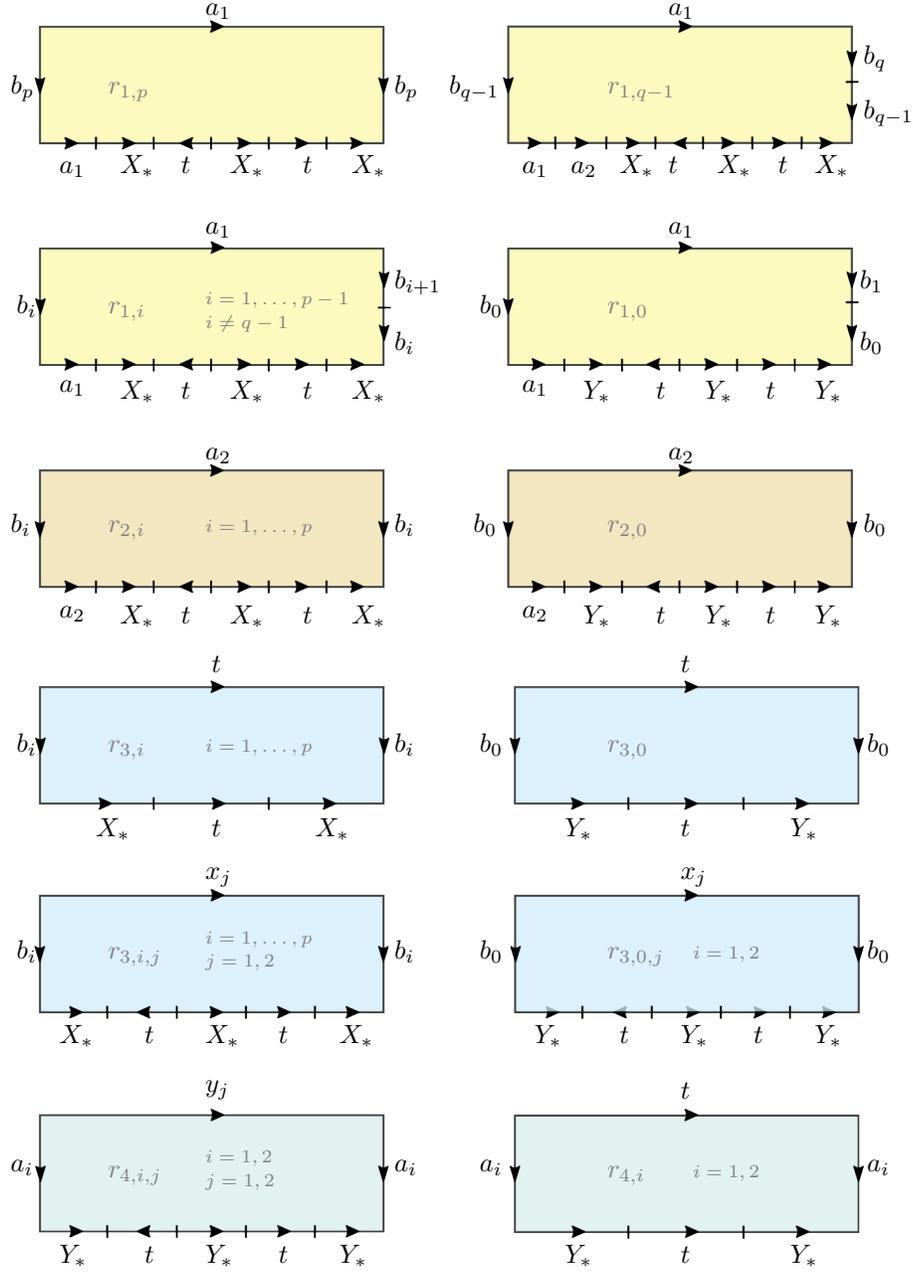
\begin{figure}[htbp]
\centering
 \begin{overpic}  
{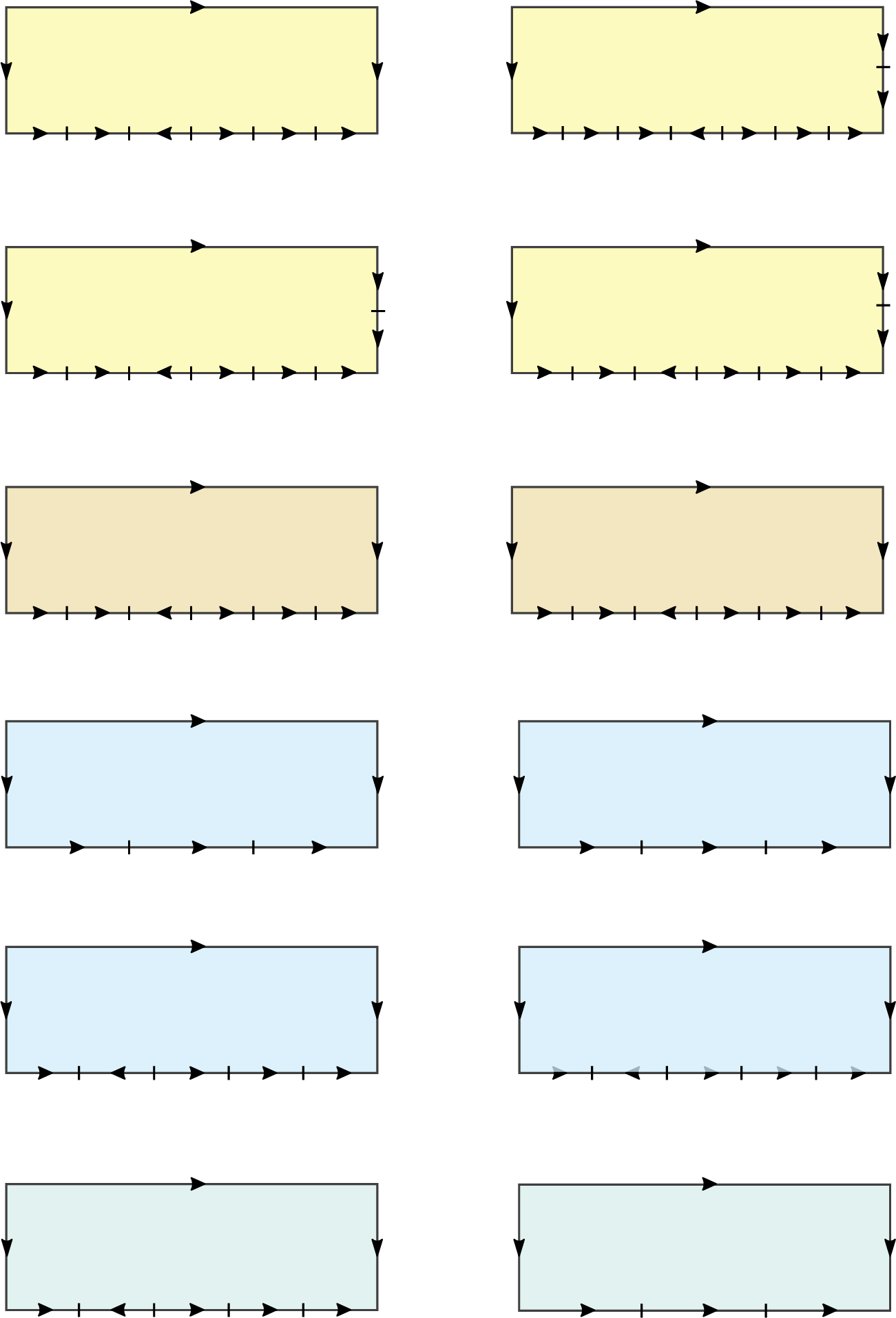}

\put(6, 94){\small{\gray{$r_{1,p}$}}}
\put(47, 94){\small{\gray{$r_{1,q-1}$}}}
\put(6, 76){\small{\gray{$r_{1,i}$}}}
\put(47, 76){\small{\gray{$r_{1,0}$}}}
\put(6, 58){\small{\gray{$r_{2,i}$}}}
\put(47, 58){\small{\gray{$r_{2,0}$}}}
\put(6, 40){\small{\gray{$r_{3,i}$}}}
\put(47, 40){\small{\gray{$r_{3,0}$}}}
\put(6, 23){\small{\gray{$r_{3,i,j}$}}}
\put(47, 23){\small{\gray{$r_{3,0,j}$}}}
\put(6, 5){\small{\gray{$r_{4,i,j}$}}}
\put(47, 5){\small{\gray{$r_{4,i}$}}}

\put(14, 75.5){\parbox{25mm}{\tiny{\gray{$ i =    1, \ldots, p-1$}} \\  \tiny{\gray{$i \neq q-1$}}}}
\put(14, 40){\tiny{\gray{$i= 1, \ldots, p$}}}
\put(14, 23){\parbox{25mm}{\tiny{\gray{$i =1, \ldots, p$}} \\  \tiny{\gray{$j =1, 2$}}}}
\put(14, 5){\parbox{25mm}{\tiny{\gray{$i =1, 2$}} \\ \tiny{\gray{$j =1, 2$}}}}
\put(54, 5){\tiny{\gray{$i =1, 2$}}}
\put(54, 23){\tiny{\gray{$j =1, 2$}}}
\put(14, 58){\tiny{\gray{$i= 1, \ldots, p$}}}

\put(14, 100.5){\small{$a_1$}}
\put(52, 100.5){\small{$a_1$}}
\put(14, 82.5){\small{$a_1$}}
\put(52, 82.5){\small{$a_1$}}
\put(2, 87.5){\small{$a_1$}}
\put(2, 69.5){\small{$a_1$}}
\put(40, 69.5){\small{$a_1$}}
\put(2, 51){\small{$a_2$}}
\put(-2, 5.5){\small{$a_i$}}
\put(29.5, 5.5){\small{$a_i$}}
\put(40, 51){\small{$a_2$}}
\put(14, 64){\small{$a_2$}}
\put(52, 64){\small{$a_2$}}
\put(36.5, 5.5){\small{$a_i$}}
\put(68.3, 5.5){\small{$a_i$}}
\put( 40, 87.5){\small{$a_1$}}
\put( 44, 87.5){\small{$a_2$}}
\put(14.5, 46.5){\small{$t$}}
\put(53, 46.5){\small{$t$}}
\put(14, 29.5){\small{$x_j$}}
\put(53, 29.5){\small{$x_j$}}
\put(14, 12){\small{$y_j$}}
\put(53, 11.5){\small{$t$}}
\put(-2, 94){\small{$b_p$}}
\put(29.5, 94){\small{$b_p$}}
\put(68, 96.5){\small{$b_q$}}
\put(29.5, 78){\small{$b_{i+1}$}}
\put(-1.5, 76){\small{$b_i$}}
\put(-1.5, 40){\small{$b_i$}}
\put(36.5, 40){\small{$b_0$}}
\put(68.3, 40){\small{$b_0$}}

\put(67.7, 78){\small{$b_1$}}
\put(67.7, 73){\small{$b_0$}}

\put(34, 94){\small{$b_{q-1}$}}
\put(68, 92){\small{$b_{q-1}$}}
\put(29.5, 40){\small{$b_{i}$}}
\put(29.5, 73){\small{$b_{i}$}}
\put(36, 58){\small{$b_0$}}
\put(68, 58){\small{$b_0$}}

\put(29.5, 58){\small{$b_i$}}
\put(-2, 58){\small{$b_i$}}
\put(-1.5, 23){\small{$b_i$}}
\put(29.5, 23){\small{$b_i$}}
\put(36.5, 23){\small{$b_0$}}
\put(36.5, 76){\small{$b_0$}}
\put(68.3, 23){\small{$b_0$}}
\put(7, 87.5){\small{$X_{\ast}$}}
\put(12, 87.5){\small{$t$}}
\put(16, 87.5){\small{$X_{\ast}$}}
\put(22, 87.5){\small{$t$}}
\put(26, 87.5){\small{$X_{\ast}$}}
\put(48, 87.5){\small{$X_{\ast}$}}
\put(52, 87.5){\small{$t$}}
\put(56, 87.5){\small{$X_{\ast}$}}
\put(61, 87.5){\small{$t$}}
\put(64, 87.5){\small{$X_{\ast}$}}
\put(7, 69){\small{$X_{\ast}$}}
\put(12, 69){\small{$t$}}
\put(16, 69){\small{$X_{\ast}$}}
\put(22, 69){\small{$t$}}
\put(26, 69){\small{$X_{\ast}$}}
\put(45, 69){\small{$Y_{\ast}$}}
\put(51, 69){\small{$t$}}
\put(55, 69){\small{$Y_{\ast}$}}
\put(60, 69){\small{$t$}}
\put(64, 69){\small{$Y_{\ast}$}}
\put(7, 50.5){\small{$X_{\ast}$}}
\put(12, 50.5){\small{$t$}}
\put(16, 50.5){\small{$X_{\ast}$}}
\put(22, 50.5){\small{$t$}}
\put(26, 50.5){\small{$X_{\ast}$}}
\put(45, 50.5){\small{$Y_{\ast}$}}
\put(51, 50.5){\small{$t$}}
\put(55, 50.5){\small{$Y_{\ast}$}}
\put(60, 50.5){\small{$t$}}
\put(64, 50.5){\small{$Y_{\ast}$}}
\put(5, 33){\small{$X_{\ast}$}}
\put(14.5, 33){\small{$t$}}
\put(23, 33){\small{$X_{\ast}$}}
\put(43.5, 33){\small{$Y_{\ast}$}}
\put(53, 33){\small{$t$}}
\put(62, 33){\small{$Y_{\ast}$}}
\put(2, 16){\small{$X_{\ast}$}}
\put(9, 16){\small{$t$}}
\put(14, 16){\small{$X_{\ast}$}}
\put(20, 16){\small{$t$}}
\put(25, 16){\small{$X_{\ast}$}}
\put(41, 16){\small{$Y_{\ast}$}}
\put(48, 16){\small{$t$}}
\put(53, 16){\small{$Y_{\ast}$}}
\put(59, 16){\small{$t$}}
\put(64, 16){\small{$Y_{\ast}$}}
\put(2, -2){\small{$Y_{\ast}$}}
\put(9, -2){\small{$t$}}
\put(14, -2){\small{$Y_{\ast}$}}
\put(20, -2){\small{$t$}}
\put(25, -2){\small{$Y_{\ast}$}}
\put(43.5, -2){\small{$Y_{\ast}$}}
\put(53, -2){\small{$t$}}
\put(62, -2){\small{$Y_{\ast}$}}

\end{overpic}
\vspace{4mm}
\caption{Defining relators for our group $G$}
\label{fig:relations}
\end{figure}

Suppose  $S$ is a set of words on $A \cup A^{-1}$ for some alphabet $A$.  A cyclic conjugate of a word $w$ is a word  $s_2 s_1$ such that $s_1$ is a prefix of $w$ and $s_2$ a suffix such that $s_1 s_2 = w$.   Let $\mathcal{C}(S)$ be the set of all cyclic conjugates of words in $S^{\pm 1}$.  Assume that all elements of $\mathcal{C}(S)$  are reduced.     A \emph{piece} is a common prefix $\pi$ of a pair of distinct words $\pi u$ and $\pi v$ in $\mathcal{C}(S)$.

We choose the Rips  subwords  $\Xb$ and $\Yb$ so that each has length at least 100 and we have:  
 \begin{itemize}
\item[i.] \emph{The uniform $C'(1/6)$-condition for $\mathcal{R}$.}  Every piece has length  strictly less than a sixth of the length of the shortest relator in $\mathcal R$.  

\item[ii.] \emph{The $C(3)$-condition for  the union $S$ of the 3-  and 5-element  generating sets of the     terminal vertex groups of Table~\ref{table:hnn}.} No element of $\mathcal{C}(S)$ is a concatenation of fewer than $3$ pieces.

\item[iii.] \emph{The $C'(1/4)$ condition for the set of Rips words $\mathcal{X} \cup \mathcal{Y}$}.  Every piece has length  strictly less than a quarter of the length of each element of  $\mathcal{C}(\mathcal{X} \cup \mathcal{Y})$ in which it occurs.

\item[iv.]  \emph{The $C(5)$-condition for $\mathcal{U} =  \set{ \, u, v \;\mid\; t^{-1} u t  v^{-1}  \in \mathcal{R} \, }$}. 
No element of  $C(\mathcal{U})$ is a concatenation of fewer than $5$ pieces.   
\end{itemize}

This can be achieved for instance by adapting the example of \cite[Remark~3.2]{Wise1} so that $\mathcal{X}$ is the set of words  $$X_{i} \  : =  \ x_1 x_2^{200ip}x_1 x_2^{200ip +1} \cdots x_1 x_2^{200ip +200p -1}$$  
for $1 \leq i \leq    14p$  
and $\mathcal{Y}$ is the set  of words $$Y_{i} \  : =  \ y_1 y_2^{200ip}y_1 y_2^{200ip +1} \cdots y_1 y_2^{200ip +200p -1}$$   for $1 \leq i \leq   30$.    
 Then $\mathcal{R}$ satisfies  $C'(1/6)$ because the longest pieces in $\mathcal R$ have the form $x_2^{\alpha-1} x_1 x_2^{\alpha}$ or $y_2^{\alpha-1} y_1 y_2^{\alpha}$ (or the inverse thereof) for some $\alpha \in \N$. 
The longest piece appears either in $X_{14p}$ with $\alpha = 200 (14p)+200p-2$ or in $Y_{30}$ with $\alpha = 200 (30) +200p -2$.  Its length is $2\alpha$, which (in either case, since $p >1$) is strictly less than $12,400p$.   On the other hand, the shortest defining relator has length at least $2|X_1|$ (see Figure~\ref{fig:relations}) which is certainly bigger than $80,000p^2$, and this number is already bigger than six times $12,400p$. Conditions~ii--iv hold similarly.

Condition~i is used in the next paragraph and will be used to achieve $\textup{CAT}(-1)$ in Remark~\ref{lem:CAT0 and CAT-1 structures}.  Condition~ii will be used in Lemma~\ref{lem: C(3)} towards establishing HNN-structures for $G$.  Condition~iii will restrict cancellation in Section~\ref{sec:lower}, where we prove a lower bound on distortion, and in Sections~\ref{sec:consequences of small-cancellation} and \ref{sec:tracks in reduced diagrams}, towards showing certain configurations of tracks do not arise in reduced diagrams.  Condition~iv achieves residual finiteness as we now explain.

All $C'(1/6)$ groups satisfy a linear isoperimetric inequality and so are hyperbolic \cite{Gersten9}. By \cite{WiseCubulating} they are cubical, and then, by \cite{Agol}, they are  virtually special, and so are residually finite. Their residually finiteness is more directly apparent via \cite[Theorem~2.1]{Wise1}, given the $C(5)$-condition for $\mathcal{U}$.

Our distorted subgroup is  $$H \ = \ \langle t, y_1, y_2 \rangle.$$

\subsection{Consequences of small-cancellation} \label{sec:consequences of small-cancellation}

Here we give three lemmas that are proximate consequences of the small-cancellation conditions in Section~\ref{sec:the defn}. 

Part~\eqref{lem part:vertex groups free S} of the first of these lemmas will be used in our proof of   Proposition~\ref{prop:hnn}. Part~\eqref{lem part:vertex groups free U}  will imply  Proposition~\ref{prop:HNNWise}. We prove it using the $C(3)$-condition for $\mathcal{U}$, which is weaker than the $C(5)$-condition we have for $\mathcal{U}$ in Section~\ref{sec:the defn}.  It is a special case of \cite[Theorem~2.11]{Wise2}, but we  include our own proof  here because the result is central to our argument and the following short argument is available in our context.

\begin{lemma} (Cf.\ \cite[Theorem~2.11]{Wise2})  \label{lem: C(3)} ${}$ \\[-12pt]  
\begin{enumerate}
\item \label{lem part:vertex groups free S} 
Let $S$ be the union of the  3- and 5-element  generating sets of the terminal vertex groups of Table~\ref{table:hnn} (that is, $S$ is  the set of all words appearing in the final column). Then $S$  freely generates a free subgroup of the free group $F = F(\mathcal{A})$, where $\mathcal{A} = \set{a_1, a_2, t, x_1, x_2, y_1, y_2}$.   	
\item \label{lem part:vertex groups free U} The set $$\mathcal{U} \ = \ \set{ \, u, v \;\mid\; t^{-1} u t  v^{-1}  \in \mathcal{R} \, }$$ freely generates a free subgroup in the free group $$F \ = \  F(a_1, a_2,   b_0, \ldots, b_p, x_1, x_2, y_1, y_2).$$ 
\end{enumerate}
\end{lemma}

\begin{proof} 
Both parts are instances of the same general result, which we will prove here in the notation of part \ref{lem part:vertex groups free S}. Suppose $w_1, \ldots, w_m \in S^{\pm 1}$  are such that $W = w_1 \cdots w_m$ is a non-empty reduced word on $S$ but $W$ freely reduces to the empty word when viewed as a word on the generators of $F$.  We will show that the existence of this $W$ contradicts $C(3)$.   

There is a planar tree $T$ whose edges are directed and are labelled by generators of $F$ so that around the perimeter of $T$ we read $W$.  As each $w_i$ is a  reduced  word on $\mathcal{A}$, the portion of the perimeter of $T$ along which one reads $w_i$ can only include a leaf of $T$ at its start or end.  It follows that if $T$ is a line, then the shorter of $w_1$ and $w_m^{-1}$ is subword of the other, and so is a piece,   contrary to $C(3)$. 

Assume, then, that $T$ is not a line.  There must be a pair of leaves $v_1$ and $v_2$ in $T$ such that the geodesic $\rho$ from $v_1$ to $v_2$ visits exactly one branching (i.e.\ valence at least 3) vertex $b$.    So the word $u$ one reads along $\rho$ is $w_j \cdots w_k$ for some $1 \leq j \leq k \leq m$.  In the remainder of our argument, read indices modulo $m$. The portion of $\rho$ along which we read $w_j$ must pass $b$ else whichever of $w_{j-1}$ and $w_j$ is shorter would be a piece.  And, in fact, then $w_j$ must be $u$, else $w_k$ or $w_{k+1}$ would be a piece.  So $j=k$.           But then, as neither  $w_{j-1}^{-1}$ nor $w_{k+1}^{-1}$ can be a subword of $w_j$ (else they would be pieces),  $w_j$ must be concatenation of two pieces: one that it shares with $w_{j-1}^{-1}$ and one that it shares with $w_{k+1}^{-1}$.  Again, this is    contrary to $C(3)$.   
  \end{proof}

In our next lemma, a stronger small-cancellation hypothesis allows the same conclusion for further subsets of free groups.  We will  call on  it in Lemma~\ref{K0Li} en route to our proof of  Proposition~\ref{prop:hnn}.

\begin{lemma} \label{lem: free subgroup C'(1/4)}
Suppose 
$Z_1, Z_2, Z_3,  Z'_1, Z'_2, Z'_3,Z_{p1}, Z_{p2}, Z_{p3}, Z_{p4}, Z_{p5}$
are words of the form $\Yb t^{-1} \Yb t \Yb$  or $\Yb t \Yb$ and each is a subword of a different 
defining relation from Figure~\ref{fig:relations} (so no $\Yb$ appears twice).  We will refer to these as $Z$-words. Then  
$$\mathcal{S}_1 \ = \   \set{t, x_1, x_2, Z_1, Z_2, Z_3,  Z'_1, Z'_2, Z'_3}$$
freely generate a free subgroup of   $F=  F(t, x_1, x_2, y_1, y_2)$.
The same is true of  
 $$\mathcal{S}_2   \ = \   \set{Z_1, Z_2, Z_3,  Z'_1, Z'_2, Z'_3, Z_{p1}, Z_{p2}, Z_{p3}, Z_{p4}, Z_{p5}}.$$
\end{lemma}

 \begin{proof}  Suppose for a contradiction that $w$ is a reduced word on $\mathcal{S}_1$ or $\mathcal{S}_2$ that represents the identity in $F$ and includes at least one of the $Z$-words. Express each   $Y_{\ast}$ as the concatenation $P_{\ast} S_{\ast}$ of a prefix and a suffix whose lengths differ by at most one.

Consider a first $P_{\ast}^{\pm 1}$ or $S_{\ast}^{\pm 1}$ that is completely cancelled away on freely reducing $w$ in $F$ by removing successive inverse pairs of adjacent letters. It must have cancelled into a neighbouring $P_{\ast}^{\pm 1}$ or $S_{\ast}^{\pm 1}$. But then, because of the  $C'(1/4)$-condition on the set of Rips words $\mathcal{X} \cup \mathcal{Y}$, some neighbouring pair of $Z$-words are inverses, contrary to $w$ being reduced as a word on $\mathcal{S}_1$ or $\mathcal{S}_2$.    
      \end{proof}

We will use the following variation on Lemma~\ref{lem: free subgroup C'(1/4)} in our proof of Lemma~\ref{lem: trapped x-noise}.

\begin{lemma} \label{lem: cancellation from C'(1/4)}
Suppose $$\overline{v} \ = \  x_{\lambda_0}^{\epsilon_0} X_{\xi_1}^{\mu_1} x_{\lambda_1}^{\epsilon_1} \cdots  X_{\xi_m}^{\mu_m} x_{\lambda_m}^{\epsilon_m}$$ is a word on $\mathcal{X} \cup \set{x_1, x_2}$   in which $m \geq 1$, each  $X_i \in  \mathcal{X}$, each $\lambda_i \in \set{1,2}$, each $\mu_i \in \set{\pm 1}$, and each $\epsilon_i \in \set{0, \pm 1}$. If $\overline{v}$  freely equals the empty word in $F(x_1, x_2)$, then for any sequence $\Sigma$ of free-reduction moves (successive removals of $x_j^{\pm 1} x_j^{\mp 1}$ subwords) that takes $\overline{v}$ to the empty word, there is some $i$ such that a subword consisting of at least a quarter of the letters of $X_{\xi_i}^{\mu_i}$   cancels with subword consisting of at least a quarter of the letters of    $X_{\xi_{i+1}}^{\mu_{i+1}}$.   
\end{lemma}

 \begin{proof}
Express each word $X_{\xi_i}^{\mu_i}$ as the concatenation $P_i S_i$ of a prefix and a suffix whose lengths differ by at most one.    Let $i$ be the index of a first $P_i$ or $S_i$ to be completely cancelled away in the course of $\Sigma$.  Assume it is $S_i$. (The argument for $P_i$ is essentially the same.)  Then $S_i$ cancels with a prefix of $x_{\lambda_i}^{\epsilon_i}  X_{\xi_{i+1}}^{\mu_{i+1}}$.     But then, $C'(1/4)$ and the fact that the $X_{\ast}$ all have length at least $100$ together imply the result. 
   \end{proof}

\subsection{Van~Kampen diagrams, corridors, and tracks} \label{sec:tracks}

Suppose $w$ is a word on the generators of a group which is given by a presentation.  A \emph{van~Kampen diagram}  for $w$ with respect to that presentation is a finite planar 2-complex in which every edge is directed and labelled by a generator in such a way that around the perimeter of the diagram (in some direction from some starting vertex) one reads $w$ and around the perimeter of each 2-cell (in some direction from some starting vertex) one reads a defining relator. A word $w$ admits a van~Kampen diagram if and only  if it represents the identity in the group.   Many introductory  texts discuss van~Kampen diagrams---e.g., \cite{BrH}.   

\begin{definition} \textbf{(Reduced diagrams)} A van~Kampen diagram is \emph{reduced} when it does not contain a pair of   back-to-back cancelling cells---that is, a pair of cells with a common edge $e$ such that the word read clockwise around the perimeter of one of these cells starting from $e$  is the same as that read anticlockwise around  the other starting from $e$.  
\end{definition}

 \begin{definition} \textbf{(Corridors)}
 Suppose $z$ is  a generator.  Suppose $C_1$, \ldots, $C_m$ is a maximal set of   distinct 2-cells in a van~Kampen diagram $\Delta$ such that for all $i$, around $\partial C_i$ one reads a word $u_iz v_i^{-1} z^{-1}$ and
 the $z$ in $\partial C_i$ is the $z^{-1}$ in $\partial C_{i+1}$.  
 Then the $C_1$, \ldots, $C_m$ concatenate in $\Delta$ to form an    \emph{$z$-corridor} $\mathcal{C}$, as shown in Figure~\ref{fig:corridor}.   A $z$-edge in $\partial \Delta$ that is not part of the boundary of a 2-cell is a corridor with no 2-cells.

\begin{figure}[ht]
\centering
\centering
\begin{overpic} 
{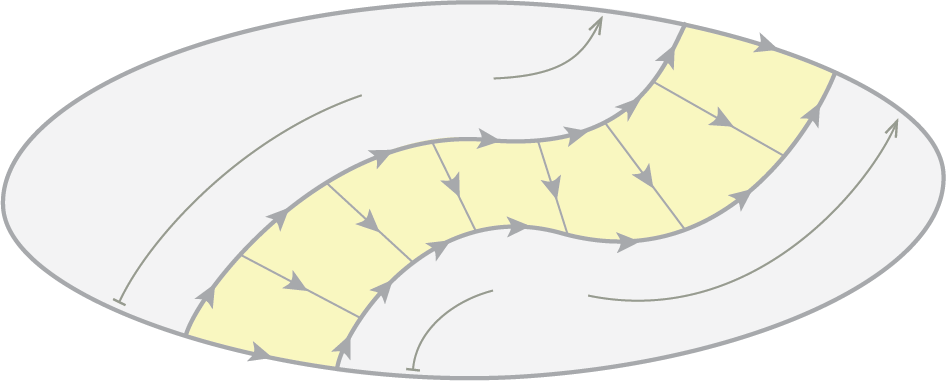}
\put(26, -0.5){\small{$z$}}
\put(32, 10.5){\small{$z$}}
\put(40, 16.5){\small{$z$}}
\put(49, 19.5){\small{$z$}}
\put(59.5, 20){\small{$z$}}
\put(69, 22){\small{$z$}}
\put(77, 28){\small{$z$}}
\put(81.5, 36.7){\small{$z$}}
\put(27, 5.5){\gray{\tiny{$C_1$}}}
\put(34, 13){\gray{\tiny{$C_2$}}}
\put(43, 18){\gray{\tiny{$C_3$}}}
\put(52, 19.5){\gray{\tiny{$C_4$}}}
\put(63, 19.2){\gray{\tiny{$C_5$}}}
\put(72, 21){\gray{\tiny{\reflectbox{$\ddots$}}}}
\put(61, 31){{\tiny{\reflectbox{$\ddots$}}}}
\put(78, 15){{\tiny{\reflectbox{$\ddots$}}}}
\put(78.3, 31.5){\gray{\tiny{$C_m$}}}
\put(17, 10){\small{$u_1$}}
\put(25, 19){\small{$u_2$}}
\put(36, 25.5){\small{$u_3$}}
\put(50, 27){\small{$u_4$}}
\put(57, 28){\small{$u_5$}}
 \put(64.4, 36){\small{$u_m$}}
\put(38, 2.5){\small{$v_1$}}
\put(41.5, 7.5){\small{$v_2$}}
\put(46.5, 10.5){\small{$v_3$}}
\put(53.5, 12.5){\small{$v_4$}}
\put(64, 11){\small{$v_5$}}
 \put(87, 25.5){\small{$v_m$}}
 \put(40, 30.5){\gray{\tiny{bottom}}}
 \put(54.5, 8.5){\gray{\tiny{top}}}
\end{overpic}
\caption{A corridor in a van~Kampen diagram.}
\label{fig:corridor}
\end{figure}
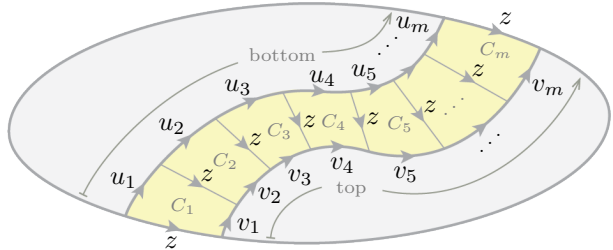

An assumption commonly made when defining corridors is that every defining relator containing a $z$ or $z^{-1}$, contains exactly one $z$ and one $z^{-1}$. Then  $z$-corridors cannot cross or self-intersect, and each one either connects a pair of $z$-edges on $\partial \Delta$ or  closes up to form a $z$-\emph{annulus}. In our presentation $\mathcal{P}$ for $G$ this assumption is met by the letters $a_1$, $b_0$, and $t$, but not, for example, by $a_2$, $b_1$, \ldots, or $b_p$: an $a_2$-corridor can terminate at an $r_{1, q-1}$-cell and a $b_i$-corridor, for $i \neq 0$, can terminate at an $r_{1, i-1}$-cell.  
 
The \emph{words along the top and bottom} of  $\mathcal{C}$ are $v_1 \cdots v_m$  and $u_1 \cdots u_m$, respectively. 
 \end{definition}

 We will reframe and generalize the definition of a corridor via the dual of a van Kampen diagram.   
  Let $\Delta^+$ be $\Delta$ with one additional  $2$-cell $e_{\infty}$ ``\emph{at infinity}'' attached along its boundary cycle.  So  $\Delta^+$ is homeomorphic to a 2-sphere.   
    Let $\mathcal{G}^+$ be the 1-skeleton of the 2-complex dual to $\Delta^+$.  Let $\mathcal{G}$ be the graph obtained from ${\mathcal G}^+$ by removing the interior of $e_{\infty}$.  So    the vertex dual to $e_{\infty}$ is absent from $\mathcal{G}$ and instead $\mathcal{G}$ has   a vertex in the middle of every edge in $\partial e_{\infty} = \partial \Delta$.  

While the following definition could be presented in more general terms, we prefer to specialize to van~Kampen diagrams $\Delta$ over our presentation $\mathcal{P}$ for $G$.

\begin{definition} \label{def: tracks}
\textbf{(Tracks, subtracks, and compound tracks)}  An \emph{$a$- or $b$-edge} in a van~Kampen diagram $\Delta$ over $\mathcal{P}$ is an edge labelled by $a_i$ or $b_i$, respectively, for some $i$.  
An \emph{$s$-subtrack} is a path $\rho: [0, k] \to \mathcal{G}$, where $k>0$ is an integer, with the following properties:
\begin{enumerate}
\item For each integer $i$ in $[0, k-1]$, the image 
$\rho([i, i+1])$ is an edge of $\mathcal G$ dual to an $s$-edge of $\Delta$.
\item All $s$-edges of $\Delta$ dual to $\rho$ are oriented the same way as one travels along $\rho$ (i.e., cross $\rho$ all right-to-left or all  left-to-right). 
\item The map $\rho$ is injective on $(0,k)$. 
\end{enumerate}
 An \emph{$s$-track} is an $s$-subtrack that is maximal---i.e., it cannot be extended to a longer path with properties (1)--(3).  
   For $s = a_1, b_0, \ldots, b_p, t$ an $s$-track traverses the 2-cells of an $s$-corridor.  When $s$ is $a$ or $b$, it gives a more general notion.    
   Figures~\ref{fig:QDiagram}, \ref{fig:G2Diagram} and ~\ref{fig:tracks_examples} show examples of tracks.  As seen in these figures, $a$- or $b$-tracks could merge. We impose a smoothness condition on these merges, which we now discuss.

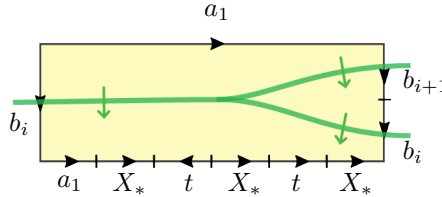
\begin{figure}[htbp]
\centering
\begin{overpic} 
{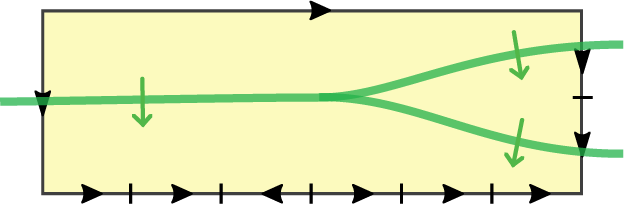}
\put(48, 38){\small{$a_1$}}
\put(11, -5){\small{$a_1$}}
\put(-1, 9){\small{$b_i$}}
\put(98, 2){\small{$b_i$}}
\put(98, 20){\small{$b_{i+1}$}}
\put(25, -6){\small{$X_{\ast}$}}
\put(43, -6){\small{$t$}}
\put(54, -6){\small{$X_{\ast}$}}
\put(70, -6){\small{$t$}}
\put(82, -6){\small{$X_{\ast}$}}
\end{overpic}
\caption{A train-track junction.}
\label{fig:junction}
\end{figure}

Let  $\mathcal{G}_a$ and $\mathcal{G}_b$   be the subgraphs of $\mathcal{G}$ made up of all edges dual to $a$- and $b$-edges, respectively.   We give $\mathcal{G}_a$ and $\mathcal{G}_b$ ``train-track'' structures by rendering some paths in them smooth and others not.  As the defining relators in $\mathcal{P}$ each have zero, two or three $b$-letters, the valence-$1$ vertices of $\mathcal{G}_b$  are precisely those in the interior of $e_{\infty}$.  The valence-$2$ vertices are those dual to 2-cells of $\Delta$ that have (for some $i$) one $b_i$ and one $b_i^{-1}$ in their boundary word. We term the valence-$3$ vertices \emph{junctions}.  They are the vertices dual to 2-cells of $\Delta$ that have (for some $i$) one $b_{i+1}$, one $b_i$, and one $b_i^{-1}$ in its boundary word.   Paths $\gamma$ in $\mathcal{G}_b$ can only fail to be smooth at junctions: per Figure~\ref{fig:junction} we make  $\gamma$ smooth at a junction if and only if the orientations of the $b$-edges  it crosses before and after $v$ agree.  So a $b$-track is a maximal path $\rho: [0, k] \to \mathcal{G}_b$ that is injective and smooth on $(0,k)$.  We will see below that if $\rho $ closes up, then $\rho$ must in fact be a smooth map of a circle into $\mathcal G$.
Corresponding statements apply to  $\mathcal{G}_a$.

\begin{figure}[htbp]
\centering
\begin{overpic}  [scale=0.9] 
{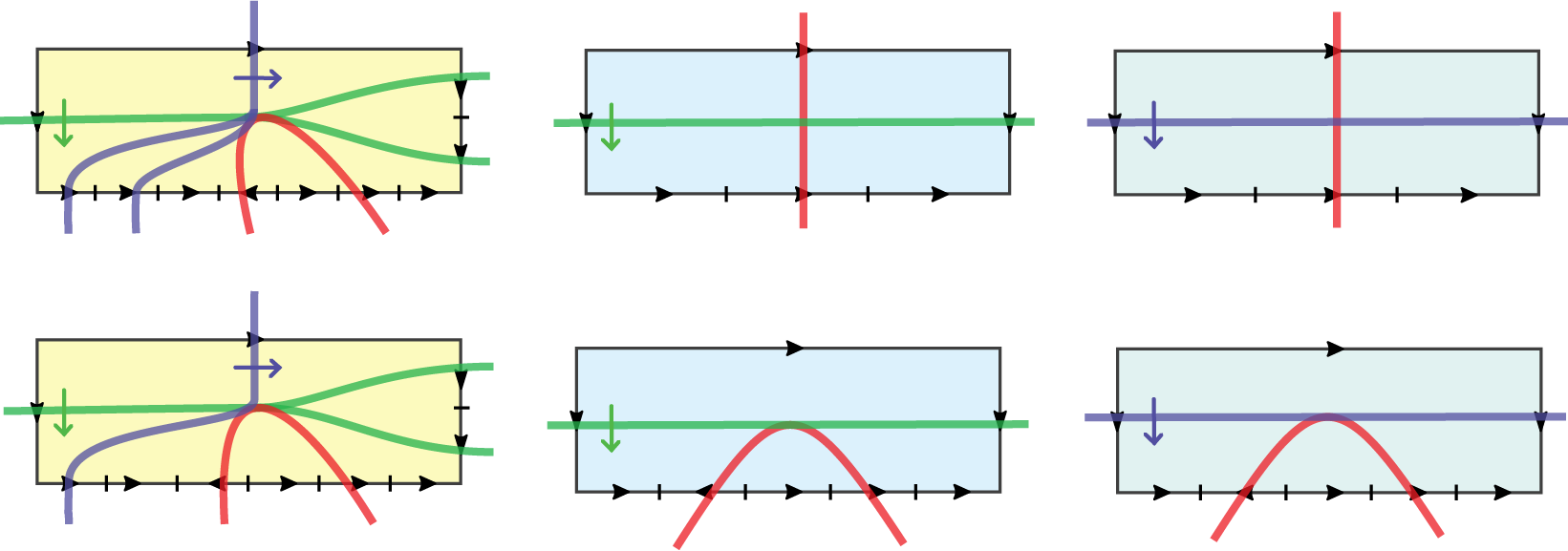}
\put(13, 34){\small{$a_{\ast}$}}
\put(13, 15){\small{$a_{\ast}$}}
\put(21.5, 1){\small{$t$}}
\put(21.5, 20){\small{$t$}}
\put(-1, 10){\small{$b_{\ast}$}}
\put(34, 10){\small{$b_{\ast}$}}
\put(-1, 29){\small{$b_{\ast}$}}
\put(34, 29){\small{$b_{\ast}$}}
\put(68, 10){\small{$a_{\ast}$}}
\put(68, 29){\small{$a_{\ast}$}}
\put(45, 0){\small{$t$}}
\put(48.5, 33){\small{$t$}}
\put(82.5, 33){\small{$t$}}
\put(80, 0){\small{$t$}}
\end{overpic}
\caption{How $a$-tracks, $b$-tracks, and  $t$-tracks  intersect in a 2-cell.  In four of the six cases, the $t$-track through the cell touches but does not cross the other tracks.}
\label{fig:cell_track_intersections}
\end{figure}
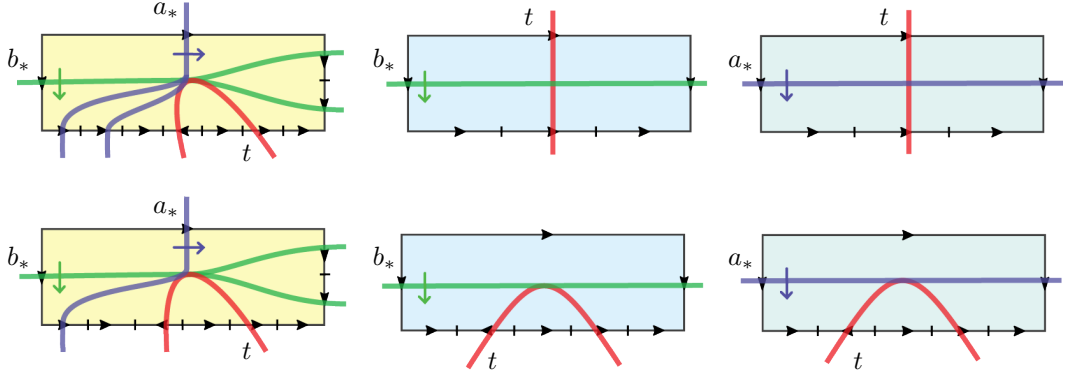

Figure~\ref{fig:cell_track_intersections} shows how we consider $a$-, $b$-, and  $t$-tracks to intersect when they traverse the same 2-cell. 

 A \emph{compound track} is a concatenation of $a$-, $b$-, and $t$-subtracks (the orientations of which are not required to agree).  
The corridor or annulus  associated to a (compound) track $\rho$
 in a van~Kampen diagram $\Delta$  is the subcomplex made up of all the 2-cells through which $\rho$ passes.     There are words along its top and bottom as for a standard corridor as explained above.  
  \end{definition}

 We will see in Section~\ref{sec:tracks in reduced diagrams} that the hypothesis that a  van~Kampen diagram  $\Delta$ over $\mathcal{P}$ is reduced  significantly restricts the behaviours of its tracks.  Then  in Section~\ref{sec: tracks in distortion diagrams} the tracks are yet more sharply restricted in diagrams pertinent to establishing upper bounds on the distortion of $H$ in $G$. Here is a  first observation in that direction.

\begin{lemma}\label{lem:teardrop} \textup{\textbf{(No teardrops)}} 
An $s$-track cannot be a teardrop---i.e., if  $\rho:[0,k]\to \mathcal G$ is an $s$-track with $\rho(0) = \rho(k)$, then $\rho$ induces a smooth map from $S^1$ to $\mathcal G$. 
\end{lemma}

\begin{proof}
Were the image of $\rho$ a teardrop, the point $\rho(0) = \rho(k)$ would be a junction.  However, as all the $s$-edges along an $s$-track are oriented the same way (in this case, either into or out of the teardrop) this would violate the orientation condition at the junction; see Figure~\ref{fig:junction}
\end{proof}

\begin{definition} \textbf{(Tracks forming loops)} 
A track that closes up is a \emph{loop}. In light of Lemma~\ref{lem:teardrop}, a track closes up without introducing a corner, and so loops are smooth.  
\end{definition}

\subsection{HNN-structures} \label{sec:hnn}

We will give two HNN-structures for $G$.  The first is an immediate consequence of Lemma~\ref{lem: C(3)}\eqref{lem part:vertex groups free U}.

 \begin{prop} \label{prop:HNNWise}
$G$ is an HNN-extension:
 $$G  \ = \ F \HNN_{t} \  \ \text{ where } \ \  F = F(a_1, a_2,   b_0, \ldots, b_p, x_1, x_2, y_1, y_2)$$
and the $r = 5p+11$ defining relators displayed in Figure~\ref{fig:relations} dictate the  isomorphism between the associated groups, both of which are  rank-$r$ free subgroups of $F$.    
   \end{prop}

We will call on the following corollary 
in our proof of Lemma~\ref{lem: no self-osculating t-corridors}.  It holds because the elements of $\mathcal{U}$ are \emph{reduced} words with no $t$-letters.    

\begin{cor} \label{cor:subwords non-trivial}
Non-trivial subwords of elements of $\mathcal{U}$  represent non-identity elements in $G$.
\end{cor}

We will learn later (in Corollary~\ref{cor:HNN undistorted}) that $F$ is undistorted in $G$, and it will follow that the same is true of the two vertex subgroups.

\bigskip

Our second HNN-structure for $G$ is: 
$$G  \ =  \    \left(  \cdots \left(  \left(  F(t, x_1, x_2, y_1, y_2) \HNN_{a_1,a_2} \right) \HNN_{b_p} \right)  \cdots  \right)  \HNN_{b_0}$$       
in the manner detailed in Proposition~\ref{prop:hnn} and Table~\ref{table:hnn} below.  

We use the notation $K\HNN_{s_1, \dots, s_l}$ to denote an $l$-fold HNN-extension with vertex group $K$, stable letters $s_1, \dots, s_l$ and subgroups $I_i, T_i<K$ for $i=1, \dots, l$, such that $
s_i^{-1} I_is_i = T_i$.  We call $I_i$ and $T_i$ the \emph{initial} and \emph{terminal} groups respectively, and say that the stable letter $s_i$ \emph{conjugates} $I_i$ to $T_i$.

\begin{definition} \label{def:Gi}
Let
$F$ be the free group on $\{t, x_1, x_2, y_1, y_2\}$. Note that this is a departure from our definition of $F$ in Proposition~\ref{prop:HNNWise}.   Let $G_{-1}$ be the group generated by 
$\{t, x_1, x_2, y_1, y_2, a_1 ,a_2\} $ 
subject to the two  $r_{4, \ast}$- and four $r_{4, \ast, \ast}$-defining-relators of Figure~\ref{fig:relations}.  Then, for      $i=0, \ldots, p$, define $G_i$ to be the group generated by $\{t, x_1, x_2, y_1, y_2, a_1 ,a_2, b_{p-i}, \ldots, b_p\}$ subject to all the relators of Figure~\ref{fig:relations} in which only these letters appear. 
In particular, $G = G_p$.
\end{definition}

We will establish that $G_{-1} =     F \HNN_{a_1 ,a_2}$ and $G_i  =  G_{i-1} \HNN_{b_{p-i}}$
for      $i\ge0$, where the  initial and terminal groups at each stage are as shown in 
 Table~\ref{table:hnn}; the words listed in the table are subwords of defining relators in $\mathcal{R}$.   
 More precisely: 
\begin{prop}
 \label{prop:hnn} 
 For $G_{-1}, G_0, \ldots, G_p$ as per Definition~\ref{def:Gi}:
\begin{enumerate} \item   \label{prop part:G1HNN} 
 $G_{-1}$ is  a double HNN-extension over $F$ with stable letters $a_1$ and $a_2$  
 conjugating the initial group $\langle t, y_1, y_2\rangle$ to the first and second terminal groups listed in Row 1 of Table~\ref{table:hnn}, respectively.

\item \label{prop part:Gi an HNN} For $i\ge0$, the group $G_i$ is an HNN-extension over $G_{i-1}$ with stable letter $b_{p-i}$ conjugating the group $K_{i}<G_{i-1}$ from Table~\ref{table:hnn} to the group $L_{i}<G_{i-1}$ from Table~\ref{table:hnn}. 

\end{enumerate}

\end{prop}
Recall that, per Section~\ref{sec:the defn}, we have chosen to suppress the indexing in our notation for the small cancellation words appearing in our construction.  Thus, the collection $\mathcal{X} \cup \mathcal{Y}$ of all the $\Xb$ or $\Yb$ satisfies $C'(1/4)$.

\begin{table}[h!] \caption{Iterated HNN structure of $G$.  The words listed are subwords of defining relators in $\mathcal{R}$.  The different instances of $\Xb$ or $\Yb$ represent different small cancellation words.} \label{table:hnn}

\ms

\begin{tabular}{|p{3.7cm}|p{7cm}|}
  \hline
    \multicolumn{2}{|l|}{\rule{0mm}{5mm}$G_{-1}$: stable letters $a_1$ and $a_2$, vertex group $F$ }\\[2mm]   \hline
  \multicolumn{1}{|c}{Initial group} & \multicolumn{1}{|c|}{Terminal groups}  \\  \hline
  \multicolumn{1}{|c}{\parbox{3.7cm}{\begin{center}\rule{0mm}{0mm}$\langle t,   y_1, y_2 \rangle$ \end{center}}} 
  & \multicolumn{1}{|c|}{\parbox{7cm}{\rule{0mm}{8mm} \ $\begin{array}{ll} a_1: & \langle \Yb t \Yb, \  \Yb t^{-1} \Yb t \Yb, \  \Yb t^{-1} \Yb t \Yb \rangle,   \\  a_2: & \langle \Yb t \Yb, \  \Yb t^{-1} \Yb t \Yb,  \ \Yb t^{-1} \Yb t \Yb \rangle \end{array}$ \\[3mm]   }}   \\
  \hline
\end{tabular}

\bs
 
\begin{tabular}{|p{3.7cm}|p{7cm}|}
  \hline
    \multicolumn{2}{|l|}{\rule{0mm}{5mm}$G_i$ for $i \geq 0$: stable letter $b_{p-i}$, vertex group $G_{i-1}$ }\\[2mm]   \hline
  \multicolumn{1}{|c}{Initial group} & \multicolumn{1}{|c|}{Terminal group}  \\  \hline
  \multicolumn{1}{|c}{\parbox{3.7cm}{\begin{center}\rule{0mm}{0mm}$K_0=\langle  a_1, a_2, t, x_1, x_2 \rangle$ \\ \rule{0mm}{10mm} $K_{i}=\langle  a_1b_{p-i+1}, a_2,$  \\ \quad  $t, x_1, x_2 \rangle$ \\ ${\scriptstyle i \ > \  0}$ \end{center}}} 
  & \multicolumn{1}{|c|}{\parbox{7cm}{\begin{center} \rule{0mm}{5mm}$L_i =\langle a_1\Xb t^{-1} \Xb t \Xb, \ a_2\Xb t^{-1} \Xb t \Xb, $ $  \Xb t \Xb, \  \Xb t^{-1} \Xb t \Xb, \ \Xb t^{-1} \Xb t \Xb  \rangle$ \\ ${\scriptstyle i \ \neq \  p,  \    p-q+1}$ \\  
\rule{0mm}{10mm}$L_{p-q+1} = \langle a_1a_2\Xb t^{-1} \Xb t \Xb,$  $a_2\Xb t^{-1} \Xb t \Xb,  \ \Xb t \Xb$, $ \Xb t^{-1} \Xb t \Xb, \ \Xb t^{-1} \Xb t \Xb  \rangle$ \\ 
\rule{0mm}{10mm}$L_p = \langle a_1\Yb t^{-1} \Yb t \Yb, \ a_2\Yb t^{-1} \Yb t \Yb,$ \\  $\Yb t \Yb, \ \Yb t^{-1} \Yb t \Yb, \ \Yb t^{-1} \Yb t \Yb  \rangle$ \\[3mm]
\end{center}}}   \\
  \hline
\end{tabular} 
 \end{table}

Before we prove Proposition~\ref{prop:hnn}, we observe that it  yields:

\begin{cor} \label{H is free}
The subgroup $H=\langle t, y_1, y_2\rangle$ of $G$ is a free group of rank 3.  
\end{cor}

\begin{proof}
Since $F$ is free on $t, x_1, x_2, y_1, y_2$, it is clear that   $\langle t, y_1, y_2\rangle$ is rank-3 free in $F$.  As vertex groups inject into HNN-extensions, Proposition~\ref{prop:hnn} yields:
$H \hookrightarrow F \hookrightarrow G_{-1} \hookrightarrow  G_0 \hookrightarrow  \cdots \hookrightarrow G_p =G$.
\end{proof}

\begin{proof}[Proof of Proposition~\ref{prop:hnn}\eqref{prop part:G1HNN}]
The group
$\langle t, y_1, y_2 \rangle<F$
 is  free of rank 3.  The two terminal vertex groups in the $G_{-1}$ row of Table~\ref{table:hnn} are free of rank 3 by 
Lemma~\ref{lem: C(3)}\eqref{lem part:vertex groups free S}. Thus the described HNN-structure follows from the definition of $G_{-1}$.
\end{proof}

To establish the HNN-structure of $G_i$ for $i\ge 0$  (thereby completing the proof of Proposition~\ref{prop:hnn}\eqref{prop part:Gi an HNN}), we must show that the groups $K_i$ and $L_i$ listed in Table~\ref{table:hnn} are free of rank 5 in $G_{i-1}$.  As a first step, we show: 

\begin{lemma}\label{K0Li}
The groups $K_0$ and $L_i$ for $i= 0, \dots, p$ are rank-5 free subgroups of $G_{-1}$.
\end{lemma}

\begin{proof}
We begin with $K_0$.  If $a_1, a_2, t, x_1, x_2$ do not generate a free subgroup 
of $G_{-1}$, then there is a non-empty freely reduced word on these letters which represents the identity in $G_{-1}$.  Let $w$ be a shortest such word and let
 $\Delta$ be a reduced van Kampen diagram with boundary label $w$.  

Observe that the group $F$ injects into $G_{-1}$, as it is  the vertex group in the HNN-structure for $G_{-1}$, by Proposition~\ref{prop:hnn}\eqref{prop part:G1HNN}. 
Thus  $\langle t, x_1, x_2\rangle<F<G_{-1}$ is free, and so no non-empty freely reduced word on these letters represents the identity. Thus
we may assume that $w$ has at least one $a_1$- or $a_2$-letter, and so $\Delta$ has at least one $a_1$- or $a_2$-corridor.    Moreover, we can assume $\Delta$ is homeomorphic to a 2-disc, because otherwise it could be broken into two subdiagrams for two words which are shorter than $w$ and represent the identity, and cannot both be freely reduced to the empty word   (since $w$ cannot be).  In particular, every $a_1$- and $a_2$-corridor is \emph{non-degenerate}, by which we mean that it is not a single $a_1$- or $a_2$-edge that is part of a 1-dimensional portion of $\Delta$.

Let $\langle  Z_1, Z_2,  Z_3  \rangle$ and $\langle  Z'_1,  Z'_2, Z'_3 \rangle$ denote the two terminal groups in the construction of $G_{-1}$ as shown in Table~\ref{table:hnn}.
No two   $a_1$- or $a_2$-corridors can cross or branch in $\Delta$, so dual to them there is an oriented tree $\mathcal{T}$ which has a vertex for each complimentary region and an edge for each corridor. Give the  edges of $\mathcal{T}$  orientations that match the directions of the $a_1$- or $a_2$-corridors they cross.   Then $\mathcal{T}$ necessarily has a sink vertex (a vertex with the property that all its incident edges are oriented towards it), 
and the boundary of the subdiagram $\Delta_0$ of $\Delta$ corresponding to this vertex consists of parts of $\partial \Delta$ between $a_1$- or $a_2$-edges at the ends of corridors and the top boundaries of $a_1$- or $a_2$-corridors.  Thus, read around $\partial \Delta_0$ is a   word $v$ on 
$$t, x_1, x_2, 
  Z_1, Z_2, Z_3,   Z'_1, Z'_2, Z'_3.$$  By Lemma~\ref{lem: free subgroup C'(1/4)}  these elements form a basis for a free subgroup $F'$ of $F$ and therefore of $G_{-1}$.   Now $v$ is non-empty (since every corridor is non-degenerate) and represents the identity in $G_{-1}$, and therefore in the free group $F'$ (since $F'\hookrightarrow G_{-1}$).  So $v$ is not freely reduced, i.e., it has a subword of the form $uu^{-1}$ for some letter or inverse letter $u$.  
  Since the subwords of $v$ on $t, x_1, x_2$ come from $w$, which is freely reduced, $u$ is one of the remaining generators of $F'$.  Then $uu^{-1}$ must be a subword of the top boundary of a single $a_1$- or $a_2$-corridor (because, if $u$ and $u^{-1}$ came from different corridors, $w$ would have a subword $a_1^{\pm1}a_1^{\mp 1}$ or $a_2^{\pm1}a_2^{\mp1}$, contradicting the fact that it is freely reduced).  This means the corridor has adjacent cells that are identical and oppositely oriented, contradicting the fact that $\Delta$ is reduced.   Thus $K_0$ is a free subgroup of $G_{-1}$.  

A near  identical proof shows that $L_p< G_{-1}$ is free.  Denoting the generators of $L_p$ by $$a_1 Z_{p1}, a_2Z_{p2}, Z_{p3}, Z_{p4}, Z_{p5},$$ let $w$ be a shortest non-empty freely reduced word on these generators which represents the identity in $G_{-1}$. Let $\Delta$ be a reduced van Kampen diagram over $G_{-1}$ with boundary label $w$.  
Since $\langle Z_{p3}, Z_{p4}, Z_{p5}\rangle < F < G_{-1}$ is free (using Lemma~\ref{lem: C(3)}\eqref{lem part:vertex groups free S}), we may assume as before that 
$w$ has at least one $a_1Z_{p1}$ or $a_2Z_{p2}$. Hence $\Delta$ has at least one $a_1$- or $a_2$-corridor.  Furthermore, we conclude as before that all  
 $a_1$- or $a_2$-corridors are non-degenerate.  Considering a sink region of the oriented dual tree 
  as above, we see that the boundary label of the sink region is a word $v$ on 
  $$Z_1, Z_2, Z_3,  Z'_1, Z'_2, Z'_3, Z_{p1}, Z_{p2}, Z_{p3}, Z_{p4}, Z_{p5}$$ which represents the identity in $G_{-1}$.   (The first six of these words appear along top boundaries of $a$-corridors while the last five appear in parts of $v$ coming from $w$.) By Lemma~\ref{lem: free subgroup C'(1/4)}, these elements form a basis for a free subgroup of $F$, and therefore of $G_{-1}$ (since $F \hookrightarrow G_{-1}$).  Then we argue as in the previous paragraph to arrive at a contradiction.   

Finally, for $i \neq p$, Lemma~\ref{lem: C(3)}\eqref{lem part:vertex groups free S} implies that $L_i$ is a rank-5 free subgroup of $K_0$.  Thus $L_i$ is a rank-5 free subgroup of $G_{-1}$ as $K_0 \hookrightarrow G_{-1}$.  
\end{proof}

In order to prove that $K_i$ is free for $i > 0$ and complete the proof of Proposition~\ref{prop:hnn} we  need three technical lemmas.  
\begin{lemma} \label{lem:bj free} In $G_i$ of Definition~\ref{def:Gi}, 
$b_p, b_{p-1}, \dots, b_{p-i}$ freely generate a free subgroup.
\end{lemma}

\begin{proof}  By examining the relators of $G_{i}$, we see that there is a quotient 
homomorphism $$G_{i} \twoheadrightarrow Q_i = \langle \, b_p, b_{p-1}, \dots, b_{p-i}, a_1 \,\mid \,
a_1^{-1} b_j a_1 = b_{j+1}b_j \text{ for } j<p; \, a_1^{-1} b_p a_1 = b_p \,
\rangle$$ 
mapping $b_j \mapsto b_j$, $a_1 \mapsto a_1$ and  killing every other generator. 
This quotient $Q_i$ is free-by-cyclic: the generator $a$ of the cyclic part acts by conjugation on a free group generated by $b_p, \dots, b_{p-i}$ by an automorphism.  Moreover, the restriction of this homomorphism to the subgroup $\langle b_p, \dots, b_{p-i}\rangle < G_{i}$ is a surjection onto the rank-$(i+1)$ free subgroup $\langle b_p, \dots, b_{p-i}\rangle < Q$.  The result follows. 
\end{proof}

The next lemma restricts the possible $b$-track systems in certain van Kampen diagrams over $G_i$.

\begin{lemma}\label{no junctions}
For $i=0, \dots p-1$, let $\Delta$ be a reduced van~Kampen diagram over the group $G_i$ of Definition~\ref{def:Gi} with boundary labelled by a word on $a_1, a_2, t, x_1, x_2,$ $ b_p, b_{p-1}, \dots, b_{p-i}$.  Then
\begin{enumerate}
\item \label{no red}
$\Delta$ has no $r_{4, \ast, \ast}$- or $r_{4, \ast}$-cells (per Figure~\ref{fig:relations}).
\item \label{no annuli} 
$\Delta$ has no $a_1$-annuli. 
\item \label{no junc} If the word read around  $\partial \Delta$ contains no letters $a_1^{\pm 1}$, then
the track system $\mathcal G_b$ of $\Delta$ has no junctions.  Thus $\mathcal G_b$ consists of a collection of disjoint tracks, each dual to a $b_j$-corridor for some $j$ such that  $0<p-i\le j \le p$. 
\end{enumerate}

\end{lemma}
\begin{proof}
For \eqref{no red}, we suppose  $\Delta_0$ is a maximal subdiagram of $\Delta$ that contains no $b$-edges and is homeomorphic to a 2-disc.   Any $r_{4, \ast, \ast}$- or  $r_{4, \ast}$-cell must be in some such $\Delta_0$.  All its 2-cells must be of type  $r_{4, \ast, \ast}$ or  $r_{4, \ast}$  since every other type of 2-cell has a $b$-edge.   So, arguing that there are no 2-cells in  $\Delta_0$ will establish \eqref{no red}.

There can be no $y$-edges in $\partial \Delta_0$ because such a $y$-edge would have to be either in $\partial \Delta$ (contrary to hypothesis) or in the boundary of a 2-cell of $\Delta$ that is not of type $r_{4, \ast, \ast}$- or  $r_{4, \ast}$ (impossible because the only such 2-cells from Figure~\ref{fig:relations} have $b_0$-edges, and $b_0 \notin G_i$ when $i <p$).  So  $\partial \Delta_0$ is labelled by a word $v$ on    $a_1, a_2, t$.  Now,  $v$ represents the identity in 
$$\langle  a_1, a_2, t, y_1, y_2 \mid r_{4, i, j}, r_{4, i}; \  i,j = 1,2 \rangle \ = \ F(a_1, a_2,  y_1, y_2) \HNN_t.$$  There can be no $t$-annulus in $\Delta_0$ since the word read around the inner boundary of an innermost $t$-annulus would be a word on $\mathcal{U}$ that freely equals the empty word,   and  Lemma~\ref{lem: C(3)}\eqref{lem part:vertex groups free U} would imply that  there must be cancellation of a pair of 2-cells, contrary to $\Delta$ being reduced.  And if there is a $t$-corridor in $\Delta_0$, then there is one that is outermost in that the freely reduced form of the word along its top or bottom follows a path in $\partial \Delta_0$.  But (since $\Delta_0$ is reduced and homeomorphic to a 2-disc) the word along the top or bottom  any  $t$-corridor in  $\Delta_0$ must contain $y$-letters, so this contradicts there being no $y$-letters in $\partial \Delta_0$.

Next we deduce \eqref{no annuli}.  Were there such an $a_1$-annulus, in  light of \eqref{no red}, one of its boundaries would be labelled by a word on $b_p, b_{p-1}, \dots, b_{p-i}$ representing the identity in $G_i$.  It would then follow from Lemma~\ref{lem:bj free} that this word would freely reduce to the empty word.  This would imply that the annulus would have adjacent 2-cells that are identical but with opposite orientation, contrary to   $\Delta$ being reduced. 

Finally, for \eqref{no junc}, suppose the word read around  $\partial \Delta$ contains no letters $a_1^{\pm 1}$.  If the track system $\mathcal G_b$ had a junction, that junction
would be in a 2-cell of $\Delta$ with an $a_1$ on its boundary, and this 2-cell would be part of an $a_1$-corridor or $a_1$-annulus.  However, there are no $a_1$-corridors since the label of $\partial \Delta$ has no $a_1$ and there are no $a_1$-annuli by~\eqref{no annuli}. 
\end{proof}

\begin{lemma} \label{lem:technical}
For $i=0, \dots p-1$, in the group $G_i$ of Definition~\ref{def:Gi}, we have 
$$\langle b_p, b_{p-1}, \dots, b_{p-i} \rangle \cap \langle a_2, x_1, x_2, t\rangle  \ =  \ \set{1}.$$  
\end{lemma}

\begin{proof}
 Suppose for a contradiction that there is a non-trivial element in 
$$\langle b_p, b_{p-1}, \dots, b_{p-i} \rangle \cap \langle a_2, x_1, x_2, t\rangle.$$
 Then there are non-empty freely reduced words $u=u(a_2, x_1, x_2, t)$ and 
 $v=v(b_p, b_{p-1}, \dots, b_{p-i})$ such that $u= v$ in $G_{i}$,  and there is a reduced van Kampen diagram $\Delta$ with boundary label $uv^{-1}$.  
 Observe that $\Delta$ satisfies the hypotheses of Lemma~\ref{no junctions}\eqref{no junc} since the word read around  $\partial \Delta$ has no instances of $a_1^{\pm 1}$.  Thus the track system $\mathcal G_b$ of $\Delta$ consists of a union of disjoint tracks, each dual to a $b_j$-corridor for some $j$.  Since $u$ has no instances of $b_j$ for any $j$, each of these tracks has both ends on the part of $\partial \Delta$ labelled $v$.  Since these $b$-tracks cannot cross each other, there must be at least one that is innermost  in that it begins and ends at consecutive letters in $v$.  This implies that $v$ 
has a subword  $b_j^{\pm1}b_j^{\mp1}$, which contradicts  $v$ being freely reduced. 
\end{proof}

We can now prove the following lemma, which establishes Proposition~\ref{prop:hnn}\eqref{prop part:Gi an HNN}.

\begin{lemma} \label{indn}
 For $i=0, \ldots, p$, 
\begin{enumerate}
\item \label{subgroups free of rank 5}  the subgroups $K_i, L_i \leq G_{i-1}$  are free of rank 5,   
\item \label{Gi an HNN}  the group $G_i$ is an HNN-extension over $G_{i-1}$ with stable letter $b_{p-i}$ conjugating $K_{i}$  to $L_{i}$.
 \end{enumerate}
\end{lemma}

\begin{proof} 
We induct on $i$.    In the case   $i=0$, Lemma~\ref{K0Li}  gives \eqref{subgroups free of rank 5}, and then \eqref{Gi an HNN} follows by definition of $G_0$.  
We now prove the induction step.   Assume the result holds up to some value of the index $i <p$.    We will show that  \eqref{subgroups free of rank 5} and \eqref{Gi an HNN} hold with the index $i$ elevated by $1$.

In Lemma~\ref{K0Li} we showed that $L_{i+1}$ is a free subgroup of $G_{-1}$ of rank 5.  By statement~\eqref{Gi an HNN} of the induction hypothesis, $G_{-1}$, $G_{0}$, \ldots,  $G_{i}$ are successive HNN extensions.  So $G_{-1} \hookrightarrow G_0 \hookrightarrow \cdots \hookrightarrow G_i$ are injective inclusions and $L_{i+1}$ is a rank-5 free subgroup of $G_i$ as well.  

Likewise, $K_0$ is   a rank-5 free subgroup of $G_i$.   
We will show    that $K_{i+1}=\langle a_1 b_{p-i}, a_2, t,  x_1, x_2\rangle$ is  also a rank-5 free subgroup of $G_i$.  This will prove \eqref{subgroups free of rank 5}, and then \eqref{Gi an HNN} will immediately follow. 

Let $w$ be a non-empty freely reduced word on the generators of $K_{i+1}$ such that $w=1$ in $G_i$.  Assume that $w$ is minimal in the sense that no shorter non-empty freely  reduced word on the generators of $K_{i+1}$  represents the identity in $G_i$.
Let $\Delta$ be a reduced van~Kampen diagram for $w$ over $G_i$.  
It contains no 2-cells of type $r_{4, \ast, \ast}$ or  $r_{4, \ast}$ by Lemma~\ref{no junctions}\eqref{no red}.

The word $w$ must include at least one instance of $a_1b_{p-i}$, as otherwise $w$ would be a non-empty freely reduced word representing the identity in the free group $K_0<G_i$, a contradiction.  Consequently, $\Delta$ has at least one $a_1$-corridor.  Moreover, every $a_1$-corridor is non-degenerate, as a degenerate corridor would cut $\partial \Delta$ into two loops (both non-trivial as $w$ is non-empty and freely reduced) and one of these would be labelled by a shorter freely reduced word on the generators of $K_0$, contradicting the minimality of $w$.  
As $\Delta$ has no 2-cells of type $r_{4, \ast, \ast}$ or  $r_{4, \ast}$, every $a_1$-corridor is made up of $r_{1, i}$-cells, where $1 \leq i \leq p$. (We exclude $r_{1, 0}$ since $i<p$.)

Let $C$ be an innermost $a_1$-corridor in $\Delta$, i.e.~an $a_1$-corridor whose complement in $\Delta$ has a component $\Delta'$ without $a_1$-corridors. Then $\partial \Delta'$ is composed of two paths between the same pair of points: a top or bottom boundary $\gamma$  of $C$ with label $v$ (which is non-empty since $C$ is non-degenerate) and a path $\delta$ along $\partial \Delta$. 
The labels $\gamma$ and $\delta$ represent the same element of $G_i$.  

There are two cases, depending on the orientation of $C$.  
If $C$ points away from $\Delta'$, then $\gamma$ is its bottom boundary and $v$ 
is a non-empty word on $b_{p-i}, \dots, b_p$, which is freely reduced since $\Delta$ is reduced.  In this case $\delta$ is labelled by a freely reduced word $u$ 
on $t, x_1, x_2,  a_2$, which is non-empty since otherwise $w$ would have an $a_1^{-1}a_1$ subword and not be freely reduced. 
Now $u=v$ in $G_i$, which contradicts  Lemma~\ref{lem:technical}.

On the other hand, if $C$ points towards $\Delta'$, then 
$\gamma$ is its top boundary and $v$  is a word on elements of the form $b_{j+1}b_j\Xb^{-1}t^{-1}\Xb^{-1}\epsilon^{-1}$, where $b_{j+1}=1$ if $j=p$, and $\epsilon = a_2$ if $j=q-1$ and 1 otherwise.  
In this case 
$\delta$ is labelled by a word  of the form $b_{p-i} u b_{p-i}^{-1}$, where $u$ is a 
word on $t, x_1, x_2,  a_2$.

We consider the track system $\mathcal G_b'$ of $\Delta'$.  
Lemma~\ref{no junctions}\eqref{no junc} applies to $\Delta'$, because it has no $a_1$-corridors,  and we conclude that $\mathcal G_b'$ 
is a disjoint union of tracks.  Each of these tracks is dual to a $b_j$-corridor for some $j$ such that $0 < p-j \le j \le p$ and inherits its label.

Suppose there exists a $b$-track with both ends on $\gamma$.  Consider an innermost such track, i.e.~one for which the subword of $v$ between its endpoints has no $b$-letters, and 
 suppose it is labelled $b_m$ for some $m$. 
Since each 2-cell of $C$ has at least one $b$-letter and at most one $b_m$, this track must begin and end at neighboring cells of $C$.   Examining the $r_{1, \ast}$-cells of  Figure~\ref{fig:relations} we see that the only possibility is that these are identical cells with opposite orientation, which contradicts $\Delta$ being reduced.  Thus tracks of $\mathcal G_b'$ have at most one end on $\gamma$.  

Since $\delta$ is labelled by $b_{p-i} u b_{p-i}^{-1}$, where $u$ has no $b$-letters, there are at most two tracks ending on $\delta$.  Since $C$ is non-degenerate, there is at least one track starting at $\gamma$, which rules out the possibility of a track with both endpoints on $\delta$.  We conclude that $\mathcal G_b'$ has exactly two tracks, each with one end on $\delta$ and one on $\gamma$, and both with label $b_{p-i}$.   It follows that $C$ has exactly two 2-cells, both of type $r_{1,p}$, and $i=0$ (as every other possible 2-cell has both $b_j$ and $b_{j+1}$ for some $j$ in its top boundary). Moreover, since tracks preserve orientation, and the two edges of $\delta$ labelled $b_p$ are oppositely oriented, it follows that the two 2-cells of $C$ are oppositely oriented.  This contradicts $\Delta$ being reduced. 

We have arrived at contradictions in all cases.  
  It follows that no such $w$ can exist, and that $K_{i+1}$ is free of rank 5, completing the induction. 
\end{proof}

\section{The lower bound} \label{ch:the lower bound}

\subsection{The lower bound on distortion}\label{sec:lower}

In this section, we will  establish  the lower bound on distortion of Theorem~\ref{main} in the case $k=1$.  In the manner outlined by the figures in this section, we prove that for all $n \in \N$, there is a 
freely reduced  word $\chi_n$ on $t^{\pm 1}$, $y_1^{\pm 1}$, and $y_2^{\pm 1}$ of length  $\simeq \! 2^{n^p}$
 which  represents the same group element  as a word $w_n$ in the generators of $G$ of length $\simeq \! n^q$.  These length estimates emerge from calculations tracing through the construction, with small-cancellation arguments ensuring that $\chi_n$ does not lose too much length through free reduction.  As  $t$,   $y_1$ and $y_2$  freely generate  $H$ (Corollary~\ref{H is free}), no shorter word than $\chi_n$ on  $t^{\pm 1}$, $y_1^{\pm 1}$ and $y_2^{\pm 1}$ equals $w_n$ in $G$. Via Lemma~\ref{lem:sparsity}, this  will establish that $\Dist_H^G(n) \succeq 2^{n^{p/q}}$.

For $w$  a word, $\abs{w}$ denotes the number of letters in $w$ and $|w|_{x}$ the exponent sum of the $x$ in $w$.  So, if $w$ is a \emph{positive} word, which is to say it contains no inverse letters, then $|w|_{x}$ is   the number of $x$ in $w$.

Recall that killing $a_2, t, x_1, x_2, y_1, y_2$  maps $G$ onto the free-by-cyclic group $$Q \  =  \   \langle \, a_1, b_0, b_1, \dots, b_p  \, \mid \, a_1^{-1} b_j a_1 = \varphi(b_j) \text{ for all } j \, \rangle,$$ where $\varphi$ is the automorphism of $F(b_0, \ldots, b_p)$  mapping  $b_j \mapsto b_{j+1}b_j$ for $j=0, \ldots, p-1$ and $b_p \mapsto b_p$.  The following lemma describes a lift of an equality $a_1 \varphi(u b_0)  = u b_0 a_1$ in $Q$ to an equality $a_1  \varphi(u b_0)    =    u b_0 a_1 \tau$ in $G$.

\begin{lemma} \label{lemma i}
Given a positive word $u=u(b_1, \ldots, b_p)$,  there is a freely reduced word $\tau = \tau(a_2, t^{\pm 1}, y_1^{\pm 1}, y_2^{\pm 1})$ such that  
\begin{align}
a_1  \varphi(u b_0)  & \ = \  u b_0 a_1 \tau  \text{ in } G \label{i1} \\ 
 | \varphi(u b_0) |_{b_i} & \ = \  | u b_0 |_{b_i} + | u b_0 |_{b_{i-1}}  \qquad \text{ for  } i = 1, \ldots, p  \label{i2} \\  
  | \varphi(u b_0) |_{b_0} & \ = \ | u b_0 |_{b_0} \label{i3}  \ = \ 1 \\ 
 | \tau |_{a_2} & \ = \  | \varphi(u b_0) |_{b_q}  - | u b_0 |_{b_{q}}.  \label{i4}   
\end{align}
Moreover,  $\tau$ has a suffix $\kappa$ that is also a \emph{long} suffix of one of the Rips words $Y_{\ast}$ used in the presentation $\mathcal{P}$ of $G$---by \emph{long} we mean that $|\kappa|$ is at least $(3/4)|Y_{\ast}|$.   
\end{lemma}

\begin{proof}
The statements \eqref{i1}--\eqref{i4} are easily verified when $u$ is empty. 
Assuming $|u| \geq 1$, express $u$ as $b_i u_0$ where $b_i$ is the first letter of $u$ and $u_0$ is the remainder of the word.   The structure of a van~Kampen diagram for \eqref{i1} is displayed in Figure~\ref{fig:i}.  It is constructed inductively, the base step being provided by the case where $u$ is empty.
The top cell in Figure~\ref{fig:i} encodes the relation $a_1 \varphi(b_i) = b_i a_1 \sigma$, where $\sigma$ is a word on 
$a_2$, $t$, $x_1$ and $x_2$ that contains no $a_2^{-1}$.  The bottom left block comes from applying the induction hypothesis to $u_0$, so $\tau_0 = \tau_0(a_2, t^{\pm 1}, y_1^{\pm 1}, y_2^{\pm 1})$. 
The bottom right block encodes the result of moving $\phi(u_0b_0)$ past $\sigma$.  
	That $\sigma_0$, and therefore $\tau$, contains letters $a_2, t^{\pm 1}, y_1^{\pm 1}, y_2^{\pm 1}$ but not $x_1^{\pm 1}, x_2^{\pm 1}$ is due to   $b_0$ conjugating $a_2$, $t^{\pm 1}$, $x_1^{\pm 1}$ and $x_2^{\pm 1}$ to words on  $a_2$, $t^{\pm 1}$, $y_1^{\pm 1}$ and $y_2^{\pm 1}$. (See the $r_{2, \ast}$-, $r_{3, \ast}$-, and $r_{3, \ast, \ast}$-cells of Figure~\ref{fig:relations}.)  

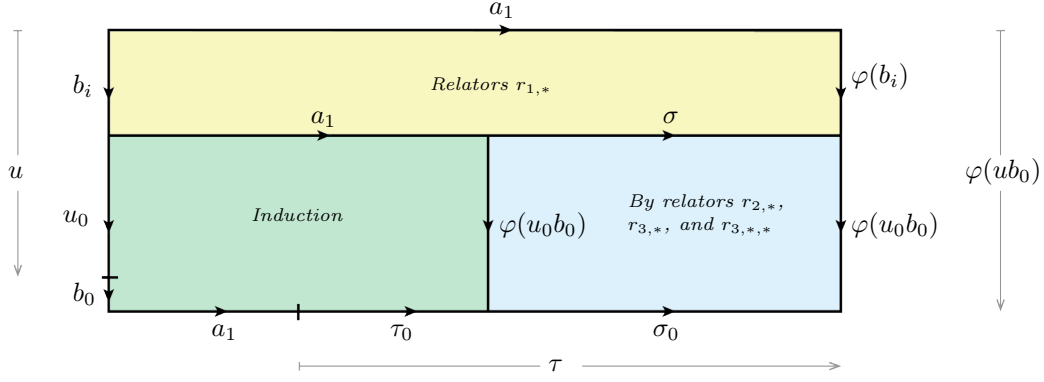
\begin{figure}[htbp]
\centering
\begin{overpic}  [scale=0.9] 
{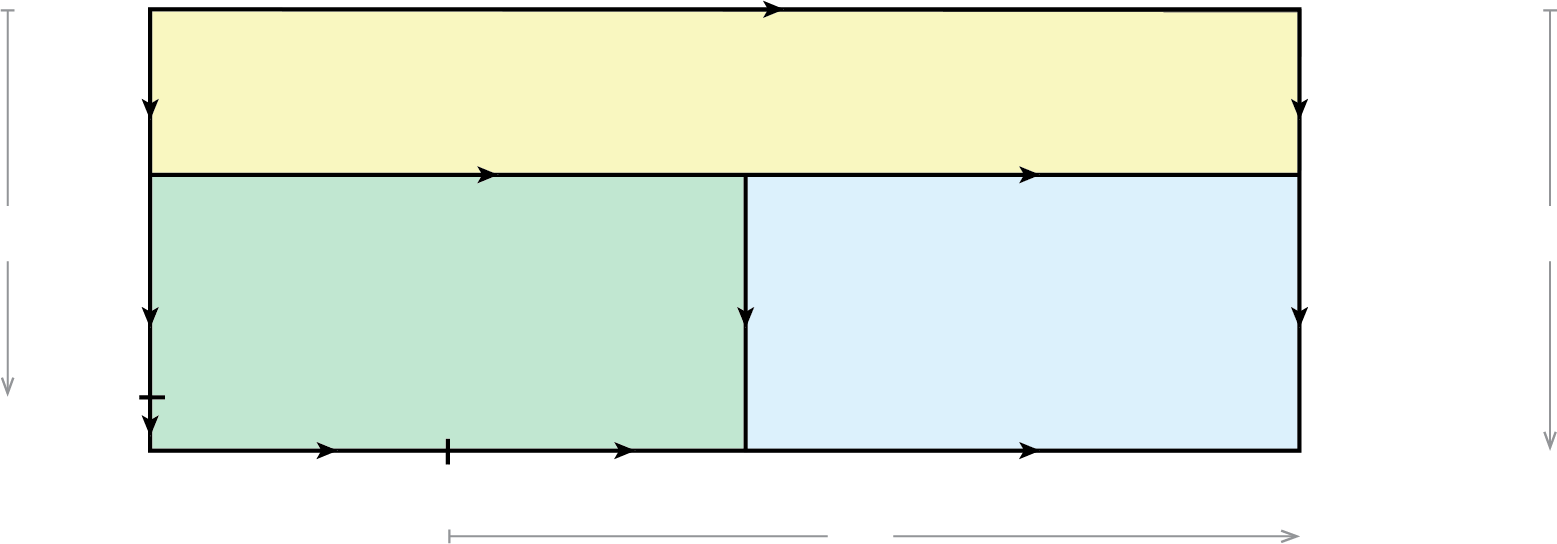}
\put(48, 36){\small{$a_1$}}
\put(30, 25){\small{$a_1$}}
\put(20, 3.5){\small{$a_1$}}
\put(6, 7){\small{$b_0$}}
\put(6, 28){\small{$b_i$}}
\put(5, 15){\small{$u_0$}}
\put(38, 3.5){\small{$\tau_0$}}
\put(64.5, 3.5){\small{$\sigma_0$}}
\put(65.5, 25){\small{$\sigma$}}
\put(84.5, 14){\small{$\varphi(u_0 b_0)$}}
\put(49, 14){\small{$\varphi(u_0 b_0)$}}
\put(84.5, 29){\small{$\varphi(b_i)$}}
\put(-0.5, 19.5){\small{$u$}}
\put(96, 19.5){\small{$\varphi(u b_0)$}}
\put(54, 0){\small{$\tau$}}
\put(42, 28){\parbox{60mm}{\tiny{\emph{Relators $r_{1, \ast}$}}}}
\put(24, 15){\parbox{30mm}{\tiny{\emph{Induction}}}}
\put(62, 15){\parbox{25mm}{\tiny{\emph{By relators $r_{2, \ast}$, \\ $r_{3, \ast}$, and $r_{3, \ast, \ast}$ }}}}
\end{overpic}
\caption{A diagram for $a_1  \varphi(u b_0) = u b_0 a_1 \tau$ in $G$.}
\label{fig:i}
\end{figure}

The equalities \eqref{i2} and \eqref{i3} follow from the definition of $\varphi$.  

We get \eqref{i4} by induction, as follows.  Assume~\eqref{i4} holds for $u_0$.  
Examining the $r_{1, \ast}$-defining relators of  Figure~\ref{fig:relations}, we see that 
$ | \sigma |_{a_2} =|\varphi(b_i)|_{b_q} - |b_i|_{b_q}$ for any $i$.  Moreover,  $ | \sigma_0 |_{a_2} = | \sigma|_{a_2} $ in the bottom right block of Figure~\ref{fig:i} as each $r_{2, \ast}$-, $r_{3, \ast}$-, and $r_{3, \ast, \ast}$-defining relator of Figure~\ref{fig:relations} satisfies this property.  
 Combining these observations with the induction hypothesis, we get:
$| \tau|_{a_2} = | \tau_0 |_{a_2} + | \sigma_0 |_{a_2} =  | \varphi(u_0b_0)|_{b_q} -  |u_0b_0|_{b_{q}}  +  | \sigma  |_{a_2} =| \varphi(u_0b_0)|_{b_q} -  |u_0b_0|_{b_{q}} + |\varphi(b_i)|_{b_q} - |b_i|_{b_q}  =
| \varphi(b_iu_0b_0)|_{b_q} -  |b_iu_0b_0|_{b_{q}} $, which completes the inductive step (since $u=b_iu_0$) and 
proves~\eqref{i4}.

 When $u$ is empty, Figure~\ref{fig:i} is a single $r_{1,0}$-cell and $\tau$ is $\Yb t^{-1} \Yb t \Yb$, which satisfies the suffix condition by construction.  
For $u$ non-empty
 we may assume by induction that  $\tau_0$ is reduced and its final letter is positive (since the  $\Yb$ are positive words).  
Now $\sigma$ is one of the subwords $\Xb t^{-1}  \Xb t \Xb$ of an $r_{1, \ast}$-defining relator of Figure~\ref{fig:relations} (as $\Yb t^{-1}  \Yb t \Yb$ is excluded since $b_i \neq b_0$).  Thus 
$\sigma$ has positive first letter and ends with $x_1$ or $x_2$.  
It follows, via the $C'(1/4)$-condition for $\mathcal{X} \cup \mathcal{Y}$ of Section~\ref{sec:the defn}, that the successive words we obtain from $\sigma$ by conjugating by a $b_i$ with $i \neq 0$ and then freely reducing have positive first letters and end with $x_1$ or $x_2$.
 Finally $\sigma_0$ is obtained by conjugating by $b_0$ and freely reducing, so it has a positive first letter and a suffix that is a \emph{long}  suffix of some 
 $\Yb t^{-1}  \Yb t \Yb$ 
 (again by $C'(1/4)$ for $\mathcal{X} \cup \mathcal{Y}$).
Therefore there is  no cancellation between $\tau_0$ and $\sigma_0$, and so 
 $\sigma_0$ gives $\tau$ the required \emph{long} suffix.  
\end{proof}

For all $j \geq 0$, define $u_j$ to be the positive word on $b_1, \ldots, b_q$ such that $u_j b_0 = \varphi^j(b_0)$ as words.  
In particular $u_0$ is the empty word $\varepsilon$, and $ u_{j+1}b_0=\varphi(u_j b_0)$. 
Now let $n \ge 1$.
For $j = 0, \ldots, n-1$, let $\tau_{j+1}$ be as per Lemma~\ref{lemma i} so that 
$a_1 u_{j+1} b_0= u_j b_0 a_1 \tau_{j+1}$ 
in $G$.  Let $v_n = a_1 \tau_1 \cdots a_1 \tau_n$.

For our next lemma, we understand the binomial coefficient $\SB{n}{i}$ to be zero when $i >n$.

\begin{lemma} \label{lemma ii}
  For all $n \geq 1$, the word $v_n$ is freely reduced and
\begin{align}
a_1^n  u_n  b_0 & \ = \  b_0  v_n  
 \text{ in } G \label{ii1} \\ 
 | v_n |_{a_1} & \ = \  n      \label{ii2} \\ 
   | u_n b_0 |_{b_i}  & \ =  \   \SB{n}{i} \text{ for } i=0, \ldots, p  \label{ii4} \\  
  | u_n b_0 | & \ = \  \SB{n}{0} + \cdots +  \SB{n}{p}.    \label{ii5} \\ 
  | v_n |_{a_2} & \ = \ | u_nb_0 |_{b_q} \ =  \   \SB{n}{q}      \label{ii3}
\end{align}
\end{lemma}
\begin{proof}
The reason $v_n$ is freely reduced is that each $\tau_i$ is freely reduced and contains no $a_1^{\pm 1}$ letters by Lemma~\ref{lemma i}.  Then \eqref{ii1} holds as per Figure~\ref{fig:ii} and \eqref{ii2}--\eqref{ii3} all follow straightforwardly from Lemma~\ref{lemma i}.  
\end{proof}

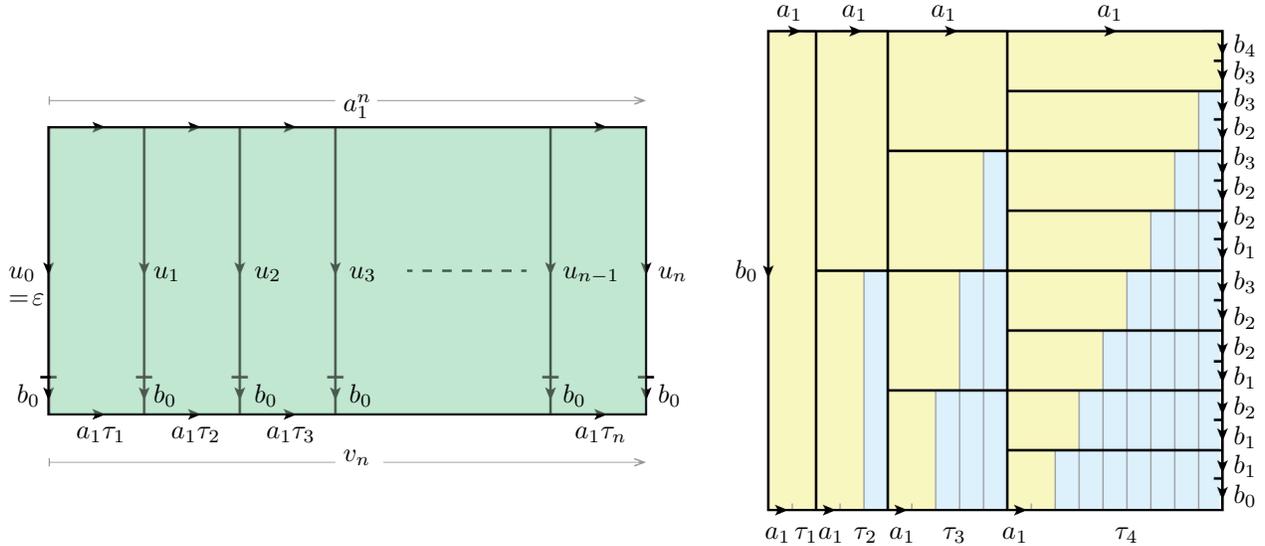
\begin{figure}[htbp]
\centering
\begin{overpic}  [scale=0.75] 
{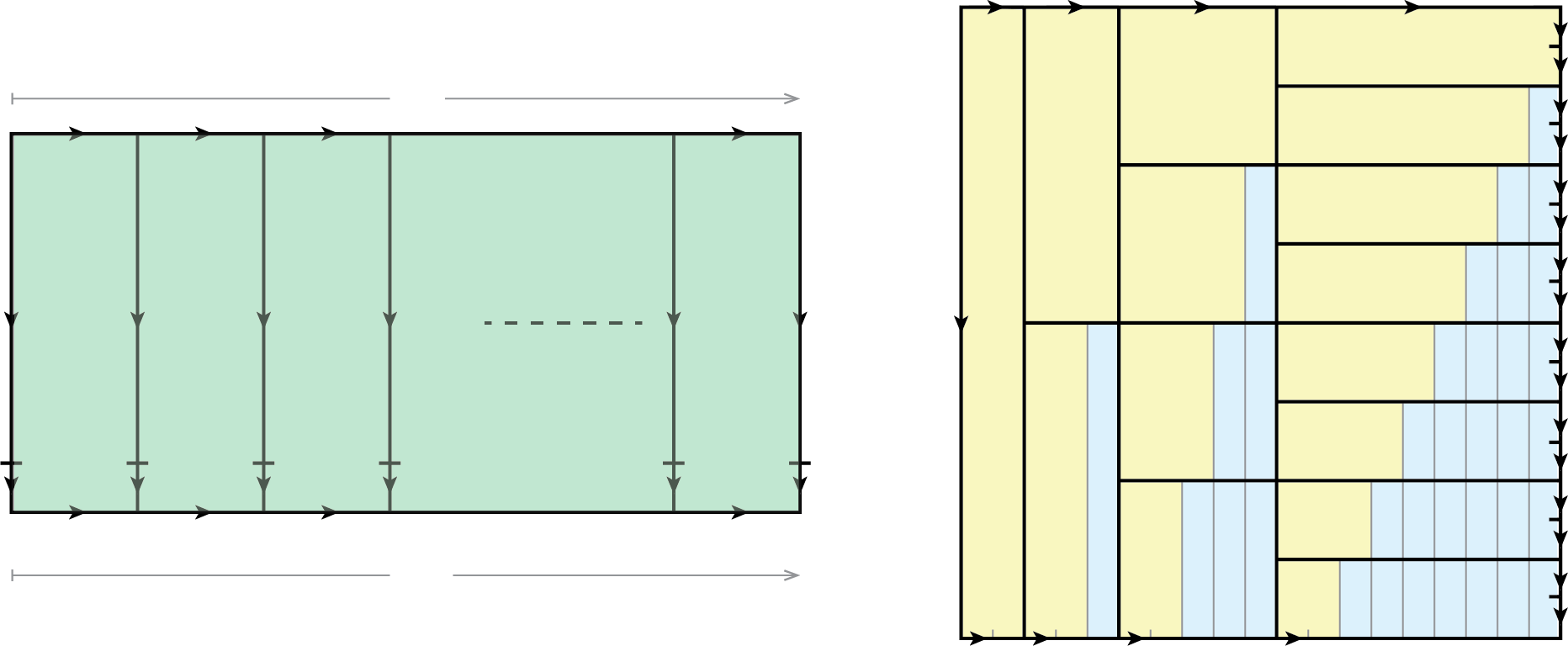}
\put(25.5, 4.5){\small{$v_n$}}
\put(3, 6.5){\small{$a_1\tau_1$}}
\put(11, 6.5){\small{$a_1\tau_2$}}
\put(19, 6.5){\small{$a_1\tau_3$}}
\put(45, 6.5){\small{$a_1\tau_n$}}
\put(25.5, 34){\small{$a_1^n$}}
\put(-2.7, 20){\small{$u_0$}}
\put(-2.7, 17.7){\small{$= \! \varepsilon$}}
\put(-2, 9.5){\small{$b_0$}}
\put(9.5, 20){\small{$u_1$}}
\put(9.5, 9.5){\small{$b_0$}}
\put(18, 20){\small{$u_2$}}
\put(18, 9.5){\small{$b_0$}}
\put(26, 20){\small{$u_3$}}
\put(26, 9.5){\small{$b_0$}}
\put(44, 20){\small{$u_{n-1}$}}
\put(44, 9.5){\small{$b_0$}}
\put(52, 20){\small{$u_n$}}
\put(52, 9.5){\small{$b_0$}}
\put(58.5, 20){\small{$b_0$}}
\put(61, -2){\small{$a_1$}}
\put(63.5, -2){\small{$\tau_1$}}
\put(65.5, -2){\small{$a_1$}}
\put(68.5, -2){\small{$\tau_2$}}
\put(71.5, -2){\small{$a_1$}}
\put(76, -2){\small{$\tau_3$}}
\put(81, -2){\small{$a_1$}}
\put(90.5, -2){\small{$\tau_4$}}
\put(62, 42){\small{$a_1$}}
\put(67.5, 42){\small{$a_1$}}
\put(75, 42){\small{$a_1$}}
\put(89, 42){\small{$a_1$}}
\put(100.5, 1){\small{$b_0$}}
\put(100.5, 3.5){\small{$b_1$}}
\put(100.5, 6){\small{$b_1$}}
\put(100.5, 8.5){\small{$b_2$}}
\put(100.5, 11){\small{$b_1$}}
\put(100.5, 13.5){\small{$b_2$}}
\put(100.5, 16){\small{$b_2$}}
\put(100.5, 19){\small{$b_3$}}
\put(100.5, 22){\small{$b_1$}}
\put(100.5, 24.5){\small{$b_2$}}
\put(100.5, 27){\small{$b_2$}}
\put(100.5, 29.5){\small{$b_3$}}
\put(100.5, 32){\small{$b_2$}}
\put(100.5, 34.5){\small{$b_3$}}
\put(100.5, 36.7){\small{$b_3$}}
\put(100.5, 39){\small{$b_4$}}

\end{overpic}
\vspace{4mm}
\caption{Why  
$a_1^n  u_n b_0   =  b_0  v_n$   in $G$.  
The  diagram on the left is assembled from $n$ instances of the diagram from Figure~\ref{fig:i}.  That on the right shows it in finer detail in the case $n=4$ and $q \geq 4$.}
\label{fig:ii}
\end{figure}

Let $\hat{v}_n$ be $v_n$ with all $t^{\pm 1}$,  $y_1^{\pm 1}$ and $y_2^{\pm 1}$ deleted.  

\begin{lemma} \label{lemma iii}
For all $n \geq 1$, there is a freely reduced word $\mu_n = \mu_n(t^{\pm 1}, y_1^{\pm 1}, y_2^{\pm 1})$, whose final letter is positive,  and such that 
\begin{align}
v_n  & \ = \   \hat{v}_n \mu_n  \text{ in } G. 
\end{align}
\end{lemma}

\begin{proof}  Use the $r_{4, \ast, \ast}$- and  $r_{4, \ast}$-defining relators of Figure~\ref{fig:relations} to shuffle the $a_1$ and $a_2$ through $v_n$  to its start to make a prefix $\hat{v}_n$.  In the process, the intervening letters $t^{\pm 1}, y_1^{\pm 1}, y_2^{\pm 1}$ become various $(\Yb t \Yb)^{\pm 1}$ and $(\Yb t^{-1} \Yb t \Yb)^{\pm 1}$.  

By Lemma~\ref{lemma i}, $\tau_n$, and therefore $v_n$, has a suffix $\kappa$ that is a  \emph{long} suffix of some $\Yb$.  The $\Yb^{\pm 1}$ that are created in the shuffling process are different from any that arise in Lemmas~\ref{lemma i}--\ref{lemma iii} (those lemmas do not use the relators $r_{4, \ast}$ or $r_{4, \ast, \ast}$). So, by $C'(1/4)$ for $\mathcal{X} \cup \mathcal{Y}$ (see Section~\ref{sec:the defn}), cancellation with these $\Yb^{\pm 1}$ cannot erode all of $\kappa$. So the final letter of $\mu_n$ is the final letter of $\kappa$, and so of some  $\Yb$,  and so is positive. 
\end{proof}

\begin{lemma} \label{lemma iv}  
There exists $K_1>1$ with the following property.   For all $n \geq 1$, there is a reduced word $Z_n$ on  $t$, $y_1$, and $y_2$, whose first letter is positive, such that    
 \begin{align}
(u_n b_0)^{-1} \, x_1 \, u_n b_0    & \ = \  Z_n \text{ in } G \\ 
K_1^{|u_n b_0|}  & \ \leq \  | Z_n |.  \label{iv2}
\end{align}
\end{lemma}

\begin{proof} The word $Z_n$ is the result of successive conjugations of $x_1$ by the letters of $u_n$ (which are $b_1, \ldots, b_p$) and then by $b_0$.   The relators $r_{3, \ast}$ and $r_{3, \ast, \ast}$ describe the  effect: conjugation produces successive words on $t$ and the $\Xb$  (so on  $t$, $x_1$ and $x_2$) until the final conjugation by  $b_0$, which results in a  word on $t$ and the $\Yb$  (so on $t$, $y_1$ and $y_2$). In any one of these words, free reduction between adjacent $\Xb^{\pm 1}$  (or adjacent $\Yb^{\pm 1}$) can only reduce the word's length by at most a half on account of the $C'(1/4)$ condition on $\mathcal{X} \cup \mathcal{Y}$ (see Section~\ref{sec:the defn}).  So, if we take $K_1$ to be half the  length of the shortest of the $\Xb$ and   $\Yb$, then each conjugation increases reduced length by a factor of at least $K_1$. The $C'(1/4)$-condition for $\mathcal{X} \cup \mathcal{Y}$ also implies that free reduction cannot erode the first letter of the word at every stage, and as the initial $x_1$ is positive and so are first letters of each $\Xb$ and $\Yb$, it follows that the first letter of $Z_n$ is positive.  
\end{proof}

\begin{lemma} \label{lemma v}
    There exist $K_2>0$ and  $K_3 >1$ with the following properties.  For all $n \geq 1$, the word  $$w_n \  = \ \hat{v}^{-1}_nb_0^{-1} a_1^n x_1 a_1^{-n} b_0 \hat{v}_n$$ 
has length at most $K_2 n^q$ and equals in $G$ a word $\chi_n = \chi_n(t^{\pm 1}, y_1^{\pm 1}, y_2^{\pm 1})$.  Moreover, freely reducing $\chi_n$ gives a word of length at least $K_3^{(n^p)}$.   
 \end{lemma}

\begin{figure}[htbp]
\centering
\begin{overpic} 
{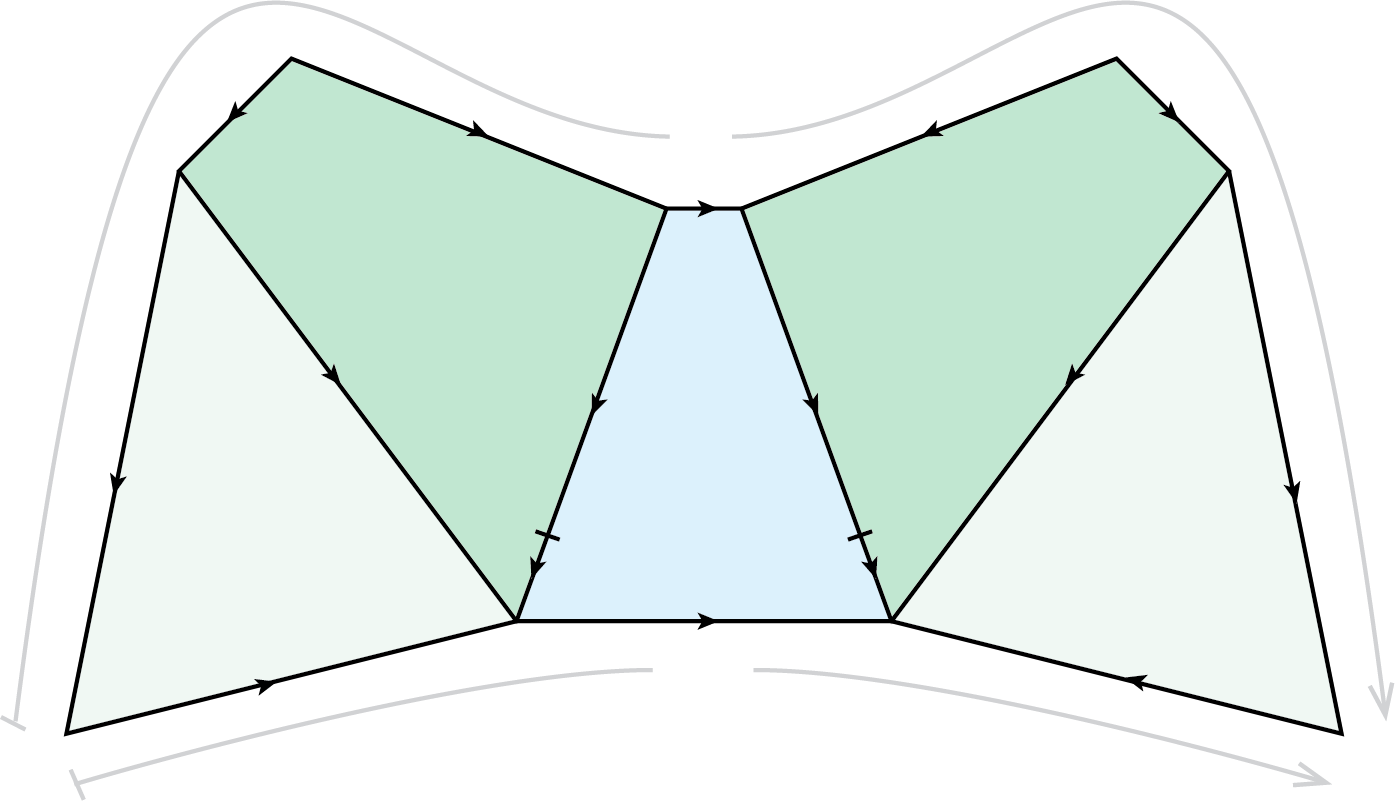}
\put(18, 47){\small{$b_0$}}
\put(81, 47){\small{$b_0$}}
\put(32, 44){\small{$a_1^n$}}
\put(65, 44){\small{$a_1^n$}}
\put(37.5, 27.5){\small{$u_n$}}
\put(60, 27.5){\small{$u_n$}}
\put(35, 17){\small{$b_0$}}
\put(63.5, 17){\small{$b_0$}}
\put(25.5, 30){\small{$v_n$}}
\put(72, 30){\small{$v_n$}}
\put(9.5, 22){\small{$\hat{v}_n$}}
\put(88.5, 22){\small{$\hat{v}_n$}}
\put(17, 11){\small{$\mu_n$}}
\put(81, 11){\small{$\mu_n$}}
\put(49, 46.5){\small{$w_n$}}
\put(49, 39.5){\small{$x_1$}}
\put(48.5, 15){\small{$Z_n$}}
\put(48.5, 8.7){\small{$\chi_n$}}
\put(22, 37){\parbox{30mm}{\tiny{\emph{Lemma~\ref{lemma ii}}}}}
\put(63, 37){\parbox{30mm}{\tiny{\emph{Lemma~\ref{lemma ii}}}}}
\put(13, 17){\parbox{30mm}{\tiny{\emph{Lemma~\ref{lemma iii}}}}}
\put(76, 17){\parbox{30mm}{\tiny{\emph{Lemma~\ref{lemma iii}}}}}
\put(44, 22){\parbox{30mm}{\tiny{\emph{Lemma~\ref{lemma iv}}}}}
 \end{overpic}
\vspace{4mm}
\caption{A diagram demonstrating that the word $w_n = \hat{v}^{-1}_nb_0^{-1} a_1^n x_1 a_1^{-n} b_0 \hat{v}_n$ on the generators of $G$ and word $\chi_n = \mu_n  Z_n \mu_n^{-1}$ on the generators of $H$ represent the same element of $G$. } \label{fig:v}
\end{figure}
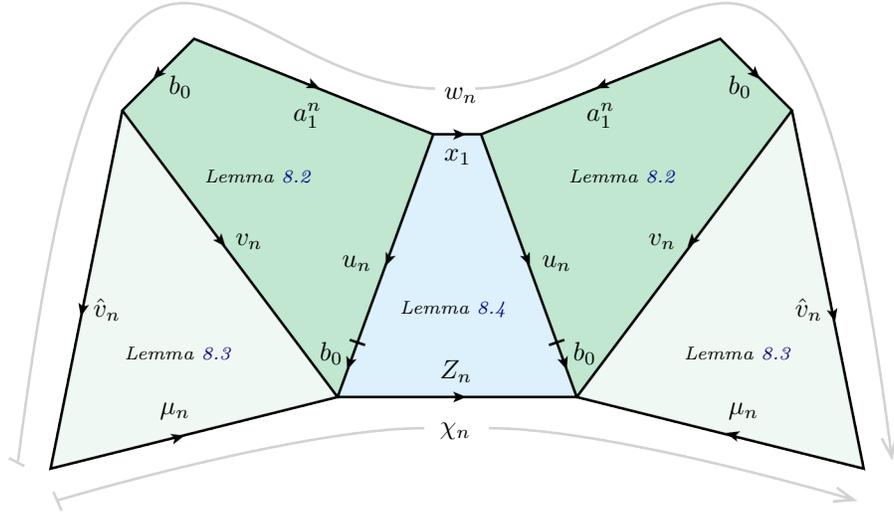

\begin{proof}  
We have $|\hat{v}_n|  = |\hat{v}_n|_{a_1} + |\hat{v}_n|_{a_2}$, which equals $|v_n|_{a_1} + |v_n|_{a_2}  =n + \SB{n}{q}$ by \eqref{ii2} and \eqref{ii3}.  So  $|w_n|  = 2   \SB{n}{q}  + 2n + (2n + 3)$, which is at most   $K_2 n^q$  for a suitable constant $K_2 >0$.

Figure~\ref{fig:v} sets out why $\chi_n = \mu_n  Z_n \mu_n^{-1}$ equals $w_n$ in $G$.  Consider freely reducing $\chi_n$ by freely reducing    $\mu_n$, $Z_n$, and  $\mu_n^{-1}$, and then performing all available cancellations where they meet.  As the final letter of the freely reduced form of $\mu_n$ and the first letter of the freely reduced form of $Z_n$ are both positive (by Lemmas~\ref{lemma iii} and \ref{lemma iv}), there is no cancellation between $\mu_n$ and $Z_n$. There may be cancellation between $Z_n$ and  $\mu_n^{-1}$ (indeed, a priori, all of $Z_n$ could cancel into $\mu_n^{-1}$). But for every letter of $Z_n$ that cancels into $\mu_n^{-1}$, there is a letter of $\mu_n$ that survives in the freely reduced form of $\chi_n$.  Therefore the length of the freely reduced form of $\chi_n$ is at least  the length of the freely reduced form of $Z_n$.      So the existence of a suitable $K_3 >1$  follows from \eqref{iv2} and the fact that, by \eqref{ii5}, $|u_n b_0|$ is a least a constant times $n^p$.   
\end{proof}

\section{Tracks and diagram rigidity} \label{ch:tracks and diagram rigidity}

\subsection{Tracks in reduced van~Kampen diagrams}  \label{sec:tracks in reduced diagrams}

As explained in Section~\ref{sec:tracks}, a van~Kampen diagram is \emph{reduced} when it does not contain a pair of   back-to-back cancelling 2-cells.  If a van~Kampen diagram is reduced, then so are its subdiagrams.  Here, we will explore the restrictions this hypothesis leads to on the arrangement of tracks in van~Kampen diagrams over our presentation $\mathcal{P}$ for $G$ of Section~\ref{sec:the defn}.

\begin{definition} \label{def:region}
A \emph{region} in a van~Kampen diagram $\Delta$ is a closed  subset that is homeomorphic to a 2-disc.  We will consider regions that have boundary circuits comprised of portions of $\partial \Delta$, other paths in the 1-skeleton $\Delta^{(1)}$,  and  subtracks.  Figure~\ref{fig:bad examples} shows two examples.  Because tracks pass through the interiors of 2-cells,   regions need not be subdiagrams.   When we say a 1-cell or 2-cell of $\Delta$ is \emph{in} $R$, we mean that it is a subset of $R$.  
\end{definition}

Before we give our first lemma, here is an   overview of this section. Every 2-cell in  a reduced van~Kampen diagram $\Delta$ over $\mathcal{P}$ has some $x$- or $y$-letters (we call these ``noise'' letters) in its boundary word.  We find it helpful to think of this noise to be \emph{flowing though the diagram and expanding} in that,  for the 2-cells to fit together,  the adjacent cells must  have  more noise (in total), and those in the next layer further beyond those have yet more noise. This continues until the noise spills out into the boundary of the diagram.  

Tracks in $\Delta$ mediate this flow of noise and provide a structure via which we can put this intuition on a firm foundation.  All $x$-noise flows across $b$-tracks in the direction of their orientations, except that on crossing a $b_0$-track, the noise is converted  to $y$-noise.  And $y$-noise flows across $a$-tracks in the direction of their orientations.  So, when a region has boundary that prevents the escape of noise, that region  cannot occur in a reduced diagram.  Lemmas~\ref{lem: trapped x-noise},  \ref{lem: trapped y-noise} and \ref{lem: no junctions no loops} are results of this nature.   As for $t$-tracks, they have noise on both sides and reflect the HNN-structure $G = F \ast_t$. Lemma~\ref{lem: no t-loop} is a consequence.  It exemplifies the following idea,  which reappears in Lemma~\ref{lem: bigons} in a more complicated guise.  If a certain feature is present (in this case, a $t$-loop), then there is an innermost instance, but an innermost instance must include cancelling 2-cells, contrary to the hypothesis that the diagram is reduced.    
 
  Lemmas~\ref{lem:corridor overlap} and \ref{lem: no self-osculating t-corridors} dig further into the structure of $t$-corridors and provide groundwork for Lemmas~\ref{lem: vertical t-corridors} and  \ref{lem:t-corridors on boundary}, which detail circumstances in which tracks and corridors show diagrams  to flare out towards a portion of their boundary. 
  These results will let us  (in Lemma~\ref{lem: Layout lemma}) simplify diagrams that demonstrate distortion.

\begin{lemma} \label{lem: no t-loop}
Reduced van~Kampen diagrams $\Delta$  over $\mathcal{P}$ contain no $t$-loops.
\end{lemma}

\begin{proof}

Were there a $t$-loop in $\Delta$, there would be one with no  $t$-loop in its interior.  The 2-cells it traverses would form an annular corridor.  Around its inner boundary we read a word which, viewed as a word on the generators of the appropriate vertex group of the HNN-structure $G = F \ast_t$ of Proposition~\ref{prop:HNNWise}, would freely equal  the empty word.  So some adjacent pair of those generators would cancel.  As those generators uniquely determine the 2-cells along whose sides they are read, a pair of 2-cells in the annulus would cancel, contrary to the diagram being reduced.  
\end{proof}

Our next lemma sets out circumstances in which $x$-edges   being absent from the boundary of a region $R$ forces there to be no $x$-edge  anywhere in $R$.  
The lemma further explains that  regions that do not contain a $y$-edge and are bounded only by $a$-subtracks, inward-oriented $b$-subtracks, and  $t$-subtracks take a highly constrained form, examples of which are shown in Figure~\ref{fig:bad examples}.

\begin{lemma} \label{lem: trapped x-noise} \textup{\textbf{(Trapped $x$-noise)}} Suppose $R$ is  a region in a reduced van~Kampen diagram $\Delta$  over $\mathcal{P}$ such that  $R$ contains no $y$-edges and is bordered by  $a$-subtracks, inward-oriented $b$-subtracks, $t$-subtracks, and paths in $\Delta^{(1)}$.   

\begin{enumerate}
	\item \label{lem part: no x-edges}  If there is an $x$-edge in $R$, then there is an $x$-edge in $\partial R$.

	\item \label{lem part: no x-edges cor} 
	If there is no $x$-edge  in $\partial R$ (in particular, if 
	$\partial R$ is made up of only  $a$-subtracks, inward-oriented $b$-subtracks, and $t$-subtracks), then 	 
\begin{enumerate}     
	\item \label{lem part part: ts} Each $t$-subtrack in $\partial R$ crosses only a single edge; indeed, it crosses between an $r_{4,1}$-cell and an $r_{4,2}$-cell as in the example in Figure~\ref{fig:bad examples} (right) and must transition to an outward-oriented $a_1$-subtrack  in  the $r_{4,1}$-cell and to an outward-oriented $a_2$-subtrack in  the $r_{4,2}$-cell.   
	\item \label{lem part part: bs} Each $b$-subtrack in $\partial R$ only crosses a single edge.  It transitions to an outward-oriented $a_1$-subtrack at one end and to an outward-oriented $a_2$-subtrack at the other.   
	\item  \label{lem part part: at least one} There is at least one $b$- or $t$-subtrack in $\partial R$. 
	\item \label{lem part part: as} The $a$-subtracks in $\partial R$ are all outward oriented. Together, they cross at least one $a_1$-edge  and at least one $a_2$-edge \end{enumerate} 
 \end{enumerate}
\end{lemma}

\begin{figure}[htbp]
\centering
\begin{overpic}  [scale=0.95]  
{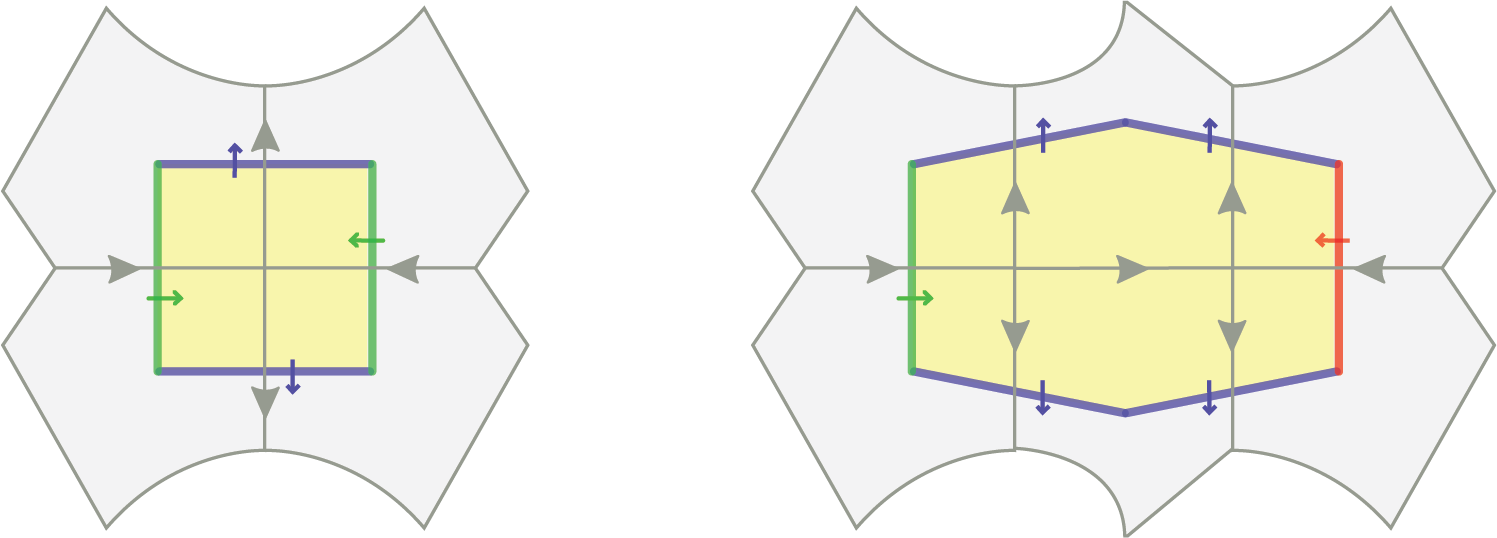}
\put(7, 15){\small{$b_1$}}
\put(26.5, 19.5){\small{$b_2$}}
\put(19, 26){\small{$a_1$}}
\put(14, 8){\small{$a_2$}}
\put(57, 15){\small{$b_0$}}
\put(69, 12){\small{$a_2$}}
\put(69, 23){\small{$a_1$}}
\put(73, 20){\small{$b_{q-1}$}}
\put(79, 12){\small{$a_2$}}
\put(79, 23){\small{$a_1$}}
\put(91, 19.5){\small{$t$}}
\put(7, 28){\gray{\tiny{$r_{1,1}$}}}
\put(7, 8){\gray{\tiny{$r_{2,1}$}}}
\put(25, 28){\gray{\tiny{$r_{1,2}$}}}
\put(25, 8){\gray{\tiny{$r_{2,2}$}}}
\put(57, 28){\gray{\tiny{$r_{1,0}$}}}
\put(57, 8){\gray{\tiny{$r_{2,0}$}}}
\put(73, 30){\gray{\tiny{$r_{1,q-1}$}}}
\put(73, 5){\gray{\tiny{$r_{2,q-1}$}}}
\put(90, 28){\gray{\tiny{$r_{4,1}$}}}
\put(90, 8){\gray{\tiny{$r_{4,2}$}}}

\end{overpic}
\caption{Examples of regions satisfying the conditions of Lemma~\ref{lem: trapped x-noise}\eqref{lem part: no x-edges cor}}
\label{fig:bad examples}
\end{figure}
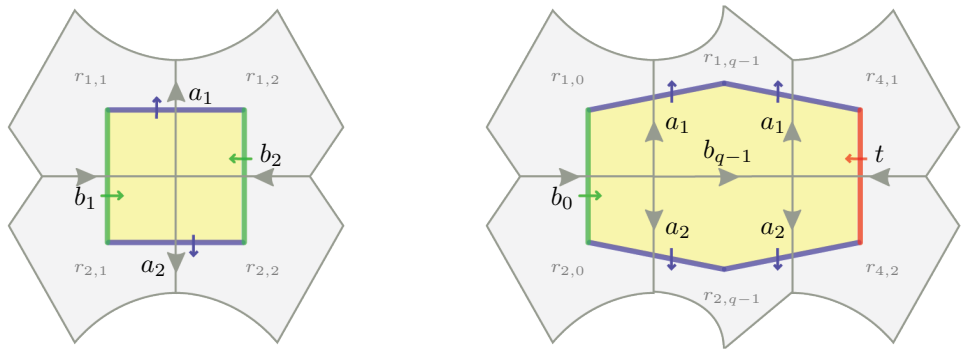

\begin{proof}  For \eqref{lem part: no x-edges},  first suppose that there is a 2-cell $c$ in $R$. By Lemma~\ref{lem: no t-loop}, there  is no $t$-loop in $\Delta$, and so the two $t$-edges in $\partial c$ are part of a $t$-subtrack that subdivides $R$ into two regions $R_1$ and $R_2$. 
If \eqref{lem part: no x-edges}  holds true for $R_1$ and $R_2$, then it holds true for $R$.  Thus, via  repeated such subdivisions, we reduce to the case where $R$ contains no 2-cell.    In that event, the subgraph $\mathcal{F}$ of $\Delta^{(1)}$ formed by the 1-cells in $R$ is a forest: were it to  contain an embedded circle, there would be a 2-cell within that circle and so in $R$.  
(In the examples of Figure~\ref{fig:bad examples}, $\mathcal{F}$ is a single vertex in the left diagram and it is the single central edge labelled $b_{q-1}$ in the right diagram.)

Assume there is  no $x$-edge in $\partial R$.  
Suppose, for a contradiction, that there is an $x$-edge in $R$, and so in some connected component $\mathcal{F}_0$ of $\mathcal{F}$.  Let $v$ be the word one reads around $\mathcal{F}_0$.    Let $\overline{v}$ be $v$ with all letters other than $x_1^{\pm 1}$ and $x_2^{\pm 1}$ deleted.

By hypothesis, there are no $y$-edges in $R$.  So $v$ is a word on $a_1, a_2, b_0, \ldots,$ $b_p, t, x_1, x_2$.  Any $x_1$ or $x_2$ in $v$ is the label of an edge $e_x$  of a 2-cell and so is either part of a  Rips subword from $\mathcal{X}$ in a defining relation, or is the lone $x_j$ at the top  (in the sense of  Figure~\ref{fig:relations}) of an $r_{3,i,j}$-cell $c$ 
(for some $i \in \set{0, \ldots, p}$).  In the latter event, no part of the $b$-track through $c$ 
can be part of $\partial R$ because then there would be an \emph{outward}-oriented $b$-subtrack, contrary to hypothesis. 
It follows that $\partial R$ contains the $t$-track of $c$ (as $c$ is not in $R$)  and that $\mathcal F_0$ contains a portion of $\partial c$ containing $e_x$  so that $x_j$ is 
part of a subword $\Xb^{-1} b_i^{-1} x_j b_i  \Xb^{-1}$ of $v^{\pm 1}$.  
So, after replacing $\overline{v}$ with a cyclic conjugate if necessary, $\overline{v}$ is a word on the $\Xb$, $\Xb^{-1}  x_1  \Xb$, and  $\Xb^{-1}  x_2  \Xb$.

Now, $v$  freely reduces to the empty word since it is read around the tree $\mathcal{F}_0$.  Therefore $\overline{v}$ also  freely reduces to the empty word. Lemma~\ref{lem: cancellation from C'(1/4)} applies to $\overline{v}$. Folding up an edge-loop labelled by $\overline{v}$ to get the  tree $\mathcal{F}_0$ equates to freely reducing $\overline{v}$.  So the lemma tells us that  parts of the   boundary cycles of some pair of 2-cells is a common path in $\mathcal{F}_0$ labelled by a subword  of some $X_{\ast}$  of at least a quarter-length.  These 2-cells are a back-to-back cancelling  pair, contrary to the diagram being reduced. So we have the contradiction we seek.

To prove \eqref{lem part: no x-edges cor}, we assume there are no $x$-edges in $\partial R$, and therefore none  in $R$ by  \eqref{lem part: no x-edges}.

For \eqref{lem part part: ts}, suppose $\tau$ is a  $t$-subtrack in $\partial R$.  It cannot intersect a $t$-edge that is part of a subword $\Yb t \Yb$  or $\Yb t^{-1} \Yb t \Yb$ in the boundary of a 2-cell, for then an adjacent $y$-edge would be in $R$, contrary to hypothesis.  It also cannot intersect a $t$-edge that is part of a subword $\Xb t \Xb$  or $\Xb t^{-1} \Xb t \Xb$ in the boundary of a 2-cell, for then an adjacent $x$-edge would be in $R$.  
The remaining possibility is that it intersects a $t$-edge  at the top of an $r_{3,i}$- or $r_{4,i}$-cell.  It cannot intersect the other $t$-edge in that cell, so $\partial R$ has to switch from a $t$-subtrack to, respectively, a $b_i$- or $a_i$- subtrack within that  cell. The former case cannot occur, as it would lead to an outward oriented $b$-track.  In the latter case, the
2-cell on the other side of that top $t$-edge must also be an $r_{4,i}$-cell.    
As the diagram is reduced, we deduce that $\tau$ crosses from an $r_{4,1}$-cell to an $r_{4,2}$-cell across their common `top' $t$-edge.  Moreover, to avoid any  $y$-edge being in $R$, $\partial R$ must exit the $r_{4,1}$-cell across an $a_1$-edge and exit the $r_{4,2}$-cell across $a_2$-edge, and these $a_1$- and $a_2$-edges must have a common end-vertex in $R$ and must both be oriented out of $R$.

For \eqref{lem part part: bs}, suppose $\beta$ is a $b$-subtrack in $\partial R$. It is impossible that $\beta$ enters and then exits a 2-cell: by hypothesis $\beta$ is inward-oriented and so 
$R$ would contain $x$- or $y$-edges from the bottom of the 2-cell (in the sense of Figure~\ref{fig:relations}).
So $\beta$ crosses only a single $b$-edge, and when doing so it travels from one 2-cell to another.  (It cannot start and end in the same 2-cell, as then two   $b$-edges in the boundary of one 2-cell would be identified in $\Delta$ and that would imply that some subword of the boundary word represents $1$ in $G$ in such a way as to contradict the HNN-structure established in Proposition~\ref{prop:HNNWise}.)   From our analysis of $t$-subtracks, we know that $\beta$ cannot transition in $\partial R$ to a $t$-subtrack, and so it  must transition to $a$-subtracks at each end.  Indeed, it must transition to outward-oriented $a$-subtracks, since the $x$- or $y$-edges of a 2-cell in which a  transition to an inward-oriented $a$-subtrack occurred would be in $R$.  And $\beta$ must connect an $a_1$-subtrack at one end and to an $a_2$-subtrack at the other, because otherwise the two 2-cells it passes through would be a cancelling pair, contrary to $\Delta$ being reduced.

For \eqref{lem part part: at least one}, all that remains is to verify that $\partial R$ is not an $a$-loop.  
It cannot be an  inward-oriented $a$-loop, for then there would be $x$- or $y$-letters in $R$.     
 Consider an inner-most outward-oriented $a$-loop $\alpha$.  The orientations on junctions in $\mathcal G_a$ force $\alpha$ to be an $a_1$- or $a_2$-loop, and the associated $a_1$- or $a_2$-annulus has inner boundary
 labelled by a non-empty word $w$ on $b_0, \ldots, b_p$, which freely reduces to the empty word.  The 2-cells in the annulus are $r_{1,i}$-cells ($i=1, \ldots, p$) in the $a_1$ case and are $r_{2,i}$-cells ($i=1, \ldots, p$) in the $a_2$ case.  In either case cancellation of an inverse-pair of  letters in $w$ implies    cancellation of a pair of 2-cells in $\Delta$, contrary to the diagram being reduced. 

We conclude that $\partial R$ has at least one $a_1$-subtrack and at least one $a_2$-subtrack, and any $a$-subtrack transitioning to a $b$- or $t$-track is outward oriented.  Were there an inward-oriented $a$-subtrack, it would have to be an $a_1$-subtrack $\alpha$ transitioning at either end to an outward oriented $a_2$-subtrack in distinct $r_{1, q-1}$-cells $c$ and $c'$.  Any 2-cell that $\alpha$ passed through between $c$ and $c'$ would lead to an $x$-or $y$-edge in $R$, so $c$ and $c'$ must be adjacent, which would be a contradiction because they are oppositely oriented.  Thus \eqref{lem part part: as} follows. 
   \end{proof}

Here is the corresponding lemma for $y$-letters. It forgoes  hypotheses excluding any particular type of edges from $R$, and it requires the $a$-subtracks,  instead of $b$-subtracks, in $\partial R$  to be inward-oriented.

\begin{lemma} \label{lem: trapped y-noise} \textbf{(Trapped $y$-noise)} Suppose $R$ is a region  in a reduced van~Kampen diagram $\Delta$  over $\mathcal{P}$, bordered by    $b$-subtracks, $t$-subtracks, inward-oriented $a$-subtracks, and   paths in $\Delta^{(1)}$. 

\begin{enumerate}
	\item \label{lem part: no y-edges}  If there is a $y$-edge in $R$, then there is a  $y$-edge in $\partial R$.   

	\item  \label{lem part: no y-edges cor} If the  $b$-subtracks in  $\partial R$ are inward oriented, then $\partial R$ must include at least one $x$-edge or $y$-edge.  In particular, in a reduced  diagram there is no region $R$ such that $\partial R$ is comprised of inward-oriented $a$-subtracks, inward-oriented $b$-subtracks, and $t$-subtracks.  (Figure~\ref{fig:forbidden examples} shows some examples of regions this precludes.)
\end{enumerate}
\end{lemma}

\begin{figure}[htbp]
\centering
\begin{overpic}[width=300pt] 
{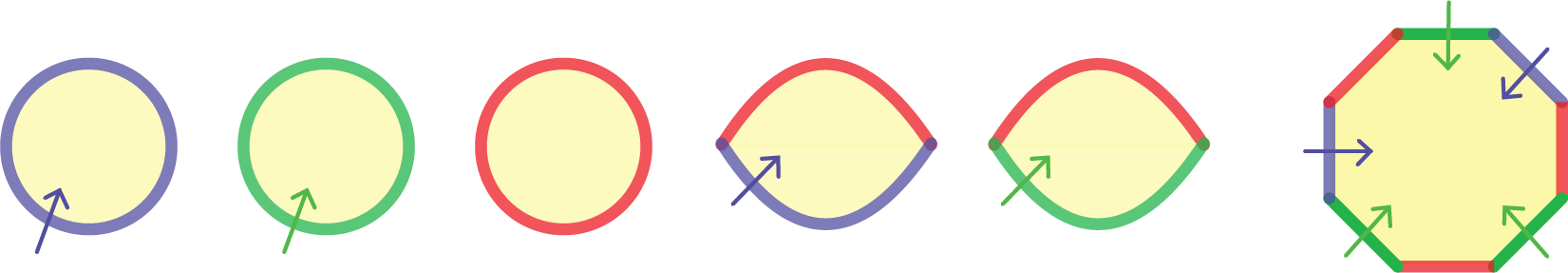}
\put(-0.4, 2.1){\small{$a$}}
\put(15.5, 1){\small{$b$}}
\put(32, 0.5){\small{$t$}}
\put(47, 1){\small{$a$}}
\put(65, 1){\small{$b$}}
\put(64.5, 12){\small{$t$}}
\put(47, 12){\small{$t$}}
\put(84, 1){\small{$b$}}
\put(81, 7){\small{$a$}}
\put(85, 13.4){\small{$t$}}
\put(99.5, 1){\small{$b$}}
\put(101, 7){\small{$t$}}
\put(99, 14){\small{$a$}}
\put(92, -2.5){\small{$t$}}
\put(93.3, 16.5){\small{$b$}}

\end{overpic}
\caption{Examples of regions precluded by Lemma~\ref{lem: trapped y-noise}\eqref{lem part: no y-edges cor}}
\label{fig:forbidden examples}
\end{figure}
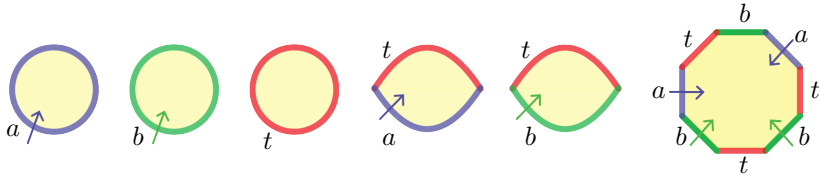

\begin{proof} For \eqref{lem part: no y-edges}, we follow the same approach as our proof of Lemma~\ref{lem: trapped x-noise}\eqref{lem part: no x-edges}.  As there, it suffices to prove the result in the case where there is no $2$-cell in $R$.  In that case, if there is a $y$-edge in $R$, then it appears in some connected component $\mathcal{F}_0$ of the forest of $1$-cells in $R$, and around $\mathcal{F}_0$ we read a word $v$ which freely reduces to the empty word.      This $v$ is a word on $a_1, a_2, b_0, \ldots, b_p, x_1, x_2, t$, and the Rips words $\mathcal{Y}$ (arising in the $\Yb t \Yb$ or $\Yb t^{-1} \Yb t \Yb$ per our presentation $\mathcal{P}$), and the $\Yb^{-1} a_i^{-1} y_j a_i  \Yb^{-1}$ around $r_{4,i,j}$-cells---the key point here is that $y_1$ and $y_2$ do not appear on their own in this list and this is because if the $y_j$ of $\Yb^{-1} a_i^{-1} y_j a_i  \Yb^{-1}$ is in $v^{\pm 1}$, then the whole of that subword is in $v^{\pm 1}$ as an $a$-subtrack across that $r_{4,i,j}$-cell would be outwards-oriented, contrary to hypothesis.    Let $\overline{v}$ be $v$ with all letters other than $y_1^{\pm 1}$ and $y_2^{\pm 1}$ deleted. Then $\overline{v}$ is a word on the $\Yb$,  $\Yb^{-1}  y_1  \Yb$, and  $\Yb^{-1}  y_2  \Yb$ which  freely reduces to the empty word.   Lemma~\ref{lem: cancellation from C'(1/4)}, translated to $y$-letters instead of $x$-letters, applies to $\overline{v}$, so as to imply that a pair of 2-cells cancel, contrary to the diagram being reduced.

For \eqref{lem part: no y-edges cor}, assume, for a contradiction, that there is no $x$- or $y$-edge in $\partial R$. Then, by \eqref{lem part: no y-edges}, there is no $y$-edge in $R$.  This, together with the hypothesis that the $b$-subtracks in  $\partial R$ are inward oriented  and the assumption that $\partial R$ has no $x$-edges, means Lemma~\ref{lem: trapped x-noise}\eqref{lem part: no x-edges cor} applies, and part \eqref{lem part part: as} tells us that $\partial R$ has non-trivial outward-oriented $a$-tracks, contradicting the hypothesis that $a$-subtracks in $\partial R$ are inward-oriented.    
\end{proof}

\begin{remark}
The analogue of Lemma~\ref{lem: trapped y-noise}\eqref{lem part: no y-edges} for $x$-edges fails.  For an example, take the van~Kampen diagram that demonstrates that $b_0^{-1}b_1^{-1} t b_1 b_0$ equals a word on $y_1$, $y_2$, and $t$, which is comprised of one $r_{3,1}$-cell and a $b_0$-corridor made up of $r_{3,0,1}$- and $r_{3,0,2}$-cells and   one $r_{3,0}$-cell.  

Lemma~\ref{lem: trapped y-noise}\eqref{lem part: no y-edges cor}  fails in the absence of  the hypothesis that the $b$-tracks be inward-orientated. A ``button'' (Definition~\ref{def: button}) provides an example.  
\end{remark}

Lemma~\ref{lem: trapped y-noise}\eqref{lem part: no y-edges cor} rules out $a$- and $b$-loops that are inward oriented.   At this stage we can also rule out outward oriented $a$-and $b$-loops in some situations: 

\begin{lemma} \label{lem: no junctions no loops}
Let  $\Delta$ be a reduced van~Kampen diagram over $\mathcal{P}$.
\begin{enumerate}

\item 
\label{lem part: r4 diagrams}
If $\Delta$ has only $r_{4, \ast}$-cells, then $\Delta$ has no $a$-loops.

\item \label{lem part: r23 diagrams}
If $\Delta$ has only $r_{2, \ast}$- and $r_{3, \ast}$-cells, then 
 $\Delta$ has no $b$-loops

\end{enumerate}
\end{lemma}
\begin{proof}
In both cases there are no $r_{1, \ast}$-cells.  Thus   
the dual graphs $\mathcal G_a$ and $\mathcal G_b$ of $\Delta$  have no junctions, so every $a$-track is an $a_i$-track and every $b$-track is a $b_j$-track, for some $i$ and $j$.   

To prove~\eqref{lem part: r4 diagrams}, suppose for a contradiction that $\Delta$ has an $a$-loop.  Then there is an innermost one $\alpha$, which is an $a_i$-loop for $i=1$ or $2$, such that the region $R$ enclosed $\alpha$ has no $a$-subtracks (as there are no junctions).   As $\Delta$ has only $r_{4,\ast}$-cells, this means that the inner boundary of the annulus associated to $\alpha$  is a closed path in $\Delta^{(1)}$  that encloses no $2$-cells, so traverses some edge $e$ twice (in opposite directions). 
Lemma~\ref{lem: trapped y-noise}\eqref{lem part: no y-edges cor} implies that $\alpha$ is outward-oriented and this, together with the fact that $\alpha$ is an $a_i$-track for a fixed $i$, means that the possible labels $y_1, y_2, t$ of $e$ determine unique $r_{4,\ast}$-cells.  It follows that 
there is an adjacent pair of oppositely oriented identical cells, contradicting the fact that $\Delta$ is reduced.

The proof of~\eqref{lem part: r23 diagrams} is identical, noting that, for an innermost $b_j$-loop in a
$\Delta$ as in~\eqref{lem part: r23 diagrams},  the possible labels  $x_1, x_2, t, a_2$ of the edge $e$   each determine a unique cell (given the orientation of $b_j$).
\end{proof}

We now define two types of diagrams containing bigons of subtracks which can occur in reduced diagrams over $\mathcal P$. 
\begin{definition} \label{def: badge} \textup{\textbf{(Badge)}}
	A \emph{badge} is a subdiagram consisting of a path with label $t^n$, where $n>0$, with $2n+2$ cells arranged around it as shown in Figure~\ref{fig:button}(left) for $n=4$.  Specifically, it has two $r_{i,j}$-cells that are connected by an $a_i$-corridor made up of $n$ $r_{4, i}$-cells and a $b_j$-corridor made of $n$ $r_{3, j}$-cells, such that the $a_i$-corridor and $b_j$-corridor are identified along their boundaries labelled $t^n$.  
	  \end{definition}

\begin{definition} \label{def: button}
\textup{\textbf{(Button)}}
	A \emph{button} is a pair of 2-cells, specifically an $r_{1,p-1}$-cell and an $r_{1,p}$-cell, in a van~Kampen diagram that are joined along the common $a_1 b_p$ subwords in their boundary word.  Figure~\ref{fig:button}(center) shows a button.  The mirror image of a button is also a button, so there are two buttons in the diagram in Figure~\ref{fig:button}(right).    
	  \end{definition}

\begin{figure}[htbp]
\centering
\begin{overpic} [scale=0.9] 
{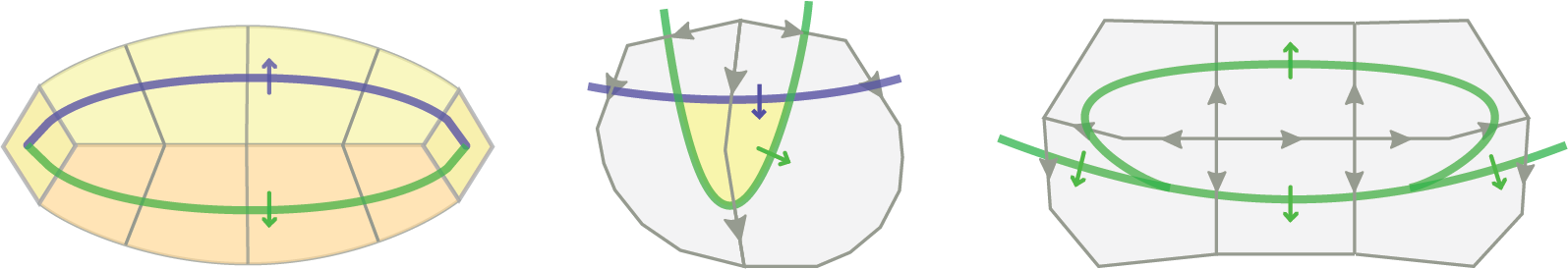}
\put(6.5, 8.4){\tiny{$t$}}
\put(13, 8.4){\tiny{$t$}}
\put(18.5, 8.4){\tiny{$t$}}
\put(24, 8.4){\tiny{$t$}}
\put(19, 1.5){\tiny{$b_{\ast}$}}
\put(20, 12.7){\tiny{$a_{\ast}$}}
\put(44, 2){\tiny{$b_p$}}
\put(43.5, 16.5){\tiny{$b_p$}}
\put(49, 16.5){\tiny{$b_{p-1}$}}
\put(43.5, 12.5){\tiny{$a_1$}}
\put(36.3, 13){\tiny{$a_1$}}
\put(56, 13){\tiny{$a_1$}}
\put(39.5, 6){\gray{\tiny{$r_{1,p}$}}}
\put(50, 4){\gray{\tiny{$r_{1,p-1}$}}}
\put(61.5, 6){\tiny{$b_{p-1}$}}
\put(69, 9.8){\tiny{$b_{p}$}}
\put(82, 9.2){\tiny{$t$}}
\put(78.5, 10.5){\tiny{$b_p$}}
\put(87.3, 10.5){\tiny{$b_p$}}
\put(98, 5){\tiny{$b_{p-1}$}}
\put(78.5, 5.5){\tiny{$b_{p-1}$}}
\put(87.3, 5.5){\tiny{$b_{p-1}$}}
\put(73.5, 9.3){\tiny{$a_1$}}
\put(90, 9.1){\tiny{$a_1$}}
\put(93.7, 10.3){\tiny{$b_p$}}
\put(70, 2.5){\gray{\tiny{$r_{1,p-1}$}}}
\put(79.5, 2.2){\gray{\tiny{$r_{3,p-1}$}}}
\put(89, 2.5){\gray{\tiny{$r_{1,p-1}$}}}
\put(71, 14){\gray{\tiny{$r_{1,p}$}}}
\put(80, 14.4){\gray{\tiny{$r_{3,p}$}}}
\put(89, 14){\gray{\tiny{$r_{1,p}$}}}
\end{overpic}
\caption{Left: a \emph{badge}.  Middle: a \emph{button}. Right: a reduced diagram that includes two buttons and contains a loop that is an outward-oriented $b$-track.  }
\label{fig:button}
\end{figure}
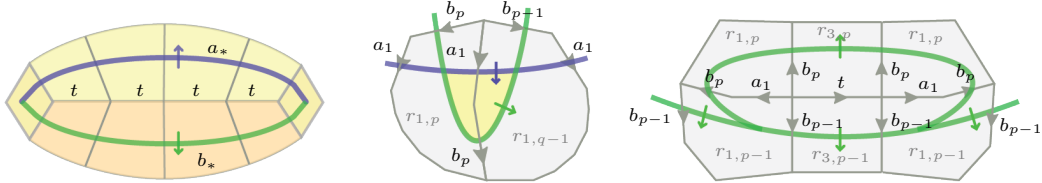

Observe that a badge or button is dual to a bigon comprised of an $a$-subtrack and an outward oriented $b$-subtrack.  The next lemma shows that such bigons always give rise to badges or buttons in the absence of $y$-edges. The second part puts a further restriction on certain bigons formed by an $a_1$-track and a $b_i$-track, which will be used in the proof of Corollary~\ref{cor: no loops}.

\begin{lemma} \label{lem: bigons}
\textup{\textbf{(Bigons, badges, and buttons)}}
Let $R$ be a region in a  reduced van~Kampen diagram $\Delta$  over $\mathcal{P}$, such that $R$ does not contain any $y$-edges, and $\partial R$ is a bigon comprised of an $a$-subtrack $\alpha$ and an outward oriented $b$-subtrack $\beta$.  Then 
\begin{enumerate}

\item \label{lem part: badge button}
The minimal subdiagram of $\Delta$  containing $R$
contains either a badge or a button. 
\item \label{lem part: bi bigon}
If $\alpha$ is an $a_1$-subtrack and $\beta$ is a $b_i$-subtrack, and $R$ has no $a_1$-subtracks in its interior,  then one of the  intersections between $\alpha$ and $\beta$ occurs in an $r_{1, i-1}$-cell.  
\end{enumerate}
 \end{lemma}

\begin{proof}
If $R$ is as in the statement of the lemma,  
we first prove that $R$ contains a minimal region of the same type.  Specifically, $R$ contains a region $S$ with boundary a bigon comprised of an $a$-track $\alpha_S$ and an outward oriented $b$-track $\beta_S$ such that the interior of $S$ contains no $a$- or $b$-subtracks.  

To construct $S$, first observe that there can be no $a$-loop in $R$, as if there were one,
it would enclose a region with no $y$-edges, contradicting Lemma~\ref{lem: trapped x-noise}\eqref{lem part part: at least one}.  Since $R$ also has no teardrops (by Lemma~\ref{lem:teardrop}), 
any $a$-subtrack $\alpha_1$ in $R$ is a path with distinct  endpoints on $\partial R$. 
 If $\alpha_1$
has both endpoints on $\alpha$, 
then (in the absence of $a$-loops and teardrops)
 we get a smooth path by replacing a subsegment of $\alpha$ with $\alpha_1$, and this forms a smaller bigon with $\beta$.  
 If one or both endpoints of $\alpha_1$ are on $\beta$, then $\alpha_1$ divides $R$ 
into two regions, one of which has boundary a bigon comprised of an $a$-subtrack and a subtrack of $\beta$.   Passing to a minimal instance, we obtain a region $R'$ with boundary a bigon comprised of an $a$-track $\alpha'$ and an outward oriented $b$-track $\beta'$ (a subtrack of $\beta$), such that $R'$ has no $a$-subtracks in its interior.

Consider the minimal diagram containing $R'$, and let $D'$ be the subdiagram consisting of 2-cells not dual to $\alpha'$.  Then $D'$ has only cells of type $r_{3, i}$ or $r_{3, i, j}$ (as any other cells would introduce $a$-subtrack in $R'$).  So $D'$ has no junctions and, by Lemma~\ref{lem: no junctions no loops}\eqref{lem part: r23 diagrams}, has no $b$-loops.
Suppose there is a $b$-subtrack $\beta_1\neq \beta$ in $R$.  Then $\beta_1$ has  both endpoints on $\alpha'$ (as there are no junctions in $D'$).  
If $\beta_1$ is oriented into the bigon that it forms with $\alpha'$, then $\alpha'$ must be oriented outward by Lemma~\ref{lem: trapped y-noise}\eqref{lem part: no y-edges cor}.  As there are no $y$-edges in $R$, Lemma~\ref{lem: trapped x-noise}\eqref{lem part: no x-edges cor} applies, and implies that $\alpha'$ transitions from $a_1$ to $a_2$.  This happens at some $r_{1, q-1}$-cell dual to $\alpha'$.  However, as $\alpha'$ is oriented outward, such a cell contributes part of an $a_1$-subtrack to the interior of $R'$, a contradiction. 

Thus any $b$-subtrack in $R'$ has both endpoints on $\alpha'$, and is oriented out of the bigon it forms with $\alpha'$.  By passing to an innermost instance, we obtain a region $S$ with boundary a 
bigon comprised of a subtrack $\alpha_S$ of $\alpha'$ and an outward oriented $b$-track $\beta_S$ such that the interior of $S$ contains no $a$- or $b$-subtracks.

To complete our proof of~\eqref{lem part: badge button}, we show that if $R$ is minimal in that its interior contains no $a$- or $b$-subtracks, then the minimal subdiagram $D$ containing $R$ is either a badge or a button. 
 
\begin{figure}[htbp]
\centering
\begin{overpic} 
{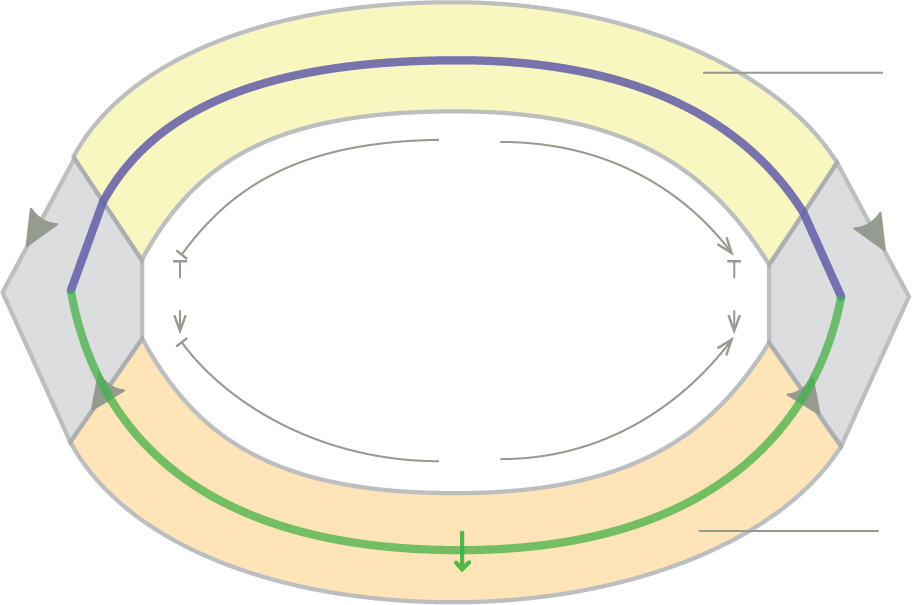}
\put(46, 1.5){\small{$\beta$}}
\put(46, 61){\small{$\alpha$}}
\put(-1.5, 42.5){\small{$b_{\ast}$}}
\put(97.5, 42.5){\small{$b_{\ast}$}}
\put(14, 21.5){\small{$b_{\ast}$}}
\put(82.5, 21.5){\small{$b_{\ast}$}}
\put(13, 43){\small{$a_{\ast}$}}
\put(83, 43){\small{$a_{\ast}$}}
\put(50, 50){\small{$u$}}
\put(17.5, 32.5){\small{$u_1$}}
\put(78.5, 33.2){\small{$u_2$}}
\put(50, 15){\small{$v$}}
\put(2, 33.5){\small{$f_1$}}
\put(93.7, 33.5){\small{$f_2$}}
\put(99, 6.3){\small{$\mathcal{C}_{\beta}$}}
\put(99, 57){\small{$\mathcal{C}_{\alpha}$}}
\end{overpic}
\caption{A bigon region per Lemma~\ref{lem: bigons}}
\label{fig: bigon be gone}
\end{figure}
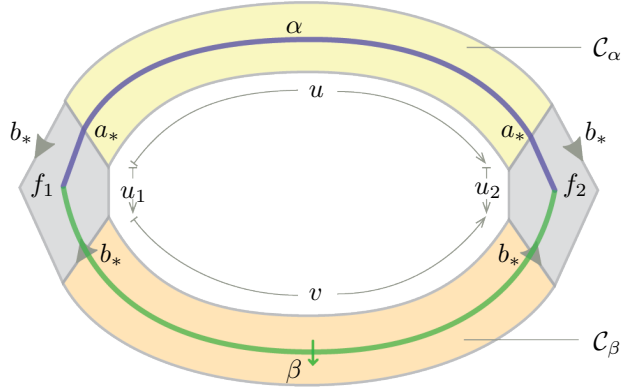

Let $C_{\alpha}$ and $C_{\beta}$ be the corridors dual to $\alpha$ and $\beta$ respectively---see Figure~\ref{fig: bigon be gone}.
They intersect in distinct 2-cells $f_1$ and $f_2$ of type 
$r_{i,j}$ with $i=1$ or $2$. (If $f_1=f_2$, then the orientation on $\beta$ would force both corners of $\partial S$ to be on the top half of some $r_{i,j}$-cell, and a terminal subpath of $\alpha$ would merge with an initial one to create a teardrop, which contradicts Lemma~\ref{lem:teardrop}.)  
Further, the $2$-cells of $D$ are exactly the 2-cells of $C_{\alpha} \cup C_{\beta}$ (because a 2-cell strictly in the interior of $R$ would result in interior $a$- or $b$ subtracks). 

The inner boundary of 
$C_{\alpha} \cup C_{\beta}$ has subpaths coming from $f_1$ and $f_2$ (labelled  $u_1$ and $u_2$ respectively), from $C_\alpha$ (labelled $u$) and from $C_{\beta}$ (labelled $v$), and these are oriented as shown in Figure~\ref{fig: bigon be gone}.  
Next, we determine which letters can occur in these labels examining Figure~\ref{fig:relations} for cells which could occur in $D$ under the given constraints.

Firstly, $f_1$ is an $r_{1, \ast}$- or $r_{2, \ast}$- cell, and given that $\beta$ is outward-oriented, one sees that the only non-empty word that could arise as $u_1$ is $b_i$ for some $i$ (when $f_1$ is a 
$r_{1, i-1}$ cell and $\alpha$ is inward-oriented).  However, as this would lead to $b$-subtracks inside $R$, we conclude that $u_1$ is empty. Likewise $u_2$ is empty.  Thus $u=v$ as group elements.

Next, each cell of $C_\beta$ apart from $f_1$ and $f_2$ is of type $r_{3, k}$, $r_{3,k,1}$ and $r_{3, k, 2}$ (as any others would introduce $a$-subtracks in the interior of $R$).   
Since $\beta$ is oriented outward, this means $v$ is a word on  $x_1, x_2, t$. 
Furthermore, the part of $\beta$ between (and excluding) $f_1$ and $f_2$ has no junctions, and so it is a $b_k$-track for some fixed $k$.   As $x_1, x_2, t$ freely generate a free group in $G$ (as a consequence of Proposition~\ref{prop:hnn}),  and each of them appears in a unique $r_{3, k}$- or $r_{3, k, j}$-cell, $\Delta$ being reduced implies $v$ is freely reduced.

If $\alpha$ is oriented outward, then 
each cell of $C_\alpha$ 
 apart from $f_1$ and $f_2$ is  of type $r_{4,i}$ (as any other cells would introduce $y$-edges or $b$-subtracks to the the interior of $R$).  
So $u$ is a reduced word of the form $t^n$ for some $n \in \Z$.   
Now, since $v$ is reduced, we have that $v = u = t^n$ as words.  
Furthermore $n\neq 0$, for otherwise $f_1$ and $f_2$ would be identified along a pair of adjacent edges in each with label $a_i^{-1} b_k$ and as each such word appears in a unique cell,  $f_1$ and $f_2$ would be oppositely oriented identical cells, contradicting the fact that $\Delta$ is reduced. 
Thus $R$ is a badge.

If $\alpha$ is oriented inward, then $C_\alpha$ cannot have any 2-cells apart from $f_1$ and $f_2$, so $u$ is empty. Then, as $v$ is reduced, it is also empty, and 
$f_1$ and $f_2$ are distinct cells identified along a corner in each with label $a_ib_j$.  Examining Figure~\ref{fig:relations} again, we see that this can only happen if they are a $r_{1, p-1}$-cell and an $r_{1, p}$-cell identified along their corners labelled $a_1b_p$, so that $R$ is a button. This completes the proof of~\eqref{lem part: badge button}.

Now assume $R$ satisfies the additional hypotheses in~\eqref{lem part: bi bigon} of this lemma (but is not necessarily minimal).  
In particular, the interior $R$ has no $a_1$-subtracks, but could have $a_2$- or $b_j$-subtracks.  
We continue with the notation of Figure~\ref{fig: bigon be gone}.  The intersection of an $a_1$-track and a $b_i$-track can only occur in an $r_{1, i}$- or $r_{1, i-1}$-cell.  Assume for a contradiction that  $f_1$ and $f_2$ are both of the former type.  Now, if $\alpha$ is oriented outwards, then  $u_1$ and $u_2$ are empty and $u$ is a word on $b_0, \dots, b_p$ (here we do not have $t$, because an $r_{4,1}$-cell would produce a $y$-edge in $R$, a contradiction). If $\alpha$  is oriented inwards, then $u_1 = b_{i+1}^{-1}$ and $u_2 = b_{i+1}$ and $u$ is a word on
$$b_p (X_{\ast} t^{-1} X_{\ast} t X_{\ast})^{-1}, \quad  b_q b_{q-1} (X_{\ast} t^{-1} X_{\ast} t X_{\ast})^{-1} , \text{ and}$$  $$ b_{i+1} b_{i} (X_{\ast} t^{-1} X_{\ast} t X_{\ast})^{-1}  \ \ (i \neq 0,\, {q-1}, \,p).$$

Now define $\overline u$ and $\overline v$ to be the images of these words in  the quotient $Q = F(b_0, \ldots, b_p) \rtimes \Z$  of $G$ from \eqref{eqn QQQ} resulting from killing $a_2, t, x_1, x_2, y_1, y_2$.  Then $\overline v$ is empty and $\overline u$ is a word on $b_1b_0, \ldots, b_p b_{p-1}, b_p$, which is a free basis for $F(b_0, \ldots, b_p)$.  So  $b_{i+1}^{-1} \overline u b_{i+1}=1$ in $Q$, and so $\overline u =1$.   So there is a canceling pair in $u$, and this implies that there is a pair of adjacent oppositely oriented cells, contradicting the hypothesis that the diagram $\Delta$ is reduced.   An analogous analysis rules out $\alpha_1$ being 
  outward-oriented. This proves~\eqref{lem part: bi bigon}.
\end{proof}

The next corollary summarizes the restrictions on loops in reduced diagrams obtained so far. 
\begin{cor} \label{cor: no loops} 
\textup{\textbf{(Loops)}} Suppose $\Delta$ is a reduced diagram.  
\begin{enumerate}
\item \label{lem part: loops absent from reduced diagrams}    $\Delta$ has no  
$t$-loops and no  inward-oriented $a$- or $b$-loops.

\item \label{lem part: no a-loops}  Every   $a$-loop in $\Delta$  encloses a $y$-edge. 

\item \label{lem part: no b-loops} 
$\Delta$ has no $b_i$-loops, and 
if $\Delta$ has  no buttons, then it has no $b$-loops. 
\end{enumerate}
\end{cor}

\begin{proof}
 Lemmas~\ref{lem: no t-loop} and \ref{lem: trapped y-noise}\eqref{lem part: no y-edges cor} establish \eqref{lem part: loops absent from reduced diagrams}.  
 
Were there an $a$-loop enclosing no $y$-edges, it would satisfy the hypotheses of 
 Lemma~\ref{lem: trapped x-noise}\eqref{lem part: no x-edges cor} but fail the conclusion in part \eqref{lem part part: at least one} of that lemma.   This proves~\eqref{lem part: no a-loops}.

  For~\eqref{lem part: no b-loops},  suppose $\beta$ is a $b$-loop in $\Delta$, as shown in Figure~\ref{fig: bad beta}. Then $\beta$ is oriented outward by~\eqref{lem part: loops absent from reduced diagrams}.   If $R$ is the region enclosed by $\beta$, then $R$ contains no $y$-edges by  Lemma~\ref{lem: trapped y-noise}\eqref{lem part: no y-edges}. Consequently, $R$ contains no $a$-loops by~\eqref{lem part: no a-loops} of this corollary.  Because $\Delta$ has no teardrops by Lemma~\ref{lem:teardrop}, any $a_1$-subtrack in $R$ must intersect $\beta$ in two distinct points, and divides $R$ into two bigons.

Let $\Delta_0$ be the minimal diagram containing $R$.  There are no 2-cells of type $r_{4 ,\ast, \ast}$ or $r_{4 ,\ast}$ in $\Delta_0$, because any such 2-cell would have to be inside $\beta$ and would give rise to a $y$-edge there.   So Lemma~\ref{lem: no junctions no loops}\eqref{lem part: r23 diagrams} tells us that $\Delta_0$ contains at least one $r_{1, \ast}$-cell. Therefore $R$ contains an $a_1$-subtrack.  Let $\alpha$ be an $a_1$-subtrack in $R$ that     
 forms a bigon with  a subtrack $\beta_1$ of $\beta$, and is \emph{innermost} in that there is no $a_1$-subtrack in the region $R_1$ enclosed by $\alpha$ and $\beta_1$.

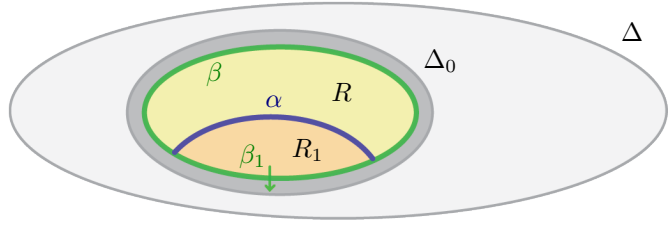
\begin{figure}[htbp]
\centering
\begin{overpic}  
{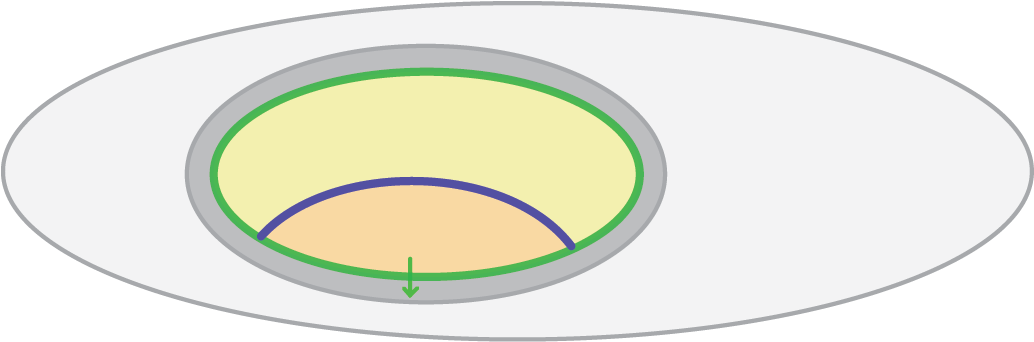}
\put(35, 8.5){\green{\small{$\beta_1$}}}
\put(30, 21){\green{\small{$\beta$}}}
\put(39, 17){\blue{\small{$\alpha$}}}
\put(93, 27){\small{$\Delta$}}
\put(63, 23){\small{$\Delta_0$}}
\put(49, 18){\small{$R$}}
\put(43, 9.5){\small{$R_1$}}
\end{overpic}
\caption{Our proof of Corollary~\ref{cor: no loops}\eqref{lem part: no b-loops}, illustrated}
\label{fig: bad beta}
\end{figure}

Now suppose $\beta$ is a $b_i$-loop for some fixed $i$, and so $\beta_1$ is a $b_i$-subtrack.  Then  applying 
Lemma~\ref{lem: bigons}\eqref{lem part: bi bigon} to $R_1$, we see that 
one of the intersections between $\alpha $ and $\beta_1$ occurs in an $r_{1, i-1}$-cell.  This is a contradiction, as $\beta$, being a $b_i$-track, cannot pass though an $r_{1, i-1}$-cell.   Thus $\Delta$ has no $b_i$-loops.  

Finally suppose that $\Delta$ has no buttons  and that $\beta$ is a $b$-loop.  
Then, by  Lemma~\ref{lem: bigons}\eqref{lem part: badge button}, the minimal subdiagram containing $R_1$ contains a badge.  The 
$a$-subtrack of this badge is dual to at least one $r_{4, i}$-cell, and this cell is in the interior of $R$.  
This is a contradiction: as already noted, each $r_{4, i}$-cell has a $y$-edge, while $R$ has none. 
This completes our proof of  \eqref{lem part: no b-loops}. 
\end{proof}

\begin{remark} 
  Figure~\ref{fig:button} shows how Corollary~\ref{cor: no loops}\eqref{lem part: no b-loops} can fail without the hypothesis absenting buttons.  
  Corollary~\ref{cor: no loops}\eqref{lem part: no a-loops} cannot be upgraded to rule out all $a$-loops:
  a  reduced diagram with an outward oriented $a_1$-track can be formed by circling an $r_{3,0}$-cell (which has $y$-edges) with an outward oriented $a_1$-annulus made up of two $r_{1,0}$-cells, two $r_{4,1}$-cells, and some $r_{4,1,j}$-cells. 
\end{remark}

Our next two lemmas concern the impact of the presence of Rips subwords in the sides of   $t$-corridors or in generalizations  defined in the following   manner.   The following expanded definition of a corridor  $\mathcal{C}$ and the lemma that follows it are motivated by   applications to our proof of Lemma~\ref{lem:t-corridors on boundary}. 

\begin{definition}\textbf{(Generalized corridors)}\label{def: generalized corridor}
Let  $\mathcal{C}$  be a set of $r$ distinct 2-cells $C_1$, $C_2$, \ldots, $C_r$ in a reduced van~Kampen diagram over our presentation $\mathcal{P}$ for $G$ such that there are edges $e_0, \ldots, e_r$ with the property that for $i = 1, \ldots, r-1$, the edge $e_i$ is in both $\partial C_i$ and $\partial C_{i+1}$.  Suppose the word read clockwise around $C_i$ is   $z_i f_i z_{i+1}^{-1} g_i$, where $z_i$ labels edge $e_i$.  Then the words along the top and bottom boundaries of $\mathcal C$ are  $f_1 f_2 \cdots f_r$ and $g_1^{-1} g_2^{-1} \cdots g_r^{-1}$ respectively. 

\end{definition}

\begin{lemma} \label{lem:corridor overlap}
\textup{\textbf{(Rips words cause the sides of corridors to be near injective and adjacent corridors to have small overlap.)}} 
There exists a constant $K \geq 1$  such that reduced van~Kampen diagrams $\Delta$ have the following properties. 

Suppose  $\mathcal{C}$ is  a generalized corridor,  $\mu$ is the path along one side of $\mathcal{C}$, and   the word  read along $\mu$ is $f := f_1 f_2 \cdots f_r$ (all per Definition~\ref{def: generalized corridor}).  Refer to $f_1,  \ldots, f_r$ as the syllables of $f$. A \emph{Rips subword} in a syllable  $f_i$ of $f$ is an element of $(\mathcal{X} \cup \mathcal{Y})^{\pm 1}$ appearing as a subword.  Suppose that if $1 \leq i \leq j \leq r$ are such that $f_i, \ldots, f_j$ do not have Rips subwords, then $f_i \cdots  f_j$ is  a reduced word on $\{a_1, a_2, b_0, \ldots, b_p\}^{\pm 1}$.

Suppose  $\overline \mu \subseteq \mu$ is an injective path from the initial vertex of $\mu$ to its terminal vertex.  So the word $\overline{f}$ read along $\overline \mu$ can be obtained from $f$ by a sequence $\Sigma$ of free reductions (successive cancellations of adjacent inverse-pairs of letters).  Then:

\begin{enumerate}
\item  \label{lem part: near injective}  
\begin{enumerate}
\item \label{lem part: survival} At least one letter of every Rips subword in a syllable survives in $\overline{f}$.   
 \item \label{lem part: linear bound on f}  $|f| \leq K  |\overline{f}| + K$. 
 \item \label{lem part part: near injective} If a subpath $\mu_0$ of $\mu$  is a loop and encloses no $2$-cells, then the subword $f_0$  of $f$ read  along $\mu_0$ has length at most $K$.  
\end{enumerate}       
\end{enumerate}

Suppose $\mu'$ is the path along one side  of another generalized corridor $\mathcal{C}'$ and $f' := f'_1 f'_2 \cdots f'_{r'}$ is the word  read along it.  Suppose that for all $i$, some element of $(\mathcal{X} \cup \mathcal{Y})^{\pm 1}$ is a subword of $f'_i$.   Suppose $\mathcal{C}$ and $\mathcal{C}'$ have no 2-cells in common and that they start and end on $\partial \Delta$ (that is, $e_0, e_r, e'_0, e'_{r'}$ are in $\partial \Delta$).  Suppose that $$I \ := \  \mathcal{C} \cap \mathcal{C}'  \ = \   \mu \cap \mu'  \ \neq  \  \emptyset.$$

\begin{enumerate} \addtocounter{enumi}{1} 

\item  \label{lem part: small overlap} Suppose $\mu_0$ and  $\mu'_0$  are the shortest subpaths  of $\mu$ and of  $\mu'$, respectively,  such that $I = \mu_0 \cap \mu'_0$.  If  $\mu_0 \cup \mu'_0$    encloses no 2-cells, then  $|\mu_0|, |\mu'_0| \leq K$.   

\end{enumerate}

\end{lemma}

 \begin{proof}  
For \eqref{lem part: near injective}, we can interpret the sequence $\Sigma$ as folding together adjacent pairs of edges in
 a  $|f \overline{f}^{-1}|$-sided simple polygonal-path in the plane until we have the planar tree  in $\Delta$  whose boundary circuit is $\mu  \overline{f}^{-1}$.  Because every cyclic conjugate of a defining relator (of Figure~\ref{fig:relations}) is freely reduced, no cancellation of a pair of letters within a syllable of $f$ occurs in the course of $\Sigma$.   

Given $\sigma   \in (\mathcal{X} \cup \mathcal{Y})^{\pm 1}$, let $P_{\sigma}$  and $S_{\sigma}$ denote 
its prefix and suffix, respectively, such that $\sigma  = P_{\sigma} S_{\sigma}$ as words, and $|P_{\sigma}| = \lfloor | \sigma|/2  \rfloor$.    Suppose of all the Rips subwords in the syllables of $f$, some subword $\sigma$  of $f_{l}$ is the first such that  either $P_{\sigma}$ and $S_{\sigma}$ is fully cancelled away in the course of $\Sigma$.  Assume it is $S_{\sigma}$ that is first cancelled away. (The argument if it is $P_{\sigma}$ will be essentially the same, and we omit it.)   Then $S_{\sigma}$  must cancel with a subword of   $f_{m}$, where   $m>l$ is  minimal such that $f_m$ has a Rips subword.  
But that is impossible: the $C'(1/4)$-condition for $\mathcal{X} \cup \mathcal{Y}$ and the fact that each of its elements has length at least 100, imply that some subword of $\sigma^{-1}$ of at least a quarter of its length is a subword of   $f_{m}$, and moreover the 2-cell $C_{l}$ cancels with $C_m$ in $\Delta$,  contrary to $\Delta$ being a reduced diagram.  
This proves \eqref{lem part: survival}.   

Now suppose that syllables $f_{i}, \ldots,  f_{j}$ do not contain Rips subwords.  Then (by hypothesis) $f_{i}  \cdots    f_{j}$ is a reduced word on  $\{a_1, a_2, b_0, \ldots, b_p\}^{\pm 1}$.  So the number of letters that can cancel away on freely reducing $f_{i-1} f_{i}  \cdots    f_{j}  f_{j+1}$ is less than four times the length of the longest defining relation for our group. Together with \eqref{lem part: survival}, this implies \eqref{lem part: linear bound on f} and \eqref{lem part part: near injective} for a suitable constant $K \geq 1$.

For \eqref{lem part: small overlap}, first we observe that $I$ is a path because, by hypothesis, $\mu_0 \cup \mu'_0$ encloses no 2-cells. Let $w_0$ and $w'_0$ be the words read along $\mu_0$ and $\mu'_0$, respectively.  Assume, without loss of generality, that $\mu_0$ and $\mu'_0$ are oriented in the same direction---which is to say that  $w_0 (w'_0)^{-1}$ is the word around $\mu_0 \cup \mu'_0$.  Then free reduction takes $w_0$ and $w'_0$ to the word $w$ read along $I$.  (We are not claiming $w$ is freely reduced---further free reduction may be possible.) 
 
 The proof can then be completed in a similar manner to part \eqref{lem part part: near injective}.  In short, if there is a  Rips subword $\sigma$ in $w'_0$,  then there must be a subword of $\sigma$ in  $w_0$ also and these two words 
 have large overlap in $w$,  
so as to imply that there are cancelling 2-cells in $\mathcal{C}$ and $\mathcal{C}'$.        So $\mu'_0$ contains no complete Rips subword   and,  because each of the syllables of $\mu'$ contains a Rips subword (by hypothesis), $\mu'_0$ has length at most a constant.  It then follows that $\mu_0$, which also contains no complete Rips subword, also has length at most a constant: within $w_0$, any  $f_i$ that contains no Rips subword can only cancel with the neighbouring $f_{i-1}$ or $f_{i+1}$ if they contain a Rips subword (so at most some constant number of letters in total can cancel away) and the remaining letters must be in $w'$, which has length at most  $|\mu'_0|$.
\end{proof}

\begin{lemma} \label{lem: no self-osculating t-corridors}
Suppose   $\mu$ is the path   along one side of a $t$-corridor $\mathcal{C}$ in a reduced van~Kampen diagram $\Delta$.  
Then the first $y$-edge $e$ of $\Delta$ traversed by $\mu$ is not  traversed a second time by $\mu$.       
\end{lemma}
 
\begin{proof}   Suppose, on the contrary, $\mu$ traverses $e$ more than once, then (because $\Delta$ is planar and $\mu$ is the side of a corridor) it does so exactly twice---once in each   direction---and the subpath $\overline{\mu}$ of  $\mu$ starting with the first traverse of $e$ and ending with the second traverse is a loop. (See Figure~\ref{fig:self-osculating}.) 

\begin{figure}[htbp]
\centering
\begin{overpic}  
{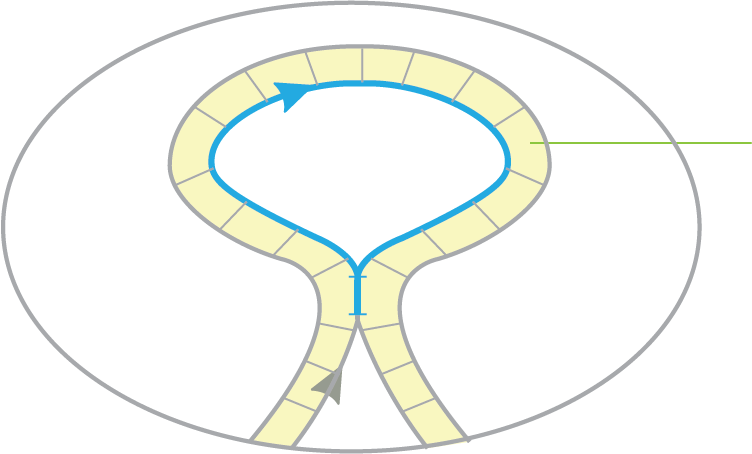}
\put(59, -4){\small{$t$}}
\put(34, -4){\small{$t$}}
\put(46, 5){\small{$\mu$}}
\put(101, 40){\small{$\mathcal{C}$}}
\put(40, 42){\small{$\overline{\mu}$}}
\put(49, 20){\small{$e$}}
\end{overpic}
\caption{The $t$-corridor of our proof of Lemma~\ref{lem: no self-osculating t-corridors}}
\label{fig:self-osculating}
\end{figure}
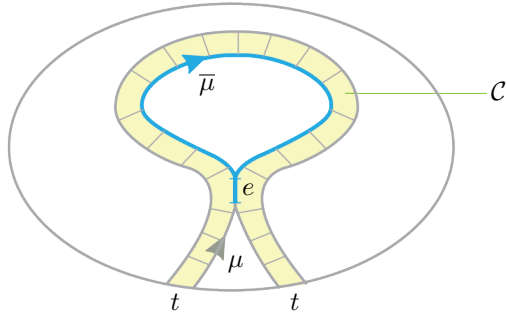

With a view to applying  Lemma~\ref{lem:corridor overlap}\eqref{lem part: near injective} to $\mathcal C$, we check its hypotheses. As $\mathcal C$ is a $t$-corridor, our defining relations imply that the label of $\mu   \cap C$ contains a Rips subword for every cell $C$ of $\mathcal C$.  
There are no $t$-edges within  the region $\overline \Delta$ enclosed by $\overline{\mu}$, for if there were, then there would be a $t$-loop within $\overline \Delta$, contradicting Lemma~\ref{lem: no t-loop}. 
So $\overline{\mu}$ does not enclose any 2-cells.  
Thus Lemma~\ref{lem:corridor overlap}\eqref{lem part: survival} applies, and tells us that the label $\overline w$ of $\overline{\mu}$ 
has no Rips subword from $(\mathcal{X} \cup \mathcal{Y})^{\pm 1}$ as a subword.

On the other hand, Corollary~\ref{cor:subwords non-trivial} implies that $\overline{w}$ cannot be a subword of the boundary word of a single 2-cell of $\mathcal C$. 
In particular, if $C_e$ is the cell of $\mathcal C$ containing the initial point of $\overline{\mu}$ (and the edge $e$), then  $\overline{\mu}$ extends beyond $C_e$, and intersects at least one other cell of $\mathcal C$.   
Thus  if $t^{\pm 1} u t^{\mp 1} = v$  is the boundary label of $C_e$, where $u$ labels $\mu \cap C_e$, then $u$ has the form $u_1 y_\ast u_2$, where $y_\ast u_2$ is a prefix of $\overline w$.  Moreover, as 
$e$ is the first $y$-edge in $\mu$, it follows that $u_1$ has no $y$-edges.  Then, 
examining Figure~\ref{fig:relations}, we see that $u_2$ necessarily contains the entirety of some Rips subword $Y_\ast$ from 
$\mathcal{Y}^{\pm 1}$ as a subword.
(This is true even if the first letter of $\overline{w}$ is the lone $y_j^{\pm 1}$ that arises in the $r_{4,i,j}$-cells.)
This contradicts our earlier conclusion that $\overline w$ has no Rips subwords. 
\end{proof}

We will use our next lemma in our proof of Lemma~\ref{lem: Layout lemma}\eqref{lem part: No y-edges above a b-track}.    Here is the intuition. Imagine a diagram consisting of a sequence of side-by-side  vertical corridors as in Figure~\ref{fig:vertical}.
If there are no $y$-edges at the bottom of the diagram, 
then we can slice horizontally through it
and discard the portion above the cut, so that the diagram that remains has no $y$-edges and the length of the cut is at most a constant times the length of the top.

\begin{lemma} \label{lem: vertical t-corridors} \textup{\textbf{($y$-edges in side-by-side   $t$-corridors)}}  There  exists a constant $C>0$ with the following property.  
Suppose $u$ and $v$ are words that represent the same element of $G$ and  that $v$  contains no $y$-letters.  
Suppose $\Delta$ is a reduced diagram for  $uv^{-1}$.  Let $\ast_0$ and  $\ast_1$ be the vertices on    $\partial \Delta$ where both $u$ and $v$ start and end (respectively).   Assume that every $t$-corridor in $\Delta$ connects a $t^{\pm 1}$ in $u$  to a  $t^{\pm 1}$ in $v$.

Then there is a word $v'$ read along some injective path through  $\Delta^{(1)}$ from $\ast_0$ to $\ast_1$ such that $|v'| \leq C |u|$ and the subdiagram $\Delta'$  (per Figure~\ref{fig:vertical}), which is a van~Kampen diagram for $v(v')^{-1}$,  contains no $y$-edges.   
\end{lemma}

\begin{figure}[htbp]
\centering
\begin{overpic}  
{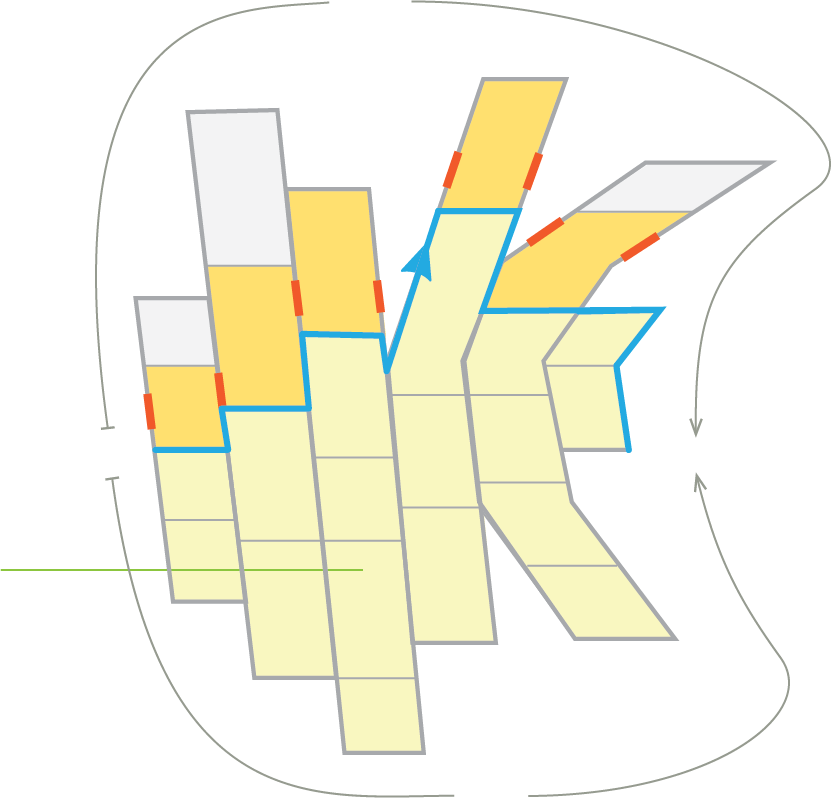}
\put(12, 41){\small{$\ast_0$}}
\put(78, 41){\small{$\ast_1$}}
\put(43.5, 94.5){\small{$u$}}
\put(58, -1){\small{$v$}}
\put(46.5, 66){\small{$v'$}}
\put(-7, 26){\small{$\Delta'$}}
\put(24, 19){\small{$t$}}
\put(34, 10){\small{$t$}}
\put(45, 1){\small{$t$}}
\put(54, 14){\small{$t$}}
\put(73, 14){\small{$t$}}
\put(72, 36){\small{$t$}}
\put(20, 62){\small{$t$}}
\put(27, 84){\small{$t$}}
\put(39, 75){\small{$t$}}
\put(62, 88){\small{$t$}}
\put(75, 60){\small{$t$}}
\put(86, 78){\small{$t$}}
\put(22, 33){\small{$\tau_1$}}
\put(31, 38){\small{$\tau_2$}}
\put(41, 40){\small{$\tau_3$}}
\put(50, 42){\small{$\tau_4$}}
\put(60, 46){\small{$\tau_5$}}
\put(68, 48){\small{$\tau_6$}}
\put(19.5, 46){\small{$C_1$}}
\put(28, 55){\small{$C_2$}}
\put(38, 62){\small{$C_3$}}
\put(57, 77){\small{$C_4$}}
\put(65.5, 63){\small{$C_5$}}
\end{overpic}
\caption{Lemma~\ref{lem: vertical t-corridors}, illustrated}\label{fig:vertical}
\end{figure}

\begin{proof}
We denote the $t$-corridors of $\Delta$ by
$\tau_1, \ldots, \tau_m$, for some $m$, where $\tau_i$ connects the $i$th $t^{\pm 1}$ in $v$  to the $i$th  $t^{\pm 1}$ in $u$.     Every $t$-corridor is of this form, by hypothesis.  Observe that $m \leq |u|$. 

For all $i$, let $\mathcal{S}_i^-$ and $\mathcal{S}_i^+$ be the paths from $v$ to $u$ along the two sides of $\tau_i$, with $\mathcal{S}_i^-$ emanating from the starting vertex of the $t^{\pm 1}$ of $\tau_i$ in $v$  and $\mathcal{S}_i^+$ from its  ending vertex. Assuming there is a $y$-edge   on $\mathcal{S}_i^{\pm}$, let $e_i^{\pm }$  be the lowest---which is to say that $e_i^{\pm }$ is  the first $y$-edge that $\mathcal{S}_i^{\pm}$ traverses.  If there are $y$-edges in one side of a 2-cell in a $t$-corridor, then there are $y$-edges in the other side of that cell.  So $e_i^{-}$ and $e_i^{+}$ (if defined) are in the boundary of the same  2-cell $C_i$ of~$\tau_i$.      
Moreover, as  Lemma~\ref{lem: no self-osculating t-corridors} guarantees that $\mathcal{S}_i^{\pm }$ does not traverse $e_i^{\pm }$ a second time and,  because $v$ has no $y$-edges, $e_i^{\pm }$ is either in $u$ or part of the neighboring $t$-corridor. 
 It follows that for all $i$, \begin{itemize} \item either both  $e_i^{+}$ and $e_{i+1}^{-}$ exist, they agree, and they are not in $u$, \item or  both  exist and are   in  $u$, \item or only one exists and is in $u$, \item or neither exists. \end{itemize}

Take $C$ to be the maximum length of a defining relator in $\mathcal{P}$.
Then there is an injective path through  $\Delta^{(1)}$ from $\ast_0$ to $\ast_1$ that follows portions of $u$ and portions of the   boundary circuits of the at most $|u|$ 2-cells $C_i$, such that the word $v'$ along  this path satisfies the required conditions.    (This path is shown in blue in Figure~\ref{fig:vertical}.)
\end{proof}

 Our final lemma is illustrated by Figures~\ref{fig:t-corridors on boundary} and \ref{fig:corridors_on_boundary_detail}.     
(The path   $\rho$ is in the   graph dual to  $\Delta^{(1)}$.)    In short, it says, in the notation of Figure~\ref{fig:t-corridors on boundary}, that the diagram cannot flare out exponentially towards $v$.   Its   application in    Lemma~\ref{lem: Layout lemma}\eqref{lem part:no sinks}  will be that certain regions can be sliced off a reduced diagram with the resulting diagram only longer by at most a constant factor. Thereby  we will simplify diagrams  that demonstrate distortion.

\begin{lemma} 
  \label{lem:t-corridors on boundary}    \textup{\textbf{(The lengths of compound-tracks between points on the boundary)}} 
  There exists a constant $C \geq 1$ with the following property.   Suppose a  region $R$ in a reduced diagram $\Delta$ is bounded by a portion $\mu$ of   $\partial \Delta$ and a compound track $\rho$ that is a concatenation  of  $a$-subtracks, inward-oriented $b$-subtracks, and $t$-subtracks.   	  Let $D$ be the minimal subdiagram of $\Delta$ containing $R$.  (That is, $D$ is the union of $R$ and the generalized corridor $\mathcal{C}$ through which $\rho$ passes.) So $D$ is a van~Kampen diagram for 
  $v {u}^{-1}$  for some words $v$ and ${u}$ 
  such that $v$ is read around $\partial \Delta$ starting and ending with the edges where $\mu$ and $\rho$ meet.                
  Suppose either
  \begin{enumerate}
\item  \label{lem part: boundary region} the $a$-subtracks in  $\rho$ are  oriented into $R$, or 
\item \label{lem part: boundary region no y}  $D$ contains no $y$-edges.  \end{enumerate} 
Then $|{u}|$ and the number of edges $| \rho|$ of $\Delta$ that $\rho$ crosses  are both at most $C |v|$.  
\end{lemma}

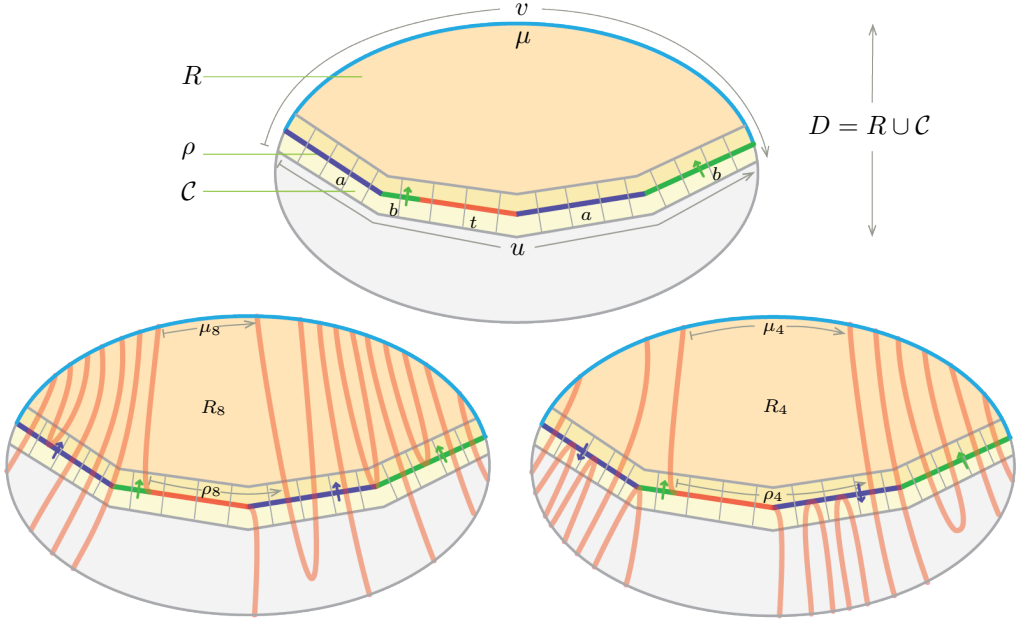
\begin{figure}[ht]
\centering
\begin{overpic}   [scale=0.9] 
{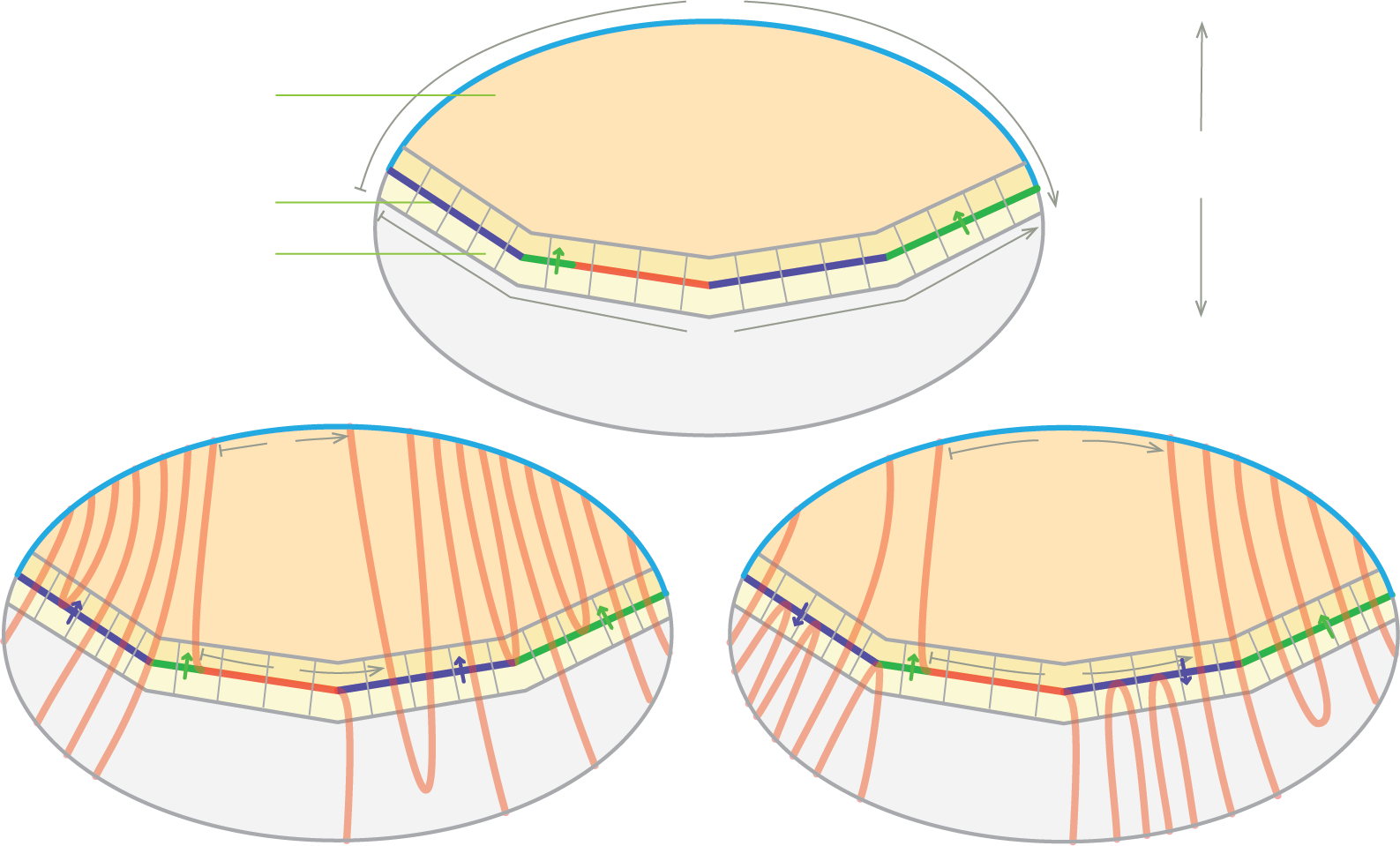}
\put(17.5, 53){\small{$R$}}
\put(17.5, 41.5){\small{$\mathcal{C}$}}
\put(17.5, 46){\small{$\rho$}}
\put(50.5, 57){\small{$\mu$}}
\put(50, 36){\small{$u$}}
\put(32.7, 43){\tiny{$a$}}
\put(38, 39.9){\tiny{$b$}}
\put(46, 38.7){\tiny{$t$}}
\put(57, 39.5){\tiny{$a$}}
\put(70, 43.5){\tiny{$b$}}
\put(50.5, 59.8){\small{$v$}}

\put(19.2, 28.4){\tiny{$\mu_8$}}
\put(19.4, 12.4){\tiny{$\rho_8$}}
\put(19.4, 20.4){\tiny{$R_8$}}

\put(75, 28.5){\tiny{$\mu_4$}}
\put(75, 11.9){\tiny{$\rho_4$}}
\put(75, 20.5){\tiny{$R_4$}}

\put(79.5, 48){\small{$D = R \cup \mathcal{C}$}}

\end{overpic}
\caption{Top: a region $R$ enclosed by a  portion $\mu$ of $\partial \Delta$ and  a compound track $\rho$ comprised of $a$- and $t$-subtracks and inward-oriented $b$-subtracks per Lemma~\ref{lem:t-corridors on boundary}.  The lower diagrams depict   the $t$-tracks incident with $\rho$  when (left) the $a$-subtracks are inward-oriented, and (right) when $R$ and $\mathcal{C}$ contain no $y$-edges.  Note that each $R_i$ could have $t$-tracks with both endpoints on $\mu_i$---these are are not pictured here, but are shown in the detail in Figure~\ref{fig:corridors_on_boundary_detail}.
}\label{fig:t-corridors on boundary}
\end{figure}

\begin{proof}       
We will establish the claimed bounds  by examining the $t$-tracks through $R$.  By Lemma~\ref{lem: no t-loop}, there are no $t$-loops in $R$ or indeed anywhere in $\Delta$, because $\Delta$ is reduced.  Next we will argue that there is no  $t$-subtrack $\tau$ in $R$ which is non-trivial (i.e.,  not a single point) 
and which starts and ends on $\rho$ and otherwise is in the interior of $R$.    If there were, then a subpath of $\tau$ together with  a subpath of $\rho$     would bound a region $R' \subseteq R$ that  cannot exist in a reduced diagram: under hypothesis \eqref{lem part: boundary region}, $R'$ would be contrary to Lemma~\ref{lem: trapped y-noise}\eqref{lem part: no y-edges cor}, and under  hypothesis \eqref{lem part: boundary region no y}, Lemma~\ref{lem: trapped x-noise}\eqref{lem part: no x-edges cor} applies to $R'$ and its conclusion \eqref{lem part part: ts} tells us there is an $r_{4,1}$-cell and an $r_{4,2}$-cell in $D$, and therefore a $y$-edge in $D$, contrary to assumption.

The tracks $\tau_1,   \ldots, \tau_m$ of $R$ which have one endpoint on $\mu$ and the other on $\rho$ divide 
$R$ into subregions $R_0, R_1, \ldots, R_m$ as illustrated in 
Figure~\ref{fig:t-corridors on boundary}, with the  lower left diagram depicting hypothesis \eqref{lem part: boundary region} and lower right, hypothesis 	\eqref{lem part: boundary region no y}.  Under either hypothesis  \eqref{lem part: boundary region} or 	\eqref{lem part: boundary region no y}, the previous paragraph implies that  every $t$-subtrack entering the interior of $R_i$ has both endpoints on $\mu$.  
In more detail, $\mu$ and $\rho$ can be expressed as  concatenations of subpaths 	$\mu_{0}$, $\mu_{1}$, \ldots,  $\mu_{m}$ and  $\rho_{0}$, $\rho_{1}$, \ldots,  $\rho_{m}$, respectively, so that for each $i$, the region  $R_i$ is bounded by $\mu_{i}$, $\rho_{i}$, $\tau_{i}$ and $\tau_{i+1}$ (with $\tau_0$ and $\tau_{m+1}$ being trivial paths).  

	 Guided by the locations of the letters $t^{\epsilon_i}$ read along the edges where 
the  $\tau_i$ meet $\mu$,  express $v$  as $$v \ = \ t^{\epsilon_0} v_{0} t^{\epsilon_1} v_{1} t^{\epsilon_2} v_{2} \cdots  t^{\epsilon_m} v_{m} t^{\epsilon_{m+1}} $$ where $\epsilon_1, \ldots, \epsilon_m \in \set{ \pm 1}$ and $\epsilon_0,  \epsilon_{m+1} \in \set{0, \pm 1}$,  and each $v_{i}$ is a subword of $v$ (which may contain further $t^{\pm 1}$).  

Fix $i \in \set{0, \ldots, m}$.  Let $\nu_i$ denote the concatenation of $\tau_i, \rho_i$ and $\tau_{i+1}$, so that $R_i$ is bounded by $\mu_i$ and $\nu_i$. Let $C_1, \dots, C_r$ denote the $2$-cells traversed by $\nu_i$, as shown in Figure~\ref{fig:corridors_on_boundary_detail} (with $i=4$ and $r=17$).  
Together they form a generalized corridor $\mathcal{C}$ in the  sense of Definition~\ref{def: generalized corridor}.  
Let $\Delta_i$ be the maximal subdiagram  that is a subset of $R$, includes the portion of $\partial \Delta$ labelled by $v_i$, and does not intersect $\tau_i$, $\rho$ or $\tau_{i+1}$.  Let $f = f_1\dots f_r$ be the word along the side of   $\mathcal{C}$ that is in $R_i$.  Then $\Delta_i$ is a van~Kampen diagram for $f v_i^{-1}$.  We refer to  $f_1, \ldots, f_r$ as the \emph{syllables} of $f$.  (It may be that $f$ is not reduced and $\Delta_i$ is not homeomorphic to a 2-disc.)

\begin{figure}[htbp]
\centering
\begin{overpic}
{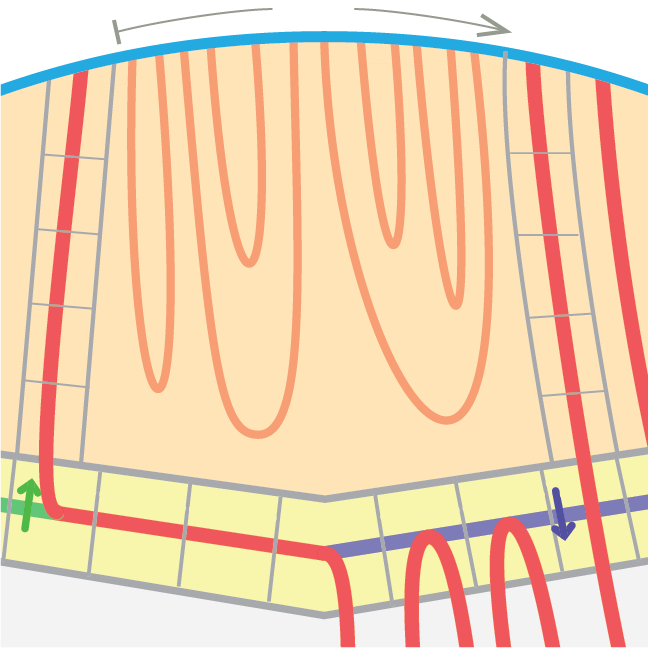}
\put(9, 92.5){\small{$t^{\epsilon_4}$}}
\put(82, 93.5){\small{$t^{\epsilon_5}$}}
\put(46, 97.5){\small{$v_4$}}
\put(47, 55){\small{$R_4$}}
\put(8.5, 82){{\tiny{$C_{1}$}}}
\put(7.5, 69){{\tiny{$C_{2}$}}}
\put(6.5, 58){{\tiny{$C_{3}$}}}
\put(5.7, 46){{\tiny{$C_{4}$}}}
\put(5, 34){{\tiny{$C_{5}$}}}
\put(5, 18){{\tiny{$C_{6}$}}}
\put(19, 16){{\tiny{$C_{7}$}}}
\put(33, 14){{\tiny{$C_{8}$}}}
\put(48, 12){{\tiny{$C_{9}$}}}
\put(61, 14){{\tiny{$C_{10}$}}}
\put(74, 16){{\tiny{$C_{11}$}}}
\put(88, 18){{\tiny{$C_{12}$}}}
\put(85, 34){{\tiny{$C_{13}$}}}
\put(83, 46){{\tiny{$C_{14}$}}}
\put(81, 58){{\tiny{$C_{15}$}}}
\put(80, 69){{\tiny{$C_{16}$}}}
\put(79, 82){{\tiny{$C_{17}$}}}
\put(18, 82){{\tiny{$f_{1}$}}}
\put(17, 69){{\tiny{$f_{2}$}}}
\put(17, 55){{\tiny{$\vdots$}}}
\put(72, 82){{\tiny{$f_{r}$}}}
\put(75, 68){{\tiny{$\vdots$}}}
\put(-14, 65){{\small{$\tau_4$ --------}}}
\put(84, 65){{\small{------------ $\tau_5$}}}

\end{overpic}
\caption{The region $R_4$ illustrated per our proof of Lemma~\ref{lem:t-corridors on boundary}.}\label{fig:corridors_on_boundary_detail}
\end{figure}
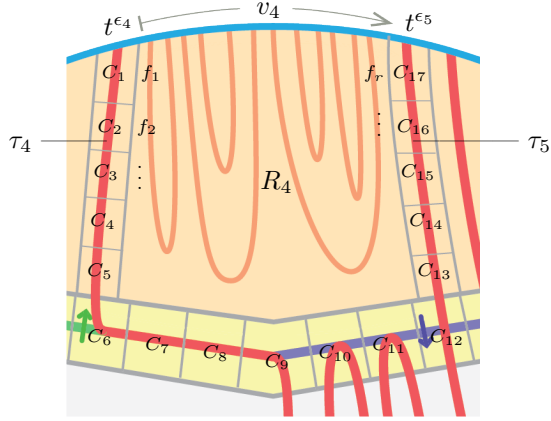

We will show  that there exists a constant $L \geq 1$ such that, if $|\nu_i|$ denotes the number of edges of $\Delta$ crossed by $\nu_i$, then 
\begin{equation} \label{bound on nu_i}
|\nu_i| \ \leq  \  L|v_{i}| + L.
\end{equation}

We will argue that $\mathcal{C}$ satisfies the hypotheses of Lemma~\ref{lem:corridor overlap}.    
The label of $C_j$, read clockwise, is of the form $\alpha f_j \beta^{-1} \hat{f_j}$, 
with  $\alpha, \beta \in \set{a_1^{\pm 1}, a_2^{\pm 1}, b_1, \ldots, b_p,  t^{\pm 1}}$ being the letters  labeling edges dual to which $\nu_i$ enters and leaves $C_j$, respectively.  (The hypothesis that the $b$-subtracks that are part of $\rho$ are oriented into $R$ precludes $\alpha$ or $\beta$ being among $b_1^{-1}, \ldots, b_p^{-1}$.)     

Suppose $f_j$ does not have a Rips subword. Inspecting the defining relators for $G$ (Figure~\ref{fig:relations}), we find that one of  $\alpha$ and $\beta$ is in $\set{a_1^{- 1}, a_2^{- 1}}$ and the other is in  $\set{a_1^{- 1}, a_2^{-1}, t}$, and this can only occur when there is an $a$-subtrack in $\rho$ that is oriented out of $R$, contrary to hypothesis \eqref{lem part: boundary region}, which means that hypothesis \eqref{lem part: boundary region no y} must apply.    
But then the only way one of $\alpha$ and $\beta$ can be $t$ is if  $C_j$ is an $r_{4,i}$-cell and $\alpha$ and $\beta$  label the top and right edges (or vice versa) in the sense of Figure~\ref{fig:relations}, which is excluded by \eqref{lem part: boundary region no y} because $r_{4,i}$-cells have $y$-edges.     So     $\alpha, \beta \in \set{a_1^{- 1}, a_2^{- 1}}$ and  $C_j$ is an $r_{1, \ast}$- or $r_{2, \ast}$-cell, with $\ast \neq 0$ lest we contradict  \eqref{lem part: boundary region no y}.  If $C_j$ is an $r_{1, \ast}$-cell, then  $f_j  \in \set{b_1, \ldots, b_p, b_{q-1}a_1}^{\pm 1}$. If  $C_j$ is an $r_{2, \ast}$-cell, then $f_j  \in \set{b_1, \ldots, b_p}^{\pm 1}$.  

Next suppose $f_{j+1}$ also does not contain Rips word.  If one of  $C_{j}$ and $C_{j+1}$ is an $r_{1, \ast}$-cell and the other is an $r_{2, \ast}$-cell, then one  of them must be an $r_{1, q-1}$-cell and they meet along an edge labelled $a_2^{-1}$.  In this event, there is no cancellation between $f_j$ and $f_{j+1}$, because $f_j f_{j+1}$ is $(b_l^{\pm 1} a_1^{-1}b_{q-1}^{-1})^{\pm 1}$ for some $l$. If, on the other hand, $C_j$ and $C_{j+1}$ are both $r_{1, \ast}$-cells or both $r_{2, \ast}$-cells, then there can be no cancellation between $f_j$ and $f_{j+1}$ lest $C_j$ and $C_{j+1}$ be a cancelling pair of 2-cells, contrary to $\Delta$ being a reduced diagram.       
Thus if consecutive syllables $f_j$,   \ldots, $f_{l}$ (for $j\leq l$) do not contain Rips words, then $f_j,   \cdots, f_{l}  \in \set{b_1, \ldots, b_p, b_{q-1}a_1}^{\pm 1}$ and $f_j  \cdots f_{l}$ is a freely reduced word.  So $\mathcal{C}$ satisfies the hypotheses of Lemma~\ref{lem:corridor overlap}.

Let $\Delta_{\mathcal{C}}$ be the minimal subdiagram of $\Delta$ containing $\mathcal{C}$ and let $\overline{\Delta_{i}}$ be the maximal subdiagram of $\Delta_i$ that contains the path labelled $v_i$ and does not intersect the interior of $\Delta_{\mathcal{C}}$.  Let  $\overline{f}$ be the word such that $\overline{\Delta_{i}}$ is a van~Kampen diagram for $\overline{f} v_i^{-1}$. There are no 2-cells in $\Delta_i \setminus \overline{\Delta_{i}}$ because there would be a $t$-track through such a 2-cell and we know that all $t$-tracks in  $\Delta_i$ connect a pair of edges in $v_i$. So $\overline{f}$ can be obtained from $f$ by freely reducing $f$ (perhaps only partially: $\overline{f}$  need not be freely reduced), so as to remove all the letters which  label any 1-dimensional \emph{spikes} of $\Delta_i$ that protrude into $\mathcal{C}$. By Lemma~\ref{lem:corridor overlap}\eqref{lem part: linear bound on f}, there is a constant $K  \geq 1$ such that  
\begin{equation} \label{eqn:f bound}
|f| \ \leq  \ K  |\overline{f}| + K.  
\end{equation}

Next, suppose $\mathcal{C'}$ is a $t$-corridor that joins a pair of $t$-letters in $v_i$.  Then $\mathcal{C}$ and $\mathcal{C}'$ have no 2-cells in common:  were there such a 2-cell, the $t$-track through $\mathcal{C'}$ would intercept $\nu_i$ (see Figure~\ref{fig:cell_track_intersections}).  Moreover,  there can be no 2-cell in any subdiagram of $\Delta_i$ whose boundary is made up of a path along one side of $\mathcal{C}$ and a path along one side of $\mathcal{C}'$:  there would be a $t$-subtrack through such a 2-cell, and it would either be part of a $t$-loop (contrary to Lemma~\ref{lem: no t-loop})  or would join two points on $\rho_i$ (which we argued at the start of this proof cannot happen).  So Lemma~\ref{lem:corridor overlap}\eqref{lem part: small overlap}  applies and tells us that the overlap between $\mathcal{C}$ and $\mathcal{C}'$ has length at most the constant $K$.  

Each edge of the $\overline{f}$-portion of $\partial \Delta_i$ is either in the $v_i$-portion  of $\partial \Delta_i$ or is the side of such a   $t$-corridor $\mathcal{C'}$. At  most $|v_i| /2$ $t$-corridors join a pair of $t$-edges in $v_i$.  We conclude that there is a constant $K' \geq 1$ such that   
\begin{equation} \label{eqn:overlinef bound}
|\overline{f}|  \ \leq \  K' |v_i|.      
\end{equation}

The existence of a constant  $L \geq 1$ such that \eqref{bound on nu_i} holds now comes from combining $|\nu_i| \leq |f|$,  \eqref{eqn:f bound}, and \eqref{eqn:overlinef bound}.

Finally, using $|\rho_{i}|  \ \leq |\nu_i|$ and summing \eqref{bound on nu_i} over all $0 \le i \le m$, we get
that   $$|\rho |  \ \leq \  \sum_{i=0}^m |\rho_{i}| \ \leq \    L|v| + L(m+1)   \ \leq  \   2L|v|.$$
So  $|\rho|$ and $|{u}|$  are both at most $C |v|$ for a suitable constant $C \geq 1$ derived from $L$ and the maximum length of a defining relation.
\end{proof}

While we will only call on the lemma above in its full generality, we note that in the case when $\rho$ is a $t$-track, it gives:

 \begin{cor}\label{cor:HNN undistorted}
The vertex groups of the HNN-structure $G  =  F \HNN_{t}$ are undistorted in $G$.
\end{cor}

\subsection{Intersection patterns for a pair of paths across a disc} \label{sec: intersection patterns}

Towards further understanding the intersection patterns of tracks, we consider here how a pair of transversely oriented paths in a disc may intersect if there are no ``sink-regions.''  
The results in this section are  formulated so as to be combinatorial, bypassing issues such as paths intersecting each other infinitely many times.  We could, equivalently, have made  the paths in this section injective combinatorial paths in the 1-skeleton of a finite 2-complex homeomorphic to a 2-disc.

\begin{definition}\label{def: sink source}
\textup{\textbf{(Sinks and sources)}}
Let $\sigma$ and $\tau$ be piecewise-linear paths in a 2-disc $D$, each of which is made up of   finitely many straight-line segments and has a  transverse orientation.   
Suppose that $\sigma$ and $\tau$ meet $\partial D$ at exactly four    points---their end points---and that their intersections are transverse.  
 A region $R$ in $D$ such that $\partial R$ is a union of subpaths of $\sigma$ and $\tau$ is called a \emph{sink region} if the orientation on each subpath in $\partial R$ points inward  and a \emph{source region} if the orientation on each subpath in $\partial R$ points outward.      Note that by definition, the boundary of a  sink or source region does not include any part of $\partial D$.  
\end{definition}

\begin{lemma} \label{lem: intersections}
Let $\sigma$ and $\tau$ be paths in a $2$-disc $D$ as per Definition~\ref{def: sink source}.    
	If there is no sink region in $D$, then, up to a homeomorphism of $D$, we have one of the cases displayed in Figure~\ref{fig: intersections}.  (The cases are arranged into four families according to the possible relative orientations of $\sigma$ and $\tau$ where they meet $S^1 = \partial D$.  
	Cases (2) and (3) include the possibility that $\sigma$ and $\tau$ do not intersect.) 
\end{lemma}

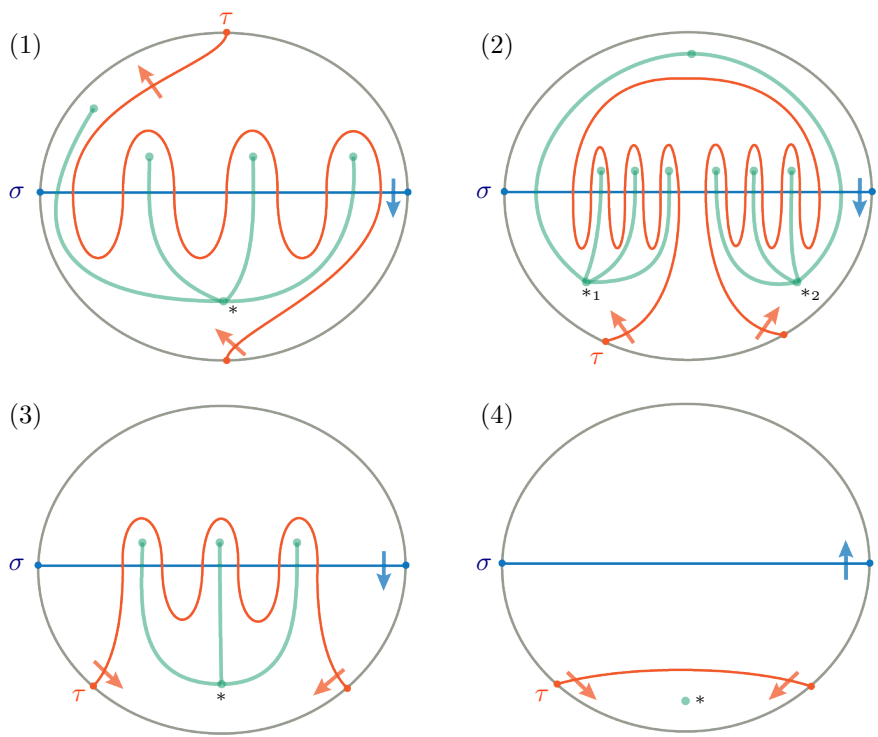
\begin{figure}[htbp]
\centering
\begin{overpic} 
{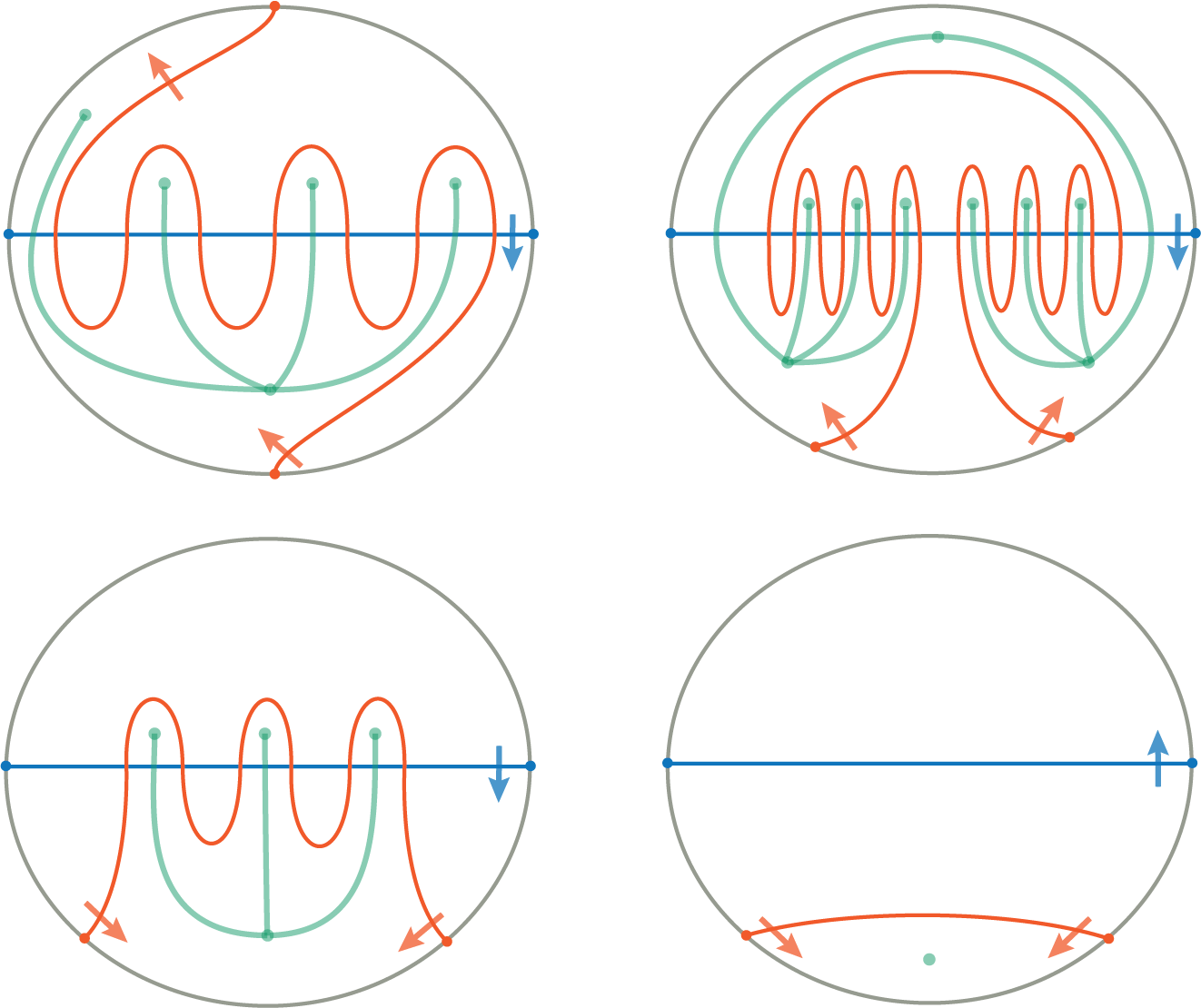}
\put(-3, 63.8){\blue{\small{$\sigma$}}}
\put(52.5, 63.8){\blue{\small{$\sigma$}}}
\put(-3, 19.5){\blue{\small{$\sigma$}}}
\put(52.5, 19.5){\blue{\small{$\sigma$}}}
\put(22, 84.5){\orange{\small{$\tau$}}}
\put(66, 44){\orange{\small{$\tau$}}}
\put(4.5, 4){\orange{\small{$\tau$}}}
\put(59.5, 4){\orange{\small{$\tau$}}}
\put(23, 50){\tiny{$\ast$}}
\put(65, 52){\tiny{$\ast_1$}}
\put(91, 52){\tiny{$\ast_2$}}
\put(21.6, 4){\tiny{$\ast$}}
\put(78.5, 3.6){\tiny{$\ast$}}
\put(-3, 81){\small{(1)}}
\put(53, 81){\small{(2)}}
\put(-3, 37){\small{(3)}}
\put(53, 37){\small{(4)}}

\end{overpic}
\caption{The intersections patterns of two transversely oriented chords $\sigma$ and $\tau$ across a disc per  Lemma~\ref{lem: intersections}, if there are no sink regions. There are four cases depending on the relative positions of the end points of $\sigma$ and $\tau$ and on their orientations.  In  (1) $\sigma$ and $\tau$ intersect $2n-1$ times for some $n \geq 1$, in (2) they intersect either $0$ times or  $(2m-1) + (2n-1)$ times for some  $m, n \geq 1$, in (3) they intersect  $2n$ times for some  
$n \geq 0$,   and in (4) they do not intersect. }\label{fig: intersections} 
\end{figure}

\begin{proof}
Consider the planar graph $\mathcal{G}$ whose vertices are  the points of intersection of $\sigma$ and $\tau$ and the four end points, and whose edges are the subpaths of $\sigma$, $\tau$, and $\partial D$ that connect them (call these \emph{$\sigma$-}, \emph{$\tau$-}, and \emph{$\partial D$-edges}, respectively).  The path $\tau$ subdivides $D$ into two subdiscs (ditto the path $\sigma$).  Let $\mathcal{T}$ be the planar graph (in fact, \emph{tree}) that has   
\begin{itemize}
\item vertices dual to every \emph{face} of $\mathcal{G}$ (i.e,   connected component of  $D \ssm \mathcal{G}$) that the orientation of $\tau$ points into,  and
\item  edges dual to all $\sigma$-edges.  
\end{itemize}
Figure~\ref{fig: intersections example}(left) shows an example---there is no loss of generality in taking $\sigma$ to be a diameter of the disc.

\begin{figure}[htbp]
\centering
\begin{overpic} 
{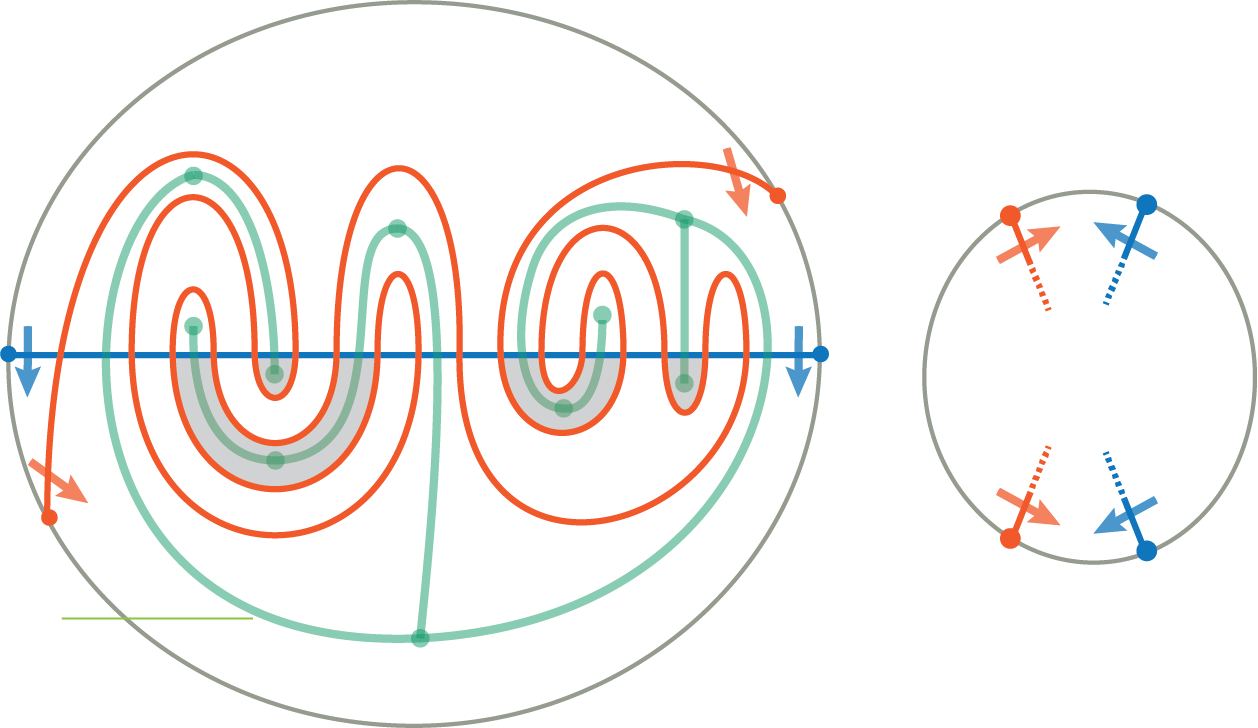}

\put(8, 52){\small{$D$}}
\put(2, 8){\small{$\mathcal{T}$}}
\put(-3, 29){\blue{\small{$\sigma$}}}
\put(1, 15){\orange{\small{$\tau$}}}
\put(32.5, 4.5){\small{$\ast$}}
\put(92, 43){\blue{\small{$\sigma$}}}
\put(79, 43){\orange{\small{$\tau$}}}
\end{overpic}
\caption{Left:  our proof of Lemma~\ref{lem: intersections}, illustrated.  Right: orientations per Corollary~\ref{cor: intersections}.}\label{fig: intersections example} 
\end{figure}
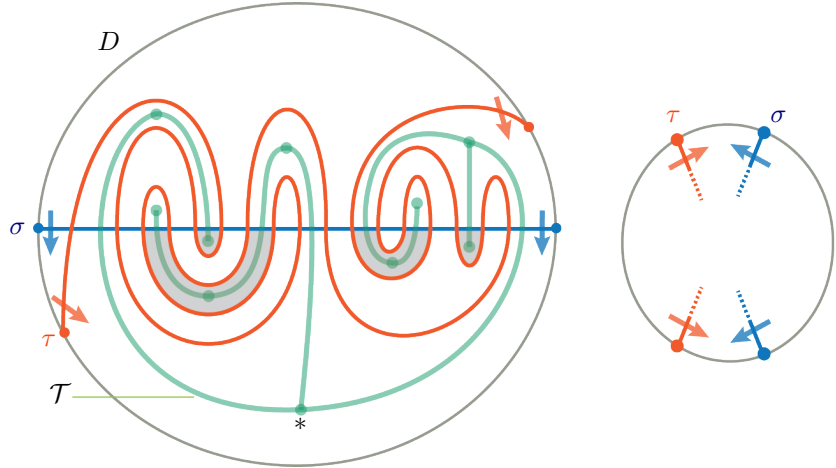

 Case (1) of Figure~\ref{fig: intersections} concerns when the end points of $\sigma$ and  $\tau$  alternate around $\partial D$.  Cases (2)--(4) subdivide the eventuality where they do not alternate to three mutually exclusive possibilities for the orientations of $\sigma$ and  $\tau$ where they meet $\partial D$,  
  namely, oriented towards each other,  in the same direction, or away from each other.    
  
Depending on whether or not $\sigma$ and $\tau$ intersect, there  are either four or three faces in $\mathcal{G}$  that have $\partial D$-edges in their boundaries.  Call these \emph{boundary faces}.        
A face $f$ of $\mathcal{G}$ either has all the $\sigma$-edges in its boundary oriented into or all out of $f$, depending on which side of $\sigma$ the face $f$ is on.  The same is true of the $\tau$-edges in $\partial f$.  
In case (1),  let $f$ be the unique boundary face that has all $\sigma$- and $\tau$-edges in $\partial f$  oriented into $f$. 
In cases (3) and (4), let $f$ be the unique boundary face that has all $\tau$-edges in $\partial f$  oriented into $f$. 
 Now,  the vertex $\ast$ dual to $f$ is a vertex of $\mathcal{T}$.  In cases (1) and (3), every other vertex of $\mathcal{T}$ that is an even distance (in $\mathcal{T}$) from $\ast$ is dual to a face that is a sink region.   (In the example of Figure~\ref{fig: intersections example} there are four such vertices, all a distance $2$ from $\ast$.  The  four faces that they are dual to are shown shaded.)  In case (4)  every vertex of $\mathcal{T}$ that is an odd distance from $\ast$ is dual to a sink region.    As our hypotheses prohibit sink regions, $\mathcal{T}$ is 
 restricted accordingly.  
 Thus $\sigma$ and $\tau$ cannot intersect in case (4), and  in cases, (1) and (3), if $\sigma$ and $\tau$ intersect, they must do so as shown in Figure~\ref{fig: intersections}, where $n$ is the valence of $\ast$.

In  the instance of case (2) if $\sigma$ and $\tau$ do  intersect, there  are two boundary faces $f_1$ and $f_2$ into which all $\sigma$- and $\tau$-edges in their boundaries are inward-oriented.  Let $\ast_1$ and $\ast_2$ be their dual vertices.  It follows that $\ast_1$ and $\ast_2$ are an even distance apart in  $\mathcal{T}$  and any there can be no other vertices in $\mathcal{T}$  that are an even distance from either.  Thus $\mathcal{T}$ is the tree shown in Figure~\ref{fig: intersections}(2),
with $m$ and $n$ being the valences of  $\ast_1$ and $\ast_2$, and moreover, no other arrangement of $\mathcal T$ along $\sigma$ is possible. 
\end{proof}

\begin{corollary} \label{cor: intersections}
	Suppose  $\sigma$ and $\tau$ are paths  in a $2$-disc $D$ as per Definition~\ref{def: sink source},  
	but we prohibit source regions instead of sink regions.  If the order and relative orientations of $\sigma$ and $\tau$ close to $\partial D$ are as shown in  Figure~\ref{fig: intersections example} (right), then   $\sigma$ and $\tau$ do not intersect.   
\end{corollary}

\begin{proof}
	This is case (4) of Lemma~\ref{lem: intersections}, but  with the orientations reversed.   
\end{proof}

Our final lemma is the observation which says, roughly, that  a pair of  oriented paths through a disc that intersect transversely, can be ``combined''  to obtain a new transversely oriented such path, so that the original paths both lie to one side of the new path.  This is illustrated in Figure~\ref{fig:combining oriented paths}, under the simplifying assumption that the intersections between the paths are transverse.  The lemma allows subpaths as intersections, so it can be applied to (compound) tracks.

\begin{lemma} \label{lem: combining oriented paths}
	Suppose for $i=1,2$,   an injective piecewise-linear path $\sigma_i$ in a 2-disc $D$ is  made up of   finitely many straight-line segments, and that $\sigma_i$  meets $\partial D$ at exactly 2 points, specifically its endpoints.   Suppose $\sigma_1$ and $\sigma_2$ have  transverse orientations. So, for $i=1,2$, there are subsets $D_i^+$ and $D_i^-$ of $D$, each  homeomorphic to a 2-disc, such that
	 $D = D_i^+ \cup D_i^-$, and $\sigma_i$ traverses the intersection of  $D_i^+$ and $D_i^-$ with   $\sigma_i$  oriented into $D_i^+$ and out of $D_i^-$.  Assume $\sigma_1$ and  $\sigma_2$ intersect in the interior of $D$.   We allow the intersection of $\sigma_1$ and $\sigma_2$ to include (finitely many) straight line segments, provided their orientations agree on the common segments.

	Suppose there is a point $p \in \partial D$ that is in $D^+_1 \cap D^+_2$ and is not on $\sigma_1$ or $\sigma_2$.  Let $C_0^+$ be the maximal connected open subset of $D$ that contains $p$ and does not intersect $\sigma_1$ or $\sigma_2$.  Let $C^+$ be the closure of $C_0^+$  and $C^-$ be  $D \ssm C_0^+$.  Then   $C^+$ and $C^-$ are homeomorphic to 2-discs.  Furthermore, 
\begin{enumerate}
\item $C^+$ contains $p$,
\item $D_1^- \cup D_2^- \subseteq C^-$.  In particular, $\sigma_1$ and $\sigma_2$ are in $C^-$, and
\item an  injective piecewise-linear path $\tau$ traverses  $C^+ \cap C^-$, connecting two different points on $\partial D$. It is a concatenation of subpaths of $\sigma_1$ and  $\sigma_2$, all oriented into $C^+$, and so has a well-defined orientation (into $C^+$). 
\end{enumerate}
\end{lemma}

\begin{figure}[htbp]
\centering
\begin{overpic} 
{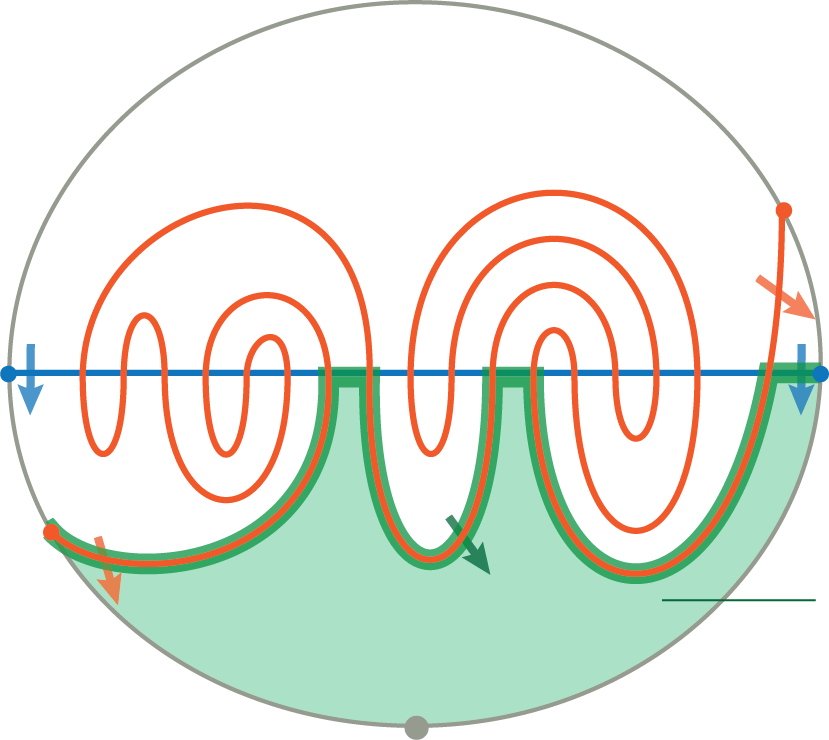}
\put(49, -4){\small{$p$}}
\put(-6, 43){\small{$\sigma_1$}}
\put(-1, 23){\small{$\sigma_2$}}
\put(55, 17){\small{$\tau$}}
\put(100, 15){\small{$C^+$}}
\end{overpic}
\caption{Lemma~\ref{lem: combining oriented paths}, illustrated.}\label{fig:combining oriented paths}
\end{figure}

\subsection{Tracks in distortion diagrams} \label{sec: tracks in distortion diagrams}

In Section~\ref{sec:tracks in reduced diagrams} we established constraints on \emph{reduced} van~Kampen diagrams over our presentation $\mathcal{P}$ for $G$.  Here, we will show that   diagrams pertinent to the distortion of $H$ in $G$ are further constrained.  The rigidity we will prove here and in Section~\ref{sec: a2bq tracks} will allow us to calculate upper bounds on distortion in Section~\ref{sec:upper}.

\begin{definition} \label{def: distortion diagram} \textbf{(Distortion diagrams, sides)}
A  \emph{distortion diagram}  $\Delta$ is a reduced van~Kampen diagram for $w \chi^{-1}$ over $\mathcal{P}$, where $\chi$ is a word on  $t, y_1, y_2$ and $w$ is a word on our generating set for $G$.  Where no confusion should result, we refer to the portions of the boundary circuit $\partial \Delta$ that are labelled by $w$ and by $\chi$ simply as $w$ and~$\chi$.     When an $a$- or $b$-track $\rho$ connects two edges in $\partial \Delta$ those edges must both be in $w$, as there are no $a$- or $b$-letters in $\chi$.  So, as shown in Figure~\ref{fig:sides},  the track $\rho$ subdivides $\Delta$ into two subsets  
whose intersection is $\rho$.  The subset  that   contains $\chi$ is the \emph{$\chi$-side} of $\rho$, and the other subset is the \emph{$w$-side}.   	
\end{definition}

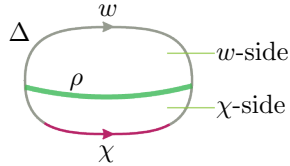
\begin{figure}[htbp]
\centering
\begin{overpic} 
{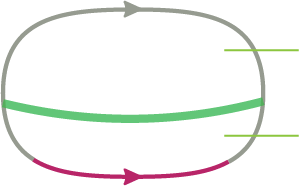}
\put(-8, 52){\small{$\Delta$}}
\put(24, 29){\small{$\rho$}}
\put(39, -9){\small{$\chi$}}
\put(39, 65){\small{$w$}}
\put(101, 14){\small{$\chi$-side}}
\put(101, 42){\small{$w$-side}}
\end{overpic}
\caption{An $a$- or $b$-track $\rho$ in a distortion diagram}\label{fig:sides}
\end{figure}

\begin{lemma} ${}$ \label{lem: Layout lemma} \textbf{($a$- and $b$-tracks in distortion diagrams.)}  There exists $C>0$ satisfying the following.    Suppose
		 $w_0$ is a word on the generators of $G$ that equals in $G$ a reduced word   $\chi$ on $t, y_1, y_2$, and  suppose  $\Delta_0$ is a distortion diagram for $w_0 \chi^{-1}$.
Assume that $\Delta_0$ is homeomorphic to a 2-disc.
   Then there is a subdiagram $\Delta$ of $\Delta_0$ that is a van~Kampen diagram for $w \chi^{-1}$, where $w$ is a word of length at most $C | w_0 |$ and the following properties are satisfied.

   \begin{enumerate}  
 
 \addtocounter{enumi}{-1}
 
 \item \label{concatenation of disc diagrams}  The portions of  $\partial \Delta$ labelled by $w$ and by $\chi$ are both  injective paths, so that $\Delta$ is a 
 concatenation of paths and distortion diagrams $\Delta'_1, \ldots, \Delta'_r$, each homeomorphic to a 2-disc and each demonstrating that some subword of $w$ equals some subword of $\chi$ (as shown on the right below). 
\bs \ms

\begin{center}
\centering
\begin{overpic} 
{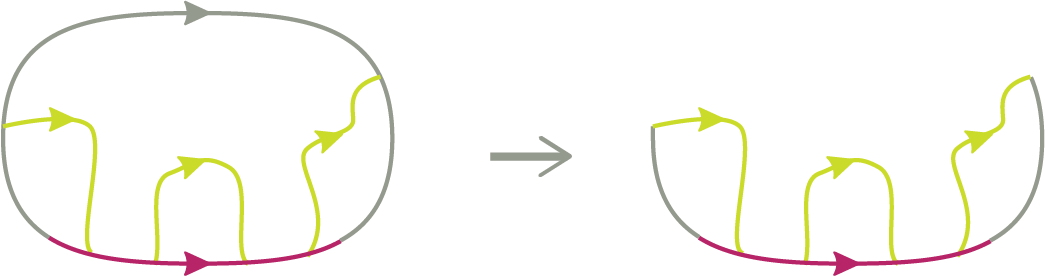}
\put(39, 17){\small{$\Delta_0$}}
\put(101, 17){\small{$\Delta$}}
\put(64.6, 9){\small{$\Delta'_1$}}
\put(79, 5){\small{$\Delta'_2$}}
\put(93, 9){\small{$\Delta'_3$}}
\put(16, 27){\small{$w_0$}}
\put(79, 12.5){\small{$w$}}
\put(17, -3){\small{$\chi$}}
 \put(79, -3){\small{$\chi$}}
\end{overpic}
\end{center}
   
   \bs

\item \label{lem part:orientations} 
No  compound track in $\Delta$ between a pair of edges in $w$ is made up of $a$-subtracks oriented towards $w$,  $b$-subtracks oriented towards $w$, and $t$-tracks (oriented either way).  In particular, 
no $t$-corridor in $\Delta$ connects two $t$-letters in  $w$ and every $a$- or $b$-track that connects a pair of edges in $\partial \Delta$ is oriented towards $\chi$.

\bs
\begin{center}
\centering
\begin{overpic}   [scale=0.75] 
{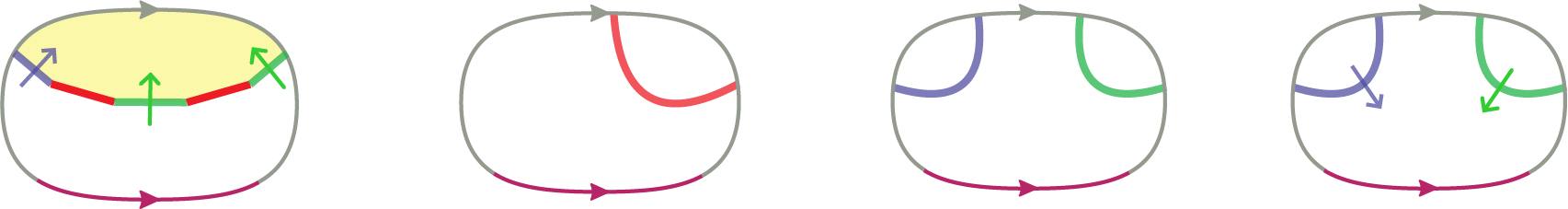}
\put(25, 9){\small{No}}
\put(-4.5, 9){\small{No}}
\put(1, 6.5){\blue{\small{$a$}}}
\put(4.5, 5){\dred{\small{$t$}}}
\put(10.5, 4.3){\green{\small{$b$}}}
 \put(14, 5){\dred{\small{$t$}}}
\put(16.8, 6){\green{\small{$b$}}}
 \put(44, 4.5){\dred{\small{$t$}}}
\put(60, 5.5){\blue{\small{$a$}}}
\put(70.5, 5){\green{\small{$b$}}}
\put(85, 5.5){\blue{\small{$a$}}}
\put(96.5, 5){\green{\small{$b$}}}
\put(76, 6){\small{$\implies$}}
\put(8.5, -1.5){\small{$\chi$}}
\put(37.5, -1.2){\small{$\chi$}}
\put(65, -1){\small{$\chi$}}
\put(90.5, -1){\small{$\chi$}}
\put(8.5, 13.8){\small{$w$}}
\put(37, 13.4){\small{$w$}}
\put(65, 13.4){\small{$w$}}
\put(90.5, 13.4){\small{$w$}}
\end{overpic}
\end{center}
\bs

\item  \label{lem part: No y-edges above a b-track} There are no $y$-edges in the $w$-side of any $b$-track $\beta$ that connects two edges in $\partial \Delta$. 

\bs
\begin{center}
\centering
\begin{overpic}  
{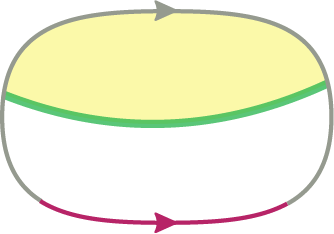}
\put(20, 45.7){\small{No $y$-edges}}
\put(44, -9){\small{$\chi$}}
\put(44, 73){\small{$w$}}
\end{overpic}
\end{center}
\bs

\item \label{lem part:no sinks} Suppose a region  $R$ is a subset of the $w$-side of a $b$-track connecting two points in $w$.  
\begin{enumerate}
\item \label{lem part part:no boundary sinks}  $\partial R$ cannot be comprised of a (non-trivial) subpath of the boundary circuit $\partial \Delta$,  $a$-subtracks,   inward oriented $b$-subtracks, and $t$-subtracks. 

\item \label{lem part part:no sinks} If  $\partial R$ is comprised of $a$-subtracks and inward-oriented $b$-subtracks, then it satisfies the constraints \ref{lem part part: bs}--\ref{lem part part: as} of Lemma~\ref{lem: trapped x-noise}.   
In particular,  $\partial R$ cannot be a bigon comprised of an $a_1$-subtrack and an inward oriented $b$-subtrack.

\end{enumerate}
\bs
\begin{center}
\centering
\begin{overpic}  [scale=0.8] 
{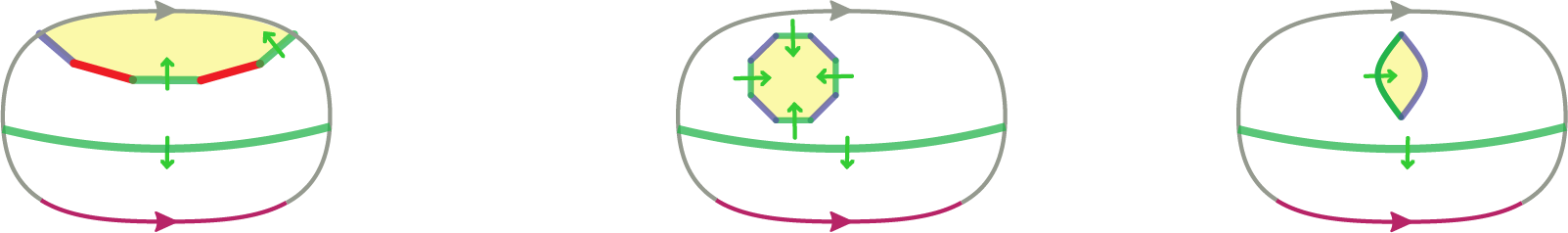}
\put(-4, 11){\small{No}}
\put(30, 11){\small{Almost no}}
\put(75, 11){\small{No}}
\put(9.9, -2){\small{$\chi$}}
\put(53, -2){\small{$\chi$}}
\put(88.5, -2){\small{$\chi$}}
\put(9.4, 15.5){\small{$w$}}
\put(52.5, 15.5){\small{$w$}}
\put(88.1, 15.5){\small{$w$}}
\put(2, 10.2){\blue{\small{$a$}}}
\put(5.5, 8){\dred{\small{$t$}}}
\put(11, 7.3){\green{\small{$b$}}}
 \put(14.5, 7.8){\dred{\small{$t$}}}
\put(18.5, 9.5){\green{\small{$b$}}}
\put(11.2, 2.8){\green{\small{$b$}}}
\put(54.8, 2.8){\green{\small{$b$}}}
\put(90.4, 2.8){\green{\small{$b$}}}
\put(86.5, 7.7){\green{\small{$b$}}}
\put(90.8, 7.7){\blue{\small{$a$}}}
\put(50, 9.2){\green{\small{$b$}}}
\put(47.2, 6.7){\blue{\small{$a$}}}
\put(52.8, 6.7){\blue{\small{$a$}}}
\put(52.8, 11.9){\blue{\small{$a$}}}
\put(47.2, 11.9){\blue{\small{$a$}}}
\end{overpic}
\end{center}
\bs

\item \label{lem part: no button} $\Delta$ contains no badge and no button (Definitions~\ref{def: badge} and \ref{def: button}).

\bs
 \begin{center}
\centering
\begin{overpic} 
{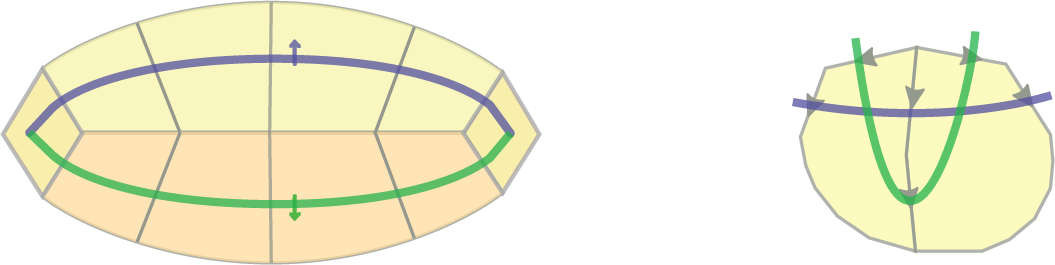}
\put(-8, 19){\small{No}}
\put(69, 19){\small{No}}
\put(11.5, 13.2){\small{\red{$t$}}}
\put(20, 13.2){\small{\red{$t$}}}
\put(30.5, 13.2){\small{\red{$t$}}}
\put(40, 13.2){\small{\red{$t$}}}
\put(21, 2){\small{\green{$b_{\ast}$}}}
\put(21, 21){\small{\blue{$a_{\ast}$}}}
\put(81, 6){\small{\green{$b_{\ast}$}}}
\put(97, 18){\small{\blue{$a_{\ast}$}}}
\end{overpic}
\end{center}
\bs

\item \label{lem part:no loops and bigons}
$\Delta$ has  no $a$- or $b$-loops and 
no bigons comprised of an $a$-subtrack and an outward oriented $b$-subtrack.

\bs

\begin{center}
\centering
\begin{overpic} 
{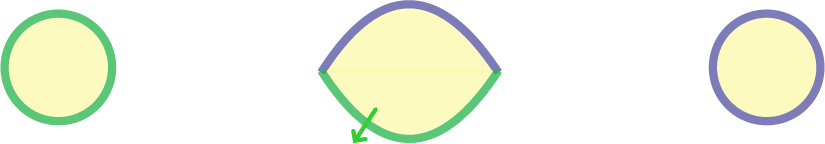}
\put(-10, 14){\small{{No}}}
\put(29, 14){\small{{No}}}
\put(76, 14){\small{{No}}}
\put(85, 1){\small{\blue{$a$}}}
\put(37, 1){\small{\green{$b$}}}
\put(56.5, 15.5){\small{\blue{$a$}}}
\put(-1, 1){\small{\green{$b$}}}
\end{overpic}
\end{center}

\bs 
 
 \item \label{lem part:no sources} 
	More generally, no region of $\Delta$ has boundary made up of \emph{consistently oriented} (meaning all inward- or all outward-oriented)  $a$-subtracks and outward-oriented $b$-subtracks.  
 \bs
 \begin{center}
\centering
\begin{overpic}  
{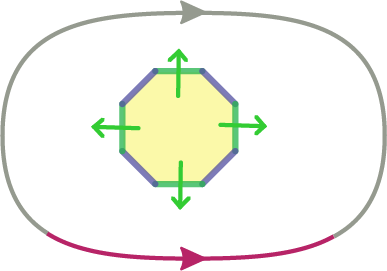}
\put(-30, 49){\small{{No}}}
\put(45, -8){\small{$\chi$}}
\put(45, 72){\small{$w$}}
\put(49, 10){\green{\small{$b$}}}
\put(29.2, 19){\blue{\small{$a$}}}
\put(29.2, 49){\blue{\small{$a$}}}
\put(57, 19){\blue{\small{$a$}}}
\put(57, 49){\blue{\small{$a$}}}
\put(67, 27.2){\green{\small{$b$}}}
\put(21, 27.2){\green{\small{$b$}}}
\put(49, 55){\green{\small{$b$}}}
\end{overpic}
\end{center}
\bs

\item \label{lem part:crossing possibilities}   Suppose $\alpha$ is an $a_1$-track and $\beta$ is a $b$-track in $\Delta$. 
	\begin{enumerate}
\item \label{lem part: both sides}
If $\alpha$ has one endpoint on either side of $\beta$ then $\alpha$ and $\beta$ intersect exactly once.
\item \label{lem part: chi side}
If both endpoints of $\alpha$ are on the $\chi$-side of $\beta$, then $\alpha$ and $\beta$ do not intersect.   

\item \label{lem part: w side} If both endpoints of $\alpha$ are on the $w$-side of $\beta$, then $\alpha$ and $\beta$ 
 intersect exactly twice.

 \bs 
\begin{center}
\centering
\begin{overpic}  [scale=0.75]  
{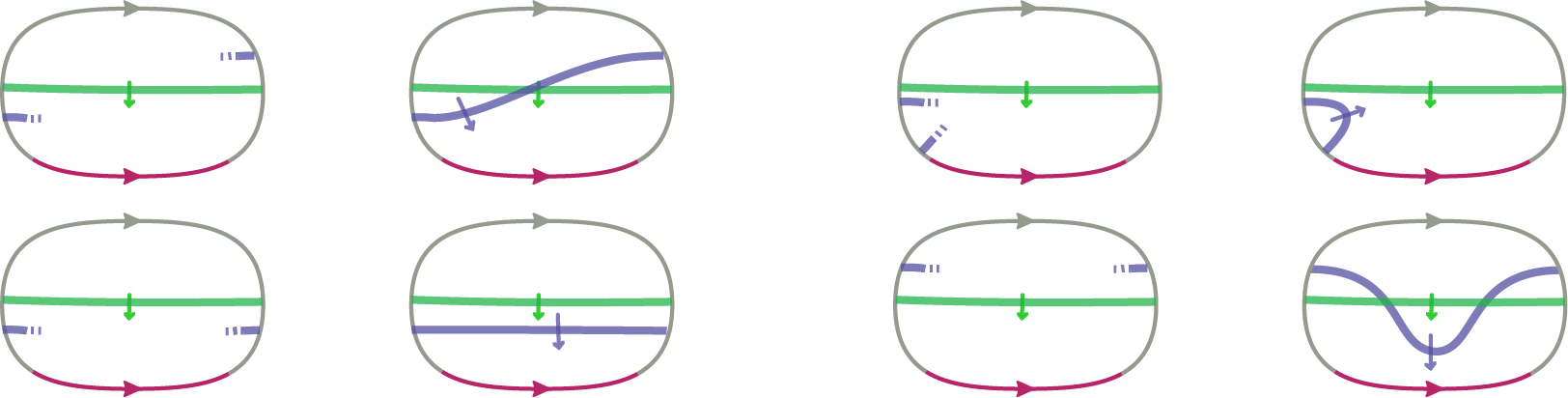}
\put(19, 5){\small{$\implies$}}
\put(20, 7.5){\small{\textit{(b)}}}

\put(19, 19){\small{$\implies$}}
\put(20, 21.5){\small{\textit{(a)}}}

\put(75.5, 5){\small{$\implies$}}
\put(76.5, 7.5){\small{\textit{(c)}}}

\put(75.5, 19){\small{$\implies$}}
\put(76.5, 21.5){\small{\textit{(b)}}}

\end{overpic}
\end{center}
\bs
	\end{enumerate}

\item \label{lem part: b_0-track outermost} There can be no $b_0$-track $\beta_0 \neq \beta$ in the $w$-side of a $b$-track $\beta$. 

\bs
 \begin{center}
\centering
\begin{overpic} 
{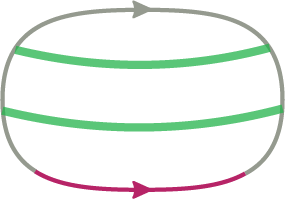}
\put(-30, 49){\small{{No}}}
\put(75, 37.3){\green{\small{$b_0$}}}
\put(75, 14.5){\green{\small{$b$}}}
\put(45, -10){\small{$\chi$}}
\put(45, 74){\small{$w$}}
\end{overpic}
\end{center}
\bs

	\end{enumerate}
	
\end{lemma}

\begin{proof}

We will   sever parts of $\Delta_0$ to obtain subdiagrams $\Delta_1$, then $\Delta_2$, and then $\Delta_3$,  that establish, respectively,  \eqref{lem part:orientations}, then   \eqref{lem part: No y-edges above a b-track}, and then   \eqref{lem part:no sinks}.  Then we will   sever parts of $\Delta_3$ to  get $\Delta$ such  that  the portion of  $\partial \Delta$ labelled by $w$ is an injective path, and we will argue that  $\Delta$  satisfies \emph{all} of   \eqref{concatenation of disc diagrams}--\eqref{lem part:no sinks}.    Then we will verify that $\Delta$ also  satisfies \eqref{lem part: no button}--\eqref{lem part: b_0-track outermost}.

\bs
For \eqref{lem part:orientations}, define a \emph{bad} path in $\Delta_0$ to be a compound track connecting a  pair of edges in $w_0$ comprised of $a$-and  $b$-subtracks oriented towards $w_0$, and $t$-tracks (oriented either way).     Let $\Delta_1$ be the maximal subdiagram of $\Delta_0$ that contains $\chi$ and intersects no bad path.  
Let $w_1$ be the word such that    $\Delta_1$ is a van~Kampen diagram for $w_1 \chi^{-1}$. 
If bad paths $\sigma_1$ and $\sigma_2$ intersect, then we may apply  Lemma~\ref{lem: combining oriented paths} with $p$ a point on $\chi$ and $\sigma_1$ and $\sigma_2$ oriented \emph{towards} $\chi$, to obtain a new path $\tau$ 
which is a concatenation of subpaths of $\sigma_1$ and $\sigma_2$ (and therefore is again a bad path), such that  both $\sigma_1$ and $\sigma_2$ are contained in the $w_0$-side of $\tau$.   
Therefore there is a collection of bad paths $\tau_1$, \ldots, $\tau_m$  that are disjoint and are such that $\Delta_1$ is the result of removing from $\Delta_0$ the subdiagrams bounded by the corridors of 2-cells through which $\tau_i$ passes and by  subwords of $w_0$.  Now 
Lemma~\ref{lem:t-corridors on boundary}\eqref{lem part: boundary region}  tells us that there exists a constant $C_1>0$ such that  $|w_1| \leq C_1|w_0|$.

\bs  For  \eqref{lem part: No y-edges above a b-track}, we first establish that
there exist disjoint $b$-tracks $\beta_1, \dots, \beta_k$, each a path between two points in $\partial \Delta_1$,  such that every $b$-track between two points in $\partial \Delta_1$ 
is on the $w_1$-side of $\beta_i$ for some $i$.  To see this, note that following~\eqref{lem part:orientations}, all $b$-tracks between pairs of points in $\partial \Delta_1$ are oriented towards $\chi$, and if two such $b$-tracks $\sigma_1$ and $\sigma_2$ intersect, then applying 
 Lemma~\ref{lem: combining oriented paths} with $p$ a point on $\chi$, we obtain a 
path $\tau$ connecting a pair of points on $\partial \Delta_1$,  
such that both  $\sigma_1$ and $\sigma_2$ are on the $w_1$ side of $\tau$, 
and $\tau$ is a concatenation  of subtracks of $\sigma_1$ and $\sigma_2$, each oriented into the component of $\Delta_1 \setminus \tau$ containing $\chi$.  Since a concatenation of consistently oriented $b$-subtracks  is again a $b$-subtrack, $\tau$ is again a $b$-track.  
 The existence of $\beta_1, \dots, \beta_k$ as above follows. 
 
 Thus, in constructing $\Delta_2$ by severing parts of $\Delta_1$, it suffices to guarantee that~\eqref{lem part: No y-edges above a b-track} holds for  $\beta=\beta_i$ for each $1\le i \le k$.  
Our argument in this case is illustrated by Figure~\ref{fig:y-track}.

By Lemma~\ref{lem: trapped y-noise}\eqref{lem part: no y-edges}, there is  no $y$-edge in any region $R_i$ enclosed by a subpath of $\beta$ and  a $t$-subtrack on the $w_1$-side of $\beta$ (such as   regions $R_1$, $R_2$,  and $R_3$ in Figure~\ref{fig:y-track}), as $\partial R_i$ has no edges in this case.  
Define $\Delta'_{\beta}$ to be the maximal subdiagram  of $\Delta_1$ that is contained in the $w_1$-side of $\beta$ and intersects no $t$-subtracks that start and end on $\beta$.
Then $\Delta'_{\beta}$ is a van~Kampen diagram for $uv^{-1}$, where $u$ is a subword of $w_1$ and $v$ is the word 
along the remainder of $\partial \Delta'_\beta$, as shown in Figure~\ref{fig:y-track}.

We will apply Lemma~\ref{lem: vertical t-corridors} to $\Delta'_{\beta}$. Let us check the hypotheses. 
To see that there are no $y$-letters in $v$, observe that $v$ is comprised of subpaths that run along the corridor associated to $\beta$, on the side that $\beta$ is oriented away from, and subpaths that run along the sides of $t$-corridors.  
The defining relations of $G$ (see Figure~\ref{fig:relations}) imply that the first type of subpath cannot have any $y$-edges, and if there were a $y$-edge in a subpath of the second type, then then there would be one on the other side of the $t$-corridor also, and so in one of the regions $R_i$, a contradiction.  

Next, we observe that all $t$-corridors in $\Delta'_{\beta}$ connect a $t$-edge in $u$ to a $t$-edge in $v$.  This  is because there are no $t$-loops by  Lemma~\ref{lem: no t-loop};   were there a $t$-track connecting a pair of edges in $u$, it would be a part (or whole) of a bad path in $\Delta_0$, and would have been cut off in the construction of $\Delta_1$;
and no $t$-corridor joins  pair of $t$-edges in $v$ by construction.  

Lemma~\ref{lem: vertical t-corridors} now implies that there is a constant $C_2>0$ (depending only on $\mathcal{P}$) and a word $v'$ labeling a path in $\Delta_\beta'^{(1)}$ with the same endpoints as $u$ and $v$ with $|v'| \le C_2 |u|$ such that the subdiagram enclosed by $v$ and $v'$ has no $y$-edges.  We now cut $\Delta_\beta'$ along $v'$, discarding the subdiagram bounded by $u$ and $v'$.  As $\beta_1, \dots, \beta_k$ are disjoint and non-nested, we do this independently for each $\beta=\beta_i$, resulting in a subdiagram $\Delta_2$ of $\Delta_1$ for a relation $w_2 \chi^{-1}$, where $w_2$ is obtained from $w_1$ by replacing a disjoint collection of subwords with words whose lengths are greater by at most a factor of $C_2$. It follows that $|w_2| \le C_2 |w_1|$, and by construction, there are no $y$-edges on the $w_2$ side of $\beta_i$ for any $i$.   In particular, \eqref{lem part: No y-edges above a b-track}  holds for $\Delta_2$.

Now suppose $\Delta_2$ has a bad path $\sigma$---i.e., suppose that \eqref{lem part:orientations} fails for $\Delta_2$.  Since $\Delta_1$ had none, $\sigma$ must have at least one end on along a path labelled by one of the $v'$, and this path is on the $w$ side of some $\beta$ which is oriented towards $\chi$.  If $\sigma$ intersects $\beta$ at least twice, then, since $\beta$ is oriented towards $\chi$, a subtrack of $\beta$ and a subpath of $\sigma$ together bound a region $R$ that is precluded by Lemma~\ref{lem: trapped y-noise} (see Figure~\ref{fig:forbidden examples}).  If $\sigma$ crosses $\beta$ exactly once, then a subpath of $\beta$, together with the part of $\sigma$ on the $\chi$ side of $\beta$ form a bad path (in the sense of  \eqref{lem part:orientations}) in $\Delta_2$, which is not possible.  Thus any bad path $\sigma$ in $\Delta_2$
lies on the $w_2$ side of $\beta$.  Such paths will be removed next, in 
the construction of $\Delta_3$.

\begin{figure}[htbp]
\centering
\begin{overpic}
{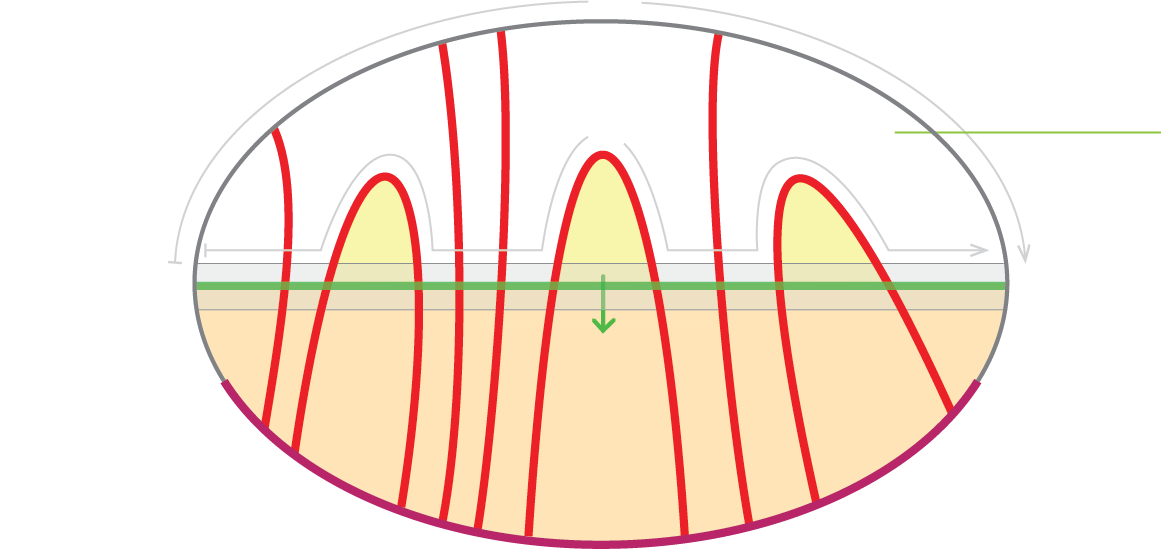}
\put(50, -3){\small{$\chi$}}
\put(53, 19){\small{$\beta$}}
\put(30.5, 26){\small{$R_1$}}
\put(50.2, 26){\small{$R_2$}}
\put(68.5, 26){\small{$R_3$}}
\put(52, 46.5){\small{$u$}}
\put(51.3, 35){\small{$v$}}
\put(83, 42){\small{$\Delta_2$}}
\put(15, 40){\small{$w_2$}}
\put(101, 35){\small{$\Delta'_{\beta}$}}
\end{overpic}
\caption{Subdiagrams and $t$-tracks per our proof of Lemma~\ref{lem: Layout lemma}\eqref{lem part: No y-edges above a b-track}}
\label{fig:y-track}
\end{figure}
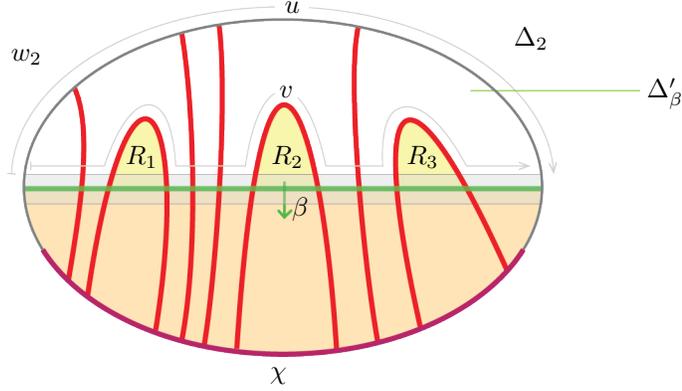

 \bs

For \eqref{lem part part:no boundary sinks}, define a region $R$ to be \emph{bad} if it is of the form \eqref{lem part part:no boundary sinks} excludes: that is, $R$ is a subset of the $w_2$-side of a $b$-track $\beta$ connecting two edges in $w$ and $\partial R$ is comprised of a  non-trivial subpath of the boundary circuit $\partial \Delta_2$   
and a compound track consisting of   $a$-subtracks,   inward oriented $b$-subtracks, and $t$-subtracks. We may assume that $\beta$ is one of the tracks $\beta_1, \dots,\beta_k$ identified above, which persist in $\Delta_2$.  
   Here are two key observations:
\begin{itemize}
\item[i.] 
If two bad regions $R_1$  and $R_2$ have intersecting interiors, they are on the $w_2$-side of a common $b$-track, say $\beta_i$.  Then, applying 
Lemma~\ref{lem: combining oriented paths} to the compound tracks in $\partial R_1$ and $\partial R_2$, we get a new bad region $R_3$ containing $R_1 \cup R_2 $ that is again on the $w_2$-side of $\beta_i$.

\item[ii.]   Suppose $R$ is a bad region on the $w$-side of a $b$-track $\beta$. Then the minimal subdiagram $D$ of $\Delta_2$ containing $R$ contains no $y$-edges.  To see this, note that no subpath of $\beta$ can contribute to $\partial R$, as $\beta$ is oriented towards $\chi$, and so no 2-cell through which $\beta$ passes can be in $D$.  Thus $D$ is a subset of the $w$-side of $\beta$ and has no $y$-edges by \eqref{lem part: No y-edges above a b-track}.
 \end{itemize}

Define $\Delta_3$ to be the maximal subdiagram of $\Delta_2$ that includes $\chi$ and does not intersect any bad region. On account of (i),  $\Delta_3$ is obtained from  $\Delta_2$ by severing a finitely many subdiagrams $D$ per Lemma~\ref{lem:t-corridors on boundary} by, in   the notation of that lemma, cutting along the paths labelled  $u_1$.  Moreover, any two of these $D$ have disjoint interiors and the associated words $u_0$ label  paths in $\partial \Delta_2$ that are non-overlapping (but can share endpoints). By (ii), hypothesis \eqref{lem part: boundary region no y} of Lemma~\ref{lem:t-corridors on boundary} holds and we can apply that lemma  to each of these $D$.    
Let $w_3$ be the word such  that $\Delta_3$ is a van~Kampen diagram for $w_3 \chi^{-1}$.  The inequality in Lemma~\ref{lem:t-corridors on boundary} then tells us that there exists a constant $C_3>0$ such that $|w_3| \leq C_3|w_2|$.    
Finally, $\Delta_3$ satisfies conditions~\eqref{lem part:orientations}--\eqref{lem part:no sinks}: as shown above, the only paths that could fail 
\eqref{lem part:orientations} were removed in the construction of $\Delta_3$; \eqref{lem part: No y-edges above a b-track}  is immediately inherited from $\Delta_2$; \eqref{lem part part:no boundary sinks} is satisfied by construction; and, in light of \eqref{lem part: No y-edges above a b-track}, 
Lemma~\ref{lem: trapped x-noise} implies \eqref{lem part part:no sinks}.

\bs

If the portion of  $\partial \Delta_3$ labelled by $w_3$ is not an injective path, then some subword labels a subdiagram which is only attached to the rest of    $\Delta_3$ at a single vertex.  We sever all subdiagrams that so arise, so as to produce a van~Kampen diagram $\Delta$ for a word  $w \chi^{-1}$, with $|w| \leq |w_3|$, such that conditions~\eqref{lem part:orientations}--\eqref{lem part:no sinks} hold, and  the portion of  $\partial \Delta$ labelled by $w$ is an injective path.  By hypothesis, 
$\chi$ is a reduced word on $t, y_1, y_2$, which freely generate a free subgroup of $G$ by Corollary~\ref{H is free}, so $\chi$ also labels an injective path in $\partial \Delta$. So $\Delta$ is  a concatenation of paths and distortion diagrams $\Delta', \ldots, \Delta'_r$, each homeomorphic to a 2-disc and each   demonstrating that some subword of $w$ equals some subword of $\chi$. This establishes \eqref{concatenation of disc diagrams}.
Further, if we let $C = C_1 C_2 C_3 C_3$, then our inequalities combine to give  $|w| \leq C | w_0 |$, as required.

For the remainder of the proof, we assume, for convenience, that $\Delta$  is homeomorphic to a 2-disc. The proofs of  \eqref{lem part: no button}--\eqref{lem part: b_0-track outermost} in the general case follow easily.  (In cases (a) and (c) of \eqref{lem part:crossing possibilities} the hypothesis forces $\alpha$ and $\beta$ to be in the same component.  In case (b), the result is automatic if they are in different components.)

 \bs
\eqref{lem part: no button}. 
Suppose there is a badge or button $\mathcal B$ in $\Delta$. Per Definition~\ref{def: tracks},
let $\mathcal{G}_b$ be the graph whose edges are the duals of the  $b$-edges in $\Delta$.  Let $\mathcal{C}$ be the connected component of $\mathcal{G}_b$ that includes the $b$-track through  $\mathcal B$.
Let $i$ be minimal such that  $\mathcal{C}$ includes the dual of a $b_i$-edge.    A $b$-track that enters a 2-cell across a $b_i$-edge can exit across another $b_i$-edge unless that 2-cell is an  $r_{1,i-1}$-cell.   So the minimality of $i$  ensures that $\mathcal{C}$ contains a $b_i$-track $\beta$.  By Corollary~\ref{cor: no loops}\eqref{lem part: no b-loops}, $\beta$ is not a loop, and so it connects two $b_i$-edges in $w$, and is oriented towards $\chi$ by  \eqref{lem part:orientations}.  So  no $b$-tracks branch off $\beta$ on its $\chi$-side and, in particular,     the $b$-tracks through $\mathcal{B}$ are on its $w$-side. (They can have subpaths in common with $\beta$.)  By  \eqref{lem part: No y-edges above a b-track}, there are no $y$-edges on the $w$-side of $\beta$.  
This ensures that $\mathcal B$ is not a badge, as if it were, it would have an $r_{4, i}$-cell contributing a $y$-edge to the $w$-side of $\beta$.

Any $a_1$-track intersecting the $w$-side of $\beta$ intersects $\beta$ exactly once---it is not a loop on the $w$-side of $\beta$ (by \eqref{lem part: No y-edges above a b-track} and Corollary~\ref{cor: no loops}\eqref{lem part: no a-loops}), it is dual to at most one edge in $\partial \Delta$ (by \eqref{lem part part:no boundary sinks}), and it  intersects $\beta$ at most once, for if it formed a bigon with $\beta$, then Lemma~\ref{lem: bigons}\eqref{lem part: bi bigon} would apply to an innermost such instance $\alpha$, and one of the intersections of $\alpha$ and $\beta$ would have to occur in an $r_{1, i-1}$-cell, contradicting the minimality of $i$.  

\begin{figure}[htbp]
\centering
\begin{overpic}  [scale=0.8] 
{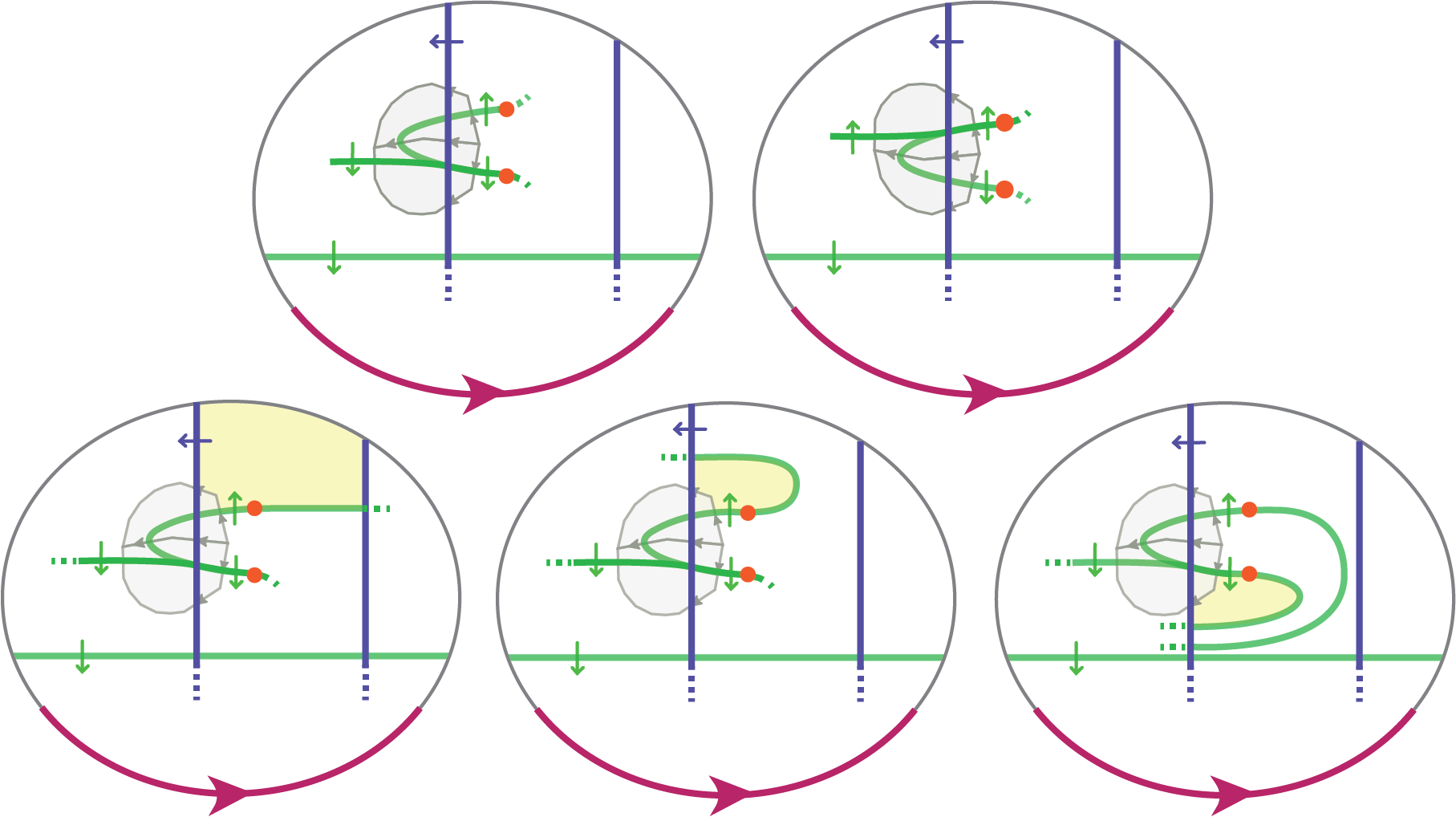}
\put(32.5, 26.5){\small{$\chi$}}
\put(67, 26.5){\small{$\chi$}}
\put(15, -1.5){\small{$\chi$}}
\put(49.5, -1.5){\small{$\chi$}}
\put(84, -1.5){\small{$\chi$}}
\put(32, 53.7){\small{$\alpha$}}
\put(66.2, 53.7){\small{$\alpha$}}
\put(14.5, 26.7){\small{$\alpha$}}
\put(49, 26.7){\small{$\alpha$}}
\put(82.6, 26.5){\small{$\alpha$}}
\put(43, 48){\small{$\alpha'$}}
\put(77.5, 48){\small{$\alpha'$}}
\put(22.4, 24){\small{$\alpha'$}}
\put(60, 20.7){\small{$\alpha'$}}
\put(94, 20.8){\small{$\alpha'$}}
\put(18, 24){\small{$R$}}
\put(51, 22){\small{$R$}}
\put(85, 14){\small{$R$}}
\put(23.5, 36){\small{$\beta$}}
\put(58, 36){\small{$\beta$}}
\put(6.5, 8.5){\small{$\beta$}}
\put(40.5, 8.5){\small{$\beta$}}
\put(74.7, 8.5){\small{$\beta$}}
\put(35.5, 47.3){\small\orange{{$A$}}}
\put(70, 43.3){\small\orange{{$B$}}}
\put(35.5, 44.5){\small\orange{{$B$}}}
\put(70, 46.2){\small\orange{{$A$}}}
\end{overpic}
\caption{Cases in our proof of Lemma~\ref{lem: Layout lemma}\eqref{lem part: no button} \label{fig:why_no_buttons}}
\end{figure}
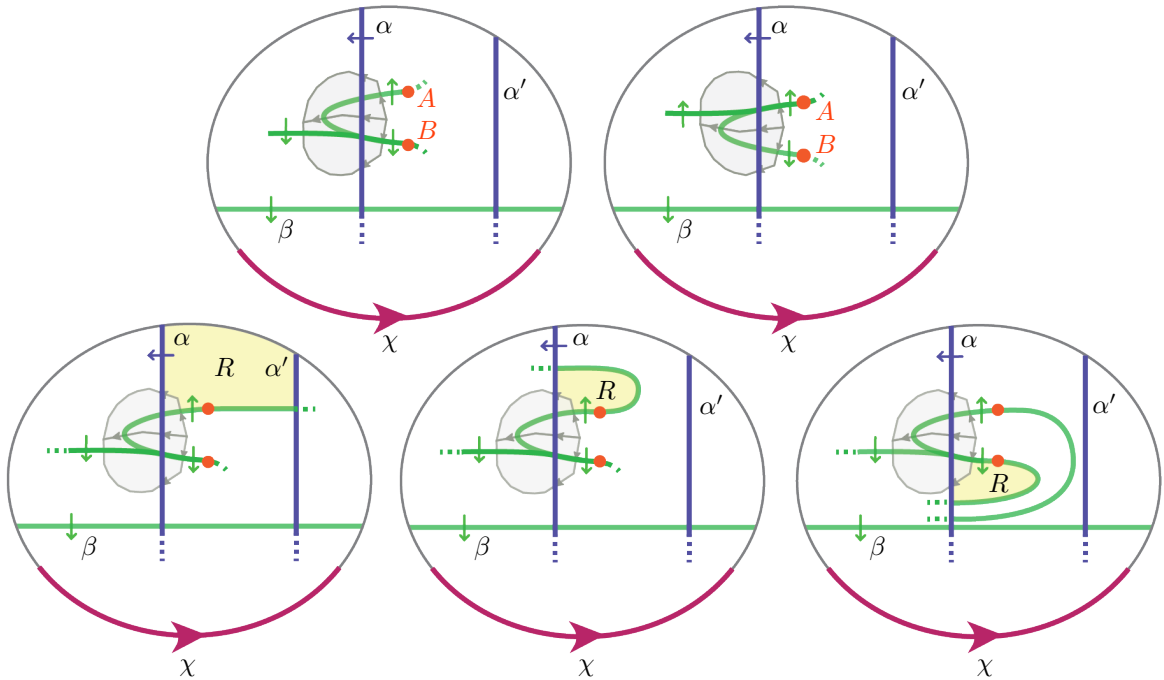

Let $\alpha$ be the $a_1$-track through $\mathcal B$, which we now know to be a button.  Figure~\ref{fig:why_no_buttons} (top-left and top-right) shows the two possible placements of $\mathcal B$ along $\alpha$, once we assume, without loss of generality, that $\alpha$ is oriented towards the left (in the sense of the figure). 
Let $A$ and $B$ be the points shown (in either case).    
Let $\alpha'$ be the first $a_1$-track   one meets on following $\beta$ to the right (in the sense of the figure) from its intersection with $\alpha$.  (If there is no such $\alpha'$ a simpler version, which we omit, of the following analysis will apply.)  
Then $\mathcal{G}_b$ can have no 
 junction
in the (closed) region bounded by $\alpha$ (on the left), $\alpha'$ (on the right), $\beta$ (below), and a portion of $\partial \Delta$ (above),  as this region has no  $a_1$-tracks.  
Thus there are three possible continuations for the $b$-track at $A$ through this region:  (i) it continues to $\alpha'$ or to $\partial \Delta$ (as shown lower left in Figure~\ref{fig:why_no_buttons}); (ii)  it returns to $\alpha$ \emph{above} the button (as shown lower middle); and (iii) it returns to $\alpha$ \emph{below} the button (as shown lower right).  In case (iii),  the $b$-track at $B$ must return to $\alpha$ \emph{below} the button also  (as otherwise there would be a junction). In all cases (i)--(iii) there  is a region $R$ (shown shaded in the figure) with boundary made up of an inward-oriented $b$-subtrack, $a_1$-subtracks, and (in   case (i)) a portion of the $w$-part of $\partial \Delta$,  contrary to  \eqref{lem part:no sinks} of this lemma.

\bs
\eqref{lem part:no loops and bigons}. 
In light of \eqref{lem part: no button}, Lemma~\ref{lem: bigons}\eqref{lem part: badge button} 
and Corollary~\ref{cor: no loops}\eqref{lem part: no b-loops} preclude bigons comprised of 
an $a$-subtrack and an outward oriented $b$-subtrack
 and $b$-loops respectively.   
 Corollary~\ref{cor: no loops}\eqref{lem part: loops absent from reduced diagrams} precludes inward
oriented $a$-loops. 
Suppose, for a contradiction, that there exists a non-trivial  $a$-loop  $\alpha$.  Then the region $R$ enclosed by $\alpha$ cannot contain a $b$-subtrack, as such a subtrack would give rise to a teardrop, a $b$-loop, or a bigon comprised of an outward oriented $b$-subtrack and an $a$-subtrack, all of which have been ruled out.  It follows that the minimal subdiagram containing $R$ contains only cells of type $r_{4, \ast}$ (as any other cells with $a$-letters would introduce $b$-subtracks), which contradicts Lemma~\ref{lem: no junctions no loops}\eqref{lem part: r4 diagrams}.

\bs
\eqref{lem part:no sources}.   Suppose, for a contradiction, that $R$ is a region of $\Delta$  whose boundary is comprised of $a$-subtracks and outward-oriented $b$-subtracks.     We may assume that no $a$- or $b$-track intersects the interior of $R$, because such a track would subdivide $R$ into two regions, at least one of which would satisfy the hypotheses of \eqref{lem part:no sources}.

 By \eqref{lem part:no loops and bigons},
 $\partial R$ cannot be an  $a$- or $b$-loop or a bigon comprised of an $a$-track and an outward-oriented $b$-track.   
 Any two  adjacent   $b$-subtracks in the circuit $\partial R$ are together  a single $b$-subtrack.  
As the $a$-subtracks in $\partial R$ are consistently oriented,  the same is true for $a$-subtracks. 
   So, $\partial R$ is a concatenation of non-trivial paths $\overline{\alpha}_1$, $\overline{\beta}_1$, \ldots, $\overline{\alpha}_m$, $\overline{\beta}_m$ where $m \geq 2$ and each  $\overline{\alpha}_i$ is a subtrack of some $a$-track $\alpha_i$, each  $\overline{\beta}_i$ is a subtrack of some $b$-track $\beta_i$ and the $\overline{\beta}_i$ are all oriented out of $R$.

  As $\overline{\beta}_1$ is oriented out of $R$ and its continuation $\beta_1$ is oriented toward $\chi$ 
 (by \eqref{lem part:orientations}), $R$ is in the $w$-side of $\beta_1$.   Now, because $\beta_2$ is also oriented toward $\chi$, and because the interior of $R$ has no $b$-subtracks, $\beta_2$ must merge with $\beta_1$ either to the left or right of $R$, as shown in Figure~\ref{fig:merging}.  
 Then some  subtrack of $\beta_2$  bounds a region $R'$ either with $\overline{\alpha}_2$  (per Figure~\ref{fig:merging}, left)   or with  the concatenation $\overline{\alpha}_3 \overline{\beta}_3 \cdots \overline{\alpha}_m \overline{\beta}_m \overline{\alpha}_1$ (per Figure~\ref{fig:merging}, right).  
 In the latter case, the extension $\alpha_1$ of $\overline{\alpha}_1$ cannot enter $R$ (as $R$ contains no $a$-subtracks) so must meet the part of $\beta_2$ in $\partial R'$ 
 (after possibly passing through some other $\overline \alpha_i$'s for $3 \le i \le m$).   
In either case we get a bigon $B$ bounded by an  $a$-subtrack and an outward-oriented  $b$-subtrack, contrary to  \eqref{lem part:no loops and bigons}.
 
 \begin{figure}[htbp]
 \centering
\begin{overpic}
{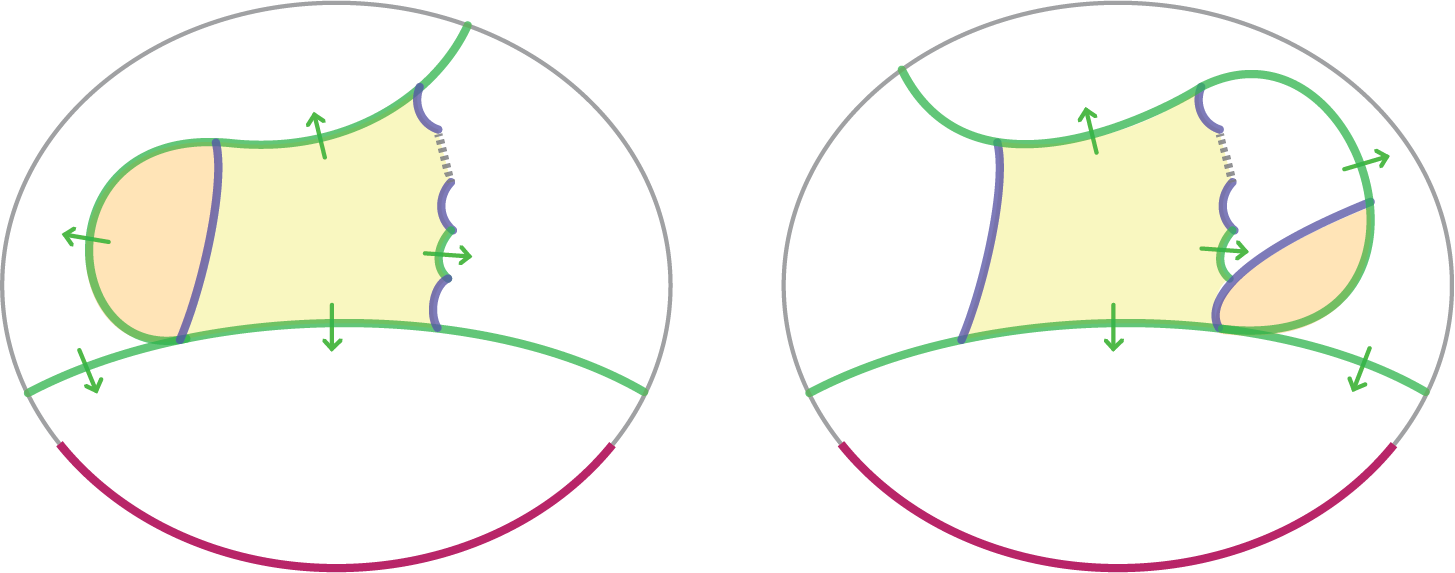}
\put(23, -2){\small{$\chi$}}
\put(76, -2){\small{$\chi$}}
\put(22, 23){\small{$R$}}
\put(75, 23){\small{$R$}}
\put(7, 24){\tiny{$R'=B$}}
\put(88, 19){\small{$B$}}
\put(87, 28){\small{$R'$}}
\put(24, 13.5){\tiny{$\overline{\beta}_1$}}
\put(22.5, 32.5){\tiny{$\overline{\beta}_2$}}
\put(33, 21){\tiny{$\overline{\beta}_m$}}
\put(78, 13.5){\tiny{$\overline{\beta}_1$}}
\put(75.5, 32.5){\tiny{$\overline{\beta}_2$}}
\put(80, 21){\tiny{$\overline{\beta}_m$}}
\put(31, 17.5){\tiny{$\overline{\alpha}_1$}}
\put(15.5, 22){\tiny{$\overline{\alpha}_2$}}
\put(30, 31.5){\tiny{$\overline{\alpha}_3$}}
\put(32, 25){\tiny{$\overline{\alpha}_m$}}
\put(80.5, 18){\tiny{$\overline{\alpha}_1$}}
\put(64.5, 22){\tiny{$\overline{\alpha}_2$}}
\put(83.5, 31.5){\tiny{$\overline{\alpha}_3$}}
\put(85, 24.5){\tiny{$\overline{\alpha}_m$}}
\end{overpic}
\caption{Illustrating our proof of Lemma~\ref{lem: Layout lemma}\eqref{lem part:no sources}  }
\label{fig:merging}
\end{figure}
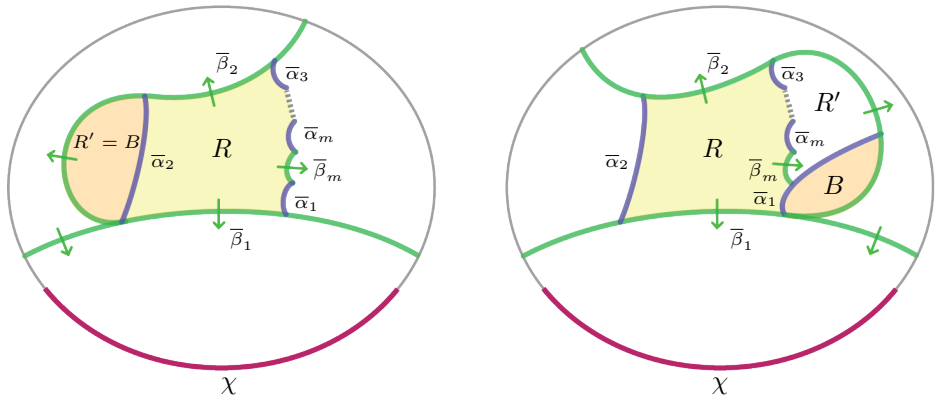

\bs 
\eqref{lem part:crossing possibilities}.   We will use  Lemma~\ref{lem: intersections} with 
$\{\tau, \sigma\} = \{\alpha, \beta\}$.
Lemma~\ref{lem: trapped y-noise}\eqref{lem part: no y-edges cor} tells us that there is no region in $\Delta$ that is bounded by inward-oriented $a$-and $b$-subtracks, which establishes the no-sink-regions hypothesis of Lemma~\ref{lem: intersections}.     

The case \eqref{lem part: both sides}  corresponds to  case~(1) of Lemma~\ref{lem: intersections}  with $\tau = \alpha$ and $\sigma = \beta$.  
By \eqref{lem part:orientations} of the present lemma,  $\alpha$ and $\beta$ are oriented towards $\chi$.  So \eqref{lem part: chi side} corresponds to either case~(2) or case~(3) of Lemma~\ref{lem: intersections} with $\tau = \alpha$ and $\sigma = \beta$,   
and \eqref{lem part: w side}   concerns  case~(3) with $\tau = \beta$ and $\sigma = \alpha$.  With just one exception, \eqref{lem part:no loops and bigons} of the present lemma (specifically the part concerning bigons)  rules out all the intersection patterns catalogued in Lemma~\ref{lem: intersections} apart from those listed in the conclusion of \eqref{lem part:crossing possibilities}.  That one exception occurs in  \eqref{lem part: w side}, where we need to further exclude the possibility that $\alpha$ and $\beta$ do not cross, which we do by invoking \eqref{lem part:no sinks} of this lemma.

\bs
\eqref{lem part: b_0-track outermost}.  
 Suppose there is a $b_0$-track $\beta_0$ in the $w$-side of a $b$-track $\beta$.   
  If $C$ is a 2-cell dual to $\beta_0$, then $C$ has a $y$-edge, so cannot be on the $w$-side of $\beta$ by \eqref{lem part: No y-edges above a b-track}.  Thus $C$ is dual to $\beta$ as well, and is an $r_{\ast, 0}$-cell.  Then $\beta$ agrees with $\beta_0$ on $C$. (This is clear if $\ast$ is 2 or 3.  If $C$ has type $r_{4, 1}$, it follows from the fact that $\beta$ and $\beta_0$ are oriented towards $\chi$ by \eqref{lem part:orientations}.)  
Consequently, $\beta_0 = \beta$. 
\end{proof}

\subsection{$(a_2,b_q)$-tracks} \label{sec: a2bq tracks}
 
 A key idea leading to the ``$p/q$'' in the subgroup distortion function of Theorem~\ref{main} is that the generation of  $b_p$ letters  within distortion diagrams is offset by generation of letters $b_q$ that must ``appear'' in $w$ either as $b_q$-letters or in the guise of $a_2$-letters.  The reason for this is that $b_q$ letters feature in \emph{$(a_2,b_q)$-tracks}, which are the subject of this section and  will be crucial to our proof of Lemma~\ref{lem:p/q}.

\begin{definition}  \textbf{($(a_2,b_q)$-tracks)} \label{def: a2bq-track} 
An \emph{$(a_2,b_q)$-track} in a   van Kampen diagram $\Delta$ over our presentation $\mathcal{P}$ for $G$ is a maximal path that is a concatenation of edges dual to  consistently oriented $a_2$-edges and $b_q$-edges in $\Delta$, such that an $(a_2,b_q)$-track entering a 2-cell of the form shown rightmost in Figure~\ref{fig:a_2bq_tracks} across an $a_2$-edge leaves across the consistently oriented $b_q$-edge. The two $(a_2,b_q)$-tracks in the 2-cell shown rightmost in Figure~\ref{fig:a_2bq_tracks}  touch, but we do not consider them to intersect.  Examples are shown in Figures~\ref{fig:QDiagram} and \ref{fig:G2Diagram}.
\end{definition}

\begin{lemma} \label{lem:(a_2,b_q)-tracks}
$(a_2,b_q)$-tracks in a   van Kampen diagram $\Delta$ have the  following properties:
\begin{enumerate}
\item \label{zeroth a2bq property} $(a_2, b_q)$-tracks inherit orientations from the orientations of their constituent subtracks.  
	\item \label{first a2bq property} Every $a_2$-edge and $b_q$-edge in $\Delta$ is dual to an edge in exactly one $(a_2,b_q)$-track.
\item \label{second a2bq property} An $(a_2,b_q)$-track cannot intersect itself or another $(a_2,b_q)$-track.   
\item \label{third a2bq property} The set of $a_2$- and $b_q$-edges in $\partial \Delta$ are paired off according to whether there is   an $(a_2,b_q)$-track whose first and last edges are dual to them.
\item \label{lem part: a2bq not a loop} If $\Delta$ is a distortion diagram 
as constructed in Lemma~\ref{lem: Layout lemma}, 
then an $(a_2,b_q)$-track in $\Delta$ cannot be a loop.  
\end{enumerate}
\end{lemma}

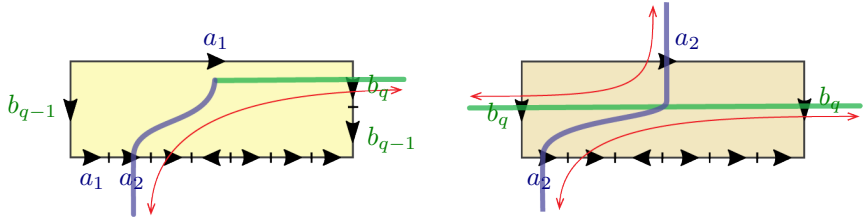
\begin{figure}[htbp]
\centering
\begin{overpic} 
{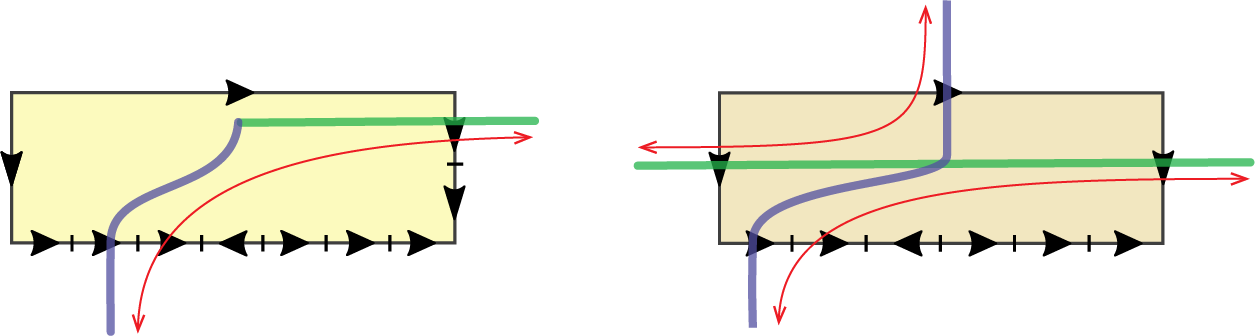}
\put(2, 4){\blue{\small{$a_1$}}}
\put(7, 4){\blue{\small{$a_2$}}}
\put(17.5, 21.5){\blue{\small{$a_1$}}}
\put(-7, 13){\green{\small{$b_{q-1}$}}}
\put(38, 9){\green{\small{$b_{q-1}$}}}
\put(38, 15.5){\green{\small{$b_{q}$}}}
\put(53, 12){\green{\small{$b_{q}$}}}
\put(94.5, 14){\green{\small{$b_{q}$}}}
\put(58, 4){\blue{\small{$a_2$}}}
\put(76.5, 21.5){\blue{\small{$a_2$}}}
\end{overpic}
\caption{\emph{How $(a_2,b_q)$-tracks progress through $r_{1,q-1}$- and $r_{2,q}$-cells}}
\label{fig:a_2bq_tracks}
\end{figure}  
 
\begin{proof}
 \eqref{zeroth a2bq property} holds because constituent  subtracks are consistently oriented edges by construction.

With the sole exception of  $r_{2, q}$ (shown rightmost in Figure~\ref{fig:a_2bq_tracks}), all our defining relators contain either none of the letters $a_2$, $a_2^{-1}$, $b_q$,  and $b_q^{-1}$, or contain exactly one of $a_2$ and $b_q$, and exactly one of $a_2^{-1}$ and $b_q^{-1}$. So \eqref{first a2bq property}--\eqref{third a2bq property} follow.  For \eqref{lem part: a2bq not a loop}, suppose there is a  $(a_2,b_q)$-loop in a distortion diagram $\Delta$. As the orientations of its constituent subtracks are consistent, it is either inward- or outward-oriented.  The former is impossible by Lemma~\ref{lem: trapped y-noise}\eqref{lem part: no y-edges cor}  and the latter by Lemma~\ref{lem: Layout lemma}\eqref{lem part:no sources}.
\end{proof}

Given $\Delta$ as per Lemma~\ref{lem: Layout lemma}, its $b_0$-tracks $\beta_1, \dots, \beta_m$ must
be arranged consecutively around $\partial \Delta$ as  per Figure~\ref{fig:a_2bq_and_b0_tracks} (since they cannot nest by Lemma~\ref{lem: Layout lemma}\eqref{lem part: b_0-track outermost}).
In short, our next lemma states that the intersections of an $(a_2,b_q)$-track with the $b_0$-tracks in $\Delta$ progress in order around the diagram.  We will use it in our proof of Proposition~\ref{prop:upper bound} at the end of Section~\ref{sec:upper}.

\begin{lemma} \label{lem: a2bq not nesting}
Suppose $\Delta$ is a distortion diagram for $w \chi^{-1}$ as per 
Lemma~\ref{lem: Layout lemma}.  Let $Q_0$ and $P_{m+1}$ be the initial and terminal vertices of the $w$ portion of $\partial \Delta$.   For distinct points $P$ and $Q$ on $w$, write $P < Q$ when one reaches $P$ first when following $w$ from $Q_0$ to $P_{m+1}$.    Suppose, as shown in  Figure~\ref{fig:a_2bq_and_b0_tracks}, $P_1 < Q_1 <  \cdots < P_m < Q_m$ are $2m$ successive points on the $w$-portion of $\partial \Delta$ and, for $i=1, \ldots, m$, $\beta_i$ is a $b_0$-track from $P_i$ to $Q_i$ oriented towards $\chi$.  Let $R$ be the maximal region of $\Delta$ that is bounded by $\beta_1$, \ldots, $\beta_m$ and the intervening subpaths of $\partial \Delta$.

Suppose $\tau$ is an $(a_2,b_q)$-track  in $\Delta$ starting at some $P$ and ending at some $Q$ in $\partial \Delta$, with $P<Q$. Let $\Sigma$ be the set of points where $\tau$ meets $\partial R$.  The order   in which $\tau$ visits the points of $\Sigma$ as it progresses from $P$ to $Q$ is the same as the order in which they occur on the boundary circuit $\partial R$ starting from $Q_0$ and following it around to $P_{m+1}$.            
\end{lemma}

\begin{figure}[htbp]
\centering
\begin{overpic} 
{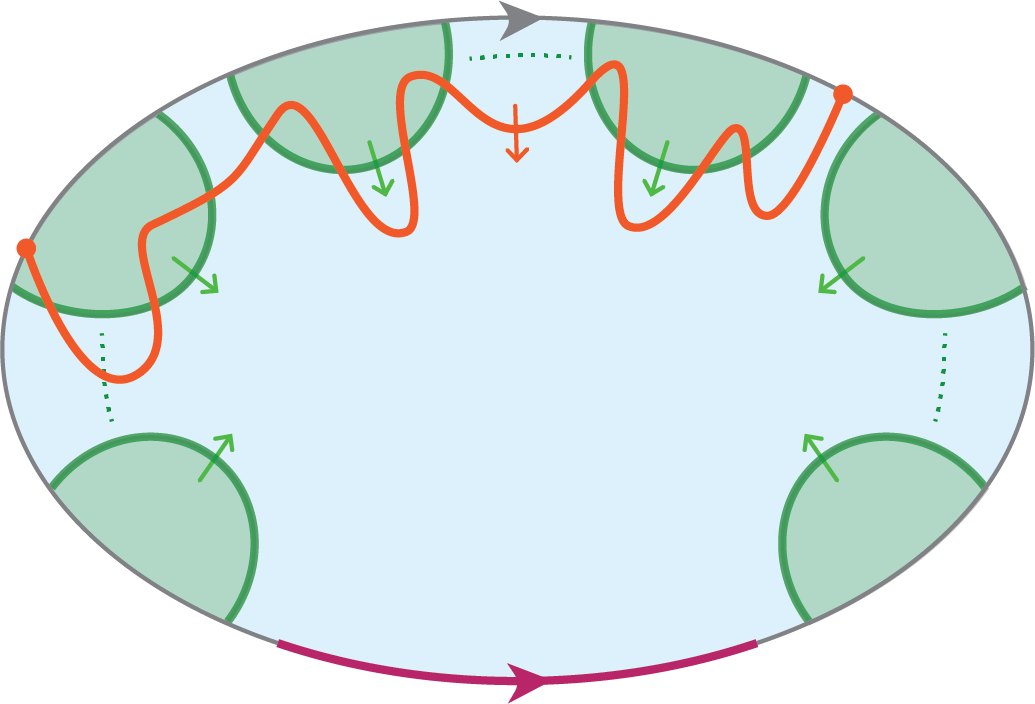}
\put(49, -2){\small{$\chi$}}
\put(49, 68.5){\small{$w$}}
\put(49, 49){\orange{\small{$\tau$}}}
\put(49, 29){\small{$R$}}
\put(-1, 44.5){\orange{\small{$P$}}}
\put(82, 60){\orange{\small{$Q$}}}
\put(23, 23){\green{\small{$\beta_1$}}}
\put(17, 36){\green{\small{$\beta_i$}}}
\put(26.5, 48){\green{\small{$\beta_{i+1}$}}}
\put(55, 54){\green{\small{$\beta_j$}}}
\put(80, 35.5){\green{\small{$\beta_{j+1}$}}}
\put(71, 22){\green{\small{$\beta_m$}}}
\put(18.5, 3.5){\green{\small{$P_1$}}}
\put(-4, 39){\green{\small{$P_i$}}}
\put(19, 64){\green{\small{$P_{i+1}$}}}
\put(57, 67.5){\green{\small{$P_j$}}}
\put(86.5, 56.5){\green{\small{$P_{j+1}$}}}
\put(95, 17){\green{\small{$P_m$}}}
\put(70, 1){\green{\small{$P_{m+1}$}}}
\put(25, 1.5){\green{\small{$Q_0$}}}
\put(2, 16){\green{\small{$Q_1$}}}
\put(10.5, 58.5){\green{\small{$Q_i$}}}
\put(39, 68){\green{\small{$Q_{i+1}$}}}
\put(76.5, 62.5){\green{\small{$Q_j$}}}
\put(100, 40){\green{\small{$Q_{j+1}$}}}
\put(78, 4){\green{\small{$Q_m$}}}
 \end{overpic}
\caption{Illustrating Lemma~\ref{lem: a2bq not nesting}}
\label{fig:a_2bq_and_b0_tracks}
\end{figure}

\begin{proof} As its constituent $a_2$- and $b_q$-subtracks are, by construction, consistently oriented, $\tau$ is a compound track which is oriented either towards or away from $\chi$.  The latter eventuality is precluded by Lemma~\ref{lem: Layout lemma}\eqref{lem part:orientations}.

 The lemma will be proved by applying either Lemma~\ref{lem: intersections}
 or Corollary~\ref{cor: intersections} to pairs consisting of  $\tau$ (or a subpath thereof) and $\beta_l$, for each $l$.   
 
Let $i, j \in \set{0, \ldots, m+1}$ be such that $Q_{i-1} < P < Q_i$ and  $P_{j} < Q < P_{j+1}$. 
 By Lemma~\ref{lem: Layout lemma}\eqref{lem part:no sources}, for all $\ell$, there is no source-region bounded by subtracks of $\tau$ and  $\beta_{\ell}$. 
If $\ell <i$ or $\ell >j$, then the orientations of $\beta_\ell$ and $\tau$ near $\partial \Delta$ are as shown in  Figure~\ref{fig: intersections example}(right),  so $\beta_\ell$ and $\tau$ cannot intersect by Corollary~\ref{cor: intersections}.

Consider traveling along $\tau$ from $P$ to $Q$.  
If $\tau$ intersects $\beta_k$ for some $k$, then $\tau$ cannot
intersect any   $\beta_{\ell}$ with $\ell <k$.    This is because were there such an $\ell$, there would be a subpath $\hat{\tau}$ of $\tau$ that connects a pair of points on $\beta_k \cup \set{Q}$ 
and intersects $\beta_\ell$. 
However, in the disc obtained from~$\Delta$ by excising the $w$-side of $\beta_k$, the orientations on $\hat{\tau}$ and $\beta_\ell$ are as shown in Figure~\ref{fig: intersections example}(right), so this intersection is contrary to Corollary~\ref{cor: intersections}.

So $\tau$ intersects none of  $\beta_1$, \ldots, $\beta_{i-1}$, $\beta_{j+1}$, \ldots, $\beta_m$ and, 
proceeding from $P$, it intersects $\beta_i, \beta_{i+1}, \dots, \beta_j$ in order (intersecting each some number of times, possibly zero). 
If $P_i < P < Q_i$, then  how $\tau$ intersects $\beta_i$ is described by case (1) of Lemma~\ref{lem: intersections}.  The other possibility is that $P < P_i$,   which is handled by case (3).  Case (3) likewise describes how $\tau$ intersects $\beta_{i+1}$, \ldots, $\beta_{j-1}$, and case (1) or (3)  how  $\tau$ intersects $\beta_j$.    
These observations combine to prove the result.   
\end{proof}

\section{The upper bound}\label{ch:upper} 

\subsection{Reduction to a free-by-cyclic quotient}\label{sec:upper}

Modulo calculations we will postpone to Section~\ref{sec:p/q}, we will  prove here: 

\begin{prop} \label{prop:upper bound} For $\chi$, $w$ and $\Delta$ as per Lemma~\ref{lem: Layout lemma}, there exists a constant $K > 1$, depending only on our presentation $\mathcal{P}$ for $G$, such that   
\begin{equation} \label{eqn:chi bound} 
 |\chi|  \ \leq  \ K^{{|w|}^{p/q}}. 
 \end{equation}        
\end{prop}

As a corollary, we obtain the desired upper bound on distortion: 

\begin{cor}	 \label{cor:upper bound}
$\mathrm{Dist}_H^G(n) \preceq \exp({n^{p/q}}).$
\end{cor}

\begin{proof}[Proof of Corollary~\ref{cor:upper bound}, assuming  Proposition~\ref{prop:upper bound}]   Suppose $n \geq 0$.  Let $\chi$ be a reduced word on the generators of $H$
 which realizes the distortion function of $H$, i.e.:
\begin{equation} \label{eqn:Dist def} \mathrm{Dist}_H^G(n) \ = \  |\chi|.\end{equation}
More precisely, $\chi$ is a maximal length reduced word on the generators of $H$  that equals,  in $G$, some word $w_0$ of length at most $n$.  We can assume $w_0$ has no subwords representing the identity in $G$.  

 Let $\Delta_0$ be a reduced van~Kampen  diagram for $w_0 \chi^{-1}$.  
If $\Delta_0$ is homeomorphic to a $2$-disc, then Lemma~\ref{lem: Layout lemma} and hence Proposition~\ref{prop:upper bound} apply, yielding $w$ such that   $|\chi|  \ \leq  \ K^{{|w|}^{p/q}} $ and $|w| \leq C |w_0|$.  This, combined with 
\eqref{eqn:Dist def} and  $|w_0| \leq n$ gives the result.

Now suppose that $\Delta_0$ is not a 2-disc.  Our choice of $w_0$ guarantees that no two vertices along the part of $\partial \Delta_0$ labelled $w_0$ are identified.  The same holds for $\chi$, as it is reduced. 
It follows that $w_0$ and $\chi$ are concatenations of subwords $w_1, w_2, \dots, w_r$ and $\chi_1, \chi_2, \dots, \chi_r$ respectively, such that for each~$i$, either $w_i = \chi_i$ and the paths with these labels along $\partial \Delta_0$ are identified, or there is a (reduced) subdiagram $\Delta_i$ of $\Delta_0$ homeomorphic to a 2-disc whose boundary reads $w_i \chi_i^{-1}$.  
In either case, we have $\chi_i \ \leq  \ K^{{|w_i|}^{p/q}} $, and 
the bound we require follows from
the superadditivity of the function $n \mapsto \exp({n^{p/q}})$.   
  \end{proof}

Let $\chi$, $w$, and $\Delta$ be as per Lemma~\ref{lem: Layout lemma}.  To prove Proposition~\ref{prop:upper bound}, we will decompose $\Delta$ into the subdiagrams we now define.  

\begin{definition}\label{def:b-block}
(\textbf{Decomposing a distortion diagram into $b$-blocks and an $a$-block.})
Given a $b$-track $\beta$ in $\Delta$, define $\Delta_{\beta}$ to be the minimal subdiagram of $\Delta$ containing the 
$w$-side of $\beta$ (see Definition~\ref{def: distortion diagram}). 
So $\Delta_\beta$ is comprised of all the 2-cells of $\Delta$ that  either have $\beta$ passing through them or are in the $w$-side of $\beta$. Say that $\beta$ is \emph{outermost} when  there is no $b$-track $\beta'$ such that  
$\Delta_{\beta'}$ properly contains $\Delta_{\beta}$.  The $\Delta_{\beta}$ such that $\beta$ is outermost are the \emph{$b$-blocks} of $\Delta$.            

Let $\mathcal{B}_1$, \ldots, $\mathcal{B}_r$ be the $b$-blocks  of $\Delta$ as per Figure~\ref{fig:blocks} (when $r=3$).
Define the \emph{$a$-block} $\mathcal{A}$ of $\Delta$ to be the maximal subdiagram of $\Delta$ that contains $\chi$ and intersects no $b$-tracks.  So $\mathcal{A}$ is obtained from $\Delta$ by severing $\mathcal{B}_1$, \ldots, $\mathcal{B}_r$.   
\end{definition}

\begin{cor}\label{cor:block cells}
${}$ \label{cor: blocks}  For $\mathcal A$ and $\mathcal{B}_1$, \ldots, $\mathcal{B}_r$ as   defined above--- 
\begin{enumerate}
\item  \label{a-block part}
$\mathcal{A}$ is a subdiagram of $\Delta$ whose 2-cells are of type $r_{4, \ast, \ast}$   and $r_{4,\ast}$ (per Figure~\ref{fig:relations}).  

\item \label{b-blocks part} 
$\mathcal{B}_1$, \ldots, $\mathcal{B}_r$ are subdiagrams of $\Delta$ whose 2-cells are of type $r_{1, \ast}$, $r_{2, \ast}$, $r_{3, \ast}$, and $r_{3, \ast, \ast}$. 
\item \label{b_0 corridor part}
For all $i$, there exists $j_i$ such that the outermost $b$-track $\beta_i$ of $\mathcal B_i$ is a $b_{j_i}$-track.  It is oriented towards $\chi$ and the cells of $\Delta$ that it traverses comprise a  $b_{j_i}$-corridor in $\mathcal{B}_i$ whose top boundary (the boundary the $b_{j_i}$-edges are oriented towards) follows $\mathcal{A} \cap \mathcal{B}_i$.  If $j_i =0$, then this is the only $b_0$-corridor in $\mathcal{B}_i$. 
\end{enumerate}
\end{cor}
\begin{proof}
Lemma~\ref{lem: Layout lemma}\eqref{lem part: No y-edges above a b-track} implies that the $b$-blocks contain   no $r_{4, \ast}$- or $r_{4, \ast, \ast}$-cells. 
Statements~(1) and (2) are then consequences of the definitions of the $a$- and $b$-blocks.   
Lemma~\ref{lem: Layout lemma}(\ref{lem part:orientations}) tells us that every $b$-track is oriented towards $\chi$.  Part (3) then follows, except we also invoke     Lemma~\ref{lem: Layout lemma}\eqref{lem part: b_0-track outermost} for its final claim.  
\end{proof}

 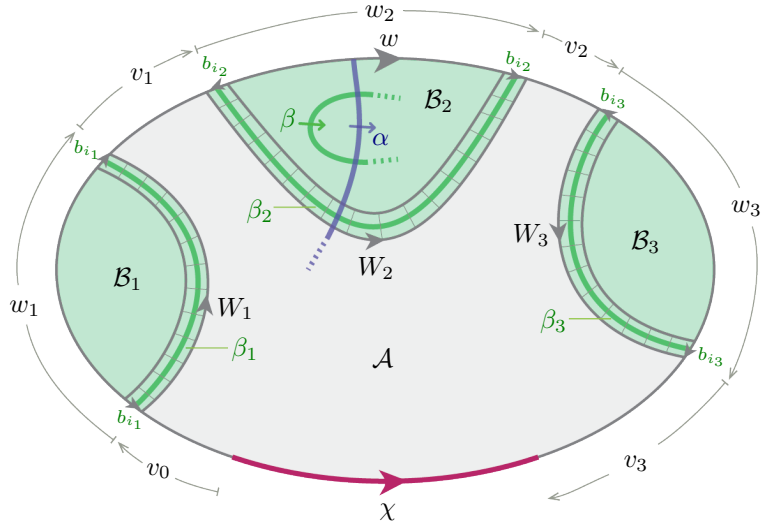
\begin{figure}[htbp]
\centering
\begin{overpic} 
{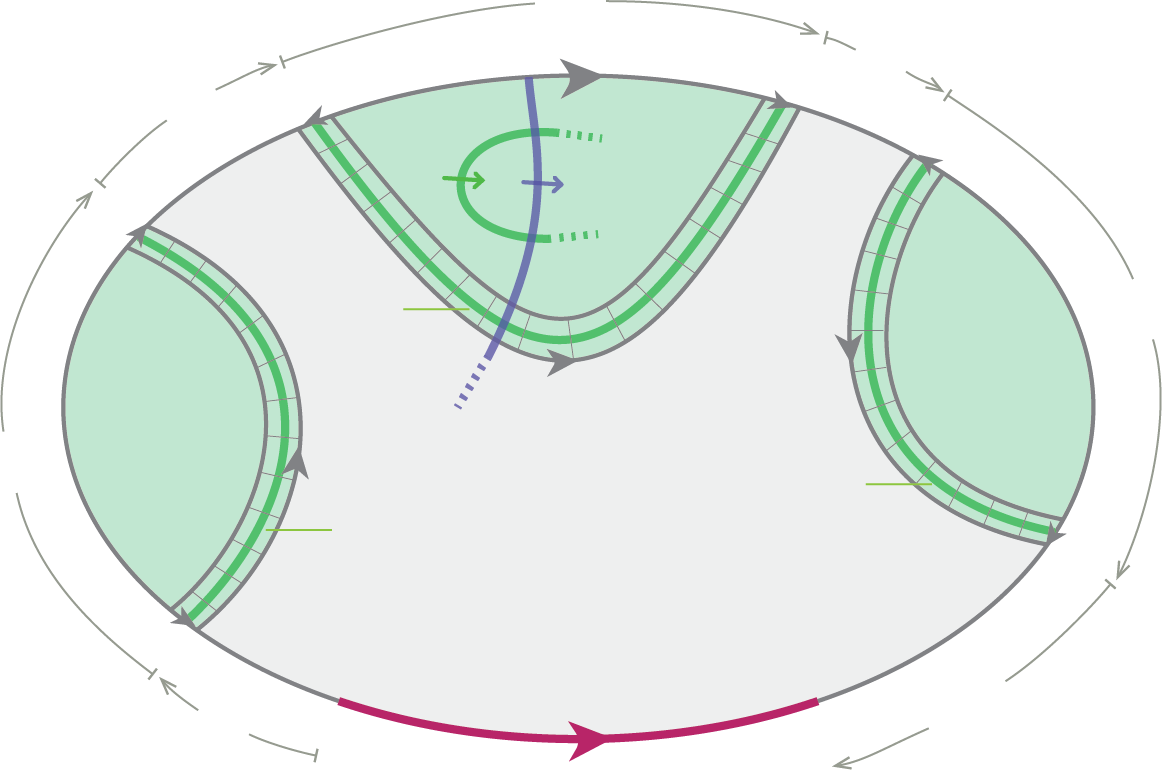}
\put(49, -2){\small{$\chi$}}
\put(49, 61.5){\small{$w$}}
\put(13.5, 10.5){\green{\tiny{$b_{i_1}$}}}
\put(7.7, 47){\green{\tiny{$b_{i_1}$}}}
\put(25, 58.3){\green{\tiny{$b_{i_2}$}}}
\put(66, 59){\green{\tiny{$b_{i_2}$}}}
\put(78.8, 54){\green{\tiny{$b_{i_3}$}}}
\put(92, 19){\green{\tiny{$b_{i_3}$}}}
\put(27, 25){\small{$W_1$}}
\put(46, 30.5){\small{$W_2$}}
\put(67, 35){\small{$W_3$}}
\put(48, 48){\blue{\small{$\alpha$}}}
\put(35.5, 50.5){\green{\small{$\beta$}}}
\put(29, 19.5){\green{\small{$\beta_1$}}}
\put(31, 38.2){\green{\small{$\beta_2$}}}
\put(70.7, 23.3){\green{\small{$\beta_3$}}}
\put(48, 18){\small{$\mathcal{A}$}}
\put(13, 29){\small{$\mathcal{B}_1$}}
\put(55, 53){\small{$\mathcal{B}_2$}}
\put(83, 34){\small{$\mathcal{B}_3$}}
\put(-1, 25.5){\small{$w_1$}}
\put(47, 65.4){\small{$w_2$}}
\put(96.5, 39){\small{$w_3$}}
\put(17.5, 3.5){\small{$v_0$}}
\put(15.3, 56.8){\small{$v_1$}}
\put(74, 60.5){\small{$v_2$}}
\put(82, 4.5){\small{$v_3$}}
\end{overpic}
\caption{The $a$-block and $b$-blocks in $\Delta$. The $a_1$-track   $\alpha$ and $b$-track $\beta$ illustrate a case in the proof of Lemma~\ref{lem:a-corridor words}. }
\label{fig:blocks}
\end{figure}

Express $w$ as the concatenation of words $$w \ = \ v_0 w_1 v_1 w_2 \cdots w_r v_r$$ where, for all $i$, $w_i$ is the word along $\partial \mathcal{B}_i \cap \partial \Delta$ as shown in  Figure~\ref{fig:blocks} and the $v_i$ are the 
(possibly empty) intervening subwords. 
Per Corollary~\ref{cor:block cells}\eqref{b_0 corridor part}, each $w_i$ has first letter $b_{j_i}^{-1}$ and final letter $b_{j_i}$.
For all $i$, let $W_i$ be the word along the other side of $\mathcal{B}_i$, so that $\mathcal{B}_i$  is a van~Kampen diagram for $w_i W_i^{-1}$. Let 
\begin{equation} \label{eq:W}
W \ = \ v_0 W_1 v_1 W_2 \cdots W_r v_r.\end{equation}
 So $\mathcal{A}$ is a van~Kampen diagram for $W \chi^{-1}$.

\bigskip

In the following lemmas we analyze the structure of a $b$-block $\mathcal B_i$ in $\Delta$.  When $\beta_i$ is a $b_0$-track, this will lead  (in Lemma~\ref{lem:Lengths of W_i}) to  an upper bound on the length of $W_i$.

\begin{lemma} \label{lem:b-block corridors}
Let $\mathcal B_i$ be a $b$-block of $\Delta$, and let $w_i$ and $W_i$ be as above.    Then 
every $a_1$-track in $\mathcal B_i$ runs from an $a_1^{\pm 1}$ in $w_i$ to   an $a_1^{\pm 1}$ in $W_i$.  
\end{lemma}

\begin{proof}
Let $\alpha$ be an $a_1$-track of $\Delta$ intersecting $\mathcal B_i$. It cannot be a loop by Lemma~\ref{lem: Layout lemma}\eqref{lem part:no loops and bigons}.   
It must have at least one endpoint in the $w$-side of $\beta_i$ by  Lemma~\ref{lem: Layout lemma}(\ref{lem part:crossing possibilities}a).   If  it has one endpoint on each side of $\beta_i$, then it intersects  $\beta_i$ exactly once by Lemma~\ref{lem: Layout lemma}(\ref{lem part:crossing possibilities}b), and so corresponds to a single $a_1$-track of $\mathcal B_i$ running from $w_i$ to $W_i$. If it has both endpoints in the $w$-side of $\beta_i$, then, by Lemma~\ref{lem: Layout lemma}(\ref{lem part:crossing possibilities}c),   it intersects $\beta_i$ exactly  twice, giving rise to two $a_1$-tracks in $\mathcal B_i$ both running from $w_i$ to $W_i$.  
\end{proof}

\begin{lemma}\label{lem:a-corridor words} Suppose $\beta_i$ is a $b_0$-track.   Let $\mathcal C$ be an $a_1$-corridor of $\mathcal B_i$.  The bottom boundary of $\mathcal C$ is labelled (in the direction from $w_i$ to $W_i$)  by a word $\lambda b_0$, where $\lambda$ is a  positive word on $b_1, \ldots, b_p$. 
\end{lemma} 
\begin{proof}
By Lemma~\ref{lem:b-block corridors},  $\mathcal C$ has one end in $w_i$ and the other in $W_i$.
By  Corollary~\ref{cor:block cells}\eqref{b-blocks part},\eqref{b_0 corridor part},  
the cells of 
$\mathcal C$ are of type $r_{1, \ast}$ (per 
Figure~\ref{fig:relations}), and only the cell where $\mathcal{C}$ meets $W_i$ has an edge labelled $b_0$, so that the bottom boundary of $\mathcal C$ (in the direction from $w_i$ to $W_i$) is labelled by a word $\lambda b_0$ where $\lambda$ is a word on $b_1^{\pm 1}, \dots, b_p^{\pm 1}$.  We will argue that $\lambda$ is a \emph{positive} word.  
 Suppose, for a contradiction, that $\lambda$ includes a letter $b_j^{-1}$ for some $j$.  
Let $\beta$ be any  $b$-track that has an edge dual to the  edge of $\partial \mathcal C$ labelled by that $b_j^{-1}$. Let $\alpha$ be the $a_1$-track  dual to $\mathcal C$. 
By Lemma~\ref{lem: Layout lemma}(\ref{lem part:orientations}),  $\beta$ is oriented towards $\chi$, and so $\beta$ intersects $\alpha$ at least one more time. So $\alpha$ and $\beta$ form a bigon.  This leads to a contradiction: that bigon violates (\ref{lem part:no sinks}b) or (\ref{lem part:no loops and bigons}) of Lemma~\ref{lem: Layout lemma}, depending on whether $\beta$ is oriented into or out of  the bigon, respectively.    (The (\ref{lem part:no sinks}b) case is illustrated in Figure~\ref{fig:blocks}.)
 We conclude that $\lambda$ is a positive word on $b_1, \dots, b_p$.   
\end{proof}

Our next lemma is illustrated by Figure~\ref{fig:b-block}.

\begin{lemma}\label{lem:clean up bs} Given  $\mathcal B_i, \mathcal C$,  and $\lambda$ as in Lemma~\ref{lem:a-corridor words}, the side of $\mathcal{C}$ labelled by $\lambda b_0$ divides $\mathcal B_i$ into two subdiagrams.  Of these two subdiagrams, let $\Lambda_0$ be that which does not contain $\mathcal{C}$. Its boundary word is $\tilde{\mu} b_0 \nu (\lambda b_0)^{-1}$, where $\nu$ and $\tilde{\mu}^{-1}$ are, respectively, some  prefix  of ($W_i$  or $W_i^{-1}$) and of ($w_i$ or $w_i^{-1}$).  (Which of these pairs it is depends on the orientation of $\mathcal{C}$.  Figure~\ref{fig:b-block} shows the case where they are prefixes of $W_i$ and $w_i$.)  Let $\Lambda_1$ be the maximal subdiagram of  $\Lambda_0$ that contains 
portions of $\partial \Lambda_0$ coming from $\lambda b_0$ and  $\hat{W}_i$,
but intersects no $b$-track in  $\Lambda_1$  that connects a pair of edges in the $\tilde{\mu}$ portion of $\partial \Lambda_0$.   (See Figure~\ref{lem:clean up bs}.)
Let $\hat{\mu}$ be the word such that   $\hat{\mu} b_0 \nu (\lambda b_0)^{-1}$ is the word read around $\partial \Lambda_1$.   Then:   
  
\begin{enumerate}
\item \label{understand the as} 
 The $a_1$-tracks in  $\Lambda_1$ all arise from removing initial subtracks from  $a_1$-tracks in $\Lambda_0$.  In particular, each runs from an $a_1^{\mp 1}$ in $\hat{\mu}$ to an $a_1^{\pm 1}$ in $\nu$, and  the number of $a_1^{\pm 1}$-letters in $\hat{\mu}$ is at most the number  in $\tilde{\mu}$, and therefore at most $|w_i|$.

\item \label{understand the b signs} In  $\hat{\mu}$ there are no letters $b_0^{\pm 1}, b^{-1}_1, \ldots, b^{-1}_p$ and

\item \label{understand the number of bs} There are  at most $|\tilde{\mu}|$ letters $b_1, \ldots, b_p$ in $\hat{\mu}$.

\item \label{sides as before}  The word read along the bottom boundary (in the direction from $\hat{\mu}$ to $\nu$) of a corridor dual to an $a_1$-track in $\Lambda_1$ is a positive word on $b_0, b_1, \ldots, b_p$.  Moreover, it has only one $b_0$, namely its final letter.     
\end{enumerate}
\end{lemma}

   \begin{figure}[htbp]
\centering
\begin{overpic} 
{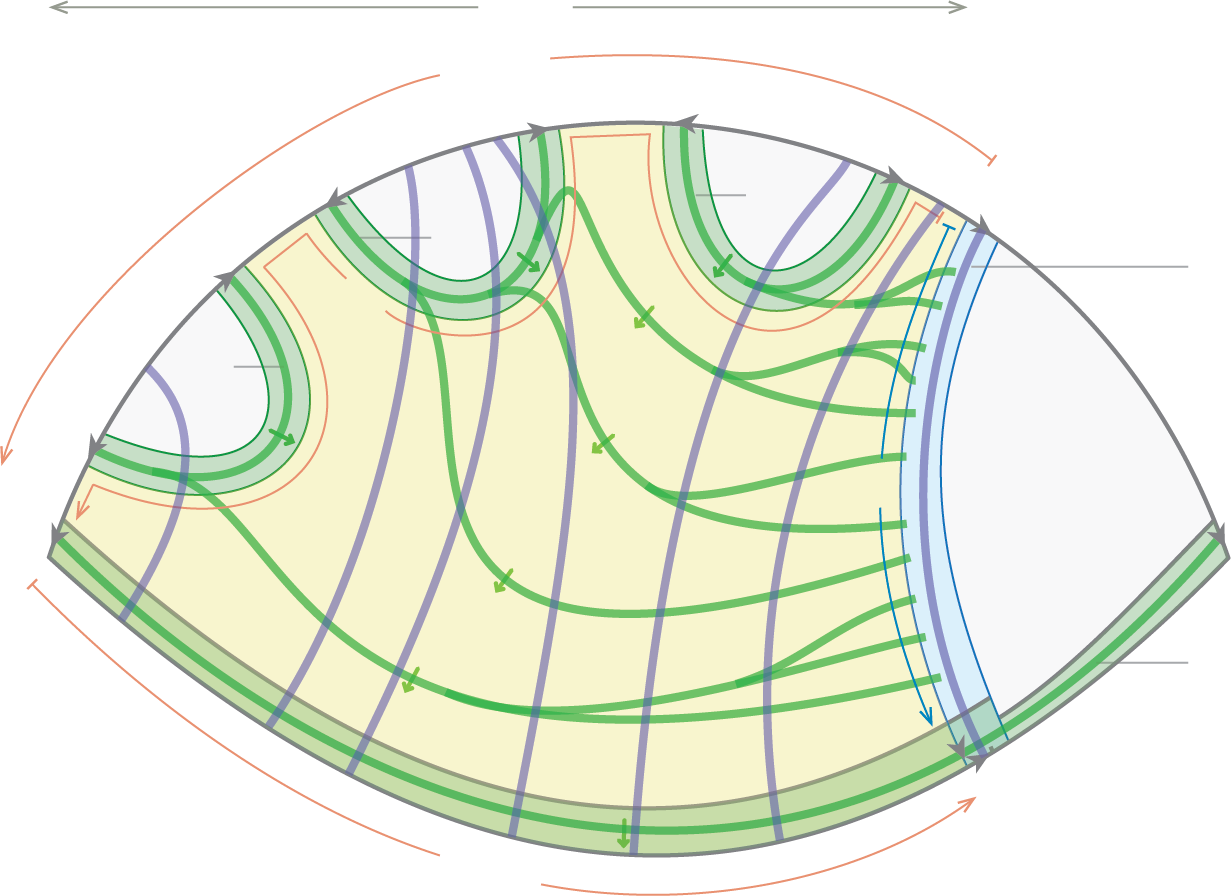}
\put(39, 1){\orange{\small{$\nu$}}}
\put(39, 67){\orange{\small{$\tilde{\mu}$}}}
\put(28.7, 46.9){\orange{\small{$\hat{\mu}$}}}
\put(70.5, 32.3){\bcyan{\small{$\lambda$}}}
\put(1, 29){\green{\small{$b_0$}}}
\put(4, 36){\green{\small{$b_{\ast}$}}}
\put(15.5, 51.5){\green{\small{$b_{\ast}$}}}
\put(25, 58){\green{\small{$b_{\ast}$}}}
\put(42, 63.5){\green{\small{$b_{\ast}$}}}
\put(55, 64){\green{\small{$b_{\ast}$}}}
\put(73, 60){\green{\small{$b_{\ast}$}}}
\put(100, 29){\green{\small{$b_0$}}}
\put(74.5, 10.5){\green{\small{$b_0$}}}
\put(15.9, 41.5){\green{\small{$\hat{\beta}_1$}}}
\put(35.5, 52.5){\green{\small{$\hat{\beta}_2$}}}
\put(60.9, 55.7){\green{\small{$\hat{\beta}_3$}}}
\put(41, 71){{\small{$\Lambda_0$}}}
\put(44, 30){{\small{$\Lambda_1$}}}
\put(98, 49.7){{\small{$\mathcal{C}$}}}
\put(98, 17.8){{\small{$\beta_i$}}}
\put(80.5, 8){\blue{\small{$a_1$}}}
\put(80.5, 55.5){\blue{\small{$a_1$}}}
\put(94, 64){\blue{\small{$\mathcal{B}_i$}}}
\end{overpic}
\caption{Illustrating our proof of Lemma~\ref{lem:clean up bs}}
\label{fig:b-block}
\end{figure}
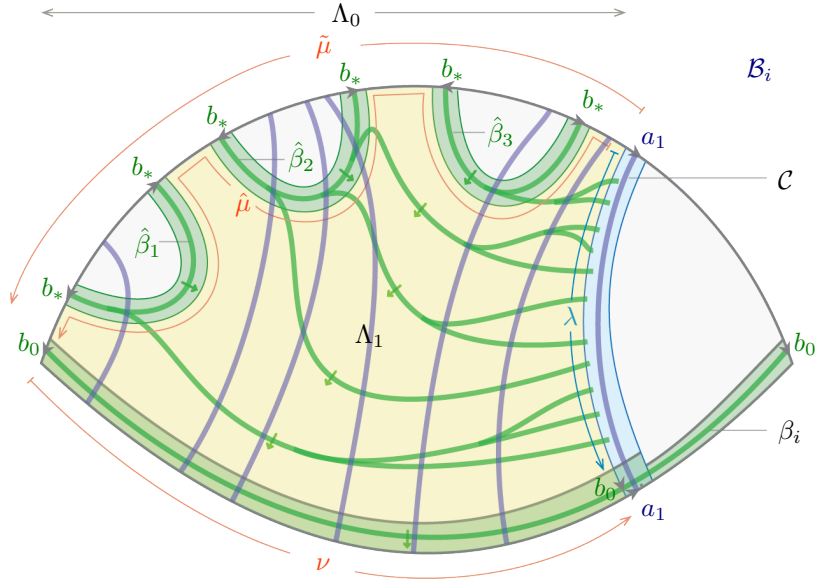

\begin{proof}
There are no letters $b_0^{\pm 1}$ in $\tilde{\mu}$ by construction. If there is a  $b^{-1}_r$ in   $\tilde{\mu}$ for some $1 \leq r \leq p$, then it is connected by a  $b$-track to some  letter $b_r$ labeling an edge in $\partial \Lambda_0$---in fact, that $b_r$ must be in  $\tilde{\mu}$, because there are no    $b^{-1}_r$ letters in $\lambda b_0$ (by Lemma~\ref{lem:a-corridor words}) or in $\nu$ (such are the 2-cells in $b_0$-corridors).  By Lemma~\ref{lem: Layout lemma}(\ref{lem part:orientations}), all such $b$-tracks are oriented towards $\chi$ in $\Delta$, and so towards $\nu$ in $\Lambda_0$.   
 So there are such $b$-tracks  $\hat{\beta}_1, \dots, \hat{\beta}_k$ (in  Figure~\ref{fig:b-block} they are shown with  $k=3$) in $\Lambda_0$ that we might call \emph{outermost}  in that 
 \begin{itemize}
 \item  the $w$-sides of any two of them are disjoint,
 \item every such $b$-track is in the  	$w$-side of one of $\hat{\beta}_1, \dots, \hat{\beta}_k$.
 \end{itemize}
Then $\Lambda_1$ is obtained from $\Lambda_0$ by cutting along the top boundaries of the corridors $C_{\hat{\beta}_1}, \dots, C_{\hat{\beta}_k}$ dual to $\hat{\beta}_1, \dots, \hat{\beta}_k$.

Then  \eqref{understand the as}  follows from Lemma~\ref{lem:b-block corridors} and the observation that, by Lemma~\ref{lem: Layout lemma}(\ref{lem part:no loops and bigons}), no $a_1$-track can cross one of the $\hat{\beta}_j$ twice.

For \eqref{understand the b signs} and \eqref{understand the number of bs}, we examine the $b$-letters in  $\hat{\mu}$. Those that arise  as letters in $\tilde{\mu}$ include no $b_0^{\pm 1}, b_1^{-1}, \ldots, b_p^{-1}$ by construction. Each of the other $b_l^{\pm 1}$ in $\hat{\mu}$ arises  on the top boundary of one of the $C_{\hat{\beta}_j}$ at some 2-cell of type $r_{1,l}$ (per Figure~\ref{fig:relations}) where some other $b$-track  branches off $\hat{\beta}_j$. There are no $b_0$-edges in $\Lambda_1$ except in the $b_0$-corridor abutting $\nu$---for otherwise there would be an additional $b_0$-corridor and therefore a $b_0^{\pm 1}$ in $\tilde{\mu}$ or $\lambda$, which is not so.  So $1 \leq l \leq p-1$.  In fact, the letter cannot be a  $b_l^{-1}$ because then there would be a $b$-track that initially follows  $\hat{\beta}_j$ until branching off into $\Lambda_1$ and eventually terminates back on $\tilde \mu$ (not on $\lambda$ because $\lambda$ is a positive word), so as to contradict $\hat{\beta}_1, \dots, \hat{\beta}_k$  being outermost.  This proves \eqref{understand the b signs}.  Then, for \eqref{understand the number of bs}, observe that each  2-cell  of type $r_{1,\ast}$ in $C_{\hat{\beta}_j}$ has a different $a_1$-track passing through it which, in light of  \eqref{understand the as}, connects  to an $a_1$-edge in $\tilde{\mu}$ between the  between   the endpoints of $\hat{\beta}_j$.     

Finally, Lemma~\ref{lem:a-corridor words} implies \eqref{sides as before}.   
\end{proof}

We will use the conclusions of Lemma~\ref{lem:clean up bs} to further analyze $\lambda$ via calculations in   
\begin{equation}  \label{eq:Q def} Q  \ = \ \langle a_1,  b_0, \ldots, b_p  \mid a_1^{-1}b_i a_1 = \varphi(b_i)  \ \, \forall i  \  \rangle, \  \ \varphi(b_j) \, = \, \begin{cases} b_{j+1}b_j & \text{if } j<p \\  b_j & \text{if } j=p, \end{cases}\end{equation}   
which is a  free-by-cyclic  quotient of $G$ via the map $G \onto Q$  killing 
 $a_2$, $t$, $x_1$, $x_2$, $y_1$, and $y_2$.

Our next simplifying step, in   Lemma~\ref{lem:clean up as}, will dispense with the positive $a_1$-letters from $\hat{\mu}$.  But first, we need two technical results concerning $Q$: 
 
\begin{lemma}\label{lem:w suffix}
Suppose $u$ and $v$ are positive words on $b_0, \dots, b_p$.  Take $\varphi^{-1}(u)$ to   denote the reduced word on $b_0, \dots, b_p$ representing that element of $Q$. Then  $\varphi^{-1}(u)v$ is reduced---that is, there is no cancellation between  $\varphi^{-1}(u)$ and $v$.  In particular, if $w$ is a positive word on $b_0, \dots, b_p$ which equals $\varphi^{-1}(u) v$ in $Q$, then  $v$ is a suffix of $w$.  
\end{lemma}

\begin{proof}
We downwards induct on the minimal index $i$ such that $u$ includes a letter $b_i$.  If $i=p$, the result holds because  $u$ is a power of $b_p$ and $\varphi^{-1} (u ) = u$. For the induction step, write $u$ as the concatenation $u_0u_1$, where $u_0$ ends in $b_i$, and $u_1$ contains no $b_i$.     

It can be checked that for $j=0, \ldots, p$,
\begin{equation*}
\varphi^{-1}(b_j)  \ = \  
\begin{cases}
   b_{j+1}^{-1}  \cdots  b_{p-3}^{-1} b_{p-1}^{-1} b_p \cdots  b_{j+2} b_{j} & \textrm{ when $p-j$ is even,}\\ \\
b_{j+1}^{-1}    \cdots  b_{p-2}^{-1} b_{p}^{-1}   b_{p-1}  \cdots  b_{j+2}  b_{j} & \textrm{ when when $p-j$ is odd,}
\end{cases}
\end{equation*}
which is a reduced word on $b_j, b_{j+1}^{\pm 1}, \ldots, b_p^{\pm 1}$ whose one and only  $b_j$ is its final letter.

So $\varphi^{-1} (u_0)$ has one $i$-letter, its last, and $\varphi^{-1} (u_1)$ has no $b_i$ letters. 
Thus $\varphi^{-1}(u) = \varphi^{-1}(u_0) \varphi^{-1}(u_1)$ as words---there is no cancellation between the two factors.  By the induction hypothesis, there is no cancellation between $\varphi^{-1}(u_1)$ and $v$, so the result follows. 
\end{proof}

\begin{lemma}\label{lem:positive in Q}
If $u$ and $\varphi^{-1}(u)$ are both positive words on $b_0, \ldots, b_p$, then $$|\varphi^{-1}(u)| \  \leq \  |u|.$$   
\end{lemma}

\begin{proof}
For $0  \leq j \leq p$, let $n_j$ and $m_j$ be the number of $b_j$-letters in $u$ and $\varphi^{-1}(u)$, respectively. Then in view of the form of $\varphi^{-1}$ given in the proof of Lemma~\ref{lem:w suffix}, we have 
$$\begin{array}{rl} 
0 \ \leq \ m_0 &  \! = \ n_0,  \   \text{  and so}  \\
0 \ \leq \  m_1  & \! = \ n_1 - n_0 \ \leq \ n_1, \   \text{  and so} \\ 
0 \ \leq  \  m_2  & \! =  \ n_2 - n_1 + n_0 \ \leq \  n_2,  \   \text{  and so on,}
\end{array}$$      	
from which the result follows.
\end{proof}

\begin{lemma} \label{lem:clean up as}
Given  $\lambda$ as in Lemmas~\ref{lem:a-corridor words} and \ref{lem:clean up bs}, there exists a word $\mu$ on $a^{-1}_1, b_1, \dots, b_p$ (so containing no $a_1, b_1^{-1}, \dots, b_p^{-1}$) such that  $|\mu| \leq 2 |w_i|$, and an integer $0 \leq l \leq |w_i|$ such that    in $Q$,
$$\mu b_0 a_1^l \ = \  \lambda b_0.$$   
\end{lemma}

\begin{proof} 
Suppose that $\pmb{\lambda} = (\lambda_0, \ldots, \lambda_l)$,   $\mathbf{u} = (u_0, \ldots, u_l)$, and $\pmb{\epsilon} = (\epsilon_1, \ldots, \epsilon_l)$, where each $\lambda_j$ is a positive word on $b_1, \ldots, b_p$, each $u_j$ is a prefix of $\lambda_j$,   each $\epsilon_i=\pm 1$, and $u_0 = \lambda_0$.  Say that \emph{$\sigma^{-1}  b_0  \tau  = \lambda b_0$ in $Q$ via $(\pmb{\lambda},  \mathbf{u}, \pmb{\epsilon})$} when  
   $$\begin{array}{rl} 
	 \sigma   & =  \  u_{0}^{-1} a_1^{\epsilon_1} u_1^{-1} a_1^{\epsilon_2} \cdots u_{l-1}^{-1} a_1^{\epsilon_l} u_{l}^{-1}    \\
	\tau  & =  \  a_1^{\epsilon_1}a_1^{\epsilon_{2}} \cdots a_1^{\epsilon_l}   \\
	\lambda & = \ \lambda_l
\end{array}
$$ 
as words, and for all $0 \leq j \leq l$,   
\begin{equation} \label{eqn:corridor side}
\lambda_j b_0  \ = \ 
(u_{j}  a_1^{-\epsilon_{j}} u_{j-1}   \cdots  a_1^{-\epsilon_{2}}  u_{1}   a_1^{-\epsilon_{1}} u_{0}) 
  \  b_0 \    (a_1^{\epsilon_1}a_1^{\epsilon_{2}} \cdots a_1^{\epsilon_j}) 
  \end{equation} in $Q$, as illustrated in Figure~\ref{fig:a_1fence}.

 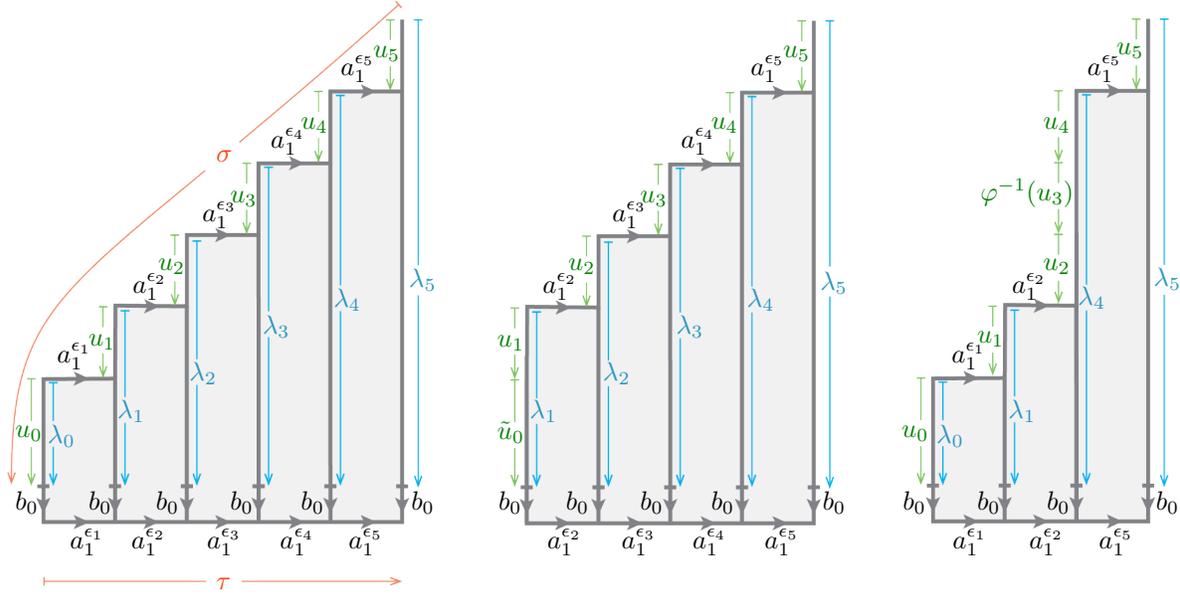
\begin{figure}[htbp]
\centering
\begin{overpic}  [scale=0.8] 
{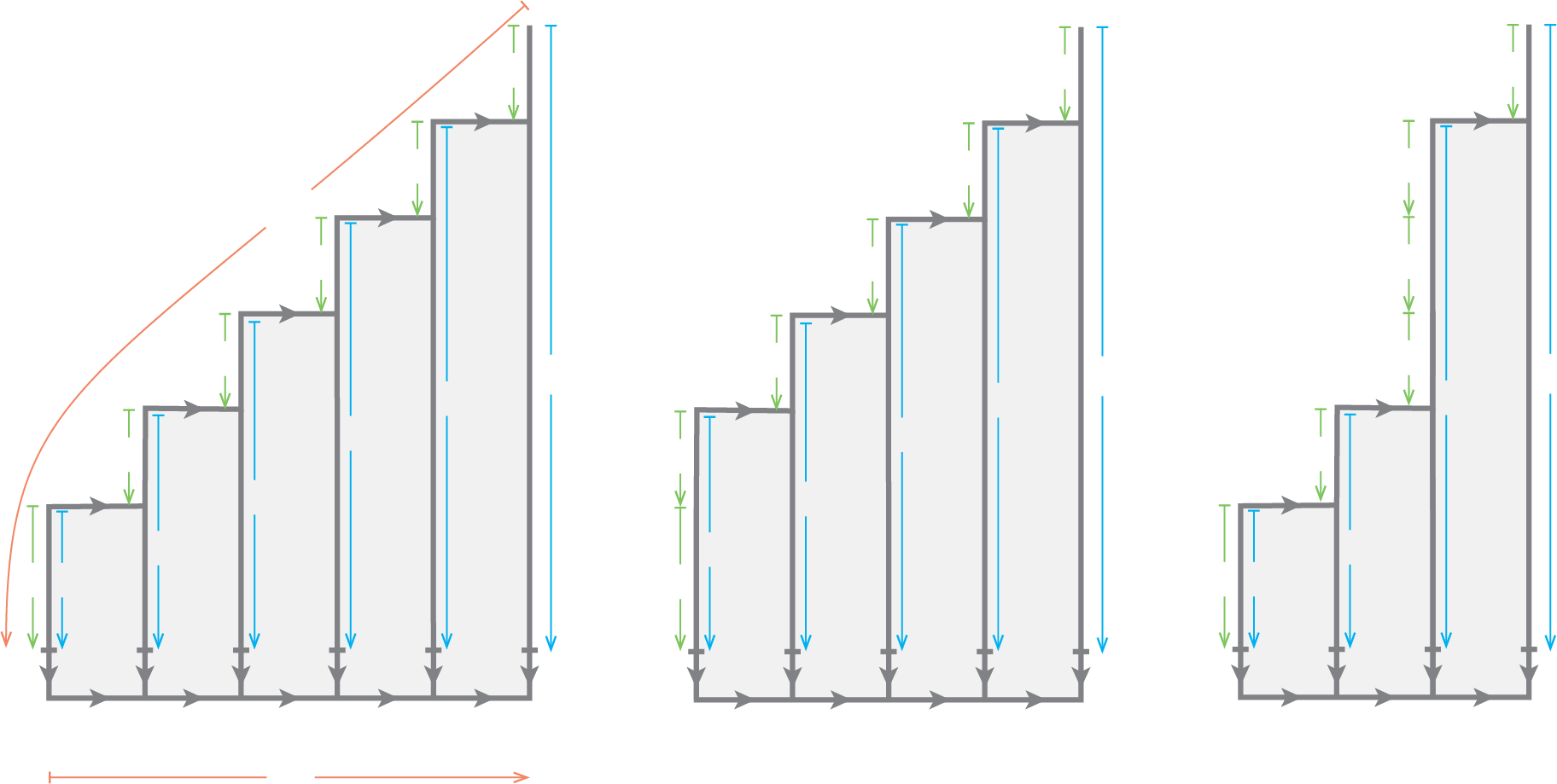}
\put(18, -0.3){\orange{\small{$\tau$}}}
\put(18, 36.5){\orange{\small{$\sigma$}}}
\put(0.7, 12.8){\green{\small{$u_0$}}} 
\put(7, 20.7){\green{\small{$u_1$}}} 
\put(13, 27,2){\green{\small{$u_2$}}} 
\put(19.2, 33){\green{\small{$u_3$}}} 
\put(25.3, 39.3){\green{\small{$u_4$}}} 
\put(31.5, 45.5){\green{\small{$u_5$}}} 
\put(0.7, 6.5){{\small{$b_0$}}} 
\put(7, 6.5){{\small{$b_0$}}} 
\put(12.8, 6.5){{\small{$b_0$}}} 
\put(19.2, 6.5){{\small{$b_0$}}} 
\put(25.3, 6.5){{\small{$b_0$}}} 
\put(34.8, 6.5){{\small{$b_0$}}} 
\put(3.5, 12.3){\bcyan{\small{$\lambda_0$}}} 
\put(9.5, 14.4){\bcyan{\small{$\lambda_1$}}} 
\put(15.7, 17.7){\bcyan{\small{$\lambda_2$}}} 
\put(22, 21.7){\bcyan{\small{$\lambda_3$}}} 
\put(28.1, 24){\bcyan{\small{$\lambda_4$}}} 
\put(34.6, 25.7){\bcyan{\small{$\lambda_5$}}} 
\put(5.4, 3.2){\small{$a_1^{\epsilon_1}$}}
\put(10.7, 3.2){\small{$a_1^{\epsilon_2}$}}
\put(17.2, 3.2){\small{$a_1^{\epsilon_3}$}}
\put(23.5, 3.2){\small{$a_1^{\epsilon_4}$}}
\put(29.4, 3.2){\small{$a_1^{\epsilon_5}$}}
\put(4.3, 19.2){\small{$a_1^{\epsilon_1}$}}
\put(10.8, 25.2){\small{$a_1^{\epsilon_2}$}}
\put(16.8, 31.4){\small{$a_1^{\epsilon_3}$}}
\put(22.7, 37,7){\small{$a_1^{\epsilon_4}$}}
\put(28.7, 44){\small{$a_1^{\epsilon_5}$}}
\put(42.2, 12.6){\green{\small{$\tilde{u}_0$}}} 
\put(42.2, 20.6){\green{\small{$u_1$}}} 
\put(48.3, 27.2){\green{\small{$u_2$}}} 
\put(54.5, 33){\green{\small{$u_3$}}} 
\put(60.6, 39.3){\green{\small{$u_4$}}} 
\put(66.8, 45.5){\green{\small{$u_5$}}} 
\put(70.2, 6.5){{\small{$b_0$}}} 
\put(42.3, 6.5){{\small{$b_0$}}} 
\put(48.1, 6.5){{\small{$b_0$}}} 
\put(54.5, 6.5){{\small{$b_0$}}} 
\put(60.6, 6.5){{\small{$b_0$}}} 
\put(45, 14.4){\bcyan{\small{$\lambda_1$}}} 
\put(51.3, 17.7){\bcyan{\small{$\lambda_2$}}} 
\put(57.6, 21.7){\bcyan{\small{$\lambda_3$}}} 
\put(63.7, 24){\bcyan{\small{$\lambda_4$}}} 
\put(70, 25.5){\bcyan{\small{$\lambda_5$}}} 
\put(46.5, 3.2){\small{$a_1^{\epsilon_2}$}}
\put(52.9, 3.2){\small{$a_1^{\epsilon_3}$}}
\put(59, 3.2){\small{$a_1^{\epsilon_4}$}}
\put(65.2, 3.2){\small{$a_1^{\epsilon_5}$}}
\put(46.1, 25.2){\small{$a_1^{\epsilon_2}$}}
\put(52.1, 31.4){\small{$a_1^{\epsilon_3}$}}
\put(58, 37,7){\small{$a_1^{\epsilon_4}$}}
\put(64, 44){\small{$a_1^{\epsilon_5}$}}
\put(77.1, 6.5){{\small{$b_0$}}} 
\put(83.5, 6.5){{\small{$b_0$}}} 
\put(89.5, 6.5){{\small{$b_0$}}} 
\put(98.9, 6.5){{\small{$b_0$}}} 
\put(77, 12.8){\green{\small{$u_0$}}} 
\put(83.5, 20.9){\green{\small{$u_1$}}} 
\put(89.2, 27.2){\green{\small{$u_2$}}} 
\put(83.8, 33.2){\green{\small{$\varphi^{-1}(u_3)$}}} 
\put(89.2, 39.3){\green{\small{$u_4$}}} 
\put(95.5, 45.8){\green{\small{$u_5$}}} 
\put(80, 12.3){\bcyan{\small{$\lambda_0$}}} 
\put(86.2, 14.4){\bcyan{\small{$\lambda_1$}}} 
\put(92.2, 24.2){\bcyan{\small{$\lambda_4$}}} 
\put(98.7, 25.8){\bcyan{\small{$\lambda_5$}}} 
\put(81.4, 3.2){\small{$a_1^{\epsilon_1}$}}
\put(88, 3.2){\small{$a_1^{\epsilon_2}$}}
\put(94, 3.2){\small{$a_1^{\epsilon_5}$}}
\put(81.3, 19.2){\small{$a_1^{\epsilon_1}$}}
\put(86.5, 25.2){\small{$a_1^{\epsilon_2}$}}
\put(93, 44){\small{$a_1^{\epsilon_5}$}}
\end{overpic}
\caption{Illustrating our proof of Lemma~\ref{lem:clean up as} (with $l=5$).  Left: a diagram for  $\sigma^{-1}  b_0  \tau  = \lambda b_0$ in $Q$ via $(\pmb{\lambda},  \mathbf{u}, \pmb{\epsilon})$.  Centre: the result of applying move I. Right: the result of applying move II (with $j=4$).}
\label{fig:a_1fence}
\end{figure}

Let $\lambda_0,  \ldots, \lambda_{l-1}$ be  the positive words on $b_1, \ldots, b_p$  such that  $\lambda_0 b_0,  \ldots, \lambda_{l-1}b_0$ are the words   along the bottom boundaries (read in the direction from $\hat{\mu}$ to $\nu$) of the $a_1$-corridors in $\Lambda_1$.  Let  $\lambda_{l} = \lambda$.  
Per Lemma~\ref{lem:clean up bs},   $\hat{\mu} b_0 \nu  = \lambda b_0$ in $G$ and, given how the $a_1$-corridors in $\Lambda_1$ pair off the $a_1^{\pm 1}$ in $\nu$ with the   $a_1^{\pm 1}$ in  $\hat{\mu}$,  if we define  $\sigma$ and  $\tau$ to be  $\hat{\mu}^{-1}$ and   $\nu$   with all letters $a_2$, $t$, $x_1$, $x_2$, $y_1$, and $y_2$ deleted,  then they have the forms displayed above. Accordingly, they define  $\mathbf{u}$ and $\pmb{\epsilon}$  so that $\sigma^{-1}  b_0  \tau  = \lambda b_0$ in $Q$ via $(\pmb{\lambda},  \mathbf{u}, \pmb{\epsilon})$.  Moreover, $l \leq |w_i|$ and $|\mathbf{u} | : = \sum_{j=0}^l |u_i| \leq 2|w_i|$, the last inequality coming from summing the bounds from Lemma~\ref{lem:clean up bs} \eqref{understand the as} and \eqref{understand the number of bs}.

We will simplify $(\pmb{\lambda},  \mathbf{u}, \pmb{\epsilon})$ in two ways: 

\begin{itemize}  
\item[I.] Suppose that $\epsilon_1 =-1$.   Then \eqref{eqn:corridor side} in the case $j=1$   gives that in $Q$,
$$
\lambda_1 b_0  \ = \ u_{1} a_1  u_{0}  
  \  b_0 \     a_1^{-1} \ = \ u_1 \varphi^{-1}(u_0 b_0).$$
 Now,   $u_1$ is a prefix of  $\lambda_1$ and so $\varphi^{-1}(u_0 b_0)$ is a suffix of $\lambda_1 b_0$, and so is a positive word. Therefore  Lemma~\ref{lem:positive in Q} applies and tells us that  $|\varphi^{-1}( u_{0} b_0)| \leq |u_0 b_0|$.  Define  $\tilde{u}_0$ to be the word obtained from  $\varphi^{-1}( u_{0} b_0)$ by removing its final letter $b_0$.  Then $|\tilde{u}_0| \leq |u_0|$ and $\lambda_1 = u_1 \tilde{u}_0$.   Define $\hat{\pmb{\lambda}}$
to be $\pmb{\lambda}$ with $\lambda_0$ discarded, define  
$\hat{\mathbf{u}}$ to be $\mathbf{u}$ with $u_{0}$  discarded
and $u_{1}$ replaced by $u_{1}  \tilde{u}_{0}$,  and define $\hat{\pmb{\epsilon}}$ 
 to be $\pmb{\epsilon}$ with   $\epsilon_{1}$ discarded.   
Then $\sigma^{-1}  b_0  \tau  = \lambda b_0$ in $Q$ via $(\hat{\pmb{\lambda}},  \hat{\mathbf{u}}, \hat{\pmb{\epsilon}})$, the lengths of the three sequences have all decreased by $1$.  And because $|\tilde{u}_0| \leq |u_0|$, we get $|\hat{\mathbf{u}}| \leq |\mathbf{u}|$.

\item[II.] Suppose $\epsilon_{j-1}=1$ and $\epsilon_{j}=-1$ for some $2 \leq j \leq l$.    
  Using \eqref{eqn:corridor side} to relate $\lambda_{j-2} b_0$ and $\lambda_{j} b_0$, we get  
$$\lambda_{j}b_0 
\ = \ u_{j} a_1  u_{j-1} a_1^{-1}    \ \lambda_{j-2}  b_0  \ a_1  a_1^{-1} \  = \  u_{j}   \varphi^{-1}( u_{j-1})      \ \lambda_{j-2}  b_0$$ in $Q$.
Now, $u_{j}$ is a prefix of $\lambda_{j}$ and $\lambda_{j}b_0$ is a positive word, so the word $\varphi^{-1}( u_{j-1})      \ \lambda_{j-2}  b_0$ is equal in $Q$ to a positive word, and then by Lemma~\ref{lem:w suffix},  $\varphi^{-1}( u_{j-1})$ is a prefix of that positive word.  Given that both $\varphi^{-1}( u_{j-1})$ and $u_{j-1}$ are positive words,  Lemma~\ref{lem:positive in Q} tells us that $|\varphi^{-1}( u_{j-1})| \leq |u_{j-1}|$.   
Now define
$\hat{\pmb{\lambda}}$
to be $\pmb{\lambda}$ with $\lambda_{j-1}$ and $\lambda_{j}$ discarded, define  
$\hat{\mathbf{u}}$
to be $\mathbf{u}$ with $u_{j-2}$ and $u_{j-1}$ discarded
and $u_{j}$ replaced with $u_{j} \varphi^{-1}( u_{j-1}) u_{j-2}$,  and define $\hat{\pmb{\epsilon}}$ 
 to be $\pmb{\epsilon}$ with   $\epsilon_{j-1}$ and $\epsilon_{j}$ discarded.   
Then $\sigma^{-1}  b_0  \tau  = \lambda b_0$ in $Q$ via $(\hat{\pmb{\lambda}},  \hat{\mathbf{u}}, \hat{\pmb{\epsilon}})$, the lengths of the three sequences have all decreased by $2$, and $|\hat{\mathbf{u}}| \leq |\mathbf{u}|$. 
\end{itemize}

Repeat I and II until we have  $(\pmb{\lambda},  \mathbf{u}, \pmb{\epsilon})$ via which $\sigma^{-1}  b_0  \tau  = \lambda b_0$ in $Q$  with $\pmb{\epsilon} = (1, \cdots, 1)$.  Throughout, the bounds $l \leq |w_i|$ and $|\mathbf{u} |  \leq 2|w_i|$ are maintained.  The resulting  $\mu = \sigma^{-1}$ and $\tau = a_1^l$ have the required properties.
\end{proof}

A calculation in $Q$ now bounds the length of $\lambda$.  We state the result in the following lemma, deferring the proof to Section~\ref{sec:p/q}.
\begin{lemma}\label{lem:Qp/q} There exists $C_0>1$ with the following property. 
Suppose there are words $\mu$ on $a_1^{-1}, b_1, \dots, b_p$ (so containing no $a_1, b_1^{-1}, \dots, b_p^{-1}$)  and  $\lambda$ on $b_1, \dots, b_p$ (so containing only positive letters), and a number $l\ge 1$ such that in $Q$ 
\begin{equation}
\label{eq:Qp/q}
\mu b_0 a_1^l \ =  \  \lambda b_0.\end{equation}  
Then, if $|\cdot|_q$ counts the number   of $b_q$ in a given word, we have: 
$$|\lambda|  \ \le \  C_0(|\mu|+ |\lambda|_q)^{p/q}.$$
\end{lemma}

In the situation of Corollary~\ref{cor:block cells}, this leads to an upper bound on the lengths of the $a_1$-corridors in $\mathcal B_i$ for all $i$ such that $\beta_i$ is a $b_0$-corridor.  

\begin{lemma} \label{lem:p/q}
There exists $C_1>1$ such that if $\mathcal C$ is as in Lemma~\ref{lem:a-corridor words} and $\xi b_0$ and $\lambda b_0$ are the words read along the top and bottom boundaries (respectively) of $\mathcal C$, then 
$$\max\{|\lambda|, |\xi|\} \ \le \  C_1 |w|^{p/q}.$$
   \end{lemma}

\begin{proof}
First consider the word $\lambda b_0$ along the bottom boundary of $\mathcal C$. Use  Lemma~\ref{lem:clean up bs} and \ref{lem:clean up as} to obtain a word $\mu= \mu(b_1, \dots, b_p, a_1^{-1})$ and a number $l\ge 1$ such that Lemma~\ref{lem:Qp/q} applies.  Then $|\lambda| \le C_0(|\mu|+ |\lambda|_q)^{p/q}$. 
By Lemma~\ref{lem:clean up as}, we have  $|\mu| \le 2 |w_i| \le 2|w|$.  

We estimate $|\lambda|_q$ using $(a_2, b_q)$-tracks (see Definition~\ref{def: a2bq-track}).
The dual of every edge labelled $b_q$  in $\lambda$  is part of an $(a_2, b_q)$-track of $\Delta$ with endpoints on $w$ (by parts~\eqref{first a2bq property} and~\eqref{lem part: a2bq not a loop} of Lemma~\ref{lem:(a_2,b_q)-tracks}).  Suppose some  $(a_2, b_q)$-track $\gamma$ crosses $\mathcal C$ twice.  Then the edges of $\lambda$ dual to $\gamma$ are necessarily labelled by $b_q^{\pm 1}$, as $\lambda$ has no $a_2$, and since $\gamma$ is oriented (Lemma~\ref{lem:(a_2,b_q)-tracks}\eqref{zeroth a2bq property})  at least one of these must be $b_q^{-1}$.  This contradicts the fact, established in Lemma~\ref{lem:a-corridor words}, that $\lambda$ is a positive word. 
Thus any $(a_2, b_q)$-track crosses $\lambda$ at most once.
It follows that 
$|\lambda|_q \le |w|$.  Thus
\begin{equation} \label{eq:lambda length}
|\lambda|  \ \le \  C_0(|\mu|+ |\lambda|_q)^{p/q} \ \le \  C_	0 ( 4 |w|)^{p/q} \ \le \  C_0' |w|^{p/q},
\end{equation}
for a suitable constant $C_0'$.  

Now if $\xi b_0$ is the top boundary of an $a_1$-corridor, then we have  a relation $\xi b_0  = a_1^{-1} (\lambda b_0)  a_1$, where $\lambda$ is a positive word on $b_1, \dots, b_p$.  Inspecting the $r_{1, \ast}$-defining relations (of Figure~\ref{fig:relations}), we see that 
$|\xi| \le C_0''|\lambda|$ for a suitable constant $C_0'' \ge 1$.  Combining this with \eqref{eq:lambda length}, we obtain $\max\{|\lambda|, |\xi|\} \le C_1 |w|^{p/q}$ for a suitable constant $C_1>1$. 
\end{proof}

Our next lemma is illustrated by Figure~\ref{fig:b-block simplify}. 
We can now derive:

\begin{lemma}\label{lem:Lengths of W_i} There exists a constant $C_2 >1$ such that for all $i$ such that $\beta_i$ is a $b_0$-track,
	   \begin{equation} \label{W_i upper bound}    | W_i|   \ \leq \  C_2^{|w| ^{p/q}}.
   \end{equation} 	
\end{lemma}

   \begin{figure}[htbp]
   \centering
\begin{overpic}
{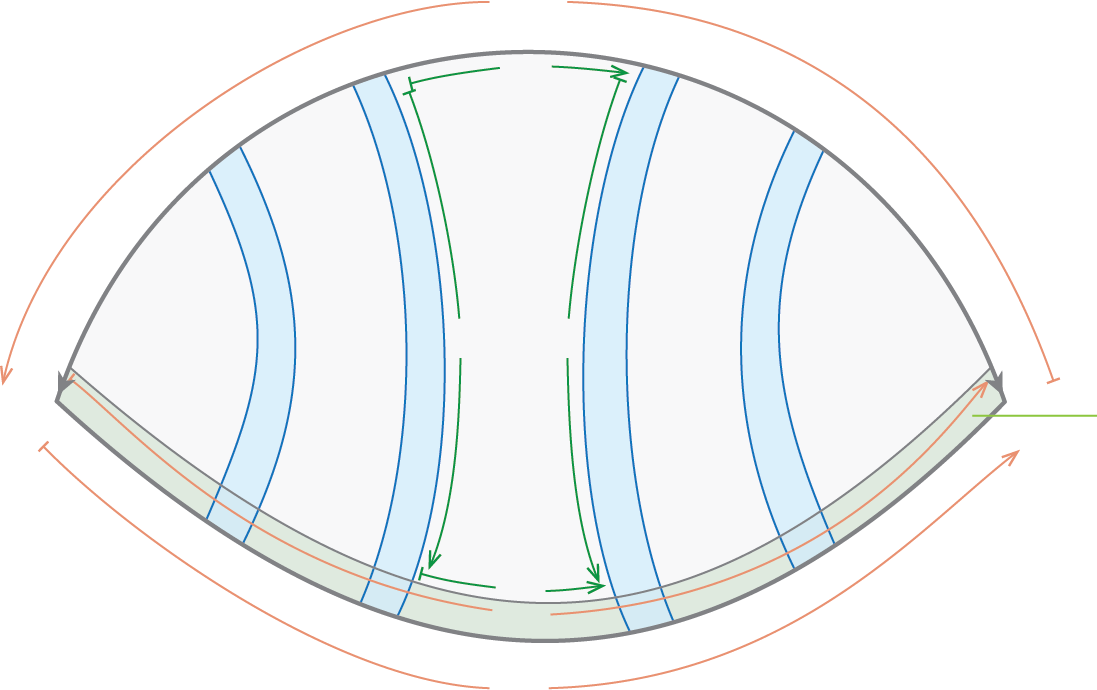}
\put(46, 62.5){\orange{\small{$w_i$}}}
\put(45.6, -.5){\orange{\tiny{$W_i$}}}
\put(45, 5.6){\orange{\tiny{$W'_i$}}}
\put(101, 24){{\small{$\mathcal{C}$}}}
\put(45.9, 55.7){\green{\small{$\alpha_2$}}}
\put(45.5, 9.4){\green{\small{$u'_2$}}}
\put(40.7, 31.7){\green{\small{$\mu_2$}}}
\put(49.2, 31.7){\green{\small{$\nu_2$}}}
\put(23.7, 30){{\tiny{$\mathcal{C}_1$}}}
\put(37.3, 27){{\tiny{$\mathcal{C}_2$}}}
\put(53.7, 27){{\tiny{$\mathcal{C}_3$}}}
\put(67.8, 30){{\tiny{$\mathcal{C}_4$}}}
\put(14, 34){{\tiny{$\mathcal{D}_0$}}}
\put(30, 37){{\tiny{$\mathcal{D}_1$}}}
\put(45, 38){{\tiny{$\mathcal{D}_2$}}}
\put(61, 37){{\tiny{$\mathcal{D}_3$}}}
\put(77, 34){{\tiny{$\mathcal{D}_4$}}}
%
\put(73.5, 9.5){\blue{\small{$a_1$}}}
\put(73.5, 51){\blue{\small{$a_1$}}}
\put(18, 11.5){\blue{\small{$a_1$}}}
\put(18, 50){\blue{\small{$a_1$}}}
\put(32, 5){\blue{\small{$a_1$}}}
\put(30.5, 57){\blue{\small{$a_1$}}}
\put(58, 3){\blue{\small{$a_1$}}}
\put(59.5, 58){\blue{\small{$a_1$}}}
\put(90, 57){\blue{\small{$\mathcal{B}_i$}}}
\put(2, 29){\green{\small{$b_0$}}}
\put(91.4, 28.5){\green{\small{$b_0$}}}
\end{overpic}
\caption{Illustrating our proof of Lemma~\ref{lem:Lengths of W_i} (with $l=4$)}
\label{fig:b-block simplify}
\end{figure}
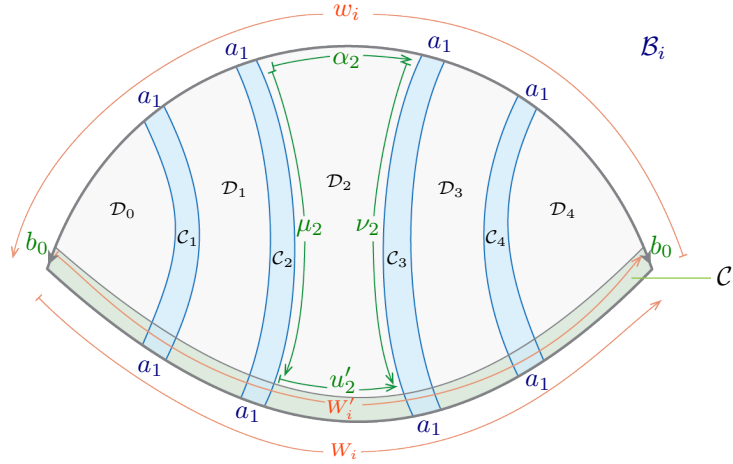

\begin{proof}
Let $\mathcal C$ be the (unique) $b_0$-corridor in $\mathcal B_i$ and let $W_i'$ be its bottom boundary, so we have the relation
$b_0^{-1} W_i' b_0 = W_i$.  Then there exists a constant $K_0 \ge 1$ such that 
\begin{equation} \label{eq:W_i'}
 |W_i|  \ \le  \  K_0  |W_i'|.
\end{equation}
Let $\mathcal C_1, \dots \mathcal C_l$ be the $a_1$-corridors of $\mathcal B_i$ and let $\mathcal D_0, \dots, \mathcal D_{l}$ be the (closures of the) components of $\mathcal B_i \setminus (\mathcal C \cup C_1\cup  \dots \cup \mathcal C_l)$.   
Then, for all $j$, $\mathcal D_j$ is a van Kampen diagram for the relation $\mu_j^{-1} \alpha_j \nu_j = u'_j$, where 
 $\alpha_j$ is a subpath of $w_i$, the paths $\mu_j$ and $\nu_j$ (which are possibly empty) run along the $a_1$-corridors bounding $\mathcal D_j$, and $u_j'$ is a subpath of $W_i'$.  We know from Corollary~\ref{cor:block cells}\eqref{b-blocks part} that  the  2-cells in $\mathcal D_j$ are of type $r_{1, \ast}$, $r_{2, \ast}$, $r_{3, \ast}$, and $r_{3, \ast, \ast}$.  And, as $\mathcal D_j$ has no $a_1$- or $b_0$-corridors, the relation $\mu_j^{-1} \alpha_j \nu_j = u'_j$ holds in (in the notation of Figure~\ref{fig:relations})  
 $$\langle \  a_2, t, x_1, x_2, \ b_1, \ldots, b_p \mid  \ \set{r_{2,i}, \,  r_{3,i}, \,  r_{3,i,j} : 1\leq i \leq p \text{ and } 1\leq j \leq 2}  \     \rangle,$$
which is a multiple HNN-extension of $F(a_2, t, x_1, x_2)$ with stable letters $b_1, \ldots, b_p$.  
 So, by repeated use of Britton's Lemma $|u_j'| \le |\alpha_j | K_1^M$, where $K_1 \geq 1$ is a constant multiplicative factor bounding the increase in length on eliminating a \emph{pinch}, and $M= \max(|\mu_j|, |\nu_j|)$.  So  $|u_j'| \le C_1 |w|^{p/q}$ by Lemma~\ref{lem:p/q}.  Then, because the number of $a_1$-corridors is $l$, we have 
 $$
 |W_i'| \ \le \  l+ \sum_{i=0}^{l+1}  |u'_j|  \ \le \  l+ \sum_{i=0}^{l+1}  |\alpha_j | K_1^M  \  \le \  \left(l+   \sum_{i=0}^{l+1}  |\alpha_j | \right) K_1^M \  \le \  |w| K_1^{C_1 |w|^{p/q}}. 
$$
  This and \eqref{eq:W_i'} together establish \eqref{W_i upper bound} for a suitable constant $C_2>1$. 
\end{proof}

We  can now   complete: 

\begin{proof}[Proof of Proposition~\ref{prop:upper bound}] 
Recall that $\Delta$ is a van Kampen diagram for $w \chi^{-1}$ and $\mathcal A$ is a subdiagram  for $W \chi^{-1}$, where $W$ is as defined in~\eqref{eq:W} and all the 2-cells of $\mathcal A$ are $r_{4, \ast}$- or $r_{4, \ast, \ast}$-cells (per Figure~\ref{fig:relations}).  
Now, $\mathcal A$ is a tree-like arrangement of 2-disc components connected by 1-dimensional portions (trees).  
As $r_{4, \ast}$- and $r_{4, \ast, \ast}$-cell have no $x$-edges on their boundaries,  
any $x$-edges in   $\mathcal A$ are in $1$-dimensional portions.    Let  $\widehat{\mathcal{A}}$ be the subdiagram of $\mathcal A$ consisting of the path $\chi$ and all its 2-disc components that share at least one edge with $\chi$. 
Then   $\widehat{\mathcal{A}}$ is a van Kampen diagram for $\widehat{W} \chi^{-1}$, where $\widehat{W}$ is a word obtained from $W$ by deleting some of its letters. Then $\widehat{W}$ contains no $x$-letters: its letters  are either 
 along the path $\chi$ or are on the boundaries of $2$-cells, neither of which have $x$-edges. 

 If $\beta_i$ is not a $b_0$-track, then $W_i$ is a word on $a_1 \Xb t \Xb$, $a_1  \Xb t \Xb$, $a_1 a_2  \Xb t \Xb$, $\Xb t^{-1} \Xb t \Xb$ and $\Xb t \Xb$.  And (because $\Delta$ is reduced and thanks to the   $C'(1/4)$ small-cancellation  condition  of Section~\ref{sec:the defn}   for the set  $\mathcal{X}$ of the  $\Xb$),  if a  subword of   the freely reduced form of $W_i$ contains no $x$-letters, then it has length at most $2$.    It follows that $W_i$ can contribute  at most two letters to    $\widehat{W}$.   
 
Therefore, in the notation of \eqref{eq:W},  $|\widehat{W}|$  is at most   $\sum_{i=0}^{r} |v_i|$, plus twice the number of $W_i$ such that $\beta_i$ is not $b_0$-track, plus  the lengths of the  remaining  $W_i$.   So, using Lemma~\ref{lem:Lengths of W_i} and that there are at most $|w|$ subwords in $W_i$ in $W$,    for a suitable constant $C_3>1$, we get
\begin{equation}\label{eq:length of What}
|\widehat{W}|  \ \leq \ |w| + 2 |w| + |w|   C_2^{|w|^{p/q}} \ \le \  C_3^{|w|^{p/q}}.
\end{equation}

Next we claim that there exists a constant $C_4>1$   such that 
\begin{equation} \label{eq:chi first upper bound}
| \chi |  \ \leq \ |\widehat{W}| \, C_4^{|w|}.  	
\end{equation}

 Since the 2-cells in $\widehat{\mathcal{A}}$ are all of type $r_{4, \ast}$ or $r_{4, \ast, \ast}$ (per Figure~\ref{fig:relations}),  $\widehat{\mathcal A}$ is a union of non-intersecting $a_1$- and $a_2$-corridors.  Each $a_1$-corridor of $\widehat{\mathcal A}$ is part of an $a_1$-corridor of~$\Delta$ whose ends are in $w$, and Lemma~\ref{lem: Layout lemma}\eqref{lem part:crossing possibilities} implies that no two $a_1$-corridors of $\widehat{\mathcal A}$ are part of the same $a_1$-corridor in $\Delta$.  On the other hand, several $a_2$-corridors of $\widehat{\mathcal A}$ could be part of the same $(a_2, b_q)$-corridor of $\Delta$.  However, by Lemma~\ref{lem: a2bq not nesting}, if a pair of  $a_2$-corridors  of $\mathcal A$ nest (meaning one is entirely in the $W$-side of the other), then they cannot be part of the same $(a_2, b_q)$-corridor of $\Delta$.  It follows that the same is true of $\widehat{\mathcal A}$: no pair of  $a_2$-corridors  of $\widehat{\mathcal A}$ have the property that one is entirely in the $\widehat{W}$-side of the other. 
Distinct $(a_2, b_q)$-corridors end on distinct pairs of edges of $w$.

Thanks to these observations, we can strip away successive portions of $\widehat{\mathcal A}$ by at most $|w|$ moves, each of which either
\begin{itemize}
\item removes an $a_1$-corridor, or 
\item removes \emph{all} the $a_2$-corridors  of $\widehat{\mathcal A}$ that are part of the same  $(a_2,b_q)$-corridor of $\Delta$.\end{itemize}
The result is 
a sequence of diagrams which demonstrate that  each word in a sequence of words equals
 $\chi$ in $G$.  Moreover, this sequence of words starts with  $\widehat{W}$ and ends with a word freely equal to $\chi$, and the length of each word is longer than the last by at most a constant factor.  This proves  \eqref{eq:chi first upper bound} for a suitable constant $C_4>1$.

Finally, \eqref{eq:length of What} and~\eqref{eq:chi first upper bound} combine to yield 
$$|\chi|  \ \le \  |\widehat{W}| \, C_4^{|w|}  \ \le  \  C_3^{|w|^{p/q}} C_4^{|w|} \ \le \  K^{|w|^{p/q}}$$
 for a suitably chosen constant $K>1$. 
 \end{proof}

\subsection{Why $p/q$?} \label{sec:p/q}

This section is devoted to a proof of Lemma~\ref{lem:Qp/q}, which we used in our proof of Proposition~\ref{prop:upper bound}.  The lemma concerns the group \begin{equation*}   Q  \ = \ \langle a_1,  b_0, \ldots, b_p  \mid a_1^{-1}b_i a_1 = \varphi(b_i)  \ \, \forall i  \  \rangle,  \ \ \varphi(b_j) \, = \, \begin{cases} b_{j+1}b_j & \text{if } j<p \\  b_j & \text{if } j=p. \end{cases}\end{equation*}   
  We begin with two preparatory lemmas.  We use the convention that the binomial coefficient ${n \choose r}$ equals $0$ for all $r \notin \set{0, \ldots, n}$.

\begin{lemma}\label{lem:Bgrowth}
Consider the relation $a_1^{-m}b_i a_1^m = \lambda$ in $Q$, where $m \geq 0$,  $0\le i \le p$, and  $\lambda$ is a word in  $b_0, \dots, b_p$. Then  
\begin{enumerate}
\item \label{binomial1} For $0 \le j \le p-i$, there are
 $m \choose j$ instances of $b_{i+j}$ in $\lambda$.  Also, $\lambda$ has no $b_k$ for $k<i$.
\item \label{binomial2}
If $m>2p$, then 
$|\lambda|\le (p+1) {m \choose p-i}$.
\item \label{binomial3} If $m \le 2p$, then 
$|\lambda|\le (p+1) (2p)^p $
\end{enumerate}
\end{lemma}

\begin{proof} For \eqref{binomial1}, induct on $m$ or refer to \cite{BR1}.  For \eqref{binomial2}, note that   if $0 \le i \le p$ and $m > 2p$, then $p-i \le p < m/2$, and so ${m\choose j} \le {m\choose p-i}$ for all $j \le p-i$. Then from (1), we have  
$$
|\lambda| \ =   \ \sum_{j=0}^{p-i} {m \choose j}  \ \le \ 
\sum_{j=0}^{p-i} {m \choose p-i} \ \le \  (p-i+1){m \choose p-i} \ \le \  (p+1){m \choose p-i}.
$$
For \eqref{binomial3}, we use the fact that ${m \choose j} \le m^j$ for any $j \le m$, and 
$$
|\lambda| \ = \  \sum_{j=0}^{p-i} {m \choose j} \  \le \ \sum_{j=0}^{p-i} m^j \ \le \  (p-i+1) m^{p-i}  \ \le 
 \ (p+1)(2p)^p. 
$$
\end{proof}

\begin{lemma}\label{lem:binomialbound}
Let $K = (2p)^{p^2}$.  For all $m, k, l \in \Z$ such that $m>2p$ and $1 \le k, l \le p$, 
\begin{enumerate}
\item \label{more binomial1}
$
{m \choose k} \le K {m \choose l}^{k/l}
$

\item \label{more binomial2} If $l<k$, then ${m \choose {k}} \le K {m \choose l}{m \choose {k-l}}$
\end{enumerate}
\end{lemma}

\begin{proof}
Let $m> 2p$.  Now, if $t$ satisfies $1 \le t \le p$, then  
$m> 2t$, or equivalently $-t > -m/2$.  Consequently, $m-t+1 > m-m/2+1 > m/2$, which gives the ``$>$'' in: 
\begin{equation}\label{eq:nq}
m^t \ge {m\choose t} \ = \  \frac{m(m-1) \dots (m-t+1)}{t!} \  > \ \left(\frac{m} {2}\right)^t \frac{1}{t!} \ \ge \  \frac{m^t} {2^p p!} \  \ge  \ \frac{m^t} {(2p)^p}. 
\end{equation}

Now, ${m \choose k}  \le m^{k}$,  \eqref{eq:nq}, and $k <p$, respectively,  imply the first, second, and third of the following inequalities: 
$$ {m \choose k}^{l} \ \le \  m^{kl} \ \le \  (2p)^{pk} {m \choose l}^k \ \le \  (2p)^{p^2} {m \choose l}^k.$$
Then \eqref{more binomial1} follows since $(2p)^{p^2/l} \le (2p)^{p^2}=K$.

For \eqref{more binomial2}, now  apply \eqref{eq:nq} to $t=l$ and $t=k-l$, and note that $2p \le p^2$ (since $1 \le l < k  \le p$ implies that $p \geq 2$): 
\begin{align*}
{m \choose k}  \ &\le \  m^k = m^{l} m^{k-l}  \ \le \  (2p)^p {m \choose l} (2p)^p {m \choose {k-l}} \\  
&= \  (2p)^{2p} {m \choose l} {m \choose {k-l}} \ \le \  K  {m \choose l} {m \choose {k-l}}.
\end{align*}
\end{proof}

 For a word $\pi$, we write $| \pi |_b$ and $| \pi |_q$ to denote the number of $b$-letters and the number of $b_q$-letters (respectively) in  $\pi$. 
 
 Suppose $\mu$ is a word on $a_1^{-1}, b_1, \dots, b_p$ (no $a_1, b_1^{-1}, \dots, b_p^{-1}$ letters), $\lambda$ is a positive word on $b_1, \dots, b_p$, and $l\ge 1$ is an integer such that in $Q$ 
\begin{equation}
\label{eq:Qp/q repeat}
\mu b_0 a_1^l \ =  \  \lambda b_0.\end{equation}
Lemma~\ref{lem:Qp/q} asserts  that 
\begin{equation} \label{p/q inequality}
|\lambda|  \ \le \  C_0(|\mu|+ |\lambda|_q)^{p/q}
\end{equation}
 for a suitable constant $C_0 >1$.

Here is the idea behind this. When we shuffle the $a_1^{\pm 1}$ letters   through $\mu b_0 a_1^l$, in order to  collect them together and cancel them away and obtain $\lambda b_0$, the effect is to apply $\varphi$ to the intervening $b$-letters.  Lemma~\ref{lem:Bgrowth}\eqref{binomial1} indicates how the number of $b$-letters then grows: as a function of $l$, the number of $b_i$-letters in $\lambda$ is at most a polynomial of degree $i$.  Whether this rate of growth is achieved depends on $\mu$. What \eqref{p/q inequality} states is how the total number of $b$-letters produced is contingent on the length of $\mu$ \emph{and} the number of $b_q$-letters produced.

\begin{proof}[Proof of Lemma~\ref{lem:Qp/q}]  
Let $C_0 = (p+1)(2p)^{2p^2}$.  We induct on   $|\mu|_b$.  

\emph{Base case.} 
In the base case, $|\mu|_b=0$, and so $\mu = a_1^{-l}$ and  \eqref{eq:Qp/q repeat} is 
$a_1^{-l} b_0 a_1^l = \lambda b_0$.  Then $ |\lambda|_q = {l \choose q}$ by 
Lemma~\ref{lem:Bgrowth}\eqref{binomial1}, and so
\begin{equation} \label{eqn:basecase}
|\mu|+ |\lambda|_q \ge {l \choose q}.
\end{equation}
If $l > 2p$, then Lemmas~\ref{lem:Bgrowth}\eqref{binomial2} and \ref{lem:binomialbound}\eqref{more binomial1} apply so as to give the first and second (respectively) of the following inequalities;  the definition of $C_0$ and  \eqref{eqn:basecase} give the third:
$$
 |\lambda|  \ \le \  (p+1) {l\choose p} \  \le \  (p+1) (2p)^{p^2} {l\choose q}^{p/q}  \ \le  \  C_0 (|\mu|+ |\lambda|_q )^{p/q}.
$$
If, on the other hand,  $l \le 2p$, then, by
Lemma~\ref{lem:Bgrowth}\eqref{binomial3}, we have that $$|\lambda| \  \le \  (p+1) (2p)^p \ \le \   C_0 \ \le  \ C_0(|\mu|+ |\lambda|_q )^{p/q},$$ with the final inequality true because  $l \ge 1$. 
This completes our proof of the base case.

\emph{Inductive step.}
Suppose we have $\hat{\mu} b_0   a_1^{l} =\hat{\lambda} b_0$  as 
per \eqref{eq:Qp/q repeat}  with  $|\hat{\mu}|_b = k+1$.
We will show that $|\hat{\lambda}|^q \le C_0^q \hat{n}^{p}$, where \begin{equation}\label{eq:n_{k+1}}
\hat{n}  \ = \  |\hat{\mu}|+|\hat{\lambda}|_q.\end{equation}  

Suppose $b_i$ is the first $b$-letter in $\hat{\mu}$.  Then $\hat{\mu} =  a_1^{-m} b_i  \beta$ for some integer $m$ such that  $0 \leq m \leq l$, and word $\beta$  that contains $l-m$ instances of $a_1^{-1}$ and satisfies $|\beta|_b =k$.  The exponent sums of the $a_1$-letters in $a_1^{-m} b_i a_1^{m}$ and $a_1^{-m} \beta a_1^{l}$ are both $0$, so there exist  positive words $\gamma$ and $\lambda$, respectively, on $b_1, \ldots, b_p$ representing them in $Q$.   
   Then in $Q$, 
$$\hat{\lambda}b_0  \ = \  \hat{\mu}b_0a_1^l  \ = \  (a_1^{-m} b_i a_1^{m})(a_1^{-m} \beta b_0a_1^{l}) \  =  \ \gamma \lambda b_0.$$ 
Thus $|\hat{\lambda}| = |\lambda|+ |\gamma|$. We will bound  $|\hat{\lambda}|^q$ by combining bounds on  $|\lambda|$ and $|\gamma|$.

Setting $\mu = a_1^{-m}\beta $, we have  $\mu b_0  a_1^l = \lambda b_0$ in $Q$, where $\mu$ satisfies the hypotheses of the present lemma and $|\mu|_b  = k$.  By the induction hypothesis,  
$|\lambda|  \le C_0 n^{p/q}$, where $n = |\mu| + |\lambda|_q$.   
 
 Before bounding $|\gamma|$ 
 we make some observations about $n$ and  $\hat{n}$.
Firstly, the presence of $b_0$ in the relation   $a_1^{-m} \beta b_0 a_1^{l} = \lambda$, together with 
Lemma~\ref{lem:Bgrowth}(1) implies that $|\lambda|_q \ge {m \choose q}$,
  and so 
\begin{equation}\label{eq:nlower}
n  \ \ge \   {m \choose q}.
\end{equation} 
Note that $|\hat{\mu}| = |\beta|+1+ m = |\mu|+1$, leading to:
\begin{equation}\label{eq:nn_k}
\hat{n}  \ = \  |\hat{\mu}| + |\hat{\lambda}|_q \  = \  
|\mu| + 1 +|\lambda|_q +  |\gamma|_q  \  = \
n + 1 + |\gamma|_q. 
\end{equation}

\smallskip
Then, since $|\hat{\lambda}| = |\lambda|+ |\gamma|$, we have
\begin{align}
|\hat{\lambda}|^q  \ \le \  ( |\lambda| + |\gamma|)^q 
& \ =  \ \sum_{j=0}^q {q \choose j} |\lambda|^{q-j} |\gamma|^j&&  \nonumber \\
& \ \le \  \sum_{j=0}^q {q \choose j} 
\left(C_0 n^{p/q}\right)^{q-j} 
|\gamma|^j && \text{(by the induction hypothesis)} \nonumber \\
& \ \le \  \sum_{j=0}^q {q \choose j} C_0^{q-j} n^{p-\frac{pj}{q}}
|\gamma|^j.&& 
\label{eq:gamma}
\end{align}

Similarly to the base case, we treat the cases $m \le 2p$ and $m > 2p$ separately.  When $m>2p$, our estimate depends on whether $i \geq q$, in which case no new $b_q$ letters are created in $\gamma$, or $i < q$, in which case new $b_q$ letters are created in $\gamma$.  Thus, we have three cases as follows.

\medskip
\noindent\emph{Case 1:} $m \le 2p$.
 In this case, $|\gamma| \le C_0$ by Lemma~\ref{lem:Bgrowth}(3).  
 Moreover, since $p>q$, we have   $n^{p-\frac{pj}{q}} \le n^{p-j}$  and  ${q \choose j} \le {p \choose j}$ for each $j$. Continuing from \eqref{eq:gamma}, we get  

$$|\hat{\lambda}|^q 
\ \le \   \sum_{j=0}^q {q \choose j} 
C_0^{q-j} n^{p-j} C_0^j
\ \le \ 
 C_0^q \sum_{j=0}^q {p \choose j}  n^{p-j} 
 \ \le \  C_0^q(n+1)^p 
$$
Finally, since $\hat{n}\ge n+1$ by \eqref{eq:nn_k}, we obtain $|\hat{\lambda}|^q \le C_0^q \hat{n}^p$, as desired.

\medskip
\noindent\emph{Case 2:} $m> 2p$ and $q \le i \le p$.  
 We have that for $K = (2p)^{p^2}$:
\begin{align*}
|\gamma| 
& \ \le \  (p+1) {m\choose {p-i}} &&\text{by Lemma~\ref{lem:Bgrowth}(2)}\\
& \ \le \ (p+1) {m\choose {p-q}} && \text{as } p-i \le p-q \le p \text{ and }m > 2p \\
& \ \le \ (p+1) K  {m\choose {q}}^{\frac{p-q}{q}} && \text{by Lemma~\ref{lem:binomialbound}(1), as } m> 2p \text{ and } q, p-q \le p\\
& \ \le \  C_0 n^{\frac{p}{q}-1} && \text{by } \eqref{eq:nlower}. 
\end{align*}

Then, continuing from \eqref{eq:gamma}, and using  that $\hat{n}\ge n+1$ by \eqref{eq:nn_k} and  that ${q \choose j} \le {p \choose j}$ for each $j$, we get
$$
|\hat{\lambda}|^q   \ \le \   \sum_{j=0}^q {q \choose j} 
C_0^{q-j} n^{p-\frac{pj}{q}} ( C_0n^{\frac{p}{q}-1})^j
 \ \le \  C_0^{q} \sum_{j=0}^q {p \choose j} 
  n^{p-j}  \ \le \  C_0^q (n+1)^p \  \le \ 
C_0^q\hat{n}^p.
$$

\medskip
\noindent\emph{Case 3:} $m> 2p$ and $1 \le i <q$. 
In this case, $|\gamma|_q={m \choose q-i} $ by Lemma~\ref{lem:Bgrowth}\eqref{binomial1} and 
\begin{align*}
|\gamma| 
& \ \le \  (p+1) {m \choose p-i} &&\text{by Lemma~\ref{lem:Bgrowth}\eqref{binomial2}}\\
& \ \le \  (p+1) K {m \choose p-q} {m \choose q-i} &&\text{by Lemma~\ref{lem:binomialbound}\eqref{more binomial2}, where $K = (2p)^{p^2}$}\\
& \ \le \ (p+1)K^2 {m \choose q}^{\frac{p-q}{q}}{m \choose q-i} && \text{by Lemma~\ref{lem:binomialbound}\eqref{more binomial1}, as } m> 2p \text{ and } \\
&&&1 \leq q, p-q \le p\\
& \ \le \ 
C_0 n^{\frac{p-q}{q}} |\gamma|_q && 
\text{by } \eqref{eq:nlower}, K = (2p)^{p^2}, \text{ and } |\gamma|_q={m \choose q-i}. 
\end{align*}

Then, continuing from \eqref{eq:gamma}, we have
\begin{align*}
|\hat{\lambda}|^q 
& \ \le \   \sum_{j=0}^q {q \choose j} 
C_0^{q-j} n^{p-\frac{pj}{q}} \left(C_0n^{\frac {p-q}{q}} |\gamma|_q \right)^j \\
& \ \le \  C_0^q \sum_{j=0}^q {p \choose j} n^{p-j} |\gamma|_q ^j \\
&\ \le \  
 C_0^q \left(n+ |\gamma|_q\right)^p \\
& \ \le \  C_0^q \hat{n}^p,
\end{align*}
where the last inequality follows from \eqref{eq:nn_k}.

This concludes the proof of inductive step, as    
$|\hat{\lambda}| \le C_0 \hat{n}^{p/q}$ in all three cases.
\end{proof}

\section{Leveraging our groups}  \label{ch:leveraging}

\subsection{Iterated exponential functions} \label{sec:general k}
 
Recall that $\exp^k$ denotes the $k$-fold iterated exponential-function. More precisely,
$\exp^1(x)= \exp(x)$ and  $\exp^i(x) = \exp(\exp^{i-1} (x))$ for integers $i > 1$.   
Here we will leverage our examples $H \leq G$ from Section~\ref{sec:the defn} to  construct free subgroups of hyperbolic groups whose distortion functions are $\simeq$-equivalent to $n \mapsto \exp^k(n^{p/q})$, where $p>q\ge 1$ and $k>1$ are integers, proving Theorem~\ref{main}.  We will take iterated amalgamated products of $G$  with certain hyperbolic free-by-free groups constructed by Brady and Tran \cite{BrT}.  We begin by reviewing the parts of their construction we need.  We write $F_m$ to denote the free group on $m$ generators.   

\begin{thm}\cite[Theorem 5.2]{BrT} \label{BT1}
Given $m \ge 1$, there exists $l >m$ and a group $F_l \rtimes F_m$ that is $\CAT(0)$ and hyperbolic.   
\end{thm}

\begin{definition}\label{BT2}
 Let $G_1$ be a finitely generated group and let $F_{m_1} < G_1$ be a free subgroup of rank $m_1$. Take $m_1< m_2 <  \cdots$ so that 
$F_{m_{i+1}} \rtimes F_{m_{i}}$ is the group of Theorem~\ref{BT1} (with $m_{i+1}=l$ and $m_i=m$).  
For $i>1$, define $G_i$ by
$$
 G_i  \ = \  (F_{m_{i}} \rtimes F_{m_{i-1}} ) \ast_{F_{m_{i-1}}}  G_{i-1}. 
$$
\end{definition}

  \begin{prop}\cite[Proposition 4.4]{BrT}\label{BT3}
In the notation of Definition~\ref{BT2}, if
  $\Dist_{F_{m_1}}^{G_1} \simeq f$ for some  non-decreasing superadditive function $f$, then for all integers $k \ge 1$, 
 $$
 \Dist_{F_{m_{k}}}^{G_k} (n) \ \simeq \  \exp^{k-1}(f(n)).
 $$
  \end{prop}

To complete the proof of Theorem~\ref{main}, we will take $G_1$ and $F_{m_1}$ to be our groups $G$ and $H \cong F_3$, respectively, from Section~\ref{sec:the defn}. 
We will then use the following two results to conclude that $G_k$ is hyperbolic when $k >1$.  
 
 \begin{thm} \textbf{(Hyperbolicity of amalgams)} \label{thm:combination}
If a finitely generated group $C$ is a subgroup of two hyperbolic groups $A$ and $B$, and $C$ is quasi-convex and malnormal in $A$, then $$\Gamma \ = \   A \ast_{C} B$$ is hyperbolic.  (We make no assumption on how $C$ sits in $B$.) 
\end{thm}

\begin{proof}
 Since $C$ is finitely generated and is  quasi-convex and malnormal in the hyperbolic group $A$, \cite[Theorem 7.11]{Bowditch2012} tells us that $A$ is hyperbolic relative to $C$. We then get that $\Gamma$ is hyperbolic relative to $B$ by \cite[Theorem 0.1(2)]{Dahmani2003}. A group that is hyperbolic relative to a hyperbolic subgroup is itself hyperbolic by \cite[Corollary 2.41]{Osin2006}.  So $\Gamma$ is hyperbolic.    
  \end{proof}

  \begin{lemma}\label{lem:malnornal qc}
 If $A$ and $B$ are finitely generated free groups and $G = A \rtimes  B$ is a hyperbolic group,   then $B$ is quasiconvex and malnormal in $G$.  
 \end{lemma}

\begin{proof}
For quasiconvexity, observe that $B$ is a retract of $G$, so it is in fact convex in $G$ (with respect to standard generating sets).  

To see that $B$ is malnormal, recall that the group $G$ can be identified with the Cartesian product $A\times B$ endowed with the multiplication $(a, b)(c, d) = (a \varphi_b(c), bd)$, where $\varphi_b(x) = bxb^{-1}$ for all $x \in A$.  
Note that for all $(c,d)\in G$ we have $(c, d)^{-1} = (\varphi_{d^{-1}}( c^{-1}), d^{-1})$. We identify $B$ with $\{1\} \times B$.

Now if $B$ is not malnormal, then there exists  some $(c, d) \in G\setminus B$ such that $(c, d)^{-1} B (c,d) \cap B$ is non-trivial.  Thus, there exists $b\in B$ with $b\neq 1$, such that 
\begin{align*}
(c, d)^{-1} (1, b) (c,d)  \   &= \  (\varphi_{d^{-1}}(c^{-1}), d^{-1}) (\varphi_b(c), bd)  \\  
&= \  (\varphi_{d^{-1}}(c^{-1}) 
\varphi_{d^{-1}}( \varphi_b(c)), d^{-1}bd) \in B.
\end{align*}

In particular, we must have $1 = \varphi_{d^{-1}}(c^{-1}) 
\varphi_{d^{-1}}( \varphi_b(c)) =\varphi_{d^{-1}}(c^{-1} \varphi_b(c))$, and since $\varphi_{d^{-1}}$ is an automorphism, we have $c^{-1} \varphi_b(c)=1$, or equivalently $c^{-1}bcb^{-1} = 1$.  
Observe that $c \neq 1$ as  $(c, d) \in G\setminus B$.  So $b$ and $c$ are commuting elements of infinite order (since $A$ and $B$ are free and inject into $G$) in a hyperbolic group, a contradiction.  We conclude that $B$ is malnormal. 
\end{proof}

\begin{proof}[Proof of Theorem~\ref{main}]
Given $p>q\ge 1$, let $G$ and $H$ be the groups we constructed in Section~\ref{sec:the defn} and proved in Sections~\ref{sec:lower}--\ref{sec:p/q} to have $\Dist_H^G (n) \simeq \exp({n^{p/q}})$.
Define  $G_1=G$ and $m_1 =3$, so that $F_{m_1} = F_3 \cong H$.  Define the groups  $G_k$ for $k>1$ as in Definition~\ref{BT2}.  Then, since $G_1$ is hyperbolic, and $F_{m_{i+1}} \rtimes F_{m_i}$ is hyperbolic for each $i$, we inductively conclude that each $G_i$ is hyperbolic, using Theorem~\ref{thm:combination} and Lemma~\ref{lem:malnornal qc}.   Finally, since 
$\exp(n^{p/q})$ is a non-decreasing superadditive function, Proposition~\ref{BT3} implies 
$$
 \Dist_{F_{m_{k}}}^{G_k} (n) \ \simeq \  \exp^k(n^{p/q}),
$$
as desired. 
\end{proof}

\begin{remark} \label{lem:CAT0 and CAT-1 structures}  \textbf{($\CAT(0)$ and $\CAT(-1)$ structures for the groups $G_k$})
For all $p>q\ge 1$, our group $G$ of Section~\ref{sec:the defn} satisfies a uniform $C'(1/6)$ condition, so can be given a $\CAT(0)$ structure by~\cite{WiseCubulating} or even a $\CAT(-1)$ structure by~\cite{BrownS, Gromov7, Martin}.  
The $F_l \rtimes F_m$ groups constructed by Brady--Tran have a piecewise Euclidean $\CAT(0)$ structure and furthermore, $F_m$ is 
\emph{ultra-convex} in $F_l \rtimes F_m$---a property they use to show that if the Gromov link condition holds in the complex associated to a group $\Gamma$, then it continues to hold for  an amalgamated product of the form $(F_l \rtimes F_m) \ast_{F_m} \Gamma$. 
See~\cite[Lemma 5.10]{BrT} for the precise statement. 
Moreover, the strategy  used in~\cite{BrownS, Gromov7} to obtain $\CAT(-1)$ structures  
by changing  each Euclidean polygon to a hyperbolic one by slightly shrinking each angle  can be applied to the Brady--Tran groups to obtain $\CAT(-1)$  groups for the form $F_l\rtimes F_m$. Thus, we expect that by choosing $\CAT(0)$ or $\CAT(-1)$ structures on the building blocks and using the ultra-convexity as in~\cite{BrT}, the groups $G_k$ in Definition~\ref{BT2} can be shown to be $\CAT(0)$ or $\CAT(-1)$ for all $k$. 
\end{remark}

 \subsection{Distortion of hyperbolic subgroups of hyperbolic groups} \label{sec:Realizing others}

Here we use ideas originating in  I.~Kapovich's \cite{Kap99} to prove Theorem~\ref{thm:distorted hyp grp}, which, in particular, extends our main result (Theorem~\ref{main}) in that it allows the distorted subgroup $H$ to be any  non-elementary torsion-free hyperbolic group rather than $F_3$.

For each of the   functions $f$ listed in Theorem~\ref{thm:distorted hyp grp}, there are constructions in the literature consisting of a hyperbolic group $K$ and a finite-rank free group $F \leq K$ such that $\Dist_F^K \simeq f$: see~\cite{Mitra, BBD} for \eqref{thm part:p/q} when $p=q$, this article for \eqref{thm part:p/q} when $p>q$, and \cite{BDR} for \eqref{thm part:Ackermann}.    We will prove the theorem by amalgamating $H$ with $K$ along a subgroup of $H$ that is isomorphic to $F$ and is supplied by the following lemma.

\begin{lemma} \label{lem:there are free subgroups}
Suppose $H$ is a non-elementary torsion-free  hyperbolic group.   For all $k \geq  2$,   $H$ contains  a malnormal quasiconvex  free subgroup $F$ of rank $k$.
\end{lemma}

\begin{proof}
 Kapovich showed that such an $H$ has a malnormal  quasiconvex rank-2 free subgroup $F(x,y)$ \cite[Theorem~C]{Kap99}.    There are malnormal rank-3 free subgroups in  $F(x,y)$---for example $$\langle x^{10}, y^{10}, (xy)^{10} \rangle$$   is malnormal by the criterion of \cite[Theorem~10.9]{KM02}, which can be interpreted as being that there is no reduced word which read from two different vertices in the Stallings graph of the subgroup makes a loop.  Likewise, for all $k \geq 2$, for sufficiently large $n$,  the subgroup $$\left\langle \, x^{n}, \ y^{n}, \ (xy)^{n},  \ (x^2y^2)^{n}, \ \ldots, \   (x^{k-2}y^{k-2})^{n} \, \rule{0mm}{4mm} \right\rangle$$ of $F(x,y)$ is malnormal and  rank-$k$.   The result then follows from the following three facts.    If  $A \leq B \leq C$ are groups such that $A$ is malnormal in $B$ and   $B$ is malnormal in $C$, then   $A$ is malnormal in $C$.  Quasiconvexity is similarly transitive.  Finitely generated  subgroups  of $F_2$ are quasiconvex.   
\end{proof}

Now, given $H$ and $f$ as in Theorem~\ref{thm:distorted hyp grp}, let $F \leq K$ be as above so that $K$ is hyperbolic, $F$ is finite-rank free, and $\Dist_F^K \simeq f$.  By Lemma~\ref{lem:there are free subgroups}, $H$ has  a quasiconvex malnormal subgroup which is isomorphic to $F$.  We will also refer  this subgroup of $H$ as $F$, so that we can define
\begin{equation}\label{eq:G def}
G \ = \  H \ast_F K.
\end{equation}
The last ingredient we require for Theorem~\ref{thm:distorted hyp grp} is:

\begin{thm} \label{thm:amalgam distortion} 
Let $\Gamma= A \ast_C B$, where $A, B,$ and $C$ are finitely generated groups and let  $f$ be a superadditive function such that $n \le f(n)$ for all $n$. 
\begin{enumerate}
\item If $\Dist_C^A  \preceq f$ and $\Dist_C^B \preceq f$, then 
$\Dist_A^\Gamma  \preceq f$ and $\Dist_B^\Gamma  \preceq f$. \label{lem part: dist1}
\item  If $\Dist_C^A (n)\simeq n $ and $\Dist_C^B \simeq f$,  then $\Dist_A^\Gamma \simeq  f$. \label{lem part: dist2}
\end{enumerate}
\end{thm}

\begin{proof}[Proof of Theorem~\ref{thm:distorted hyp grp} assuming Theorem~\ref{thm:amalgam distortion}]
Given $H$ and $f$ as in the theorem, let $G$ be the group defined in~\eqref{eq:G def}. Since $F$ is malnormal and quasiconvex in $H$,  Theorem~\ref{thm:combination} tells us that $G$ is hyperbolic.
Now $\Dist_F^K \simeq f$ by construction, and note that every function $f$ listed in the statement of Theorem~\ref{thm:distorted hyp grp} is superadditive and superlinear. 
  Since $F$ is quasiconvex in $H$ and $H$ is hyperbolic, we have  $\Dist_F^H (n)\simeq n \preceq f(n)$, and Theorem~\ref{thm:amalgam distortion}\eqref{lem part: dist2}  implies that $\Dist_H^G \simeq f$.
\end{proof}
\begin{proof}[Proof of Theorem~\ref{thm:amalgam distortion}] We begin with some setup. 
For $X= A, B, C$, let $S_X$ be a generating set for $X$, and let $K_X$ be a $K(X, 1)$ with 1-skeleton a rose on $|S_X|$ petals.  We assume that $S_C \subset S_A$ and  $S_C \subset S_B$.   Then $\Gamma $ is generated by $S_\Gamma = S_A \cup S_B$. 
Let $K$ be the standard graph of spaces with fundamental group $\Gamma$, i.e., 
$$
K \ = \  (K_A \sqcup(K_C\times [0, 1 ] )\sqcup K_B)/\sim
$$
where $\sim$ identifies $K_C \times \{0\}$ and $K_C \times \{1\}$ with the images of the maps induced by the inclusion of $C$ in $A$ and $B$ respectively.     
For convenience, we subdivide the cell structure so that $K_C^{(1)}\times \{1/2\} \subset K^{(1)}$. 

Let $c$ be the unique vertex of $K_C$, and let $p =  \{c\} \times \{1/2\} \in K_C\times [0, 1] \subset K$.  We identify $\Gamma$ with  $\pi_1(K, p)$. More precisely, identify $S_C$ with 
the set of petals of $K_C^{(1)} \times \{1/2\}$
and  $S_A \setminus S_C$ with the collection of loops $\delta \alpha \bar \delta$, where $\delta$ and $\bar \delta$ are the interval 
$\{c\} \times [0, 1/2] \subset K_C \times [0,1]$ oriented towards and away from $K_A$ respectively, and $\alpha$ is a petal of $K_A^{(1)}$ representing an element of $S_A \setminus S_C$.  Identify $S_B \setminus S_C$ with the analogous set of loops, replacing $\{c\} \times [0, 1/2]$ with $\{c\} \times [1/2, 1]$.  
Let $\mathcal S_\Gamma$ be the set of the loops defined in this paragraph. 
Each element of $\mathcal S_\Gamma$ is contained in $K^{(1)}$. 

The associated Bass--Serre tree is obtained by collapsing each lift of $K_A$ or $K_B$ in $\widetilde K$ to a vertex (called the $A$- and $B$- vertices, respectively) and each lift of $K_C \times [0,1]$ to an edge.  We subdivide each edge by adding a midpoint, obtained by collapsing a lift of $K_C \times \{1/2\}$; we call each such midpoint a $C$-vertex. Let $T$ denote this subdivided tree, and let 
$\psi: \widetilde K \to T$ denote the collapsing map.  Given an $A$- or $B$- vertex $v$ of $T$,
define $s_v$ to be the star of $v$ in $T$.  Since $T$ is subdivided, every vertex of $s_v$ besides $v$ is a $C$-vertex.

 If $\gamma \in \mathcal S_\Gamma$ corresponds to $g \in S_\Gamma$, 
 then each lift $\tilde \gamma$ of $\gamma$ in $\widetilde K^{(1)}$ is considered to be labelled by $g$. 
 By construction, the image of $\psi \circ \tilde \gamma$ is a $C$-vertex if $g \in S_C$, and otherwise it is contained in a star $s_v$ for an  
 $A$- or $B$-vertex $v$.   More generally, if $w$ is any word over $S_\Gamma$, then for each lift $\tilde p$ of $p$, there is a path $\xi_w$ starting at $\tilde p$ with label $w$ in $\widetilde K^{(1)}$.  (We abuse notation by  suppressing $\tilde p$.)
 
 Now if, in addition, $w=1$ in $\Gamma$, then $\xi_w$ is a loop based at some (any) $\tilde p$ and  $\psi\circ \xi_w$ is a loop based at $\psi(\tilde p)$ in $T$.  
The image of $\psi\circ \xi_w$ is a subtree of $T$, which we denote $\tau_w$.  
We measure the complexity of $w$  by $n(w)$, which counts the number of $A$- or $B$-vertex stars intersecting $\tau_w$: 
$$n(w)  \ = \  \# \{ v \, \mid \, v \text{ is an } A\text{- or }B\text{-vertex and } s_v \cap \tau_w  \neq \emptyset \}. 
$$ 
Note that $n(w)$ is finite since $\tau_w$ is compact, and $n(w) \ge 1$ as $\psi(\tilde p) \in \tau_w$.

We are now ready to prove \eqref{lem part: dist1}.  In this proof, a \emph{geodesic word in $X$} or \emph{over $S_X$} will mean a word of minimal length over $S_X$ representing an element of $X$, where $X = A, B, $ or $\Gamma$. 
Let $u$ be a geodesic word in either $A$ or $B$ and let $w$ be a geodesic word in $\Gamma$ with 
$u^{-1}w =1$ in $\Gamma$.
We wish to show that $|u| \le f(|w|)$.  The proof is by induction on $n(u^{-1}w)$.

If $n(u^{-1}w)=1$, then $\tau_{u^{-1}w}$ is contained in some $s_v$, where $v$ is an $A$- or $B$-vertex, depending on whether $u$ is in $A$ or $B$. We assume without loss of generality that $v$ is an $A$-vertex.  By construction, $s_v = \psi(Y)$, where $Y \subset \widetilde K$ consists of some lift of $K_A$, and all the lifts of $K_C \times [0, 1/2]$ intersecting it. 
Now $\xi_{u^{-1}w}$ is contained $Y^{(1)}$ and it follows that 
its label $u^{-1}w$ is a word over $S_A$. Thus $u$ and $w$ are both geodesics over  $S_A$ representing the same element of $A$, so $|u| = |w|$.  This proves the base step of the induction.

For the induction step, assume that $|u'| \le f(|w'|)$ whenever $u'^{-1}w'=1$  with $u'$ a geodesic in $A$ or $B$
and   $w'$ a geodesic in $\Gamma$  and $n(u'^{-1}w') < n (u^{-1}w)$.   Again, assume 
without loss of generality that $u$ is a geodesic in $A$.  
Write $\xi_{u^{-1}w}$ as a concatenation $ \xi_{u^{-1}} \xi_w$.   Then $\psi(\xi_{u^{-1}}) \subset s_a$ for some $A$-vertex $a$ (since $u$ is a geodesic over $S_A$). 
Now, by considering $\psi^{-1} (\tau_{u^{-1}w}\setminus s_a^{\circ})$, where $s_a^{\circ}$ denotes the interior of $s_a$, we obtain a concatenation $\xi_w = \xi_{x_0} \xi_{y_1} \xi_{x_1} \cdots \xi_{y_k} \xi_{x_k}$  (so $w =  x_0 y_1 x_1  \cdots  y_k x_k$, as words),
such that
for each $i$, we have that $\psi(\xi_{x_i}) \subset s_a$  (so $x_i$ is a word over $S_A$) and that $\psi \circ \xi_{y_i} $  is a loop in $\tau_{u^{-1}w}\setminus s_a^{\circ}$ based at a $C$-vertex $p_i$ of $s_a$.  

By construction, each  $\xi_{y_i}$ has its endpoints in some lift of $K_C \times \{1/2\}$, and so $y_i$ represents an element of $C$, and therefore of $B$.   Let 
$z_i$ be a geodesic word over $S_B$ with $z_i = y_i$ in $\Gamma$, and let $\xi_{z_i}$ be the path in $\widetilde K$ with the same endpoints as $\xi_{y_i}$.  Then  $\psi(\xi_{z_i}) \subset s_{b_i}$ where $b_i$ is the unique $B$-vertex adjacent to $p_i$.  Now  consider $\xi_{z_i^{-1} y_i} = \overline \xi_{z_i} \xi_{y_i}$ and note that $\psi( \xi_{y_i})$ intersects $s_{b_i}$, since the endpoints of $\xi_{y_i}$ map to $p_i$.  It follows that $\tau_{z_i^{-1}y_i}$  intersects the same number of $A$- and $B$-vertex stars as $\psi ( \xi_{y_i})$,  
and, by construction, this number is less than $n(u^{-1}w)$ (since $\tau_{u^{-1}w}$ intersects the additional vertex star $s_a$).   
So $n(z_i^{-1} y_i) < n(u^{-1}w)$.  Since $y_i$ is a geodesic (being a subword of a geodesic) over $S_\Gamma$, we may apply the induction hypothesis to conclude that 
$|z_i| \le f(|y_i|)$.  
Moreover, in $\Gamma$ we have $u = w = x_0z_1x_1 \cdots z_k x_k$ (as elements). 
So the facts that $u$ is a geodesic and that  $n \le f(n)$ combined with 
the superadditivity of $f$ yield: 

$$\begin{aligned}
|u| \ & \le \  \sum_{i=0}^k|x_i| + \sum_{i=1}^k|z_i| \ \le \  \sum_{i=0}^k|x_i| + \sum_{i=1}^kf(|y_i|) \  \\ &  \le  \  f\left( 
\sum_{i=0}^k|x_i| + \sum_{i=1}^k|y_i|  \right) \  =   \  f(|w|).
\end{aligned}$$
This completes the induction step and proves \eqref{lem part: dist1}.  The bound $\Dist_A^\Gamma (n) \preceq f(n)$ of \eqref{lem part: dist2} immediately follows.  

For the reverse bound in \eqref{lem part: dist2}, by the definition  of $\Dist_C^B$, there exist for each $n \ge 1$, geodesic words $u_n$ and $w_n$ over $S_C$ and $S_B$, respectively, with  $u_n = w_n$ in $\Gamma$, such that  $|w_n|\le n$ and 
$|u_n| = \Dist_C^B(n)$.  Since $u_n$ is an element of $C$, it is also an element of $A$.  Let $v_n$ be 
a geodesic word over $S_A$ representing $u_n$.  Since $C$ is undistorted in $A$, there exists a 
constant $K\ge 1$ such that $|u_n| \le K |v_n|$.  Then, for each $n$, we have found a geodesic word 
$v_n$ in $A$ which represents the same element as the word  $w_n$ over $S_\Gamma$ of length at 
most $n$, and  
$|v_n| \ge \frac{1}{K} |u_n| = \frac{1}{K} \Dist_C^B(n)$.
It follows that $\Dist_A^\Gamma (n) \succeq \Dist_C^B (n)$.  Combined with the hypothesis $\Dist_C^B (n) \simeq f(n)$, this gives  $\Dist_A^\Gamma (n) \succeq f(n)$, which  completes our proof of \eqref{lem part: dist2}.  
\end{proof}

 \section{Height}  \label{ch:Height}   
 
 \subsection{Why our examples have infinite height}  \label{sec:Height}

An infinite subgroup $H$ of a group $G$ has \emph{infinite height} when, for all $n$, there exist $g_1, \ldots, g_n \in G$ such that $\bigcap_{i=1}^n {g_i}^{-1} H g_i$ is infinite and $H g_i \neq  H g_j$ for all $i \neq j$.  Otherwise it has \emph{finite height}.  New constructions of non-quasiconvex subgroups of hyperbolic groups are natural test cases for 
this longstanding question attributed to Swarup in \cite{Mitra2}:  if   a finitely presented subgroup $H$ of a hyperbolic group $G$  has \emph{finite height}, is $H$ necessarily quasiconvex in $G$?

So we note here that our examples do not speak to Swarup's question:   
\begin{prop}
If $H$ is the non-quasiconvex subgroup of the hyperbolic group $G$ we construct to prove  Theorem~\ref{main}  or, more generally, to prove Theorem~\ref{thm:distorted hyp grp} in case \eqref{thm part:p/q}  with $p>q$,    then  $H$ has infinite height.
\end{prop}

\begin{proof}
Consider $\Gamma= F(t, x_1, x_2, y_1, y_2) \HNN_{a_1,a_2}$ with the HNN-structure from Proposition~\ref{prop:hnn}, the defining relators being those specified by the $r_{4, \ast}$-cells of Figure~\ref{fig:relations}.  

We first show that $F=F(t, y_1, y_2)$ has infinite height in $ \Gamma$.  
 It is evident from the defining relators for $\Gamma$ that $a_{1}^{-1} F a_{1} \subset  F$.  So, for $i=0, 1,  \ldots$, we define $g_i = a_1^i$, and conclude that $g_{i+1}^{-1} F g_{i+1} \subset g_i^{-1} F g_i$.   Then,  for all $n \geq 0$, we have  $\bigcap_{i=1}^n g_i^{-1} F g_i =  g_n^{-1} F g_n$, which is a non-trivial subgroup of the free group $F$ and so is infinite.  And  $F g_i \neq  F g_j$ for all $i \neq j$ because, by virtue of the HNN-structure of $\Gamma$,  we find that $a_1^k \in F$  only when $k=0$.  So $F$ has infinite height in $\Gamma$. 

For the $G$ of Section~\ref{sec:the defn} constructed to prove Theorem~\ref{main} when $k=1$, we have 
$H =F( t, y_1, y_2) =F < \Gamma < G$ as a consequence of the HNN structure discussed in Section~\ref{sec:hnn}. When $k>1$, the same is true because of the graph of groups structure of Definition~\ref{BT2}.  Since $H$ has infinite height in $\Gamma$, it has infinite height in $G$ also.  

For the groups $G$ we  constructed to prove Theorem~\ref{thm:distorted hyp grp}\eqref{thm part:p/q}  when $p>q$, we have $G = H \ast_F K$, where  $K$ is one of the groups we constructed to prove Theorem~\ref{main}.  So $F<H$ and $F<\Gamma<K$. 
Moreover, the amalgamated product structure implies that $a_1^k \in H$ only when $k=0$, so, using the same group elements $g_i$  as before,  $H g_i \neq  H g_j$ when $i \neq j$.  And, for all $n \ge 0$, 
$$
g_n^{-1} F g_n  \ = \  \bigcap_{i=1}^n g_i^{-1} F g_i  \ \subset  \ \bigcap_{i=1}^n g_i^{-1} H g_i. 
$$
As $g_n^{-1} F g_n$ is infinite, we conclude that $H$ has infinite height. 
\end{proof}

\bibliographystyle{alpha}
\bibliography{bibli}



\end{document}